\documentclass[]{report}

\usepackage{anysize}
\usepackage{xspace}
\usepackage{pdfsync}
\usepackage{amsmath}
\usepackage{fncylab}
\usepackage[usenames]{color}
\usepackage{amsfonts}
\usepackage{amssymb}  
\usepackage{amsthm} 
\usepackage{amsmath} 
\usepackage{caption,subcaption}
\usepackage{etoolbox}
\usepackage{stmaryrd} 
\usepackage[dvipsnames]{xcolor}
  \definecolor{editcolour}{rgb}{0.7,0.1,0}
  \definecolor{hrefcolour}{rgb}{0,0,0.7}
\usepackage{graphicx}
\usepackage{mathtools}
\usepackage[T1]{fontenc}
\usepackage{pdfpages}

\usepackage{tikz}
\usetikzlibrary{
  matrix,
  arrows,
  shapes,
  decorations.markings,
  decorations.pathreplacing,
  patterns,
  decorations.pathmorphing
}

\tikzset{
  cd/.style={
    ->,
    scale=6,
    >=angle 90,
    font=\scriptsize}
  }

\tikzset{
  graph/.style={
    ->,
    scale=2,
    >=triangle 45,
    font=\scriptsize}
}

\tikzset{
  ob/.style={
    shape=circle,
    draw,
    thick,
    inner sep=2.75
  }
}  

\tikzset{
  -|->/.style={
    decoration={
      markings,
      mark=at position .5 with {\arrow{|}},
      mark=at position 1 with {\arrow{>}}
    },
    postaction={decorate}
  }
}

\tikzset{
  every loop/.style={
    in=60,
    out=120,
    looseness=10
  }
}

\definecolor{rewritecolor}{rgb}{0,.9,1}

\tikzset{
  zxgreen/.style={
    shape=circle,
    draw,
    thick,
    fill=green
  }
}
\tikzset{
  zxred/.style={
    shape=circle,
    draw,
    thick,
    fill=red
  }
}
\tikzset{
  zxyellow/.style={
    shape=rectangle,
    draw,
    thick,
    fill=yellow
  }
}
\tikzset{
  zxblack/.style={
    shape=diamond,
    fill=black,
    inner sep=2.75
  }
}
\tikzset{
  zxwhite/.style={
    shape=circle,
    draw,
    thick
  }
}

\usepackage[]{hyperref} 
  \hypersetup{
    colorlinks,
    linkcolor={hrefcolour},
    citecolor={hrefcolour},
    urlcolor={hrefcolour}
  }
\usepackage[left=1.5in,top=1.5in,right=1in,bottom=1in]{geometry}
\usepackage{todonotes}

\renewcommand{\epsilon}{\varepsilon}
\newcommand{\op}{^{\scriptsize{ \textrm{op} } }}
\newcommand{\iso}{\cong}

\newcommand{\bydef}{\coloneqq}
\newcommand{\hcirc}{\circ_{\textup{h}}}
\newcommand{\vcirc}{\circ_{\textup{v}}}
\renewcommand{\hat}{\widehat}

\newcommand{\A}{\cat{A}}

\newcommand{\C}{\cat{C}}
\newcommand{\D}{\cat{D}}

\newcommand{\M}{\cat{M}}

\renewcommand{\S}{\cat{S}}

\newcommand{\T}{\cat{T}}

\newcommand{\X}{\cat{X}}

\newcommand{\BB}{\bicat{B}}
\newcommand{\CC}{\bicat{C}}

\newcommand{\CCC}{\dblcat{C}}

\newcommand{\RRR}{\dblcat{R}}

\newcommand{\FinSet}{\cat{FinSet}}
\newcommand{\Set}{\cat{Set}}
\newcommand{\Rel}{\cat{Rel}}
\newcommand{\RRel}{\bicat{Rel}}
\newcommand{\FinHilb}{\cat{FinHilb}}
\newcommand{\Ab}{\cat{Ab}}
\newcommand{\Vect}{\cat{Vect}}
\newcommand{\Mod}{\cat{Mod}}
\newcommand{\Pos}{\cat{Pos}}
\newcommand{\FinGraph}{\cat{FinGraph}}
\newcommand{\Graph}{\cat{Graph}}
\newcommand{\RGraph}{\cat{RGraph}}

\newcommand{\Cat}{\cat{Cat}}
\newcommand{\CCat}{\bicat{Cat}}

\newcommand{\DblCat}{\cat{DblCat}}

\newcommand{\Cospan}{\cat{Csp}}

\newcommand{\Gram}{\cat{Gram}}
\newcommand{\StrCsp}{\cat{StrCsp}}

\newcommand{\SSStrCsp}{\dblcat{S} \bicat{trCsp}}
\newcommand{\StrCspGram}{\cat{StrCspGram}}
\newcommand{\BBoldRewrite}{\bicat{BoldRewrite}}
\newcommand{\BBBoldRewrite}{
  \dblcat{B}\bicat{old}\dblcat{R}\bicat{ewrite}}
\newcommand{\FFineRewrite}{\bicat{FineRewrite}}
\newcommand{\FFFineRewrite}{
  \dblcat{F}\bicat{ine}\dblcat{R}\bicat{ewrite}}
\newcommand{\ZX}{\cat{ZX}}
\newcommand{\ZZX}{\bicat{ZX}}
\newcommand{\ZZZX}{\dblcat{ZX}}
\newcommand{\FinGraphGamma}{\FinGraph \downarrow \Gamma}
\newcommand{\AdjTopos}{\cat{AdjTopos}}

\newcommand{\Lang}{\mathrm{Lang}}

\newcommand{\defn}[1]{\textbf{\textup{#1}}}
\newcommand{\cat}[1]{\mathsf{#1}}
\newcommand{\bicat}[1]{\mathbf{#1}}
\newcommand{\dblcat}[1]{\mathbb{#1}}

\newcommand{\from}{\colon}
\newcommand{\rel}{\nrightarrow}
\newcommand{\To}{\Rightarrow}
\newcommand{\xto}[1]{\xrightarrow{#1}}
\newcommand{\monicto}{\rightarrowtail}
\newcommand{\dderiv}[2]{#1 \rightsquigarrow #2}
\newcommand{\deriv}[2]{#1 \rightsquigarrow^\ast #2}
\renewcommand{\gets}{\leftarrow}

\newcommand{\xgets}[1]{\xleftarrow{#1}}
\newcommand{\spn}[3]{#1 \gets #2 \to #3}
\newcommand{\xspn}[5]{#1 \xgets{#2} #3 \xto{#4} #5}
\newcommand{\csp}[3]{#1 \to #2 \gets #3}

\newcommand{\lrto}{\rightleftarrows}

\DeclareMathOperator{\id}{id}

\DeclareMathOperator{\Sub}{Sub}
\DeclareMathOperator{\colim}{colim}
\DeclareMathOperator{\ob}{ob}
\DeclareMathOperator{\arr}{arr}

\DeclareMathOperator{\decat}{decat}

\theoremstyle{plain}
\newtheorem{theorem}{Theorem}

\newtheorem{corollary}[theorem]{Corollary}

\newtheorem{definition}[theorem]{Definition}
\newtheorem{example}[theorem]{Example}

\newtheorem{lemma}[theorem]{Lemma}

\newtheorem{proposition}[theorem]{Proposition}
\newtheorem{remark}[theorem]{Remark}

\renewenvironment{proof}[1][Proof]{\textbf{#1.} }{\ \rule{0.5em}{0.5em}}

\newlength\mylen
\newcommand{\hto}{%
  \to\kern-0.55\mylen\vline height 1.2ex depth
  -0.4pt\kern0.55\mylen}

\newcommand{\adjunction}[5]{%
  \begin{tikzpicture}
    \node (1) at (0,0) {\( #1 \)};
    \node (2) at (#5,0) {\( #2 \)};
    \draw [cd]
      (1) edge[bend left] node[above]{$ #3 $} (2)
      (2) edge[bend left] node[below]{$ #4 $} (1);
    \path (1.center) -- node[]{$ \bot $} (2.center);
  \end{tikzpicture}
}

\begin{document}



## Acknowledgments

First and foremost, I would like to thank my advisor John Baez. In these past few years, I have learned more than I could have imagined about mathematics and the job of doing mathematics. I also want to thank the past and current Baez Crew for the many wonderful discussions. I am indebted to Math Department at the University of California, Riverside, which has afforded me numerous opportunities to travel to conferences near and far.

Almost certainly, I would never have had a chance to pursue my doctorate had it not been for my parents who were there for me through every twist and turn on this, perhaps, too scenic route that I traveled.

Most importantly, this project would have been impossible without the full-hearted support of my love, Elizabeth.

Elizabeth. It's finally over, baby!

ABSTRACT OF THE DISSERTATION

Rewriting Structured Cospans: A Syntax For Open Systems

by

Daniel Cicala

Doctor of Philosophy, Graduate Program in Mathematics
University of California, Riverside, June 2019
Dr. John C. Baez, Chairperson

The concept of a system has proliferated through natural and social sciences. While myriad theories of systems exist, there is no mathematical general theory of systems. In this thesis, we take a first step towards formulating such a theory. Our focus is on developing a syntax for compositional systems equipped with a rewriting theory. We pull from category theory and linguistics to accomplish this. The basic syntactical unit is a structured cospan and rewriting is introduced via the double pushout method. Two versions of rewriting are proposed: one that tracks intermediate steps and another disregards them. Benefits and drawbacks of both versions are discussed. We apply our results to the decomposition of closed systems, obtaining a structurally inductive viewpoint of rewriting such systems.



\chapter{Introduction}
\label{sec:introduction}


Systems exist everywhere and there are many different
languages used to describe them.  The diversity of languages
reflect those who study systems. Physicists, chemists,
biologists, ecologists, economists, sociologists, linguists,
mathematicians, computer scientists all work with systems
and all have their own idiosyncratic methods to describe
them.  This parallels diversity in the natural languages
where location and communication needs are but two factors
contributing to a language's development.  

Just as linguists glean knowledge about humanity from
studying languages, we can glean knowledge about our world
from studying languages of systems. Still, no fully
general mathematical theory of systems exists. Should it?

We say `yes'. To develop a fully general mathematical theory
of systems is a worthy pursuit. Successfully creating a
formal language of systems can bestow many gifts. For
instance, with a better understanding of systems, engineers
get a better toolkit for their designs.  One such engineered
system, the power grid---a keystone to our way of life---is
vulnerable due to increased energy demands inflicted by
climate change
\cite{mukherjee-nateghi_electricity-demand}. A better
understanding of systems eases translation across
disciplines.  By placing, say, systems ecology
\cite{odum-ecology} and the programming language \emph{R}
\cite{mailund_functional-in-r} in the same formalism,
ecological models can be more faithfully translated into
mathematical models. A better understanding of systems
directs us to new paths of inquiry. An abstract
understanding of systems places them into a ``space of
systems'' where they can be compared and contrasted. With
this space, we can craft analogies and narratives. This new
perspective should present questions previously not
apparent. So yes, aspiring to a general mathematical theory
of systems is worthwhile.

Often, one studies \emph{a} system. The social network
described by Facebook is a single system frequently studied.
Another is the logistics of shipping Amazon packages the
world over.  In reality, systems rarely exist in isolation.
The Facebook network is affected by other social media
networks.  Amazon's shipping networks are affected by the
economics of oil prices.  That is, systems interact with
each other to form new systems and this ought to be a
component of an honest general mathematical systems
theory. One way systems interact is to not exert any
influence over each other, which should evoke to a
mathematician the disjoint union operation. But to exert
influence necessitates each system to have points on which
the interaction can occur. For example, a point of
interaction of a building's electrical system is an outlet,
where one can connect a blender forming a composite
electrical-blender system. A point of interaction with a
pulley system is a dangling rope that one can pull, upon
which we obtain the composite pulley-musculoskeletal system.

When connecting systems together, one may veer
into the \emph{principle of
  compositionality}. Compositionality is present
when the whole of a system is equal to the sum of
its parts. This can be exploited to great effect
when analyzing complicated systems by allowing for
its decomposition into simpler pieces. For
instance, the physical system of two pendulums
connected together with a spring (see Figure
\ref{fig:pendulum}) can be fully analyzed by
separately considering the two pendulums and the
spring. In mathematical terms, this amounts to
coupling the corresponding differential equations.

Compositionality lies in contrast to so-called
\emph{emergent systems} where new features burst
into existence upon connection. Life is believed
to have emerged from complex systems of
ribonucleic acid (better known as RNA). No sign of
life is present in a single RNA molecule but somehow life
appears in a system comprising only RNA.

\begin{figure}
  \centering
  \fbox{
    \begin{minipage}{\linewidth}
      \centering
      \begin{tikzpicture}[thick,>=latex,->]
        \begin{scope} 
          \draw[decoration={
            aspect=0.3,
            segment length=1.5mm,
            amplitude=3mm,coil},
          decorate]
          (3.2,-3) -- (6.9,-3);
        \end{scope}
        \begin{scope} 
          \clip(-5,2) rectangle (5,-5);
          \draw[double distance=1.6mm]
          (0,0) -- (3,-3);
          \draw[fill=white]
          (-1.2,1.0) -- (-.5,0)
          arc(180:360:0.5) -- (1.2,1.0) -- cycle;
          \draw[draw=black,fill=white]
          (0, 0) circle circle (.3cm);
          \draw[draw=black,fill=white]
          (3,-3) circle circle (.3cm);
          \draw[pattern=north east lines]
          (-1.4,1.3) rectangle (1.4,1);
        \end{scope}
        \begin{scope}[shift={(8,0)}] 
          \clip(-5,2) rectangle (5,-5);
          \draw[double distance=1.6mm]
          (0,0) -- (-1,-3);
          \draw[fill=white]
          (-1.2,1.0) -- (-.5,0)
          arc(180:360:0.5) -- (1.2,1.0) -- cycle;
          \draw[draw=black,fill=white]
          (0, 0) circle circle (.3cm);
          \draw[draw=black,fill=white]
          (-1,-3) circle circle (.3cm);
          \draw[pattern=north east lines]
          (-1.4,1.3) rectangle (1.4,1);
        \end{scope}
      \end{tikzpicture}
      \caption{A compositional physical system}
      \label{fig:pendulum}
    \end{minipage}
  }
\end{figure}

The two methods of interaction described above, disjoint
union and connecting along points of interaction, have clear
analogies to fundamental mathematical concepts: addition and
composition. From the many areas of mathematics, the one
that stands out in its singular focus on addition and
composition is the theory of monoidal categories. Category
theory takes as fundamental the composition of `arrows' and
endowing a category with a `monoidal structure' allows us to
``add'' the arrows together. Therefore, monoidal categories
are an excellent foundation on which to base a general
mathematical theory of compositional systems.

\section*{What is this thesis about?}
\label{sec:what-this-thesis}

Here, we take first steps in towards building a theory of
compositional systems. What do these first steps look like?
In short, we are setting up a syntax for compositional
systems.

The term `syntax' appears most often in linguistics where it
refers to rules and principles that an arrangement of
words must satisfy to be a well-formed sentence.
It means roughly the same for us except that we are working
with compositional systems, not words and sentences. In this
analogy, compositional systems correspond to both words and
sentences in that, instead of building sentences by
arranging words, we are building larger systems by
connecting smaller systems.  To do so, we need a set of
rules and principals governing how to connect systems
together.

The yin to syntax's yang is \emph{semantics}. This concept,
also from linguistics, refers to the meaning of a sentence.
In our context, semantics refers to the \emph{behavior} of a
system. Resistor circuits are a nice example to highlight
the distinction between syntax and semantics.  First, recall
that resistors wired in series have the same resistance as a
single resistor with the aggregate resistance.  Now, while a
circuit with a $ 25 \Omega $ and $ 35 \Omega $ resistors
wired in series is syntactically different from a circuit
with a single $ 60 \Omega $ resistor, their resistance is
equal meaning they have the same semantics.  While semantics
is important to any theory of systems, we do not directly
consider it in this thesis. However, we do consider it
indirectly.

Granting that syntax and semantics are separate entities, it
is often useful for syntax to \emph{reflect} semantics. We
do not want to say that the two resistor circuits are
\emph{equal}.  That is too strong. But we do want to
establish a formal \emph{relationship} between them.  More
than that, we want a way to propagate this relationship
through a suitable space of circuits so that every circuit
with resistors wired in series relates to the circuit with a
single resistor in their place. Of course, our method of
propagating such a relationship must be abstract enough to
handle more systems than just resistor circuits.

Again we turn to linguistics, this time the study of formal
languages. These are different from natural languages like
English, Italian, or Afrikaans that ebb and flow under so
many social forces.  Formal languages are designed and can
be controlled. They can approximate natural languages. This
makes them useful in studying natural languages.  However,
the ``formal languages'' we are interested in do not contain
words and sentences. The formal languages we are interested
in are systems connected together.

From the study of formal languages comes \emph{rewriting
  theory}. Originally used to generate well-formed
sentences, rewriting has since evolved through being studied
by mathematicians, logicians, and computer scientists for
whom it provides a mechanism to replace terms with distinct
but equivalent terms. As mentioned above, rewriting is
syntactic but meant to reflect semantics.  This means that
rewriting relates syntactical terms if they behave in the
same way. For example, a programming language that can
perform addition would have a `rewrite rule' saying that
`2+2' can be rewritten into `4' because they mean the same
thing.  There would \emph{not} be a rule rewriting `2+2'
into `5' because they never mean the same thing. Moreover,
rewriting theory provides a way to extend this rule to
longer strings containing `2+2', for instance, the string
`(3*(2+2))/(2+2+3)' can be rewritten into
`(3*4)/(4+3)'. Crucially, rewriting also prevents erroneous
applications such as rewriting `2+2(x+y)' into `4(x+y)'.
The first expansion of rewriting theory beyond the realm of
characters and words was into combinatorial graphs where
rewrite rules tell us when one graph can replace another. If
we were modeling the internet as a directed graph with
websites as nodes and a link from one website to another
as edges, then we are likely uninterested in self-loops,
which represent a webpage that links to itself.  So we can
introduce a rule that deletes self-loops. Informally, this
would say that the graph
\begin{center}
  \begin{tikzpicture}
    \node (a) at (0,0) {$ \bullet $};
    \draw [graph]
      (a) edge[loop above] (a);
    \draw [rounded corners] (-1,-1) rectangle (1,2);
  \end{tikzpicture}
\end{center}
can be rewritten into the graph
\begin{center}
  \begin{tikzpicture}
    \node (a) at (0,0) {$ \bullet $};
    \draw [rounded corners] (-1,-1) rectangle (1,1);
  \end{tikzpicture}
\end{center}
This rule can be extended to remove loops from more
complicated graphs like
\begin{center}
  \begin{tikzpicture}
    \node (a) at (0,0) {$ \bullet $};
    \node (b) at (2,0) {$ \bullet $};
    \node (c) at (1,2) {$ \bullet $};
    \draw [graph]
      (a) edge[loop left] (a)
      (b) edge[loop right] (b)
      (c) edge[loop above] (c)
      (a) edge[] (b)
      (b) edge[] (c);
    \draw [rounded corners] (-2,-1) rectangle (4,4);
  \end{tikzpicture}
\end{center}
being rewritten into
\begin{center}
  \begin{tikzpicture}
    \node (a) at (0,0) {$ \bullet $};
    \node (b) at (2,0) {$ \bullet $};
    \node (c) at (1,2) {$ \bullet $};
    \draw [graph]
      (a) edge[] (b)
      (b) edge[] (c);
    \draw [rounded corners] (-1,-1) rectangle (3,3);
  \end{tikzpicture}
\end{center}
To formalize this requires abstract mathematics, namely
category theory. Fortunately, because the category theory
involved in rewriting graphs is so abstract, we can use
it to rewrite syntax developed for compositional
systems.

What does rewriting do for us?  It allows us to
simplify our syntax, whether that syntax is based on
characters or combinatorial graphs or other types of systems. 
The ability to simplify syntax is a powerful tool for any
would-be analyst simply because of how complex syntactical
terms can grow.  The graph model of the internet is massive
with over 1.5 billion nodes, each an individual website.

Our goal in this thesis is to present a syntax for
compositional systems proposed by Baez and Courser
\cite{baez-courser_str-csp} called `structured cospans' and
combine it with a theory of rewriting.

\section*{A road map for the thesis}
\label{sec:road-map}

The larger goal of creating a general mathematical theory
for compositional systems is still aspirational, but we
stride within these several chapters, developing a syntax
and rewriting theory.  To assist the reader in navigating
these chapters, we sketch their contents and give
the highlights. We visualize the dependencies between
the chapters with Figure \ref{fig:chapter-dependency}.

In Chapter \ref{sec:structured-cospans}, we present a syntax
for compositional systems.  Baez and Courser introduced this
syntax under the name `structured cospans'.  A cospan is a
diagram in a category with shape
\begin{center}
  \begin{tikzpicture}
    \node (0) at (0,0) {$ a $};
    \node (1) at (2,0) {$ b $};
    \node (2) at (4,0) {$ c $};
    \draw [cd] 
      (0) edge node[above]{$ f $} (1)
      (2) edge node[above]{$ g $} (1); 
  \end{tikzpicture}
\end{center}
where $ a,b,c $ are objects in the category and $ f,g $ are
arrows in the category. For a structured cospan, we have a
specific interpretation in mind: the object $ b $ is a
system with inputs $ a $ and outputs $ c $.  The arrows $ f
$ and $ g $ maps the inputs and outputs to the system.

To formalize this perspective, our starting data is an
adjunction
\[
  \adjunction{\A}{\X}{L}{R}{2}
\]
between topoi $ \A $ and $ \X $.  We interpret $ \A $ as a
topos whose objects are the \emph{interface types}; that is
the objects that can serve as inputs or outputs to our
systems, and $ \X $ as a topos whose objects are the system
types. Often, $ \A $ is the topos $ \Set $ of sets and
functions. And $ \X $ can be whatever system we are working
with, for example a category whose objects are resistor
circuits.  The functor $ L \from \A \to \X $ translates the
interface types into degenerate system types so that they can
interact via a structured cospan, which is a cospan of the
form
\begin{center}
  \begin{tikzpicture}
    \node (0) at (0,0) {$ La $};
    \node (1) at (2,0) {$ x $};
    \node (2) at (4,0) {$ Lb $};
    \draw [cd] 
      (0) edge node[above]{$ f $} (1)
      (2) edge node[above]{$ g $} (1); 
  \end{tikzpicture}
\end{center}
This structured cospan is a system $ x $ with inputs $ La $
and outputs $ Lb $. A resistor circuit as a structured
cospan would look like
\begin{center}
\begin{tikzpicture}
    \begin{scope}
    \node at (0,0) {$ \bullet a $};
    \node at (0,2) {$ \bullet b $};
    \draw [rounded corners] (-1,-1) rectangle (1,3);
    \node (l) at (1.1,1) {};  
    \end{scope}
    \begin{scope}[shift={(3,0)}]
    \node                        (1)  at (0,0) {$a\bullet$};
    \node                        (2)  at (0,2) {$b\bullet$};
    \node [rectangle,draw=black] (3)  at (1,2) {$10\Omega$};
    \node                        (3') at (2,2) {};
    \node [rectangle,draw=black] (4)  at (1,0) {$5\Omega$};
    \node                        (4') at (2,0) {};
    \node                        (5)  at (2,1) {{$ \bullet $}};
    \node [rectangle,draw=black] (6)  at (3,1) {$15\Omega$};
    \node                        (7)  at (4,1) {$\bullet c$};
    \draw  
    (1.center) -- (4)
    (2.center) -- (3)
    (3) -- (3'.center) -- (5.center)
    (4) -- (4'.center) -- (5.center)
    (5.center) -- (6)
    (6) -- (7.center);
    \draw [rounded corners] (-1,-1) rectangle (5,3);
    \node (ml) at (-1.1,1) {};
    \node (mr) at (5.1,1)  {};
    \end{scope}
    \begin{scope}[shift={(10,0)}]
    \node at (0,1) {$ \bullet c $};
    \draw [rounded corners] (-1,-1) rectangle (1,3);
    \node (r) at (-1.1,1) {};
    \end{scope}
    \draw [cd] 
    (l) edge[] (ml)
    (r) edge[] (mr);
  \end{tikzpicture}
\end{center}
The left-hand graph $ L ( \{ a,b \} ) $ gives the inputs and
the right-hand graph $ L ( \{c\} ) $ gives the outputs.  

We devote Section \ref{sec:StrCsp-as-Arrows} to composing
structured cospans. As is standard in cospan categories,
composition uses pushout.  For example, any resistor circuit
with a single input, say
\begin{center}
\begin{tikzpicture}
    \begin{scope}
    \node at (0,1) {$ \bullet c $};
    \draw [rounded corners] (-1,-1) rectangle (1,3);
    \node (l) at (1.1,1) {};  
    \end{scope}
    \begin{scope}[shift={(3,0)}]
    \node                        (1) at (0,1) {$c\bullet$};
    \node [rectangle,draw=black] (2) at (1,1) {$5\Omega$};
    \node                        (3) at (2,1) {{$ \bullet d $}};
    \draw (1.center) -- (2) -- (3.center);
    \draw [rounded corners] (-1,-1) rectangle (3,3);
    \node (ml) at (-1.1,1) {};
    \node (mr) at (3.1,1)  {};
    \end{scope}
    \begin{scope}[shift={(8,0)}]
    \node at (0,1) {$ \bullet d $};
    \draw [rounded corners] (-1,-1) rectangle (1,3);
    \node (r) at (-1.1,1) {};
    \end{scope}
    \draw [cd] 
      (l) edge[] (ml)
      (r) edge[] (mr);
  \end{tikzpicture}
\end{center}
can be connected to the resistor circuit above that has a
single output as follows
\begin{center}
\resizebox{1.0\linewidth}{!}{
\begin{tikzpicture}
    \begin{scope} 
    \node at (0,0) {$ \bullet a $};
    \node at (0,2) {$ \bullet b  $};
    \draw [rounded corners] (-1,-1) rectangle (1,3);
    \node (l) at (1.1,1) {};  
    \end{scope}
    \begin{scope}[shift={(3,1)}] 
    \node                        (1)  at (0,0) {$a\bullet$};
    \node                        (2)  at (0,2) {$b\bullet$};
    \node [rectangle,draw=black] (3)  at (1,2) {$10\Omega$};
    \node                        (3') at (2,2) {$  $};
    \node [rectangle,draw=black] (4)  at (1,0) {$5\Omega$};
    \node                        (4') at (2,0) {$  $};
    \node                        (5)  at (2,1) {{$\bullet$}};
    \node [rectangle,draw=black] (6)  at (3,1) {$15\Omega$};
    \node                        (7)  at (4,1) {$\bullet c$};
    \draw  
      (1)   -- (4)
      (2)   -- (3)
      (3) -- (3'.center) -- (5.center)
      (4) -- (4'.center) -- (5.center)
      (5.center)   -- (6)
      (6)   -- (7.center);
    \draw [rounded corners] (-1,-1) rectangle (5,3);
    \node (ml) at (-1.1,1) {};
    \node (mr) at (5.1,1)  {};
    \end{scope}
    \begin{scope}[shift={(10,0)}] 
    \node at (0,1) {$ \bullet c $};
    \draw [rounded corners] (-1,-1) rectangle (1,3);
    \node (r) at (-1.1,1) {};
    \node (l') at (1.1,1) {};
    \end{scope}
    \begin{scope}[shift={(13,1)}] 
    \node                        (1) at (0,1) {$c\bullet$};
    \node [rectangle,draw=black] (2) at (1,1) {$5\Omega$};
    \node                        (3) at (2,1) {$ \bullet d $};
    \draw (1.center) -- (2) -- (3.center);
    \draw [rounded corners] (-1,-1) rectangle (3,3);
    \node (ml') at (-1.1,1) {};
    \node (mr') at (3.1,1)  {};
    \end{scope}
    \begin{scope}[shift={(18,0)}] 
    \node at (0,1) {$ \bullet d $};
    \draw [rounded corners] (-1,-1) rectangle (1,3);
    \node (r') at (-1.1,1) {};
    \end{scope}
    \draw [cd] 
      (l') edge[] (ml')
      (r') edge[] (mr')
      (l) edge[] (ml)
      (r) edge[] (mr);
  \end{tikzpicture}
  }  
\end{center}
We then pushout over the common interface
\begin{center}
  \begin{tikzpicture}
    \node at (0,1) {$ \bullet c $};
    \draw [rounded corners] (-1,-1) rectangle (1,3);
  \end{tikzpicture}
\end{center}
to get the single structured cospan
\begin{center}
\begin{tikzpicture}
    \begin{scope} 
    \node at (0,0) {$ \bullet a $};
    \node at (0,2) {$ \bullet b  $};
    \draw [rounded corners] (-1,-1) rectangle (1,3);
    \node (l) at (1.1,1) {};  
    \end{scope}
    \begin{scope}[shift={(3,0)}] 
    \node                        (1) at (0,0) {$a\bullet$};
    \node                        (2) at (0,2) {$b\bullet$};
    \node [rectangle,draw=black] (3) at (1,2) {$10\Omega$};
    \node                        (3') at (2,2) {};
    \node [rectangle,draw=black] (4) at (1,0) {$5\Omega$};
    \node                        (4') at (2,0) {};
    \node                        (5) at (2,1) {{$ \bullet $}};
    \node [rectangle,draw=black] (6) at (3,1) {$15\Omega$};
    \node                        (7) at (4,1) {$\bullet$};
    \node [rectangle,draw=black] (8) at (5,1) {$5\Omega$};
    \node                        (9) at (6,1) {$\bullet d$};
    \draw  
      (1.center) -- (4)
      (2.center) -- (3)
      (3)        -- (3'.center) -- (5.center)
      (4)        -- (4'.center) -- (5.center)
      (5.center) -- (6)
      (6)        -- (7.center)
      (7.center) -- (8)
      (8)        -- (9.center);
    \draw [rounded corners] (-1,-1) rectangle (7,3);
    \node (ml) at (-1.1,1) {};
    \node (mr) at (7.1,1)  {};
    \end{scope}
    \begin{scope}[shift={(12,0)}] 
    \node at (0,1) {$ \bullet d $};
    \draw [rounded corners] (-1,-1) rectangle (1,3);
    \node (r) at (-1.1,1) {};
    \end{scope}
    \draw [cd] 
      (l) edge[] (ml)
      (r) edge[] (mr);
  \end{tikzpicture}
\end{center}
that represents a single circuit with input nodes $ a,b $
and output node $ d $.

Starting with the adjunction
$ L \from \A \lrto \X \from R $, where $ \A $ and $ \X $ are
symmetric monoidal categories with their respective
coproducts, we then package structured cospans into a
compact closed category $ ( _L \Cospan, \otimes, 0_A ) $
whose objects are the interface types, that is objects of
$ \A $, and the arrows of type $ a \to b $ are the
structured cospans $ \csp{La}{x}{Lb} $.

Our stated goal is to introduce a rewriting theory to
structured cospans. To do this, we must ensure that
structured cospans are sufficiently nice to accommodate
rewriting. This entails designing a topos where structured
cospans are the objects. Constructing this topos is the
topic of Section \ref{sec:StrCspAsObject}.  We define a
category $ _L\StrCsp $ whose objects are structured cospans
and whose arrows between the structured cospans
$ \csp{La}{x}{Lb} $ and $ \csp{La'}{x'}{Lb'} $ are commuting
diagrams
\begin{center}
  \begin{tikzpicture}
    \node (La)  at (0,2) {$ La $};
    \node (x)   at (2,2) {$ x $};
    \node (Lb)  at (4,2) {$ Lb $};
    \node (La') at (0,0) {$ La' $};
    \node (x')  at (2,0) {$ x' $};
    \node (Lb') at (4,0) {$ Lb' $};
    \draw [cd] 
    (La)  edge[] node[above]{$  $} (x)
    (Lb)  edge[] node[above]{$  $} (x)
    (La') edge[] node[below]{$  $} (x')
    (Lb') edge[] node[below]{$  $} (x')
    (La)  edge[] node[left]{$ Lf $} (La')
    (x)   edge[] node[left]{$ h $} (x')
    (Lb)  edge[] node[right]{$ g $} (Lb'); 
  \end{tikzpicture}
\end{center}
in $ \X $.  The main result of this section is

\textbf{Theorem \ref{thm:strcsp-istopos}.}  For any
adjunction
\[
  \adjunction{\A}{\X}{L}{R}{2}
\]
between topoi, the category $ _L\StrCsp $ is a topos.

This result is the keystone that stabilizes the combination
of structured cospans and rewriting. Because of this fact,
structured cospans do accommodate a rewriting theory. By
this, we mean that the local Church--Rosser and concurrency
properties hold.  We do not investigate these properties in
this thesis, but Corradini, et.~al.~thoroughly discuss these
properties \cite{corradini-ehrig_algebraic-graph-grammars}.
We also show in Theorem \ref{thm:strcsp-isfunctorial} that
constructing $ _L\StrCsp $ is functorial in $ L $.

Viewing structured cospans through the two categories
$ _L \Cospan $ and $ _L \StrCsp $ in which they appear, we
note that they play two roles. In $ _L \Cospan $, structured
cospans form the arrows. In $ _L\StrCsp $, structured cospans
form the objects.  We combine these two perspectives into a
single framework using double categories in Section
\ref{sec:DblCatOfStrCsp}. The final section of Chapter
\ref{sec:structured-cospans} sets the groundwork for
rewriting structured cospans by defining spans of structured
cospans.

In Chapter \ref{sec:dpo-rewriting}, we discuss the theory
of rewriting with just enough detail to provide the reader
with an appreciation for the subject and enough tools to
read this text. We begin with its linguistic beginnings but
quickly move to the axiomatization of the double pushout
method of rewriting. The axioms of rewriting theory are
captured in their full generality by so-called `adhesive
categories'.  However, this is too general for our needs, so
we restrict to rewriting in a topos, a type of adhesive
category.

By fixing a topos $ \T $, we learn how to apply a rewrite
rule, which manifests as a span
\[
  \spn{\ell}{k}{r}
\]
in $ T $. We interpret this rule to say $ \ell $ can be
rewritten into $ r $. We apply this rule by identifying a
copy of $ \ell $ inside another object $ \ell' $ via an
arrow $ \ell \to \ell' $ of $ \T $ and there are objects
$ k' $ and $ r' $ of $ \T $ fitting into a `double pushout
diagram'
\begin{center}
  \begin{tikzpicture}
    \node (ell) at (0,2) {$ \ell $};
    \node (k) at (2,2) {$ k $};
    \node (r) at (4,2) {$ r $};
    \node (ell') at (0,0) {$ \ell' $};
    \node (k') at (2,0) {$ k' $};
    \node (r') at (4,0) {$ r' $};
    \draw [cd] 
      (k) edge[] (ell)
      (k) edge[] (r)
      (k') edge[] (ell')
      (k') edge[] (r')
      (ell) edge[] (ell')
      (k) edge[] (k')
      (r) edge[] (r');
      \draw (0.3,0.4) -- (0.4,0.4) -- (0.4,0.3);
      \draw (3.7,0.4) -- (3.6,0.4) -- (3.6,0.3);
  \end{tikzpicture}
\end{center}
We then say that $ \ell' $ can be rewritten to $ r' $. The
double pushout diagram encodes that we first identify a copy
of $ \ell $ in $ \ell' $, remove and replace it by $ r $,
and this results in $ r' $.  In this way, an initial set of
rewrite rules propagate throughout $ \T $ by collecting all
possible applications of all the initial rules.

In Chapter \ref{sec:fine-rewriting}, we introduce the first
of two styles of rewriting structured cospans. A `fine
  rewrite rule' of structured cospans is a diagram with shape
\begin{center}
  \begin{tikzpicture}
    \node (La) at (0,4) {$ La $};
    \node (x) at (2,4) {$ x $};
    \node (La') at (4,4) {$ La' $};
    \node (Lb) at (0,2) {$ Lb $};
    \node (y) at (2,2) {$ y $};
    \node (Lb') at (4,2) {$ Lb' $};
    \node (Lc) at (0,0) {$ Lc $};
    \node (z) at (2,0) {$ z $};
    \node (Lc') at (4,0) {$ Lc' $};
    \draw [cd] 
      (La) edge[] (x)
      (La') edge[] (x) 
      (Lb) edge[] (y)
      (Lb') edge[] (y)
      (Lc) edge[] (z)
      (Lc') edge[] (z)
      (Lb) edge[] (La)
      (Lb) edge[] (Lc)
      (Lb) edge[] node[left]{$ \iso $} (La)
      (Lb) edge[] node[left]{$ \iso $} (Lc)
      (y) edge[>->] (x)
      (y) edge[>->] (z)
      (Lb') edge[] node[right]{$ \iso $} (La')
      (Lb') edge[] node[right]{$ \iso $} (Lc'); 
  \end{tikzpicture}
\end{center}
taken up to isomorphism. The marked arrows are monic and an
isomorphism to another fine rewrite of structured cospans
\begin{center}
  \begin{tikzpicture}
    \node (La) at (0,4) {$ La $};
    \node (x) at (2,4) {$ x $};
    \node (La') at (4,4) {$ La' $};
    \node (Lb) at (0,2) {$ Lb $};
    \node (y) at (2,2) {$ y' $};
    \node (Lb') at (4,2) {$ Lb' $};
    \node (Lc) at (0,0) {$ Lc $};
    \node (z) at (2,0) {$ z $};
    \node (Lc') at (4,0) {$ Lc' $};
    \draw [cd] 
      (La) edge[] (x)
      (La') edge[] (x) 
      (Lb) edge[] (y)
      (Lb') edge[] (y)
      (Lc) edge[] (z)
      (Lc') edge[] (z)
      (Lb) edge[] (La)
      (Lb) edge[] (Lc)
      (Lb) edge[] node[left]{$ \iso $} (La)
      (Lb) edge[] node[left]{$ \iso $} (Lc)
      (y) edge[>->] (x)
      (y) edge[>->] (z)
      (Lb') edge[] node[right]{$ \iso $} (La')
      (Lb') edge[] node[right]{$ \iso $} (Lc'); 
  \end{tikzpicture}
\end{center}
is an invertible arrow $ y \to y' $ such that the evident
diagrams commute. Admittedly, we are being rather brusque by
saying `evident', though Definition
\ref{def:morphism-span-str-cospans} spells this out in
detail.  The main result of this section is the construction
of a double category $ _L \FFFineRewrite $ whose objects are
interface types from $ \A $, horizontal arrows are
structured cospans, and squares are fine rewrites of
structured cospans.  This result is listed as Proposition
\ref{thm:fine-rewrite-double-cat}. Proving the interchange
law is quite technical, so we devote all of Section
\ref{sec:interchange-law} to this.  In Section
\ref{sec:sm-structure} we equip the double category
$ _L \FFFineRewrite $ with a symmetric monoidal structure.
In the final section of Chapter \ref{sec:fine-rewriting}, we
appease those readers who prefer bicategories to double
categories. There, we extract from the double category
$ _L \FFFineRewrite $ a compact closed bicategory
$ _L \FFineRewrite $.

In Chapter \ref{sec:bold-rewriting}, we introduce the
counterpart to fine rewriting called `bold rewriting'.  A
bold rewrite rule is the connected component of a diagram
\begin{center}
  \begin{tikzpicture}
    \node (La) at (0,4) {$ La $};
    \node (x) at (2,4) {$ x $};
    \node (La') at (4,4) {$ La' $};
    \node (Lb) at (0,2) {$ Lb $};
    \node (y) at (2,2) {$ y $};
    \node (Lb') at (4,2) {$ Lb' $};
    \node (Lc) at (0,0) {$ Lc $};
    \node (z) at (2,0) {$ z $};
    \node (Lc') at (4,0) {$ Lc' $};
    \draw [cd] 
      (La)  edge[] (x)
      (La') edge[] (x) 
      (Lb)  edge[] (y)
      (Lb') edge[] (y)
      (Lc)  edge[] (z)
      (Lc') edge[] (z)
      (Lb)  edge[] (La)
      (Lb)  edge[] (Lc)
      (Lb)  edge[] node[left]{$ \iso $} (La)
      (Lb)  edge[] node[left]{$ \iso $} (Lc)
      (y)   edge[] (x)
      (y)   edge[] (z)
      (Lb') edge[] node[right]{$ \iso $} (La')
      (Lb') edge[] node[right]{$ \iso $} (Lc'); 
  \end{tikzpicture}
\end{center}
By connected component, we mean the equivalence class
generated by relating the above diagram to
\begin{center}
  \begin{tikzpicture}
    \node (La) at (0,4) {$ La $};
    \node (x) at (2,4) {$ x $};
    \node (La') at (4,4) {$ La' $};
    \node (Lb) at (0,2) {$ Lb $};
    \node (y) at (2,2) {$ y' $};
    \node (Lb') at (4,2) {$ Lb' $};
    \node (Lc) at (0,0) {$ Lc $};
    \node (z) at (2,0) {$ z $};
    \node (Lc') at (4,0) {$ Lc' $};
    \draw [cd] 
      (La) edge[] (x)
      (La') edge[] (x) 
      (Lb) edge[] (y)
      (Lb') edge[] (y)
      (Lc) edge[] (z)
      (Lc') edge[] (z)
      (Lb) edge[] (La)
      (Lb) edge[] (Lc)
      (Lb) edge[] node[left]{$ \iso $} (La)
      (Lb) edge[] node[left]{$ \iso $} (Lc)
      (y) edge[->] (x)
      (y) edge[->] (z)
      (Lb') edge[] node[right]{$ \iso $} (La')
      (Lb') edge[] node[right]{$ \iso $} (Lc'); 
  \end{tikzpicture}
\end{center}
if there is an arrow $ y \to y' $ such that the evident
diagrams commute.  This chapter largely mirrors that on fine
rewriting.  We define a double category
$ _L \BBBoldRewrite $ whose objects are the interface types
from $ \A $, whose horizontal arrows are structured cospans,
and whose squares are bold rewrites.  Again, we extract a
bicategory from the double category. We show that this
bicategory $ _L \BBoldRewrite $ is a bicategory of
relations.

In the final section of Chapter \ref{sec:bold-rewriting}, we
illustrate bold rewriting with the ZX-calculus.  This is a
language consisting of string diagrams used to reason about
a corner of quantum mechanics favored by quantum computer
theorists. Coecke and Duncan, the inventors of the
ZX-calculus, organized it into a dagger compact category
whose arrows are the very diagrams that constitute the
ZX-calculus. Using the machinery laid out in this chapter,
we expand this dagger compact category to a symmetric
monoidal double category that encodes the ZX-calculus.  The
benefit of this is that, instead of merely equating
ZX-calculus diagrams when there exists a rewrite rule between
them, the squares of our double category actually witness
these equations. This should satisfy mathematical
constructivists.  Overall, the double category structure we
build is richer than the category.

We complete this thesis with Chapter
\ref{sec:structural-induction}. Most academic work on
systems focuses on \emph{closed systems}, those with an
empty interface.  Physicists often represent a closed system
with a phase space.  Chemical reactions are worked out as if
the rest of the world does not exist (or is reduced to a
triviality).  Petri nets do not interact with each other.
Markov chains are never combined.  One hope of this research
program is to provide the mathematical resources to change
this, so that \emph{open networks} become the norm. Then the
phase spaces of two different systems could be
connected. Chemical reactions could more easily consider their
environment. Petri nets and Markov chains could be composed
together. This final chapter motivates using open systems to
study closed systems.

Specifically, we construct a mechanism to rewrite closed
systems using structural induction.  That is, we can
decompose a given closed system into open sub-systems each of
which can be rewritten independently of each other. After
simplifying each sub-system via this rewriting procedure,
we reconnect them together into an equivalent version
of the original closed system. In short, we introduce an
inductive process that simplifies closed systems.  This is
characterized by the following theorem.

\textbf{Theorem \ref{thm:inductive-rewriting}.}  Fix a
adjunction $ L \from \A \lrto \X \from R $ with monic
counit. Let $ ( \X , P ) $ be a grammar such that for every
$ \X $-object $ x $ in the apex of a production of $ P $,
the Heyting algebra $ \Sub (x) $ is well-founded. Given
$ g $, $ h \in \X $, then $ \deriv{g}{h} $ in the rewriting
relation for a grammar $ ( \X , P ) $ if and only if there
is a square
\[
  \begin{tikzpicture}
      \node (1t) at (0,4) {$ LR 0 $};
      \node (2t) at (2,4) {$ g $};
      \node (3t) at (4,4) {$ LR 0 $};
      \node (1m) at (0,2) {$ LR 0 $};
      \node (2m) at (2,2) {$ d $};
      \node (3m) at (4,2) {$ LR 0 $};
      \node (1b) at (0,0) {$ LR 0 $};
      \node (2b) at (2,0) {$ h $};
      \node (3b) at (4,0) {$ LR 0 $};
      \draw [cd] (1t) to node [] {\scriptsize{$  $}} (2t);
      \draw [cd] (3t) to node [] {\scriptsize{$  $}} (2t);
      \draw [cd] (1m) to node [] {\scriptsize{$  $}} (2m);
      \draw [cd] (3m) to node [] {\scriptsize{$  $}} (2m);
      \draw [cd] (1b) to node [] {\scriptsize{$  $}} (2b);
      \draw [cd] (3b) to node [] {\scriptsize{$  $}} (2b);
      \draw [cd] (1m) to node [] {\scriptsize{$  $}} (1t);
      \draw [cd] (1m) to node [] {\scriptsize{$  $}} (1b);
      \draw [cd] (2m) to node [] {\scriptsize{$  $}} (2t);
      \draw [cd] (2m) to node [] {\scriptsize{$  $}} (2b);
      \draw [cd] (3m) to node [] {\scriptsize{$  $}} (3t);
      \draw [cd] (3m) to node [] {\scriptsize{$  $}} (3b);
    \end{tikzpicture}
  \]
in the double category $ \Lang ( _{L}\StrCsp , P' ) $.

In less technical terms, this theorem says that, under
suitable hypotheses, one closed system can be rewritten into
another precisely when there is a square between their
corresponding structured cospans. This square is built
inductively from rewrites between open sub-systems.   

This marks the end of the thesis proper. However, we
anticipate that the results contained within may be of
interest to a wide audience including certain network
theorists, systems theorists, computer scientists, and
mathematicians.  Therefore, we organized the thesis so that
background material is mostly confined to the appendices.
This way, it will not distract those familiar with it and it
is readily available to those readers who are not.  Here are
the topics of the appendices.
\begin{description}
\item[Appendix \ref{sec:extern-bicats}] 
  
  Enriched categories and bicategories. This material is
  used in Sections
  \ref{sec:compact-closed-bicategory-spans-of-cospans} and
  \ref{sec:cartesian-bicategory-spans-of-cospans} where
  bicategories are extracted from double categories;
  
\item[Appendix \ref{sec:internal-double-cats}] 

  Internalization and double categories are useful
  throughout as double categories are a main character in
  our story. Also, this section covers internal monoids which are
  used to show that the bicategory of bold rewrites
  $ _L \BBoldRewrite $ is a bicategory of relations in
  Section \ref{sec:cartesian-bicategory-spans-of-cospans};

\item[Appendix \ref{sec:cartesian-bicategories}] 

  Bicategories of relations, which are used in Section
  \ref{sec:cartesian-bicategory-spans-of-cospans};
  
\item[Appendix \ref{sec:duality-in-bicategories}] 
  
  Duality in bicategories, which is used for the
  bicategories in both Sections
  \ref{sec:compact-closed-bicategory-spans-of-cospans} and
  \ref{sec:cartesian-bicategory-spans-of-cospans};
    
\item[Appendix \ref{sec:adhesive-categories}] 

  Adhesive categories, which are the result of axiomatizing
  rewriting theory and, though useful throughout because of
  the central role played by rewriting in this thesis, we
  pack most of the required information into the next
  section of the appendix;
  
\item[Appendix \ref{sec:topoi}] 

  Topoi, which are used throughout.
  
\end{description}

\begin{figure}[h!]
  \centering
  \fbox{
    \begin{minipage}{1.0\linewidth}
      \centering
      \vspace{2em}
      \begin{tikzpicture}
        \node [rectangle,draw=black] (2) at (-3,4) {
          Ch.~1 Structured cospans};
        \node [rectangle,draw=black] (3) at (3,4) {
          Ch.~2 Double pushout rewriting};
        \node [rectangle,draw=black] (4) at (-5,2) {
          Ch.~3 Fine rewriting};
        \node [rectangle,draw=black] (5) at (5,2) {
          Ch.~4 Bold rewriting};
        \node [rectangle,draw=black] (6) at (0,0) {
          Ch.~5 Structural induction for rewriting};
        \draw[graph]
        (4) edge[shorten >=0.1cm] (6.west) 
        (2) edge[shorten >=0.1cm] (4)
        (2) edge[shorten >=0.3cm] (5)
        (3) edge[shorten >=0.3cm] (4)
        (3) edge[shorten >=0.1cm] (5)
        (2) edge[shorten >=0.1cm] (6)
        (3) edge[shorten >=0.1cm] (6);
      \end{tikzpicture}
      \caption{Chapter dependencies}
      \label{fig:chapter-dependency}
    \end{minipage}
  }
\end{figure}

\section*{Global notation and assumptions}
\label{sec:set-notation}

As usual in mathematics, we systematically select notation
to orient the reader.  Here, we lay out the logic behind our
notation.

\begin{description}
\item[Categorical structures]

  Three types of categorical structure are used throughout:
  \begin{itemize}
  \item

    Categories and topoi, which we denote with the font
    $ \A $, $ \X $, $ \C $, $ \T $.  $ \A $ and $ \X $
    are used are topoi used to build structured cospans, $
    \C $ is a generic category, and $ \T $ is a generic
    topos.
    
  \item

    Bicategories, which we denote with bold font $ \CC $.
    The two most important bicategories for us are $
    \FFineRewrite $ and $ \BBoldRewrite $.
    
  \item

    Double categories, which we denote with blackboard bold
    font $ \CCC $. The two most important double categories
    for us are $ \FFFineRewrite $ and $ \BBBoldRewrite $.
    
  \end{itemize}
  
\item[Objects]

  Objects in a category are denoted by lower case letters.
  The most common categories we work with are labeled as $
  \A $ and $ \X $ and we refer to their respective objects
  are $ a,b,c,\dotsc $ and $ \dotsc x,y,z $.
  
\item[Arrows]

  Both categories and graphs frequent these pages. To
  distinguish whether a drawing is of a graph or a diagram
  in a category, look at the arrow tips.  An arrow in a
  category uses
  \[
  \begin{tikzpicture}
    \draw [cd] (0,0) to (0.5,0);
  \end{tikzpicture}
  \]
  while an arrow in a graph uses
  \[
  \begin{tikzpicture}
    \draw [graph] (0,0) to (0.5,0);
  \end{tikzpicture}
  \]
  Also, we reserve tailed arrows 
  \[
  \begin{tikzpicture}
    \draw [cd,>->] (0,0) to (0.5,0);
  \end{tikzpicture}
  \] 
  to mean a monic arrow in a category.  We do not often
  refer to specific arrows, but when we do, we use lower
  case letters $ f,g,h, $ etc. Occasionally, if an arrow is
  of particular importance we distinguish it with a lower
  case Greek letter.
  
\item[2-arrows] 

  We refer to 2-arrows in higher categories using Greek
  letters.  In particular, when using $ \lambda $, $ \rho $,
  $ \alpha $ without explicitly stating what they are, then
  they are monoidal coherence maps for left unity, right unity, and
  associativity.  
  
\item[Rewrite relation]

  Central to the theory of rewriting is the `rewriting
  relation'.  This is built in two steps from a given rewriting
  system.  First, $ \dderiv{a}{b} $ says that $ a $ can be
  rewritten into $ b $ by applying a single rewrite rule.
  The rewriting relation, which we denote by $ \deriv{}{} $,
  is the reflexive and transitive closure of $ \dderiv{}{} $.

\item[Systems and networks] 

  Our work concerns both open and closed systems, the former
  more prominently. Therefore, when using the term system or
  network without a qualifier, we mean `open' by default.
  Only when we explicitly say `closed' do we mean a closed
  system or network. 

\item[Cospans of graphs]

  Many graph morphisms are drawn throughout the following
  pages. Too much detail tends to clutter the drawings, so
  we leverage the geometry of the page to suggest the
  definition of the morphisms. Only in cases where this
  suggestion lacks clarity do we explicitly spell out the
  meanings. In Chapter \ref{sec:dpo-rewriting}, we see the
  drawing
  \begin{center}
    \begin{tikzpicture}
    \draw [rounded corners] (-1,-1) rectangle (1,2);
    \node (a)  at (0,0)     {$ \bullet $};
    \node (a') at (0,1)     {$ \bullet $};
    \node (lr) at (1.1,0.5) {$  $};
    \draw [graph] (a) to (a');
    \draw [rounded corners] (2,-1) rectangle (4,2);
    \node (b)  at (3,0)     {$ \bullet $};
    \node (b') at (3,1)     {$ \bullet $};
    \node (ml) at (1.9,0.5) {$  $};
    \node (mr) at (4.1,0.5) {$  $};
    \draw [rounded corners] (5,-1) rectangle (7,2);
    \node (c)  at (6,0.5)   {$ \bullet $};
    \node (rl) at (4.9,0.5) {$  $};
    \draw [cd]
      (ml) edge (lr)
      (mr) edge (rl); 
  \end{tikzpicture}
  \end{center}
  which consists of three directed graphs each in a box and
  two graph morphisms. Note the differences between the
  arrow heads.  Also, the definitions of these graph
  morphisms are not explicitly spelled out, but they are
  apparent nonetheless because of the location of the graph
  nodes on the page.  
  
\end{description}


\chapter{Structured cospans}
\label{sec:structured-cospans}

Researchers traditionally study \emph{closed systems}, those
that lack the ability to interact with outside agents.  A
research program initiated by John Baez centralizes the
study of \emph{open systems}, those with ability to interact
with outside agents
\cite{
  baez-courser-markov,
  baez-fong-pollard_markov,
  baez-linear-networks,
  baez-pollard-reaction-networks}.

In this thesis, our primary example of an open system is an
\emph{open graphs}.  We use them throughout to illustrate
new concepts and definitions. For this reason, we start with
a set theoretical definition of open graphs and modify our
understanding of them in parallel to building our structured
cospan formalism. We use this approach to provide a concrete
example to ground us through the development of our
theory. Open graphs are not new
\cite{dixon-duncan-kiss_open-graphs,gadducci_ind-graph-transf},
but our structured cospan perspective is new.

\begin{definition}[Open and closed graphs]
\label{def:open-graph-informal}
  An \defn{open graph} $ G \bydef (E,N,r,s,t,I,O) $ is a
  directed reflexive multi-graph $ (E,N,r,s,t) $ equipped
  with two non-empty subsets $ I,O \subseteq N $ of
  nodes. We call elements of $ I $ the inputs of the graph
  and the elements of $ O $ the outputs of the graph. In the
  case that $ I $ and $ O $ are empty, then we call $ G $ a
  \defn{closed graph}.
\end{definition}

This definition deserves several remarks. First, note that a
closed graph is simply a graph in the classical sense.  We
append the qualifier `closed' to highlight the fact that it
has no inputs or outputs. Second, the terms `input' and
`output' \emph{do not} imply causal structure or
directionality.  Finally, the author prefers reflexive
graphs to non-reflexive graphs because \emph{(i)} they are
truncated simplicial sets so have nicer topological
features, \emph{(ii)} unlike graphs, the ``points'' of
reflexive graphs (the nodes) correspond to maps from the
terminal object, and \emph{(iii)} the category of reflexive
graphs $ \RGraph $ is monadic over $ \Set $.

An open graph is illustrated in Figure
\ref{fig:open-graph-informal}. We suppress the reflexive
loops in drawing reflexive graphs. In that figure, the nodes
are $a$, $b$, $c$, $d$, $e$, and $f$. The input nodes are
$ a $, $ b $, and $ d $. The output nodes are $ c $ and
$ d $.

\begin{figure}
  \centering
  \fbox{
  \begin{minipage}{1.0\linewidth}
  \[
    \begin{tikzpicture}
      \node (1) at (0,0) {{$ _a \bullet $}};
      \node (2) at (2,0) {{$ \bullet_b $}};
      \node (3) at (1,1) {{$ _c \bullet $}};
      \node (4) at (0,2) {{$ _d \bullet $}};
      \node (5) at (2,2) {{$ \bullet_e $}};
      \node (6) at (1,3) {{$ \bullet_f $}};
      \draw [graph]
      (1) edge (2)
      (2) edge (3)
      (5) edge (3)
      (6) edge (4)
      (4) edge (5)
      (4) edge (1); 
      \node () at (4,2.5) {$ a,b,d \in I $};
      \node () at (4,0.5) {$ c,d \in O  $};
    \end{tikzpicture}
  \]
  \caption{An open graph}
  \label{fig:open-graph-informal}
  \end{minipage}
}\end{figure}

A non-exhaustive list of other systems of interest to Baez's
research program are Petri nets
\cite{master-open-petri-nets}, Markov processes
\cite{baez-courser-markov}, passive linear circuits
\cite{baez-linear-networks}, reaction networks
\cite{baez-pollard-reaction-networks}, the ZX-calculus
\cite{cicala-zx}. See Figure \ref{fig:various-networks} for
depictions of these various systems. These systems are
traditionally studied as closed systems. To ``open'' them,
they need an interface along which compatible systems can be
connected. This is the purpose of introducing the input and
output nodes.

\begin{figure}
  \centering
  \fbox{
  \begin{minipage}{\linewidth}
  \begin{minipage}{0.45\linewidth}
    \centering
    \includegraphics[scale=0.3]{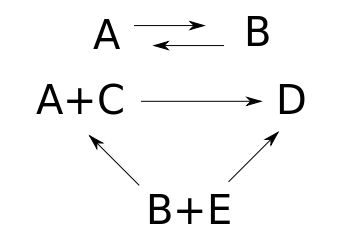}\\
    \textsc{Chemical Reaction Network}
  \end{minipage}
  \begin{minipage}{0.45\linewidth}
    \centering
    \includegraphics[scale=0.1]{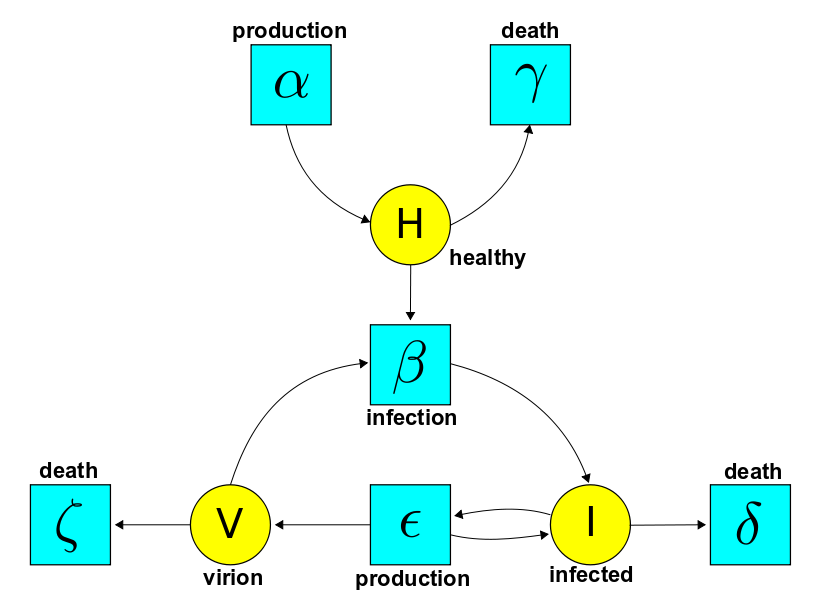}\\
    \textsc{Petri Net}
  \end{minipage}
  \linebreak
  \begin{minipage}{0.45\linewidth}
    \centering
    \includegraphics[scale=0.1]{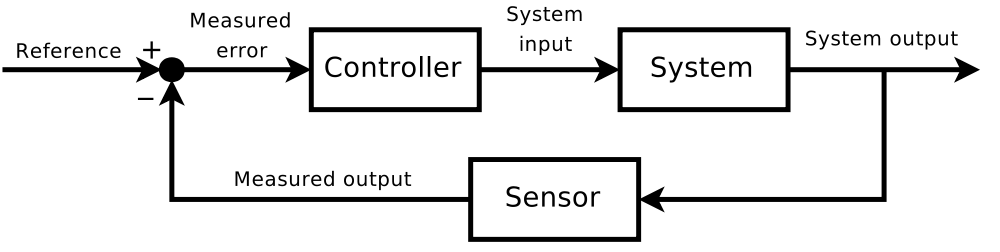}\\
    \textsc{Control Network}
  \end{minipage}
  \begin{minipage}{0.45\linewidth}
    \centering
    \includegraphics[scale=0.1]{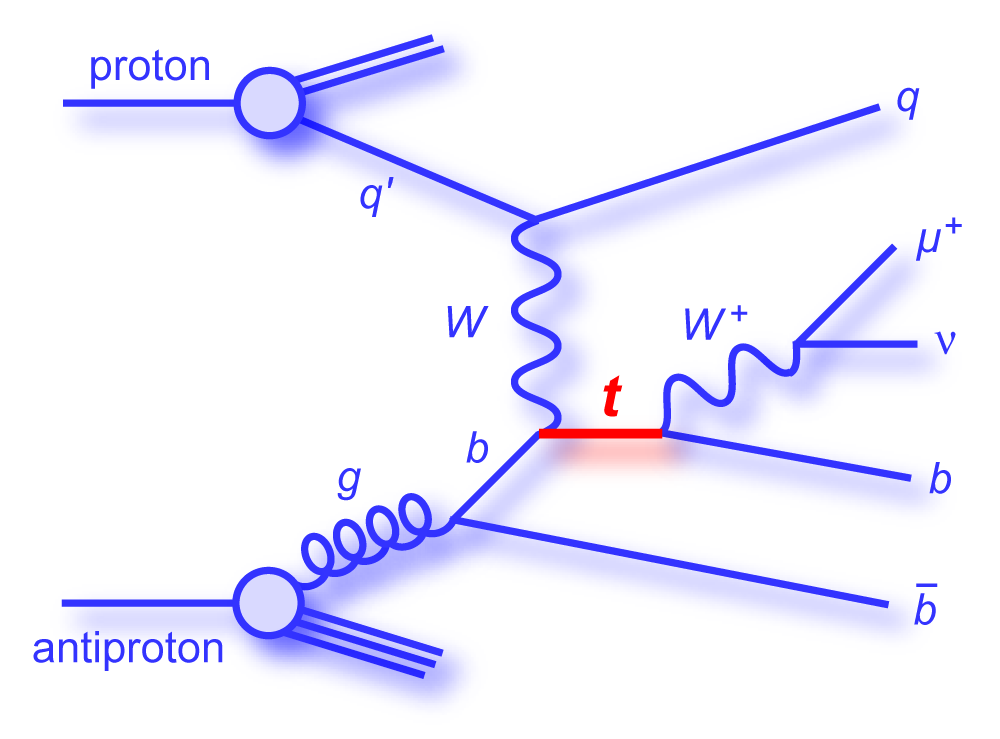}\\
    \textsc{Feynman Diagram}
  \end{minipage}
  \caption{Various systems}
  \label{fig:various-networks}
  \end{minipage}
  }
\end{figure}

\pagebreak

\begin{example}[Connecting open graphs]
  \label{ex:connecting-open-graphs}
  We can connect together two open graphs when the
  inputs of one is equal to the outputs of the
  other. To illustrate this, consider the open
  graphs
  \[
  \begin{tikzpicture}
    \begin{scope}
    \node (1) at (0,0) {{$ _a \bullet $}};
    \node (2) at (0,2) {{$ _b \bullet $}};
    \node (3) at (2,2) {{$ \bullet_d $}};
    \node (4) at (2,0) {{$ \bullet_e $}};
    \node (5) at (1,4) {{$ \bullet_c $}};
    \draw [graph]
    (1) edge[] (2)
    (2) edge[] (3)
    (3) edge[] (1)
    (4) edge[] (3)
    (3) edge[] (5)
    (5) edge[] (2);
    \node () at (4,3) {$ a,c,d \in \mathtt{ inputs} $ };
    \node () at (4,1) {$ d,e \in \mathtt{ outputs} $};
    \end{scope}
    \begin{scope}[shift={(8,0)}]
    \node (1) at (0,3) {{$ _{d} \bullet $}};
    \node (2) at (0,1) {{$ _{e} \bullet $}};
    \node (3) at (2,2) {{$ \bullet_{f} $}};
    \draw [graph] 
    (1) edge[] (2)
    (2) edge[] (3)
    (3) edge[] (1); 
    \node () at (4,3) {$ d,e \in \mathtt{ inputs} $};
    \node () at (4,1) {$ e,f \in \mathtt{ outputs} $};  
    \end{scope}
  \end{tikzpicture}
  \]
  Connect these open graphs by gluing like-nodes
  together. This results in 
  \[
  \begin{tikzpicture}
    \node (1) at (0,0) {{$ _a \bullet $}};
    \node (2) at (0,2) {{$ _b \bullet $}};
    \node (3) at (2,2) {{$ \bullet_{d} $}};
    \node (4) at (2,0) {{$ \bullet_{e} $}};
    \node (5) at (1,4) {{$ \bullet_c $}};
    \node (6) at (4,1) {{$ \bullet_f$}};
    \draw [graph]
      (1) edge[] (2)
      (2) edge[] (3)
      (3) edge[] (1)
      (3) edge[bend left=15] (4)
      (3) edge[] (5)
      (5) edge[] (2)
      (4) edge[bend left=15] (3)
      (4) edge[] (6)
      (6) edge[] (3); 
    \node () at (6,3) {$ a,c,d \in I $};
    \node () at (6,1) {$ e,f \in O $};
  \end{tikzpicture}
  \]
\end{example}

The operation of gluing open graphs together can be defined
set theoretically.  However, we prefer to define this
operation as a composition of morphisms in an appropriate
category. This ensconces the gluing operation as
fundamental. In this chapter, we discuss the formalism of
structured cospans. These offer a language better equipped
to describe open systems than do more traditional set theory
styled definitions.

A \defn{cospan} in a category is a pair of arrows
\[
  x \to y \gets z
\]
with common codomain. A structured cospan is a special sort
of cospan. The rough idea of a structured cospan is that the
common codomain is some system and the domains are the
inputs and outputs of that system. In other words, we
interpret a structured cospan as the diagram
\[
  \mathtt{inputs} \xto{\iota}
  \mathtt{system} \xgets{\omega}
  \mathtt{outputs}
\]
where $ \iota $ chooses the part of the system to serve as
inputs and $ \omega $ chooses the outputs.  Section
\ref{sec:StrCsp-as-Arrows} is devoted to constructing a
category whose arrows are the structured cospans.

The motivation for using composition to describe the
connection of open systems also has a philosophical
component. We study systems through the lens of
compositionality.  A pithy description of compositionality
is ``the opposite of emergent''. That is, the behavior
of a compositional system is fully determined by the
behavior of the sub-systems comprising it. Here
are some examples of compositionality.

\begin{itemize}
\item

  Set functions are compositional.  Given functions
  $ f \from X \to Y $ and $ g \from Y \to Z $, then we know
  everything about the composite function
  $ g \circ f \from X \to Z $.
  
\item

  Given two computer programs, one that approximates a
  smooth solution to a given differential equation and
  another that outputs a visualization of a smooth function,
  then we know that the composite program renders a drawing
  of an approximate smooth solution to a given differential
  equation.
  
\item

  If one manufacturing line inputs various wood pulp and
  outputs paper and another manufacturing line inputs paper
  and outputs notebooks, then the composite manufacturing
  line inputs wood-pulp and outputs notebooks.

\end{itemize}

Already, we have mentioned examples of systems we are
interested in. Each of these examples are useful tools applied by
various scientists or engineers.  Naturally, each formalism
has developed idiosyncrasies, inflating the differences
between them. However, there remain clear qualitative
similarities between the different formalisms that ought to
be exploited to transport results determined with one
formalism to results about another formalism.  As
cross-disciplinary collaboration increases, the importance
of translating between formalisms grows. We propose the
structured cospan serve as a medium of translation.

The analogy to languages runs deeper than mere
translation. Indeed, languages have both syntactic and
semantic content. Systems do too. We intend to clearly
delineate between the two. William Lawvere's `functorial
semantics' \cite{lawvere_func-semantics} serves as
inspiration. This is a categorical approach to universal
algebra where algebraic theories are separated into two
pieces: one category capturing the structure and properties
of a type of algebraic object $ A $ and another category
containing the ``stuff'' underlying an instance of $ A $
(e.g.~the underlying set).  A functor between the categories
selects an instance of an algebraic object of type $ A $. In
our context, we separate open systems, not algebraic object
types, into two categories. One category contains the system
syntax and the other category the system semantics. In this
perspective, categories with structured cospans for arrows
serve as syntax and their compositionality manifests as
a functor into a category of semantics.

In this chapter, we define structured cospans and two
categories in which they appear.  The first categories
$ {}_L \Cospan $ was introduced by Baez and Courser
\cite{baez-courser_str-csp} and encodes open systems are
arrows.  The second category $ {}_L \StrCsp $ houses the
morphisms of structured cospans which are used to define
their rewriting. To ensure that structured cospans support a
good theory of rewriting, we show that $ _L\StrCsp $ is a
topos.  We close this chapter by combining $ _L\Cospan $ and
$ _L\StrCsp $ into a double category. Most of the work in
this chapter appeared previously in \cite{cicala_rewriting}.


\section{Structured cospans as a compositional framework}
\label{sec:StrCsp-as-Arrows}

In this section, we define a structured cospan and fit them
as arrows into a category.  There are several technical
components we need to consider, each serving a purpose. So
instead of providing the definition here, we build up to it
discussing each technicality along the way.

When thinking of a structured cospan, we have in mind
a diagram
\[
  \mathtt{inputs} \to
  \mathtt{system} \gets
  \mathtt{outputs}
\]
sitting in a category. Often, the inputs and outputs of a
system will be sets. For sets to exist in the same category
as the systems---as is needed to have the inputs, outputs,
and system represented in the same diagram---we consider sets as
degenerate systems.  For instance, the open graph
  \[
  \begin{tikzpicture}
    \begin{scope}
    \node (1) at (0,0) {{$ _a \bullet $}};
    \node (2) at (0,2) {{$ _b \bullet $}};
    \node (3) at (2,2) {{$ \bullet_d $}};
    \node (4) at (2,0) {{$ \bullet_e $}};
    \node (5) at (1,4) {{$ \bullet_c $}};
    \draw [graph]
    (1) edge[] (2)
    (2) edge[] (3)
    (3) edge[] (1)
    (4) edge[] (3)
    (3) edge[] (5)
    (5) edge[] (2);
    \node () at (4,3) {$ a,c,d \in \mathtt{ inputs} $ };
    \node () at (4,1) {$ d,e \in \mathtt{ outputs} $};
    \end{scope}
  \end{tikzpicture}
\]
presented using Definition \ref{def:open-graph-informal} is
realized as the structured cospan
\begin{equation}
\label{eq:open-graph-as-strcsp} 
  \begin{tikzpicture}
    \begin{scope}
    \node at (0,0) {$ \bullet $};
    \node at (0,2) {$ \bullet  $};
    \node at (0,4) {$ \bullet $};
    \draw [rounded corners]
      (-0.5,-0.5) rectangle (0.5,4.5);
    \node (l) at (0.6,2) {};  
    \end{scope}
    \begin{scope}[shift={(2,0)}]
    \node (1) at (0,0) {{$ \bullet $}};
    \node (2) at (0,2) {{$ \bullet $}};
    \node (3) at (2,2) {{$ \bullet $}};
    \node (4) at (2,0) {{$ \bullet $}};
    \node (5) at (1,4) {{$ \bullet $}};
    \draw [graph] 
    (1) edge[] (2)
    (2) edge[] (3)
    (3) edge[] (1)
    (4) edge[] (3)
    (3) edge[] (5)
    (5) edge[] (2); 
    \draw [rounded corners]
      (-0.5,-0.5) rectangle (2.5,4.5);
    \node (ml) at (-0.6,2) {};
    \node (mr) at (2.6,2) {};
    \end{scope}
    \begin{scope}[shift={(6,0)}]
    \node at (0,0) {$ \bullet $};
    \node at (0,2) {$ \bullet $};
    \draw [rounded corners]
      (-0.5,-0.5) rectangle (0.5,4.5);
    \node (r) at (-0.6,2) {};
    \end{scope}
    \draw [cd] 
    (l) edge[] (ml)
    (r) edge[] (mr);
  \end{tikzpicture}
\end{equation}
Inside this picture, we have three graphs enclosed in the
boxes.  The left and right-most graphs are really just sets
considered as edgeless graphs or, in our parlance, as
``degenerate systems''.  The arrows between the graphs are
graph morphisms defined as suggested by the layout. These
arrows choose the components of the central graph to serve
as inputs and outputs.

To model open graphs with structured cospans, we
do not want to allow arbitrary graphs in the feet
of the cospan. We only want sets qua
edgeless graphs. To accomplish this, we define a
functor
\begin{equation}
\label{eq:edgeless-graph-functor}
  L \from \Set \to \RGraph
\end{equation}
that turns a set $ a $ into a graph $ La $ with
node set $ a $ and no non-reflexive edges. Now, the open graph
in \eqref{eq:open-graph-as-strcsp} has form
\[
  La \to x \gets Lb
\]
where $ a $ is a three element set, $ b $ is a two
element set, and $ x $ is the graph
\[
  \begin{tikzpicture}
    \node (1) at (0,0) {{$ _a \bullet $}};
    \node (2) at (0,2) {{$ _b \bullet $}};
    \node (3) at (2,2) {{$ \bullet_d $}};
    \node (4) at (2,0) {{$ \bullet_e $}};
    \node (5) at (1,4) {{$ \bullet^c $}};
    \draw [graph]
    (1) edge[] (2) 
    (2) edge[] (3) 
    (3) edge[] (1) 
    (4) edge[] (3) 
    (3) edge[] (5) 
    (5) edge[] (2);
  \end{tikzpicture}
\]

The functor $ L $ in \eqref{eq:edgeless-graph-functor} is
crucial to the definition of a structured cospan. To capture
open systems more general than open graphs, we allow $ L $
to be of type $ \A \to \X $ for categories $ \A $ and
$ \X $. Now, a structured cospan based on a functor
$ L \from \A \to \X $ is a cospan in $ \X $ of the form
$ La \to x \gets Lb $. We do not use this as a definition
because for rewriting we require more from $ L $, $ \A $,
and $ \X $.

One such need is to construct a category where structured
cospans $ La \to x \gets Lb $ are arrows. Hence, given
another structured cospan $ Lb \to y \gets Lc $, we need to
define the composite.  As is typical in cospan categories
\cite{benabou-bicategories}, we compose by pushout.  That
is, the composite of the structured cospans
\[
  La \to x \gets Lb
  \quad \text{ and } \quad
  Lb \to y \gets Lc
\]
is the structured cospan
\[
  La \to x +_{Lb} y \gets Lc
\]
Using this composition, we henceforth require
$ \X $ to have pushouts.

Let us unpack this composition.  We have a pair of
systems $ x $ and $ y $, where the outputs of
$ x $ are chosen by the arrow $ Lb \to x $ and the
inputs of $ y $ are chosen by the arrow
$ Lb \to y $.  Considered together, we have a span
$ x \gets Lb \to y $. The pushout of this span is
\[
  \begin{tikzpicture}
    \node (x)  at (0,0)   {$ x $};
    \node (y)  at (2,2)   {$ y $};
    \node (Lb) at (0,2)   {$ Lb $};
    \node (po) at (2,0)   {$ x +_{Lb} y $};
    \draw [cd]
    (Lb) edge[] (x) 
    (Lb) edge[] (y) 
    (x)  edge[] (po) 
    (y)  edge[] (po); 
    \draw (1.6,0.3) -- (1.6,0.4) -- (1.7,0.4);
  \end{tikzpicture}
\]
A useful intuition of this pushout is that the system
$ x +_{Lb} y $ is obtained by gluing the image of $ Lb $ in
$ x $ to the image of $ Lb $ in $ y $.  The composite system
$ x +_{Lb} y $ has inputs chosen by the composite
$ La \to x \to x +_{Lb} y $ and outputs chosen by the
composite $ Lc \to y \to x +_{Lb} y $. The composite
structured cospan is then
\[
  La \to x +_{Lb} y \gets Lc
\]

From this composition, a functor
$ L \from \A \to \X $ where $ \X $ has pushouts
gives a category whose objects are those of $ \A $
and whose arrows of type $ a \to b $ are
structured cospans $ La \to x \gets Lb $. For our
needs, however, we ask more of $ L $, $ A $, and $ X $.

In Chapter \ref{sec:dpo-rewriting}, we introduce a theory of
rewriting structured cospans.  To do so, we need a
\emph{topos}---discussed in Appendix \ref{sec:topoi}---in
which structured cospans are the objects. We find this topos
in Theorem \ref{thm:strcsp-istopos} and so our theory
requires the assumptions held there. Precisely, we need
$ L $ to be a pullback preserving left adjoint and for both
$ \A $ and $ \X $ to be topoi. Section
\ref{sec:StrCspAsObject} contains further discussion about
how these assumptions figure into our goal of modeling
systems. In the meantime, we fix these assumptions once and
for all.

Fix a adjunction
\[
  \adjunction{\A}{\X}{L}{R}{2}
\]
with $ L $ preserving pullbacks. How does our theory of
systems map onto this adjunction? Interpret the topos $ \X $
as a category whose objects are systems and whose arrows are
the homomorphism of systems. These systems are
\emph{closed}, in that they cannot interact with outside
agents, specifically other systems of the same type. To
provide a compositional structure to these systems, we
introduce a topos $ \A $ that we interpret as a category of
interfaces types and their morphisms. By transporting the
interface types along $ L $, we can include them in the
cospans with systems in $ \X $. The arrows of a structured
cospan equip a system with its interface. Once equipped with
a (non-empty) interface, a system is \emph{open} in that
they can interact with compatible systems. There is no
explicit role for $ R $. It is the properties of $ L $ that
exists in light of $ L $ being an adjunction that we
use. However, we can still interpret $ R $ as returning the
maximal (by inclusion) interface of a system. The existence
of $ R $ is a side-effect that we leverage in Theorem
\ref{thm:strcsp-istopos}.

Using the adjunction $ L \from \A \lrto \X \from R $ we
construct a compositional framework having systems as arrows
in a cospan category.  Composition of arrows uses pushout
which encodes connecting a pair of compatible
systems. Because cospans are too general for our needs, we
restrict our attention to structured cospans.

\begin{definition}[Structured cospan] \label{df:strcsp}
  A \defn{structured cospan} is a cospan of the
  form $ La \to x \gets Lb $.  When we want to
  emphasize $ L $, we use the term
  \defn{$ L $-structured cospans}.
\end{definition}

Structured cospans fit into two different categories that
are central to our theory. The first one, that we meet now,
was proved by Baez and Courser to actually be a category
\cite{baez-courser_str-csp}. To start, we define an
isomorphism of structured cospans from $ La \to x \gets Lb $
to $ La \to x' \gets Lb $ to be an invertible arrow
$ h \from x \to x' $ in $ \X $ that fits into the commuting
diagram
\[
  \begin{tikzpicture}
    \node (La) at (0,0)  {$ La $};
    \node (Lb) at (4,0)  {$ Lb $};
    \node (x)  at (2,2)  {$ x $};
    \node (x') at (2,-2) {$ x' $};
    \draw [cd]
    (La) edge[] node[above,left]{$ f $} (x)
    (Lb) edge[] node[above,right]{$ g $} (x)
    (La) edge[] node[below,left]{$ f' $} (x')
    (Lb) edge[] node[below,right]{$ g' $} (x')
    (x)  edge[] node[right]{$ h $} (x');
  \end{tikzpicture}
\]

\begin{definition} \label{def:LCsp}

  The category $ _{L} \Cospan $ has as objects the objects
  of $ \A $ and arrows $ a \to b $ are structured cospans
  $ La \to x \gets Lb $ up to isomorphism.
  
\end{definition}

Composing $ La \to x \gets Lb $ with
$ Lb \to y \gets Lc $ uses pushout
\[
  \begin{tikzpicture}
    \begin{scope}[]
    \node (1) at (0,0) {\( La \)};
    \node (2) at (2,1) {\( x +_{Lb} y \)};
    \node (3) at (4,0) {\( Lc \)};
    \draw [cd] (1) edge (2);
    \draw [cd] (3) edge (2); 
    \end{scope}
  \end{tikzpicture}
\]
In a sense, pushouts glue objects together making
it a sensible way to model system connection.  The
composition above is like connecting along $ Lb
$. Using structured cospans, we now improve our earlier definition of open graphs.

\begin{example} \label{ex:open-graph-as-arrow}

  There is a geometric morphism (see Definition
  \ref{def:geometric-morphism})
  \[
    \adjunction{\Set}{\RGraph}{L}{R}{4}
  \]
  where $ Rx $ is the node set of graph $ x $ and
  $ La $ is the edgeless graph with node set
  $ a $. An \defn{open graph} is a cospan
  $ La \to x \gets Lb $ for sets $ a $, $ b $, and
  graph $ x $. An illustrated example, with the
  reflexive loops suppressed, is
  \[
    \begin{tikzpicture}
      \begin{scope} 
      \node (1) at (0,1) { \( \bullet \) };
      \node (2) at (0,0) { \( \bullet \) };
      \draw [rounded corners]
        (-0.5,-0.5) rectangle (0.5,1.5);
      \end{scope}
      \begin{scope}[shift={(2,0)}] 
      \node (1) at (0,1) {\( \bullet \)};
      \node (2) at (0,0) {\( \bullet \)};
      \node (3) at (1,0.5) {\( \bullet  \)};
      \node (4) at (2,0.5) {\( \bullet  \)};
      \draw [graph]
        (1) edge (3)
        (2) edge (3)
        (3) edge (4);
      \draw [rounded corners]
        (-0.5,-0.5) rectangle (2.5,1.5);
      \end{scope}
      \begin{scope}[shift={(6,0)}] 
      \node (1) at (0,0.5) {\( \bullet \)};
      \draw [rounded corners] (-0.5,-0.5) rectangle (0.5,1.5);
      \end{scope}
      \begin{scope} 
        \node (1) at (0.5,0.5) {};
        \node (2) at (1.5,0.5) {};
        \node (3) at (4.5,0.5) {};
        \node (4) at (5.5,0.5) {};
        \draw [->] (1) to (2);
        \draw [->] (4) to (3);
      \end{scope}
    \end{tikzpicture}
  \]
  The boxed items are graphs and the arrows
  between boxes are graph morphisms defined as
  suggested by the illustration.  In total, the
  three graphs and two graph morphisms make up a
  single open graph whose inputs and outputs are,
  respectively, the left and right-most graphs.
    
  Open graphs are compositional. For instance, we
  can compose
  \[
    \begin{tikzpicture}
      \begin{scope} 
      \node (1) at (0,1) { \( \bullet \) };
      \node (2) at (0,0) { \( \bullet \) };
      \draw [rounded corners]
        (-0.5,-0.5) rectangle (0.5,1.5);
      \end{scope}
      \begin{scope}[shift={(2,0)}] 
      \node (1) at (0,1) {\( \bullet \)};
      \node (2) at (0,0) {\( \bullet \)};
      \node (3) at (1,0.5) {\( \bullet  \)};
      \node (4) at (2,0.5) {\( {}_{j} \bullet  \)};
      \draw [graph]
        (1) edge (3)
        (2) edge (3)
        (3) edge (4);
      \draw [rounded corners]
        (-0.5,-0.5) rectangle (2.5,1.5);
      \end{scope}
      \begin{scope}[shift={(6,0)}] 
      \node (1) at (0,0.5) {\( \bullet_{j} \)};
      \draw [rounded corners]
        (-0.5,-0.5) rectangle (0.5,1.5);
      \end{scope}
      \begin{scope} 
      \node (1) at (0.5,0.5) {};
      \node (2) at (1.5,0.5) {};
      \node (3) at (4.5,0.5) {};
      \node (4) at (5.5,0.5) {};
      \draw [cd] (1) to (2);
      \draw [cd] (4) to (3);
      \end{scope}
    \end{tikzpicture}
  \]
  with
  \[
    \begin{tikzpicture}
      \begin{scope} 
      \node (1) at (0,0.5) { \( {}_{j} \bullet \) };
      \draw [rounded corners]
        (-0.5,-0.5) rectangle (0.5,1.5);
      \end{scope}
      \begin{scope}[shift={(2,0)}] 
      \node (1) at (0,0.5) {\( {}_{j} \bullet \)};
      \node (2) at (2,0) {\( \bullet \)};
      \node (3) at (2,0.5) {\( \bullet  \)};
      \node (4) at (2,1) {\( \bullet  \)};
      \draw [graph]
        (1) edge (2)
        (1) edge (3)
        (1) edge (4);
      \draw [rounded corners]
        (-0.5,-0.5) rectangle (2.5,1.5);
      \end{scope}
      \begin{scope}[shift={(6,0)}] 
      \node (2) at (0,0) {\( \bullet \)};
      \node (3) at (0,0.5) {\( \bullet  \)};
      \node (4) at (0,1) {\( \bullet  \)};
      \draw [rounded corners]
        (-0.5,-0.5) rectangle (0.5,1.5);
      \end{scope}
      \begin{scope} 
      \node (1) at (0.5,0.5) {};
      \node (2) at (1.5,0.5) {};
      \node (3) at (4.5,0.5) {};
      \node (4) at (5.5,0.5) {};
      \draw [cd]
        (1) edge (2)
        (4) edge (3);
      \end{scope}
    \end{tikzpicture}
  \]
  to get the open graph
  \[
    \begin{tikzpicture}
      \begin{scope} 
      \node (1) at (0,1) { \( \bullet \) };
      \node (2) at (0,0) { \( \bullet \) };
      \draw [rounded corners]
        (-0.5,-0.5) rectangle (0.5,1.5);
      \end{scope}
      \begin{scope}[shift={(2,0)}] 
      \node (1) at (0,1) {\( \bullet \)};
      \node (2) at (0,0) {\( \bullet \)};
      \node (3) at (1,0.5) {\( \bullet  \)};
      \node (4) at (2,0.5) {\( {}^{j} \bullet  \)};
      \node (5) at (3,0) {\( \bullet \)};
      \node (6) at (3,0.5) {\( \bullet  \)};
      \node (7) at (3,1) {\( \bullet  \)};
      \draw [graph]
        (1) edge (3)
        (2) edge (3)
        (3) edge (4)
        (4) edge (5)
        (4) edge (6)
        (4) edge (7);
      \draw [rounded corners]
        (-0.5,-0.5) rectangle (3.5,1.5);
      \end{scope}
      \begin{scope}[shift={(7,0)}] 
      \node (2) at (0,0) {\( \bullet \)};
      \node (3) at (0,0.5) {\( \bullet  \)};
      \node (4) at (0,1) {\( \bullet  \)};
      \draw [rounded corners]
        (-0.5,-0.5) rectangle (0.5,1.5);
      \end{scope}
      \begin{scope} 
      \node (1) at (0.5,0.5) {};
      \node (2) at (1.5,0.5) {};
      \node (3) at (5.5,0.5) {};
      \node (4) at (6.5,0.5) {};
      \draw [cd]
        (1) edge (2)
        (4) edge (3);
      \end{scope}
    \end{tikzpicture}
  \]
  which is obtained by composing structured cospans. Note
  that this is composition in $ _L \Cospan $ for $ L \from
  \Set \to \RGraph $. 
\end{example}

In general, interpret $ La \to x \gets Lb $ as
consisting of a system $ x $ equipped with an
interface comprised of inputs $ La $ and outputs
$ Lb $. The terms `input' and `output' do not
imply any causal structure.  They are merely meant
to provide a way to connect a pair of systems
along a proper subset of their interfaces.
Decomposing the interface into inputs and outputs
distinguish the portion of the interface that is
used in a connection from the portion of the
interface that is not used. The specific
connection formed determines the interface
decomposition and every possibility exists as an
arrow in $ _{L} \Cospan $. This is reflected in the
fact that $ _{L} \Cospan $ is compact closed (see Definition \ref{def:DualPairCat}).

\begin{proposition} \label{thm:cospan-compact-closed}
  $ ( _L \Cospan , \otimes , 0_\A ) $, where
  \begin{align*}
    \label{eq:cospan-compact-closed}
    & \otimes \from _L\Cospan  \times {}_L\Cospan  \to {}_L\Cospan \\
    a & \otimes b \mapsto a + b  \\
    \left( La \xto{f} x \xgets{g} Lb \right) & \otimes
                                               \left( La' \xto{f'} x' \xgets{g'} Lb' \right) 
                                               \mapsto
                                               \left(
                                               L(a+a') \xto{f+f'} x+x'
                                               \xgets{g+g'} L(b+b') \right)
  \end{align*}
  is compact closed.
\end{proposition}

\begin{proof}
  It is a matter of course to show that
  $ ( _L \Cospan , \otimes , 0_\A ) $ is a
  symmetric monoidal category. Though, we point
  out that we are being a bit casual with our
  definition of $ \otimes $.  The tensor
  product actually returns the structured cospan
  \[
      L(a+a') \xto{\sigma^{-1}_a}   La+La'
      {}      \xto{f+f'}            x+x'
      {}      \xgets{g+g'}          Lb+Lb'
      {}      \xgets{\sigma^{-1}_b} L(b+b')
  \]
  where $ \sigma $ is the structure map arising
  from the preservation of $ + $ by $ L
  $. The symmetry rests on the fact that both $ (\A,+,0_\A)
  $ and $ (\X,+,0_\X) $ are symmetric monoidal categories.
  
  Regarding compactness, each object is self-dual.
  For an object $ a $, the evaluation map
  $ a \otimes a \to I $ is
  \[
    L(a+a) \xto{L\nabla} La \xgets{!} L0_\A
  \]
  and the coevaluation map is
  \[
    L0 \xto{!} La \xgets{L\nabla} L(a+a)
  \]
  where $ \nabla $ denotes the
  codiagonal. Checking the triangle identities are
  straightforward.
 \end{proof}


\section{Structured cospans as objects}
\label{sec:StrCspAsObject}

Lack and Sobocinski provided a way to rewrite objects in
what are called adhesive categories
\cite{lack-sobo_adhesive-cats}.  To provide a theory of
rewriting structured cospans using adhesive categories, we
need a category in which structured cospans are the
objects. This, of course, requires a notion of structured
cospan morphism. 

\begin{definition} \label{df:morph-of-strcsp}
  A morphism between $ L $-structured cospans
  \[ La \to x \gets Lb \text{ and } Lc \to y \gets Ld \] is a
  triple of arrows $ ( f,g,h ) $ that fit into the commuting
  diagram
  \[
    \begin{tikzpicture}
      \node (1) at (-3,2) {\( La \)};
      \node (2) at (0,2) {\( x \)};
      \node (3) at (3,2) {\( Lb \)};
      \node (4) at (-3,0) {\( Lc \)};
      \node (5) at (0,0) {\( y \)};
      \node (6) at (3,0) {\( Ld \)};
      \draw [cd]
        (1) edge[] (2)
        (3) edge[] (2)
        (4) edge[] (5)
        (6) edge[] (5)
        (1) edge[] node[left]{$ Lf $} (4)
        (2) edge[] node[left]{$g$} (5)
        (3) edge[] node[right]{$Lh$} (6);
    \end{tikzpicture}
  \]
  There is a category $ _{L}\StrCsp $ whose objects are
  structured cospans and arrows are these morphisms.
\end{definition}

We now come to the first of our main results: that
$ _{L}\StrCsp $ is a topos. This result is critical for our
theory because, as each topos is adhesive
\cite{lack-sobo_topoi-adhesive}, it allows the introduction
of rewriting onto structured cospans.

\begin{theorem} \label{thm:strcsp-istopos}
  For any adjunction 
  \[
    \adjunction{\A}{\X}{L}{R}{2}
  \]
  between topoi $ \A $ and $ \X $, the category
  $ _{L}\StrCsp $ is a topos.
\end{theorem}

\begin{proof}
  Note that $ _{L}\StrCsp $ is equivalent to the category
  whose objects are cospans of form
  \(
    a \to Rx \gets b
  \)
  and morphisms are triples $ ( f,g,h ) $ fitting into the
  commuting diagram
  \[
    \begin{tikzpicture}
      \node (1) at (-4,2) {\( w \)};
      \node (2) at (0,2)  {\( Ra \)};
      \node (3) at (4,2)  {\( x \)};
      \node (4) at (-4,0) {\( y \)};
      \node (5) at (0,0)  {\( Rb \)};
      \node (6) at (4,0)  {\( z \)};
      \draw [cd] 
      (1) edge[] (2)
      (3) edge[] (2)
      (4) edge[] (5)
      (6) edge[] (5)
      (1) edge[] node[left]{$ f $} (4)
      (2) edge[] node[left]{$ Rg $} (5)
      (3) edge[] node[right]{$ h $} (6); 
    \end{tikzpicture}
  \]
  This, in turn, is equivalent to the comma category
  $ ( \A \times \A \downarrow \Delta R ) $, where
  $ \Delta \from \A \to \A \times \A $ is the diagonal
  functor. But this diagonal functor is right adjoint to the
  coproduct functor. Therefore, $ \Delta R $ is also a right
  adjoint so $ ( \A \times \A \downarrow \Delta R ) $ is an
  instance of Artin gluing \cite{wraith_artin-glue}, hence a
  topos.
\end{proof}

We now show that constructing $ _{L}\StrCsp $ is functorial
in $ L $. The codomain of this functor is comprised of topoi
and adjoint pairs, the left of which preserves pullbacks.
We call this category $ \AdjTopos $. The domain this functor
is the arrow category of $ \AdjTopos $, which we denote by
$ [\bullet \to \bullet , \AdjTopos] $. In this category, the
objects are adjunctions between topoi, the left adjoint
preserving pullbacks, and an arrow from
$ L \from \A \lrto \X \from R $ to
$ L' \from \A' \lrto \X' \from R' $ is a pair of adjoints
$ F \dashv G $ and $ F' \dashv G' $ fitting into a diagram
\begin{center}
  \begin{tikzpicture}
    \node (1) at (-2,2) {\( \A \)};
    \node (2) at (-2,-2) {\( \A' \)};
    \node (3) at (2,2) {\( \X \)};
    \node (4) at (2,-2) {\( \X' \)};
    \draw[cd]
    (1) edge[bend left] node[right]{$ G $} (2)
    (1) edge[bend left] node[above]{$ L $} (3)
    (2) edge[bend left] node[left]{$ F $} (1)
    (2) edge[bend left] node[above]{$ L' $} (4)
    (3) edge[bend left] node[below]{$ R $} (1)
    (3) edge[bend left] node[right]{$ G' $} (4)
    (4) edge[bend left] node[below]{$ R' $} (2)
    (4) edge[bend left] node[left]{$ F' $} (3); 
    \node (5) at (0,-2) {{\( \bot \)}};
    \node (6) at (0,2)  {{\( \bot \)}};
    \node (7) at (-2,0) {{\( \dashv \)}};
    \node (8) at (2,0)  {{\( \dashv \)}};
  \end{tikzpicture}
\end{center}
such that $ LF=F'L' $ and $ GR=R'G' $.

\pagebreak 

\begin{theorem} \label{thm:strcsp-isfunctorial}
  There is a functor
  \[
    _{(-)}\StrCsp \from
    [ \bullet \to \bullet , \AdjTopos ]
    \to
    \AdjTopos
  \]
  defined by  
  \[
    \begin{tikzpicture}
      \begin{scope}
        \node (1) at (2,2) {\( \X \)};
        \node (2) at (2,-2) {\( \X' \)};
        \node (3) at (-2,2) {\( \A \)};
        \node (4) at (-2,-2) {\( \A' \)};
        \draw[cd]
        (1) edge[bend left] node[right]{$ G' $} (2)
        (2) edge[bend left] node[left]{$ F' $} (1) 
        (3) edge[bend left] node[above]{$ L $} (1) 
        (1) edge[bend left] node[below]{$ R $} (3) 
        (2) edge[bend left] node[below]{$ R' $} (4)
        (4) edge[bend left] node[above]{$ L' $} (2) 
        (4) edge[bend left] node[left]{$ F $} (3) 
        (3) edge[bend left] node[right]{$ G $} (4); 
        \node (5) at (0,-2) {{\( \bot \)}};
        \node (6) at (0,2)  {{\( \bot \)}};
        \node (7) at (-2,0) {{\( \dashv \)}};
        \node (8) at (2,0)  {{\( \dashv \)}};
      \end{scope}
      \begin{scope}[shift={(4,0)}]
        \node (1) at (0,0) { $\xmapsto{ _{(-)}\StrCsp }$ };
      \end{scope}
      \begin{scope}[shift={(6,0)}]
      \node (1) at (0,0) {\( _{L}\StrCsp \)};
      \node (2) at (4,0) {\( _{L'}\StrCsp \)};
      \node (3) at (2,0) {{\( \perp \)}};
      \draw [cd]
        (1) edge[bend left] node[above]{$\Theta$} (2)
        (2) edge[bend left] node[below]{$\Theta'$} (1);  
      \end{scope}
    \end{tikzpicture}
  \]
  which is in turn given by
  \[
    \begin{tikzpicture}
      \begin{scope}
      \node (1) at (-2,2) {\( La \)};
      \node (2) at (0,2) {\( x \)};
      \node (3) at (2,2) {\( Lb \)};
      \node (4) at (-2,0) {\( Lc \)};
      \node (5) at (0,0) {\( y \)};
      \node (6) at (2,0) {\( Ld \)};
      \draw [cd] (1) to node [above] {$ m $}  (2);
      \draw [cd] (3) to node [above] {$ n $}  (2);
      \draw [cd] (4) to node [below] {$ o $}  (5);
      \draw [cd] (6) to node [below] {$ p $}  (5);
      \draw [cd] (1) to node [left]  {$ Lf $} (4);
      \draw [cd] (2) to node [left]  {$ g $}  (5);
      \draw [cd] (3) to node [right] {$ Lh $} (6);
      \end{scope}
      \begin{scope}[shift={(3,0)}]
        \node (1) at (0,1) { $ \xmapsto{ \Theta } $ };
      \end{scope}
      \begin{scope}[shift={(6.5,0)}]
      \node (1) at (-2,2) {\( L'G'a \)};
      \node (2) at (0,2)  {\( Gx \)};
      \node (3) at (2,2)  {\( L'G'b \)};
      \node (4) at (-2,0) {\( L'G'c \)};
      \node (5) at (0,0)  {\( Gy \)};
      \node (6) at (2,0)  {\( L'G'd \)};
      \draw [cd] (1) to node [above] {$ Gm $}    (2);
      \draw [cd] (3) to node [above] {$ Gn $}    (2);
      \draw [cd] (4) to node [below] {$ Go $}    (5);
      \draw [cd] (6) to node [below] {$ Gp $}    (5);
      \draw [cd] (1) to node [left]  {$ L'G'f $} (4);
      \draw [cd] (2) to node [left]  {$ Gg $}    (5);
      \draw [cd] (3) to node [right] {$ L'G'h $} (6);  
      \end{scope}
    \end{tikzpicture}
  \]
  and
  \[
    \begin{tikzpicture}
      \begin{scope}
      \node (1) at (-2,2) {\( L'a' \)};
      \node (2) at (0,2)  {\( x' \)};
      \node (3) at (2,2)  {\( L'b' \)};
      \node (4) at (-2,0) {\( L'c' \)};
      \node (5) at (0,0)  {\( y' \)};
      \node (6) at (2,0)  {\( L'd' \)};
      \draw [cd] (1) edge node [above] {$ m' $} (2);
      \draw [cd] (3) edge node [above] {$ n' $} (2);
      \draw [cd] (4) edge node [above] {$ o' $} (5);
      \draw [cd] (6) edge node [above] {$ p' $} (5);
      \draw [cd] (1) edge node [left]  {$ L'f' $} (4);
      \draw [cd] (2) edge node [left]  {$ g' $} (5);
      \draw [cd] (3) edge node [left]  {$ L'h' $} (6);
      \end{scope}
      \begin{scope}[shift={(3,0)}]
      \node (1) at (0,1) { $ \xmapsto{ \Theta' } $ };
      \end{scope}
      \begin{scope}[shift={(6.5,0)}]
      \node (1) at (-2,2) {\( LF'a' \)};
      \node (2) at (0,2) {\( Fx' \)};
      \node (3) at (2,2) {\( LF'b' \)};
      \node (4) at (-2,0) {\( LF'c' \)};
      \node (5) at (0,0) {\( Fy' \)};
      \node (6) at (2,0) {\( LF'd' \)};
      \draw [cd] (1) edge node[above]{$ Fm $}   (2);
      \draw [cd] (3) edge node[above]{$ Fn' $}  (2);
      \draw [cd] (4) edge node[above]{$ Fo' $}  (5);
      \draw [cd] (6) edge node[above]{$ Fp' $}  (5);
      \draw [cd] (1) edge node[left]{$ LF'f' $} (4);
      \draw [cd] (2) edge node[left]{$ Fg' $}   (5);
      \draw [cd] (3) edge node[left]{$ LF'h' $} (6);  
      \end{scope}
    \end{tikzpicture}
  \]  
\end{theorem}

\begin{proof}
  In light of Theorem \ref{thm:strcsp-istopos}, it suffices
  to show that $ \Theta \dashv \Theta' $ gives an
  adjunction and $ \Theta $ preserves pushouts.

  Denote the structured cospans
  \[
    La \xto{m} x \xgets{n} Lb
  \]
  in $ _{ L }\StrCsp $ by $ \ell $ and  
  \[
    L'a' \xto{m'} x' \xgets{n'} L'b'
  \]
  in $ _{ L' }\StrCsp $ by $ \ell' $. Denote the unit and
  counit for $F \dashv G$ by $ \eta $, $ \varepsilon $ and
  for $ F' \dashv G' $ by $ \eta' $, $ \varepsilon' $.  The
  assignments
  \begin{align*}
    \left(
      ( f,g,h ) \from \ell \to \Theta' \ell'
      \right)
    & \mapsto
    \left(
      ( \varepsilon' \circ F'f , \varepsilon \circ Fg , \varepsilon'
      \circ F'h )
      \from \Theta \ell \to \ell'
      \right) \\
      \left(
      ( f',g',h' ) \from \Theta \ell \to \ell'
      \right)
    & \mapsto
      \left(
      ( G'f' \circ \eta', Gg' \circ \eta , G'h' \circ \eta' )
      \from \ell \to \Theta' \ell'
      \right) 
  \end{align*}
  give a bijection
  $ \hom ( \Theta \ell , \ell' ) \simeq \hom ( \ell ,
  \Theta' \ell' ) $. The naturality of $ \ell $ and
  $ \ell' $ rest on natural maps $ \eta $, $ \varepsilon $,
  $ \eta' $, and $ \varepsilon' $. The left adjoint
  $ \Theta' $ preserves finite pullbacks because they are taken
  pointwise and $ L $, $ F $, and $ F' $ all preserve finite
  limits.
\end{proof}

The arrows $ _{L}\StrCsp \to _{L'}\StrCsp $ that we are
interested in act on the systems and their interfaces.

\begin{definition} \label{def:str-csp-functor}
  Fix a pair of structured cospan categories $ _{L}
  \StrCsp $ and $ _{L'} \StrCsp $ using the adjunctions
  \[
    \raisebox{-0.5\height}{\adjunction{\A}{\X}{L}{R}{2}}
    \quad \text{ and } \quad
    \raisebox{-0.5\height}{\adjunction{\A'}{\X'}{L'}{R'}{2}}
  \]
  with $ L $ and $ L' $ preserving pullbacks.  A
  \defn{structured cospan functor} of type
  \[ _{L}\StrCsp \to _{L'}\StrCsp \] is a pair of
  finitely continuous and cocontinuous functors
  $ F \from \X \to \X ' $ and
  $ G \from \A \to \A' $ such that the diagrams
  \[
    \begin{tikzpicture}
      \node (A) at (0,2) {$ \A $};
      \node (X) at (2,2) {$ \X $};
      \node (A') at (0,0) {$ \A' $};
      \node (X') at (2,0) {$ \X ' $};
      \draw [cd] (A)  to node[above]{$ L $}   (X);
      \draw [cd] (A') to node[below]{$ L' $}  (X');
      \draw [cd] (A)  to node[left]{$ G $}    (A');
      \draw [cd] (X)  to node[right]{$ F $}   (X');
    \end{tikzpicture}
    \quad \quad
    \begin{tikzpicture}
      \node (A) at (0,2) {$ \A $};
      \node (X) at (2,2) {$ \X $};
      \node (A') at (0,0) {$ \A' $};
      \node (X') at (2,0) {$ \X ' $};
      \draw [cd] (X) to node[above]{$ R $}   (A);
      \draw [cd] (X') to node[below]{$ R' $} (A');
      \draw [cd] (A) to node[left]{$ G $}    (A');
      \draw [cd] (X) to node[right]{$ F $}  (X');
    \end{tikzpicture}
  \]
  commute.
\end{definition}

Structured cospan categories and their morphisms form a
category which we leave unnamed.


\section{A double category of structured cospans}
\label{sec:DblCatOfStrCsp}

We use (pseudo) double categories (see Definition
\ref{def:dbl-cat}) to combine into a single instrument the
competing perspectives of structured cospans as objects and
as arrows.

\begin{definition}[Structured cospan double category]
  There is a double category
  $ _{L}\SSStrCsp $ given by the following data:
  \begin{itemize}
  \item the objects are the $ \A $-objects
  \item the vertical arrows $ a \to b $ are the $ \A $-arrows, 
  \item the horizontal arrows $ a \hto b$ are the cospans
    $ La \to x \gets Lb $, and
  \item the squares are the commuting diagrams
    \[
    \begin{tikzpicture}
    \node (1) at (0,2) {\( La \)};
    \node (2) at (2,2) {\( x \)};
    \node (3) at (4,2) {\( Lb \)};
    \node (4) at (0,0) {\( Lc \)};
    \node (5) at (2,0) {\( y \)};
    \node (6) at (4,0) {\( Ld \)};
    \draw [cd] (1) to node [] {\scriptsize{\(   \)}} (2);
    \draw [cd] (3) to node [] {\scriptsize{\(  \)}} (2);
    \draw [cd] (4) to node [] {\scriptsize{\(  \)}} (5);
    \draw [cd] (6) to node [] {\scriptsize{\(  \)}} (5);
    \draw [cd] (1) to node [left] {\scriptsize{\( Lf \)}} (4);
    \draw [cd] (2) to node [left] {\scriptsize{\( g \)}} (5);
    \draw [cd] (3) to node [right] {\scriptsize{\( Lh \)}} (6);
    \end{tikzpicture}
  \]
  \end{itemize}
\end{definition}

Baez and Courser proved that this truly is a double category
\cite[Cor.~3.9]{cicala_spans-cospans}. Moreover, when $ \A $
and $ \X $ are cocartesian, their coproducts can be used to
define a symmetric monoidal structure on $ _{L} \SSStrCsp
$. The meaning of this structure is that the disjoint union
of two systems can be considered a single system. The
following example illustrates the squares and tensor
product.

\begin{example}
  Consider the double category $ _L \SSStrCsp $ where $ L $
  is left adjoint to the underlying node functor
  $ R \from \RGraph \to \Set $. A square in this double
  category is a diagram in $ \RGraph $ such as
  \begin{center}
    \begin{tikzpicture}
      \begin{scope}[shift={(0,5)}] 
        \node (a) at (0,2) {$ \bullet $};
        \node (b) at (0,0) {$ \bullet $};
        \draw [rounded corners] (-1,-1) rectangle (1,3);
        \node (01r) at (1.1,1) {$  $};
        \node (01b) at (0,-1.1) {$  $};
      \end{scope}
      \begin{scope}[shift={(3,5)}] 
        \node (a) at (0,2) {$ \bullet $};
        \node (b) at (0,0) {$ \bullet $};
        \node (c) at (2,1) {$ \bullet $};
        \draw [graph] 
          (a) edge[] (c)
          (b) edge[] (c);
        \draw [rounded corners] (-1,-1) rectangle (3,3);
        \node (11l) at (-1.1,1) {$  $};
        \node (11r) at (3.1,1) {$  $};
        \node (11b) at (1,-1.1) {$  $};
      \end{scope}
      \begin{scope}[shift={(8,5)}] 
        \node (a) at (0,1) {$ \bullet $};
        \draw [rounded corners] (-1,-1) rectangle (1,3);
        \node (21l) at (-1.1,1) {$  $};
        \node (21b) at (0,-1.1) {$  $};
      \end{scope}
      \begin{scope}[shift={(0,0)}] 
        \node (b) at (0,1) {$ \bullet $};
        \draw [rounded corners] (-1,-1) rectangle (1,3);
        \node (00r) at (1.1,1) {$  $};
        \node (00t) at (0,3.1) {$  $};
      \end{scope}
      \begin{scope}[shift={(3,0)}] 
        \node (b) at (0,1) {$ \bullet $};
        \node (c) at (2,1) {$ \bullet $};
        \draw [graph] 
          (b) edge[] (c);
        \draw [rounded corners] (-1,-1) rectangle (3,3);
        \node (10l) at (-1.1,1) {$  $};
        \node (10r) at (3.1,1) {$  $};
        \node (10t) at (1,3.1) {$  $};
      \end{scope}
      \begin{scope}[shift={(8,0)}] 
        \node (a) at (0,1) {$ \bullet $};
        \draw [rounded corners] (-1,-1) rectangle (1,3);
        \node (20l) at (-1.1,1) {$  $};
        \node (20t) at (0,3.1) {$  $};
      \end{scope}
      \draw [cd]
        (01r) edge[] (11l)
        (21l) edge[] (11r)
        (00r) edge[] (10l)
        (20l) edge[] (10r)
        (01b) edge[] (00t)
        (11b) edge[] (10t)
        (21b) edge[] (20t); 
    \end{tikzpicture}
  \end{center}
  The tensor is the disjoint union of open graphs.  For
  example, tensoring
  \[
    \begin{tikzpicture}
      \begin{scope}
        \node (11) at (0,0) {$ \bullet $};
        \node (12) at (0,2) {$ \bullet $};
        \node (1r) at (1.1,1) {$ $};
        \draw [rounded corners] (-1,-1) rectangle (1,3);
      \end{scope}
      \begin{scope}[shift={(4,0)}]        
        \node (11) at (-1,0) {$ \bullet $};
        \node (21) at (1,0) {$ \bullet $};
        \node (12) at (-1,2) {$ \bullet $};
        \draw[graph]
          (11) edge (12)
          (11) edge (21)
          (12) edge (21);
        \node (2l) at (-2.1,1) {};
        \node (2r) at (2.1,1)  {};
        \draw [rounded corners] (-2,-1) rectangle (2,3);
      \end{scope}
      \begin{scope}[shift={(8,0)}]
        \node () at (0,1) {$ 0 $};
        \node (3l) at (-1.1,1) {$  $};
        \draw [rounded corners] (-1,-1) rectangle (1,3);
      \end{scope}
      \draw [cd]
        (1r) edge (2l)
        (3l) edge (2r);
    \end{tikzpicture}
  \]
  together with
  \[
    \begin{tikzpicture}
      \begin{scope}
        \node (11) at (0,1) {$ \bullet $};
        \node (1r) at (1.1,1) {$  $};
        \draw [rounded corners] (-1,-1) rectangle (1,3);
      \end{scope}
      \begin{scope}[shift={(4,0)}]        
        \node (11) at (-1,1) {$ \bullet $};
        \node (12) at (1,1) {$ \bullet $};
        \draw[graph] (11) edge (12);
        \node (2l) at (-2.1,1) {};
        \node (2r) at (2.1,1)  {};
        \draw [rounded corners] (-2,-1) rectangle (2,3);
      \end{scope}
      \begin{scope}[shift={(8,0)}]
        \node () at (0,1) {$ \bullet $};
        \node (3l) at (-1.1,1) {$  $};
        \draw [rounded corners] (-1,-1) rectangle (1,3);
      \end{scope}
      \draw[cd]
        (1r) edge (2l)
        (3l) edge (2r);  
    \end{tikzpicture}
  \]
  gives the open graph
  \begin{center}
    \begin{tikzpicture}
      \begin{scope}
        \node (11) at (0,0) {$ \bullet $};
        \node (12) at (0,1) {$ \bullet $};
        \node (13) at (0,2) {$ \bullet $};
        \node (1r) at (1.1,1) {$  $};
        \draw [rounded corners]
          (-1,-1) rectangle (1,3);
      \end{scope}
      \begin{scope}[shift={(4,0)}]
        \node (11) at (-1,0) {$ \bullet $};
        \node (21) at (1,0) {$ \bullet $};
        \node (12) at (-1,1) {$ \bullet $};
        \node (13) at (-1,2) {$ \bullet $};
        \node (23) at (1,2) {$ \bullet $};
        \draw[graph]
          (11) edge (12)
          (11) edge (21)
          (12) edge (21)
          (13) edge (23);
        \node (2l) at (-2.1,1) {};
        \node (2r) at (2.1,1)  {};  
        \draw [rounded corners] (-2,-1) rectangle (2,3);
      \end{scope}
      \begin{scope}[shift={(8,0)}]
        \node () at (0,2) {$ \bullet $};
        \node (3l) at (-1.1,1) {$  $};
        \draw [rounded corners] (-1,-1) rectangle (1,3);
      \end{scope}
      \draw[cd]
        (1r) edge (2l)
        (3l) edge (2r);  
    \end{tikzpicture}
  \end{center}
\end{example}

This double category is explored further by Baez and Courser
\cite{baez-courser_str-csp}.  For us, it is a nice structure
in which to simultaneously present the compositional role
and the object role of structured cospans.


\section{Spans of structured cospans}
\label{sec:spans-structured-cospans}

For this final section of the chapter, we define spans of
structured cospans. These are the objects that serve as
rewrite rules. We bring the two flavors of rewriting, fine
and bold, to structured cospans in Chapter
\ref{sec:fine-rewriting} and Chapter
\ref{sec:bold-rewriting}. This section segues to
those two chapters.

We continue to work with a adjunction
\[
  \adjunction{\A}{\X}{L}{R}{2}
\]
with $ L $ preserving pullbacks.

\begin{definition}
  \label{def:span-str-cospan}
  A \defn{span of structured cospans} is a commuting diagram
  \[
    \begin{tikzpicture}
      \begin{scope}
      \node (La) at (0,4) {$ La $};
      \node (x) at (2,4) {$ x $};
      \node (La') at (4,4) {$ La' $};
      \node (Lb) at (0,2) {$ Lb $};
      \node (y) at (2,2) {$ y $};
      \node (Lb') at (4,2) {$ Lb' $};
      \node (Lc) at (0,0) {$ Lc $};
      \node (z) at (2,0) {$ z $};
      \node (Lc') at (4,0) {$ Lc' $};
      \draw[cd]
      (La) edge node[]{$  $} (x)
      (La') edge node[]{$  $} (x)
      (Lb) edge node[]{$  $} (y)
      (Lb') edge node[]{$  $} (y)
      (Lc) edge node[]{$  $} (z)
      (Lc') edge node[]{$  $} (z)
      (Lb) edge node[]{$  $} (La)
      (Lb) edge node[]{$  $} (Lc)
      (y) edge node[]{$  $} (x)
      (y) edge node[]{$  $} (z)
      (Lb') edge node[]{$  $} (La')
      (Lb') edge node[]{$  $} (Lc');
    \end{scope}
    \end{tikzpicture}
  \]
 \end{definition}

 Spans of cospans (not structured cospans) were considered
 by Kissinger in his thesis \cite{kissinger_pictures} and
 also by Grandis and Par\'{e} in
 \cite{grandis-pare_intercats}. They did not fit them into a
 categorical structure as we do in latter chapters.  For us,
 they will be squares in a double category for which we
 need to introduce horizontal composition $ \hcirc $ and
 vertical composition $ \vcirc $. The compositions use
 pushouts and pullbacks, which are only defined up to
 isomorphism. It follows that we will need to consider
 \emph{classes} of spans of cospans, the specifics of which
 we put off until introducing the fine rewriting and bold
 rewriting of structured cospans.  For now, we define a
 morphism of spans of structured cospans.
 
 \begin{definition}
   \label{def:morphism-span-str-cospans}
   A \defn{morphism of spans of structured cospans} from
  \[
    \begin{tikzpicture}
      \begin{scope}
      \node (a) at (0,4) {$ La $};
      \node (b) at (2,4) {$ x $};
      \node (c) at (4,4) {$ Lb $};
      \node (x) at (0,2) {$ Lc $};
      \node (y) at (2,2) {$ y $};
      \node (z) at (4,2) {$ Ld $};
      \node (d) at (0,0) {$ Le $};
      \node (e) at (2,0) {$ z$};
      \node (f) at (4,0) {$ Lf $};
      \draw[cd]
      (x) edge[] (a)
      (y) edge[] (b)
      (z) edge[] (c)
      (x) edge[] (d)
      (y) edge[] (e)
      (z) edge[] (f)
      (a) edge[] (b)
      (x) edge[] (y)
      (d) edge[] (e) 
      (c) edge[] (b) 
      (z) edge[] (y)
      (f) edge[] (e); 
    \end{scope}
    \node () at (5,2) {to};
    \begin{scope}[shift={(6,0)}]
      \node (a) at (0,4) {$ La $};
      \node (b) at (2,4) {$ x $};
      \node (c) at (4,4) {$ Lb $};
      \node (x) at (0,2) {$ Lc $};
      \node (y) at (2,2) {$ y' $};
      \node (z) at (4,2) {$ Ld $};
      \node (d) at (0,0) {$ Le $};
      \node (e) at (2,0) {$ z$};
      \node (f) at (4,0) {$ Lf $};
      \draw[cd]
      (x) edge[] (a)
      (y) edge[] (b)
      (z) edge[] (c)
      (x) edge[] (d)
      (y) edge[] (e)
      (z) edge[] (f)
      (a) edge[] (b)
      (x) edge[] (y)
      (d) edge[] (e) 
      (c) edge[] (b)
      (z) edge[] (y)
      (f) edge[] (e); 
    \end{scope}
  \end{tikzpicture}
  \]
  is an arrow $ \theta \from y \to y' $ that fits into a
  commuting diagram
\[
  \begin{tikzpicture}
    \node (La)  at (0,4)   {$ La $};
    \node (Lb)  at (1,2)   {$ Lb $};
    \node (Lc)  at (2,0)   {$ Lc $};
    \node (x)   at (3,4)   {$ x $};
    \node (y)   at (4,3.5) {$ y $};
    \node (y')  at (4,0.5) {$ y' $};
    \node (z)   at (5,0)   {$ z $};
    \node (La') at (6,4)   {$ La' $};
    \node (Lb') at (7,2)   {$ Lb' $};
    \node (Lc') at (8,0)   {$ Lc' $};
    \draw[cd]
    (Lb) edge node[]{$  $} (La)
    (Lb) edge node[]{$  $} (Lc)
    (Lb') edge node[]{$  $} (La')
    (Lb') edge node[]{$  $} (Lc')
    (y) edge node[]{$  $} (x)
    (y') edge node[]{$  $} (x)
    (y') edge node[]{$  $} (z)
    (La) edge node[]{$  $} (x)
    (La') edge node[]{$  $} (x)
    (Lb) edge node[]{$  $} (y')
    (Lb') edge node[]{$  $} (y');
    \draw[line width=0.5em,draw=white]
    (Lb) edge node[]{$  $} (y)
    (Lb') edge node[]{$  $} (y)
    (Lc) edge node[]{$  $} (z)
    (Lc') edge node[]{$  $} (z)
    (y) edge node[]{$  $} (z)
    (y) edge node[]{$  $} (y');
    \draw[cd]
    (Lb) edge node[]{$  $} (y)
    (Lb') edge node[]{$  $} (y)
    (Lc) edge node[]{$  $} (z)
    (Lc') edge node[]{$  $} (z)
    (y) edge node[]{$  $} (z)
    (y) edge node[left] {$ \theta $} (y');
  \end{tikzpicture}
  \]  
  If $ \theta $ is invertible, then the morphism
  is an isomorphism.
\end{definition}

We now have our syntactical device in hand.  As previously
stated, our goal is to incorporate rewriting.  To do so, we
spend the next chapter covering rewriting in a general
setting before moving on to focus solely on rewriting
structured cospans.


\chapter{Double pushout rewriting}
\label{sec:dpo-rewriting}

Our primary aim is to develop a theory of rewriting for open
systems. This goal fits into a larger program of studying
the ``linguistics'' of open systems. By this we mean
designating syntax and semantics. Rewriting lives on the
syntactical side of this divide.

To develop an intuition for rewriting, we provide a sliver
of its broader story. We chase from its beginnings in
linguistics to double pushout graph rewriting to the modern
day axioms of adhesive categories (see Appendix
\ref{sec:adhesive-categories}). The most important example
of an adhesive category for us is a topos. This fact
highlights the importance of structured cospans forming a
topos (Theorem \ref{thm:strcsp-istopos}) and it cements
our ability to rewrite open systems.


\section{A brief history of rewriting}
\label{sec:brief-history}

We prefer to sketch the theory of rewriting rather than
delve into details. For us, it is enough to build an intuition for
rewriting prior to introducing it to open systems via
structured cospans.

The theory of rewriting arose from Chomsky's work in formal
languages \cite{chomsky-linguistics}. He used rewriting as a
device to generate well-formed sentences.  While a
well-formed sentence must be grammatically sound, it need
not mean anything. Chomsky's \cite{chomsky-linguistics}
classic example of a grammatically sound but meaningless
sentence is
\begin{quote}
  \centering
  `Colorless green ideas sleep furiously.'
\end{quote}
That this sentence is syntactically good but semantically
bad helps to highlight the difference between syntax and
semantics.  How does one use rewriting, in Chomsky's sense,
to build that sentence?

We begin with a collection of rewrite rules:
\begin{enumerate}
\item a sentence is a noun phrase followed by a verb phrase;
\item a verb phrase consists of a verb and the option to
  follow with an adverb;
\item a noun phrase can be a noun with, optionally, a
  preceding determiner such as an article, demonstrative,
  quantifier, etc;
\item a noun phrase can be a noun with, optionally, a
  preceding adjective phrase or, optionally, a
  prepositional phrase.
\end{enumerate}
These rules are denoted as follows:
\begin{align*}
  \textrm{S}     & \to \textrm{NP VP}           \\
  \textrm{VP}    & \to \textrm{VP (Adv.)}       \\
  \textrm{NP}    & \to \textrm{(Det.) NP}          \\
  \textrm{NP}    & \to \textrm{(AP) NP (PP)}
\end{align*}

To derive a sentence, first apply a rule to S, then apply a
rule to that first step's output, and so on. Eventually, no
further rules are applicable at which point we are left with
a grammatically sound sentence. The derivation of the above
sentence is
\[
  \begin{tikzpicture}
    \node () at (0,0) {Colorless green ideas sleep
      furiously};
    \node (A1) at (-2.25,1) {AP};
    \node (A2) at (-1,1) {AP};
    \node (N1) at (0,1) {NP};
    \node (V1) at (1,1) {VP};
    \node (Adv) at (2.25,1) {Adv.};
    \node (NP1) at (-0.5,2) {NP};
    \node (VP) at (1.6,2) {VP};
    \node (NP2) at (-1.3,3) {NP};
    \node (S) at (0,4) {S};
    \draw [-] (S) to (NP2);
    \draw [-] (S) to (VP);
    \draw [-] (NP2) to (A1);
    \draw [-] (NP2) to (NP1);
    \draw [-] (NP1) to (A2);
    \draw [-] (NP1) to (N1);
    \draw [-] (VP) to (V1);
    \draw [-] (VP) to (Adv);
    \draw [-] (A1) to (-2.25,0.5);
    \draw [-] (A2) to (-1,0.5);
    \draw [-] (N1) to (0,0.5);
    \draw [-] (V1) to (1,0.5);
    \draw [-] (Adv) to (2.25,0.5);
  \end{tikzpicture}
\]

The success of rewriting in linguistics led to its use in
logic and mathematics. One evolution of rewriting into
mathematics is an \emph{Abstract Rewriting System}, a set
$ A $ together with a binary relation $ A \rel A $. An
element of this relation $ (a,b) $ means that you can
`reduce' $ a $ to $ b $.  Often, one studies the transitive
and reflexive closure of $ A \rel A $ which we denote by
adorning the arrow with an asterisk $ A \rel^\ast A $. This
so-called \defn{rewriting relation} $ \rel^\ast $ accounts
for reflexive and multi-step reductions.

\begin{example}
  The word problem can be expressed in terms of abstract
  rewriting. Let $ M $ be the set underlying a monoid, let
  $ FM $ be the free monoid on $ M $, and let $ \rel $ be a
  binary relation on $ FM $ given by
  $ x_1 \cdots x_n \rel x $, with $ x_i \in M $, whenever
  $ x_1 \cdots x_n = x $ in the monoid $ M $.  The word problem asks,
  ``if given words $ w $, $ w' $ in $ FM $, does
  $ w \rel^\ast w' $ and $ w' \rel^\ast w $''?
\end{example}

As just seen, we can determine whether syntactical
expressions, such as words in a free monoid, are equivalent
using rewriting.  It is in this sense, not in generating
sentences, that we are interested in rewriting.

We are particularly interested `structured cospans', a
syntactical device Baez and Courser introduced
\cite{baez-courser_str-csp} as a written language for open
systems. In order to develop a theory of rewriting for
structured cospans, we need more sophisticated machinery
than abstract rewriting systems.

A first step in that direction is graph rewriting, invented
by Ehrig, et.~al.~\cite{ehrig_graph-grammars}, where graphs
are used in place of words and sentences. Rules are used to
choose a subgraph and replace it with another
equivalent\footnote{
We mean `equivalent' in a semantic sense, thus varying with context.}
graph.  Ehrig, et.~al.~encode rewrite rules in spans of
graphs and apply a rule using pushouts. That is, a rule is a
span of graphs
\[
  \spn{\ell}{k}{r}
\]
We interpret this rule to say \emph{any instance of a
  sub-graph isomorphic to $ \ell $ can be replaced by the
  graph $ r $}.

Given such a rule and a graph $ g $, how do we identify a
copy of $ \ell $ inside of $ g $ and then replace it with
$ r $? The answer lies in the following definition.

\begin{definition}[Double pushout]
  A \defn{double pushout diagram} is a a pair of pushouts
\[
  \begin{tikzpicture}
    \node (1) at (0,2) {$ \ell $};
    \node (2) at (2,2) {$ d $};
    \node (3) at (4,2) {$ r $};
    \node (4) at (0,0) {$ g $};
    \node (5) at (2,0) {$ k $};
    \node (6) at (4,0) {$ h $};
    \draw[cd]
    (2) edge node[]{$  $} (1)
    (2) edge node[]{$  $} (3)
    (5) edge node[]{$  $} (4)
    (5) edge node[]{$  $} (6)
    (1) edge node[]{$  $} (4)
    (2) edge node[]{$  $} (5)
    (3) edge node[]{$  $} (6);
    \draw (0.3,0.4) -- (0.4,0.4) -- (0.4,0.3);
    \draw (3.7,0.4) -- (3.6,0.4) -- (3.6,0.3);
  \end{tikzpicture}
  \]
  that share an arrow as depicted. 
\end{definition}

While double pushout diagrams make sense in any category,
graph rewriting restricts to the category $ \Graph $ of
directed graphs and their morphisms. So in the diagram
above, each letter represents a graph and the arrows are
graph morphisms.  The rule being applied is
$ \spn{\ell}{k}{r} $ and the output of applying this rule to
$ g $ is the graph $ h $. The graph $ k $ is what holds
fixed as $ r $ replaces $ \ell $ and $ d $ is what holds
fixed as $ h $ replaces $ g $.  A concrete example of this
is given below in Equation \eqref{eq:example-dpo-diagram}.

Observers noticed that the mechanisms did not require
anything specific about graphs to work.  Pushouts and spans
are basic constructions in category theory, so it is
reasonable to consider extending double pushout rewriting to
a broader class of categories than just $ \Graph $.  There
were a number of attempts to axiomatize the important
properties of graph rewriting, the most prominent example
being `high level replacement systems'
\cite{ehrig_graph-to-hlr}, which we discuss in Section
\ref{sec:adhesive-categories}.  The drawback of HLRS's was
the sheer number of axioms. Lack and Sobocinski eventually
found a much shorter list of axioms. They called categories
that satisfy their axioms `adhesive categories'
\cite{lack-sobo_adhesive-cats} (see Appendix
\ref{sec:adhesive-categories}). Adhesive categories are
currently the most general setting in which rewriting theory
holds. However, we don't need the full generality of
adhesive categories and instead focus on topoi, each of
which is adhesive.

As mentioned earlier, there are different ways to interpret
what a \emph{rewriting} is. For instance, `a
rewriting is making a choice' or `a rewriting is a
simplification'.  The interpretation for our needs is `a
rewriting is to replace by a behaviorally indistinguishable
system'. The linguistic analogy is `synonym'.

Though the focus of this thesis is on the syntax
of open systems, the semantics of systems cannot
fully be ignored.  By the syntax of a system, we
mean the rules followed by its diagrammatic
representations. By the semantics of a system, we
mean the behavior of a system.  For example,
consider resistor systems. It is a syntactic
issue that the circuit diagram
\[
  \begin{tikzpicture}
    \draw (0,0) rectangle (1,0.5);
    \draw (2,0) rectangle (3,0.5);
    \draw (-1,0.25) to (0,0.25);
    \draw (1,0.25) to (2,0.25);
    \draw (3,0.25) to (4,0.25);
    \node at (0.5,0.25) {\scriptsize{$ 25 \Omega $}};
    \node at (2.5,0.25) {\scriptsize{$ 35 \Omega $}};
  \end{tikzpicture}
\]
makes sense but
\[
  \begin{tikzpicture}
    \draw (0,0) rectangle (1,0.5);
    \draw (0.2,-1) rectangle (0.8,-2);
    \draw (-1,0.25) to (0,0.25);
    \draw (1,0.25) to (2,0.25);
    \draw (0.5,0) to (0.5,-1);
    \draw (0.5,-2) to (0.5,-3);
    \node at (0.5,0.25) {\scriptsize{$ 25 \Omega $}};
    \node at (0.5,-1.5) {\scriptsize{$ 35 \Omega $}};    
  \end{tikzpicture}
\]
does not. A semantic consideration is that the resistor
circuit
\[
  \begin{tikzpicture}
    \draw (0,0) rectangle (1,0.5);
    \draw (2,0) rectangle (3,0.5);
    \draw (-1,0.25) to (0,0.25);
    \draw (1,0.25) to (2,0.25);
    \draw (3,0.25) to (4,0.25);
    \node at (0.5,0.25) {\scriptsize{$ 25 \Omega $}};
    \node at (2.5,0.25) {\scriptsize{$ 35 \Omega $}};
  \end{tikzpicture}
\]
behaves in the same exact way as the circuit
\[
  \begin{tikzpicture}
    \draw (0,0) rectangle (1,0.5);
    \draw (-1,0.25) to (0,0.25);
    \draw (1,0.25) to (2,0.25);
    \node at (0.5,0.25) {\scriptsize{$ 60 \Omega $}};
  \end{tikzpicture}
\]
This follows from Ohm's law.  Syntactically, these
are two different circuits. Building a
rewriting theory into our structured cospan
formalism provides our system syntax a mechanism
to recognize semantically (i.e.~behaviorly)
indistinguishable systems. 


\section{Rewriting in topoi}
\label{sec:rewriting-topoi}

Fix a topos $ \T $.  Rewriting starts with the notion of a
\defn{rewrite rule}, or simply \defn{rule}. In its
most general form, a rule is a span
\[
  \ell \gets k \to r
\]
in $ \T $. The arrows are left unnamed unless we need to
refer to them. For us, rules come in two flavors.  A
\defn{fine rule} is one in which both of the span arrows are
monic. A \defn{bold rule} is one without restriction on the
arrows.

\begin{remark}
  \label{rmk:rewrite-rules-monic}
  Both fine and bold approaches are considered in the
  rewriting literature, but often by the name `linear' and
  `non-linear', respectively. Fine rewriting is more common.
  Habel, Muller, and Plump compared these alternatives
  in the context of graph rewriting
  \cite{hmp_dpo-revisited}. The distinction between the two
  cases does not appear in this chapter, and everything we say
  carries through in either case. We do take care to ensure
  that constructions are well-defined in the monic case.
\end{remark}

The conceit of a rule is that $ r $ replaces $ \ell $ while
$ k $ identifies a subsystem of $ \ell $ that remains
fixed. For example, suppose we were modeling some system
using graphs where self-loops were meaningless. In the
introduction, we considered modeling the internet with a
graph with websites as nodes and links as edges. If we did
not care about websites with a link to itself, we would
introduce a rule that replaces a node with a loop with a
node

\begin{equation}
  \label{eq:squash-loop-rule}
  \begin{tikzpicture}
    \draw [rounded corners]
      (-1,-1) rectangle (1,2);
    \node (a)  at (0,0.5)   {$ \bullet $};
    \node (lr) at (1.1,0.5) {};  
    \draw [graph,loop above] (a) to (a);
    \draw [rounded corners]
      (2,-1) rectangle (4,2);
    \node (b)  at (3,0.5)   {$ \bullet $};
    \node (ml) at (1.9,0.5) {};
    \node (mr) at (4.1,0.5) {};
    \draw [rounded corners]
      (5,-1) rectangle (7,2);
    \node (c)  at (6,0.5) {$ \bullet $};
    \node (rl) at (4.9,0.5) {};
    \draw [cd]
      (ml) edge node[]{$  $} (lr)
      (mr) edge node[]{$  $} (rl); 
  \end{tikzpicture}
\end{equation}

For another example, suppose we had another system modeled
on graphs where an edge between two nodes is equivalent to
having a single node. This is captured with the rule
\begin{equation}
  \label{eq:edge-identity-rule}
   \begin{tikzpicture}
    \draw [rounded corners] (-1,-1) rectangle (1,2);
    \node (a)  at (0,0)     {$ \bullet $};
    \node (a') at (0,1)     {$ \bullet $};
    \node (lr) at (1.1,0.5) {$  $};
    \draw [graph] (a) to (a');
    \draw [rounded corners] (2,-1) rectangle (4,2);
    \node (b)  at (3,0)     {$ \bullet $};
    \node (b') at (3,1)     {$ \bullet $};
    \node (ml) at (1.9,0.5) {$  $};
    \node (mr) at (4.1,0.5) {$  $};
    \draw [rounded corners] (5,-1) rectangle (7,2);
    \node (c)  at (6,0.5)   {$ \bullet $};
    \node (rl) at (4.9,0.5) {$  $};
    \draw [cd]
      (ml) edge (lr)
      (mr) edge (rl); 
  \end{tikzpicture}
\end{equation}
This rule appears in the ZX-calculus example from Section
\ref{sec:zx-calculus}. Observe that the first example is a
fine rewrite and the second is a bold rewrite.

To \emph{apply} a rule $ \ell \gets k \to r $ to an object
$ g $, we require an arrow $ m \from \ell \to g $ such that
there exists a \defn{pushout complement}, an object $ d $
fitting into a pushout diagram
\[
  \begin{tikzpicture}
    \node (l) at (0,2) {$ \ell $};
    \node (k) at (2,2) {$ k $};
    \node (g) at (0,0) {$ g $};
    \node (d) at (2,0) {$ d $};
    \draw [cd]
       (k) edge node[]{$  $}      (l)
       (k) edge node[]{$  $}      (d)
       (l) edge node[left]{$ m $} (g)
       (d) edge node[]{$  $}      (g); 
    \draw (0.3,0.4) -- (0.4,0.4) -- (0.4,0.3);
  \end{tikzpicture}
\]
A pushout complement need not exist, but when it does and
the map $ k \to \ell $ is monic, then it is unique up to
isomorphism \cite[Lem.~15]{lack-sobo_adhesive-cats}.

For each application of a rule, we derive a new rule.

\begin{definition}[Derived rule]
  \label{def:derived-rule}
A \defn{derived rule} is any span $ g \gets d \to h $
fitting into the bottom row of the double pushout diagram
\[
  \begin{tikzpicture}
    \node (1) at (0,2) {$ \ell $};
    \node (2) at (2,2) {$ k $};
    \node (3) at (4,2) {$ r $};
    \node (4) at (0,0) {$ g $};
    \node (5) at (2,0) {$ d $};
    \node (6) at (4,0) {$ h $};
    \draw [cd]
    (2) edge (1)
    (2) edge (3)
    (5) edge (4)
    (5) edge (6)
    (1) edge (4)
    (2) edge (5)
    (3) edge (6); 
    \draw (0.3,0.4) -- (0.4,0.4) -- (0.4,0.3);
    \draw (3.7,0.4) -- (3.6,0.4) -- (3.6,0.3);
  \end{tikzpicture}
\]
\end{definition}

When the arrows of the rule $ \ell \gets k \to r $
are both monic, the arrows of the span
$ g \gets d \to h $ are also monic because
pushouts preserve monics in topoi
\cite[Lem.~12]{lack-sobo_adhesive-cats}. The
intuition of this diagram is that $ \ell \to g $
identifies a copy of $ \ell $ in $ g $ and we replace
that copy with $ r $, resulting in a new object
$ h $.

To illustrate, let us return to a system modeled with
graphs and where self-loops are meaningless.  Then we can
apply Rule \eqref{eq:squash-loop-rule} to any node with a
loop. This application is captured with the double pushout
diagram
\begin{equation} \label{eq:example-dpo-diagram}
 \begin{tikzpicture}
    \begin{scope} 
      \node (c) at (0,0.5) {$ \bullet $};
      \draw [graph,loop above]
        (c) edge (c);
      \draw [rounded corners]
        (-1,-1) rectangle (1,2);
      \node (01r) at (1.1,0)  {};
      \node (01b) at (0,-1.1) {};  
    \end{scope}
    \begin{scope}[shift={(4,0)}]
      \node (c) at (0,0.5) {$ \bullet $};
      \draw [rounded corners]
        (-1,-1) rectangle (1,2);
      \node (11l) at (-1.1,0) {};
      \node (11r) at (1.1,0)  {};
      \node (11b) at (0,-1.1) {};
    \end{scope}
    \begin{scope}[shift={(8,0)}]
      \node (c) at (0,0.5) {$ \bullet $};
      \draw [rounded corners]
        (-1,-1) rectangle (1,2);
      \node (21l) at (-1.1,0) {};
      \node (21b) at (0,-1.1) {};
    \end{scope}
    \begin{scope}[shift={(0,-5)}]
      \node (a) at (-0.5,-0.5) {$ \bullet $};
      \node (b) at (0.5,-0.5)  {$ \bullet $};
      \node (c) at (0,0.5)     {$ \bullet $};
      \draw [graph]
      (a) edge              (c)
      (b) edge              (c)
      (c) edge [loop above] (c);
      \draw [rounded corners]
        (-1,-1) rectangle (1,2);
      \node (00t) at (0,2.1)   {};
      \node (00r) at (1.1,0.5) {};  
    \end{scope}
    \begin{scope}[shift={(4,-5)}]
      \node (a) at (-0.5,-0.5) {$ \bullet $};
      \node (b) at (0.5,-0.5) {$ \bullet $};
      \node (c) at (0,1) {$ \bullet $};
      \draw [graph]
        (a) edge (c)
        (b) edge (c);
      \draw [rounded corners]
        (-1,-1) rectangle (1,2);
      \node (10t) at (0,2.1)    {};
      \node (10r) at (1.1,0.5)  {};
      \node (10l) at (-1.1,0.5) {};
    \end{scope}
    \begin{scope}[shift={(8,-5)}]
      \node (a) at (-0.5,-0.5) {$ \bullet $};
      \node (b) at (0.5,-0.5)  {$ \bullet $};
      \node (c) at (0,1)    {$ \bullet $};
      \draw [graph]
      (a) edge              (c)
      (b) edge              (c);
      \draw [rounded corners]
        (-1,-1) rectangle (1,2);
      \node (20t) at (0,2.1)    {};
      \node (20l) at (-1.1,0.5) {};
      \end{scope}
      \draw [cd]
      (01b) edge (00t)
      (11b) edge (10t)
      (21b) edge (20t)
      (11l) edge (01r)
      (11r) edge (21l)
      (10l) edge (00r)
      (10r) edge (20l);
    \end{tikzpicture}
\end{equation}
We identified a self-loop in the bottom left graph then
applied the rule to remove it. The result
is the bottom right graph.  The reader can check that the
two squares are pushouts.

Usually when modeling a system, there is a set of rewrite
rules that accompany it.  For example, in resistor circuits there
are parallel, series, and star rules. Just like in natural
languages, we call a collection of rules a grammar.

\begin{definition}[Grammar]
  \label{def:grammar}
  A topos $ \T $ together with a finite set $ P $ of rules
  $ \{ \spn{\ell_j}{k_j}{r_j} \} $ in $ \T $ is a
  \defn{grammar}. When the all rules in a grammar have monic
  arrows, we say the grammar is \defn{fine}.  Else, the
  grammar is \defn{bold}.  An arrow of (fine, bold) grammars
  $ ( \S , P ) \to ( \T , Q ) $ is a pullback and pushout
  preserving functor $ F \from \S \to \T $ such that for
  each rule $ \xspn{\ell}{f}{k}{g}{r} $ in $ P $, the rule
  $ \xspn{F\ell}{Ff}{Fk}{Fg}{Fr} $ is in $ Q $. Together
  these form a category $ \Gram $.
\end{definition}

A grammar is a seed. Like a seed, the grammar gives birth to
something entirely new and more complex called the
language. It is this language that we are interested more so
than the grammar.  We can certainly learn about the language
from the grammar, but what we actually study is the `rewrite
relation' which informs us about how different components of
the language relate. Every grammar $ ( \T , P ) $ gives rise
to a relation $ \dderiv{}{} $ on the objects of $ \T $
defined by $ \dderiv{g}{h} $ whenever there exists a rule
$ \spn{g}{d}{h} $ derived from a production in $ P $. For
instance, the above double pushout diagram would relate
\[
  \begin{tikzpicture}
    \begin{scope}[shift={(0,0)}]
      \node (a) at (-0.5,-0.5) {$ \bullet_a $};
      \node (b) at (0.5,-0.5)  {$ \bullet_b $};
      \node (c) at (0,0.5)     {$ \bullet_c $};
      \draw [graph]
        (a) edge              (c)
        (b) edge              (c)
        (c) edge [loop above] (c); 
      \draw [rounded corners]
        (-1,-1) rectangle (1,2);
    \end{scope}
    \node at (2.5,0.5) {to};
    \begin{scope}[shift={(5,0)}]
      \node (a) at (-0.5,-0.5) {$ \bullet_a $};
      \node (b) at (0.5,-0.5)  {$ \bullet_b $};
      \node (c) at (0,0.5)       {$ \bullet_c $};
      \draw [graph]
        (a) edge              (c) 
        (b) edge              (c);
      \draw [rounded corners]
        (-1,-1) rectangle (1,2);
    \end{scope}
  \end{tikzpicture}
\]
But $ \dderiv{}{} $ is too small to capture the full behavior of
the language.  For one, it is not true in general that
$ \dderiv{g}{g} $ holds. Also, $ \dderiv{}{} $ does not
capture multi-step rewrites. That is, there may be derived
rules witnessing $ \dderiv{g}{g'} $ and $ \dderiv{g'}{g''} $
but not a derived rule witnessing $ \dderiv{g}{g''} $. We
want to relate a pair of objects if one can be rewritten
into another with a finite sequence of derived
rules. Therefore, we actually want the following.

\begin{definition}[Rewrite relation]
  \label{def:rewrite-relation}
  To each grammar $ ( \T, P ) $, assign a relation on the
  objects of $ \T $ defined by setting $ \dderiv{g}{h} $
  whenever there is a rewrite rule $ \spn{\ell}{k}{r} $ in
  $ P $ and an object $ d $ of $ \T $ that fit into a double
  pushout diagram
  \begin{center}
      \begin{tikzpicture}
    \node (1) at (0,2) {$ \ell $};
    \node (2) at (2,2) {$ k $};
    \node (3) at (4,2) {$ r $};
    \node (4) at (0,0) {$ g $};
    \node (5) at (2,0) {$ d $};
    \node (6) at (4,0) {$ h $};
    \draw [cd]
    (2) edge (1)
    (2) edge (3)
    (5) edge (4)
    (5) edge (6)
    (1) edge (4)
    (2) edge (5)
    (3) edge (6); 
    \draw (0.3,0.4) -- (0.4,0.4) -- (0.4,0.3);
    \draw (3.7,0.4) -- (3.6,0.4) -- (3.6,0.3);
  \end{tikzpicture}
  \end{center}
  The \defn{rewrite relation} $ \deriv{}{} $ is the
  transitive and reflexive closure of $ \dderiv{}{} $.
\end{definition}

Every grammar determines a unique rewrite relation in a
functorial way. We devote Section \ref{sec:Rewriting-StrCsp}
to proving this fact, though, we restrict ourselves working
with grammars of structured cospan categories.


\chapter{Fine rewriting and structured cospans}
\label{sec:fine-rewriting}

In this chapter, we introduce a theory of fine rewriting to
structured cospans.  Rewriting is \emph{fine}
when the rewrite rules are spans with monic legs.  Our
primary goal is to define a double category whose squares
are fine rewrites of structured cospans.  The rough idea is
that this double category, denoted $ _L\FFFineRewrite $, has
interface types for objects, structured cospans for
horizontal arrows, isomorphisms of interface objects for
vertical arrows, and fine rewrite rules of structured
cospans for squares. We prove in Proposition
\ref{thm:fine-rewrite-double-cat} that $ _L\FFFineRewrite $
actually is a double category. The first step to proving
this is to ensure the fine rewrite rules are suitable squares for
our double category, we define them as follows.

\begin{definition}[Fine rewrite]
  A \defn{fine rewrite of structured cospans} is an
  isomorphism class of spans of structured cospans of the
  form
  \[
    \begin{tikzpicture}
    \node (a) at (0,4) {$ La $};
    \node (b) at (2,4) {$ x $};
    \node (c) at (4,4) {$ Lb $};
    \node (x) at (0,2) {$ Lc $};
    \node (y) at (2,2) {$ y $};
    \node (z) at (4,2) {$ Ld $};
    \node (d) at (0,0) {$ Le $};
    \node (e) at (2,0) {$ z$};
    \node (f) at (4,0) {$ Lf $};
    \draw [>->] (y) to (b);
    \draw [>->] (e) to (y);
    \draw[cd]
    (a) edge[] (b)
    (x) edge[] (y)
    (d) edge[] (e)
    (c) edge[] (b)
    (z) edge[] (y)
    (f) edge[] (e)
    (x) edge[] node[left]{$ \iso $}  (a)
    (z) edge[] node[right]{$ \iso $} (c)
    (x) edge[] node[left]{$ \iso $}  (d)
    (z) edge[] node[right]{$ \iso $} (f) ;
  \end{tikzpicture}
  \]
  The marked arrows are monic.
\end{definition}

In a double category, the squares have two composition
operations.  Horizontal composition uses pushout as is
typical with cospan categories. The vertical composition
uses pullback as is typical in span categories.  But because
there are no higher order arrows traversing the squares in a
double category, and because pushouts and pullbacks are only
defined up to isomorphism, we take isomorphism classes of
structured cospan rewrite rules.  With the squares of
$ _L\FFFineRewrite $ defined, we can introduce the two
composition operations.

\begin{definition} \label{def:hor-vert-composition}

  The \defn{horizontal composition} of fine rewrite rules is
  given by
  \[
  \begin{tikzpicture}
    \begin{scope}
    \node (a) at (0,4) {$ La $};
    \node (b) at (0,2) {$ Lb $};
    \node (c) at (0,0) {$ Lc $};
    \node (x) at (2,4) {$ v $};
    \node (y) at (2,2) {$ w $};
    \node (z) at (2,0) {$ x $};
    \node (d) at (4,4) {$ La' $};
    \node (e) at (4,2) {$ Lb' $};
    \node (f) at (4,0) {$ Lc' $};
    \draw [cd]  (a) to (x);
    \draw [cd]  (b) to (y);
    \draw [cd]  (c) to (z);
    \draw [cd]  (d) to (x);
    \draw [cd]  (e) to (y);
    \draw [cd]  (f) to (z);
    \draw [cd]  (b) to (a);
    \draw [>->] (y) to (x);
    \draw [cd]  (e) to (d);
    \draw [cd]  (b) to (c);
    \draw [>->] (y) to (z);
    \draw [cd]  (e) to (f);
    \end{scope}
    \begin{scope}[shift={(6,0)}]
    \node (a) at (0,4) {$ La' $};
    \node (b) at (0,2) {$ Lb' $};
    \node (c) at (0,0) {$ Lc' $};
    \node (x) at (2,4) {$ v' $};
    \node (y) at (2,2) {$ w' $};
    \node (z) at (2,0) {$ x' $};
    \node (d) at (4,4) {$ La'' $};
    \node (e) at (4,2) {$ Lb'' $};
    \node (f) at (4,0) {$ Lc'' $};
    \draw [cd] (a) to (x);
    \draw [cd] (b) to (y);
    \draw [cd] (c) to (z);
    \draw [cd] (d) to (x);
    \draw [cd] (e) to (y);
    \draw [cd] (f) to (z);
    \draw [cd] (b) to (a);
    \draw [>->] (y) to (x);
    \draw [cd] (e) to (d);
    \draw [cd] (b) to (c);
    \draw [>->] (y) to (z);
    \draw [cd] (e) to (f);
    \end{scope}
    \node () at (5,2) {$ \hcirc $};
    \node () at (11,2) {$ \bydef  $};
  \end{tikzpicture}
\]
\[
  \begin{tikzpicture}
    \node (a) at (0,4) {$ La $};
    \node (b) at (0,2) {$ Lb $};
    \node (c) at (0,0) {$ Lc $};
    \node (x) at (2,4) {$ v +_{La'} v' $};
    \node (y) at (2,2) {$ w +_{Lb'} w' $};
    \node (z) at (2,0) {$ x +_{Lc'} x' $};
    \node (d) at (4,4) {$ La'' $};
    \node (e) at (4,2) {$ Lb'' $};
    \node (f) at (4,0) {$ Lc'' $};
    \draw [cd] (a) to (x);
    \draw [cd] (b) to (y);
    \draw [cd] (c) to (z);
    \draw [cd] (d) to (x);
    \draw [cd] (e) to (y);
    \draw [cd] (f) to (z);
    \draw [cd] (b) to (a);
    \draw [cd] (y) to (x);
    \draw [cd] (e) to (d);
    \draw [cd] (b) to (c);
    \draw [cd] (y) to (z);
    \draw [cd] (e) to (f);
  \end{tikzpicture}
  \]
  The \defn{vertical composition} of fine rewrite rules is
  \[
  \begin{tikzpicture}
    \begin{scope}
    \node (a) at (0,4) {$ La $};
    \node (b) at (0,2) {$ Lb $};
    \node (c) at (0,0) {$ Lc $};
    \node (x) at (2,4) {$ v $};
    \node (y) at (2,2) {$ w $};
    \node (z) at (2,0) {$ x $};
    \node (d) at (4,4) {$ La' $};
    \node (e) at (4,2) {$ Lb' $};
    \node (f) at (4,0) {$ Lc' $};
    \draw [cd] (a) to node [] {\scriptsize{$  $}} (x);
    \draw [cd] (b) to node [] {\scriptsize{$  $}} (y);
    \draw [cd] (c) to node [] {\scriptsize{$  $}} (z);
    \draw [cd] (d) to node [] {\scriptsize{$  $}} (x);
    \draw [cd] (e) to node [] {\scriptsize{$  $}} (y);
    \draw [cd] (f) to node [] {\scriptsize{$  $}} (z);
    \draw [cd] (b) to node [] {\scriptsize{$  $}} (a);
    \draw [>->] (y) to node [] {\scriptsize{$  $}} (x);
    \draw [cd] (e) to node [] {\scriptsize{$  $}} (d);
    \draw [cd] (b) to node [] {\scriptsize{$  $}} (c);
    \draw [>->] (y) to node [] {\scriptsize{$  $}} (z);
    \draw [cd] (e) to node [] {\scriptsize{$  $}} (f);
    \end{scope}
    \begin{scope}[shift={(6,0)}]
    \node (a) at (0,4) {$ Lc $};
    \node (b) at (0,2) {$ Ld $};
    \node (c) at (0,0) {$ Le $};
    \node (x) at (2,4) {$ x $};
    \node (y) at (2,2) {$ y $};
    \node (z) at (2,0) {$ z $};
    \node (d) at (4,4) {$ Lc' $};
    \node (e) at (4,2) {$ Ld' $};
    \node (f) at (4,0) {$ Le' $};
    \draw [cd] (a) to node [] {\scriptsize{$  $}} (x);
    \draw [cd] (b) to node [] {\scriptsize{$  $}} (y);
    \draw [cd] (c) to node [] {\scriptsize{$  $}} (z);
    \draw [cd] (d) to node [] {\scriptsize{$  $}} (x);
    \draw [cd] (e) to node [] {\scriptsize{$  $}} (y);
    \draw [cd] (f) to node [] {\scriptsize{$  $}} (z);
    \draw [cd] (b) to node [] {\scriptsize{$  $}} (a);
    \draw [>->] (y) to node [] {\scriptsize{$  $}} (x);
    \draw [cd] (e) to node [] {\scriptsize{$  $}} (d);
    \draw [cd] (b) to node [] {\scriptsize{$  $}} (c);
    \draw [>->] (y) to node [] {\scriptsize{$  $}} (z);
    \draw [cd] (e) to node [] {\scriptsize{$  $}} (f);
  \end{scope}
  \node () at (5,2) {$ \vcirc $};
  \node () at (11,2) {$ \bydef  $};
\end{tikzpicture}
\]
\[
  \begin{tikzpicture}   
    \begin{scope}[shift={(0,0)}]
    \node (a) at (0,4) {$ La $};
    \node (b) at (0,2) {$ L(b \times_c d) $};
    \node (c) at (0,0) {$ Le $};
    \node (x) at (2,4) {$ v $};
    \node (y) at (2,2) {$ w \times_x y $};
    \node (z) at (2,0) {$ z $};
    \node (d) at (4,4) {$ La' $};
    \node (e) at (4,2) {$ L(b' \times_{c'} d') $};
    \node (f) at (4,0) {$ Le' $};
    \draw [cd] (a) to node [] {\scriptsize{$  $}} (x);
    \draw [cd] (b) to node [] {\scriptsize{$  $}} (y);
    \draw [cd] (c) to node [] {\scriptsize{$  $}} (z);
    \draw [cd] (d) to node [] {\scriptsize{$  $}} (x);
    \draw [cd] (e) to node [] {\scriptsize{$  $}} (y);
    \draw [cd] (f) to node [] {\scriptsize{$  $}} (z);
    \draw [cd] (b) to node [] {\scriptsize{$  $}} (a);
    \draw [>->] (y) to node [] {\scriptsize{$  $}} (x);
    \draw [cd] (e) to node [] {\scriptsize{$  $}} (d);
    \draw [cd] (b) to node [] {\scriptsize{$  $}} (c);
    \draw [>->] (y) to node [] {\scriptsize{$  $}} (z);
    \draw [cd] (e) to node [] {\scriptsize{$  $}} (f);
    \end{scope}
  \end{tikzpicture}
  \]
\end{definition}

We defined $ \hcirc $ and $ \vcirc $ using representatives
of isomorphism classes, however this operation is
well-defined. It is less clear, however, that these
operations preserve the monic arrows in the fine rewrites of
structured cospans.  In Proposition
\ref{thm:comp-preserve-monic}, we show that horizontal and
vertical composition do preserve these monic arrows. To
prove this, we require the following lemma.

\begin{lemma}
  \label{thm:quotient-map-monic-pushout}
  The diagram
  \begin{equation} \label{eq:qmmp-1}
    \begin{tikzpicture}
      \node (x) at  (0,2) {$ x $};
      \node (y) at  (2,2) {$ y $};
      \node (z) at  (4,2) {$ z $};
      \node (x') at (0,0) {$ x' $};
      \node (y') at (2,0) {$ y' $};
      \node (z') at (4,0) {$ z' $};
      \draw [cd]  (y) edge (x);
      \draw [cd]  (y) edge (z);
      \draw [cd]  (y') edge (x');
      \draw [cd]  (y') edge (z');
      \draw [>->] (x) edge (x');
      \draw [cd]  (y) edge node[left]{$ \iso $} (y');
      \draw [>->] (z) edge (z');
    \end{tikzpicture}
  \end{equation}
  induces a pushout
  \begin{equation} \label{eq:qmmp-2}
    \begin{tikzpicture}
      \node (x+z)      at (0,2) {$ x + z $};
      \node (x+_yz)    at (4,2) {$ x +_y z $};
      \node (x'+z')    at (0,0) {$ x' + z'  $};
      \node (x'+_y'z') at (4,0) {$ x' +_{y'} z $};
      \draw [cd] (x+z)   to node[above]{$ \rho $} (x+_yz);
      \draw [cd] (x'+z') to node[below]{$ \rho' $} (x'+_y'z');
      \draw [cd] (x+z)   to node[left]{$ \gamma $} (x'+z');
      \draw [cd] (x+_yz) to node[right]{$ \gamma' $} (x'+_y'z');
    \end{tikzpicture}
  \end{equation}
  such that the canonical arrows $ \gamma $ and
  $ \gamma' $ are monic.
\end{lemma}

\begin{proof}
  The universal property of coproducts implies
  that $ \gamma $ factors through $ x' + z $ as in
  the diagram
  \[
    \begin{tikzpicture}
      \node (x) at (0,4) {$ x $};
      \node (x') at (0,2) {$ x' $};
      \node (x+z) at (2,4) {$ x+z $};
      \node (x'+z) at (2,2) {$ x'+z $};
      \node (x'+z') at (2,0) {$ x'+z' $};
      \node (z) at (4,2) {$ z $};
      \node (z') at (4,0) {$ z' $};
      \draw [>->] (x) to node[above]{$ \iota_x $}  (x+z);
      \draw [>->] (x') to node[below]{$ \iota_{x'} $} (x'+z);
      \draw [>->] (z) to node[above]{$ \iota_z  $} (x'+z);
      \draw [>->] (z') to node[below]{$ \iota_{z'} $} (x'+z');
      \draw [>->] (x) to (x');
      \draw [cd] (x+z) to (x'+z);
      \draw [cd] (x'+z) to (x'+z');
      \draw [>->] (z) to (z');
    \end{tikzpicture}
  \]
  It is straightforward to check that both squares
  are pushouts. By Lemma
  \ref{lem:adhesive-properties}, it follows that
  $ \gamma $ is monic.

  Diagram \ref{eq:qmmp-2} commutes because of the
  universal property of coproducts.  To see that
  it is a pushout, arrange a cocone
  \begin{equation}
    \label{eq:qmmp-3}
    \begin{tikzpicture}
      \node (x+z) at (0,4) {$ x+z $};
      \node (x'+z') at (0,2) {$ x'+z' $};
      \node (x+_yz) at (4,4) {$ x +_y z $};
      \node (x'+_y'z') at (4,2) {$ x' +_{y'} z' $};
      \node (c) at (6,0) {$ c $};
      \draw[cd]
      (x+z) edge node[above]{$\rho$} (x+_yz)
      (x'+z') edge node[above]{$\rho'$} (x'+_y'z')
      (x+z) edge node[left]{$\gamma$} (x'+z')
      (x+_yz) edge node[right]{$\gamma'$} (x'+_y'z')
      (x'+z') edge[bend right] node[below]{$\psi'$} (c)
      (x+_yz) edge[bend left] node[right]{$\psi$} (c);
    \end{tikzpicture}
  \end{equation}
  Denote by $ \iota_x $ any map that includes $ x $.  Then
  $ \psi' \iota_{x'} $, $ \psi' \iota_{z'} $, and $ c $ form
  a cocone under the span $ x' \gets y' \to z' $ from the
  bottom face of Diagram \ref{eq:qmmp-1}. This induces the
  canonical map $ \psi'' \from x'+_{y'} z' \to c $. It
  follows that
  $ \psi' \iota_{x'} = \psi'' \rho' \iota_{x'} $ and
  $ \psi' \iota_{z'} = \psi'' \rho' \iota_{z'} $. Therefore
  $ \psi' = \psi'' \rho' $ by the universal property of
  coproducts.

  Furthermore, $ \psi \rho \iota_z $,
  $ \psi \rho \iota_z $, and $ c $ form a cocone
  under the span $ x \gets y \to z $ on the top
  face of Diagram \ref{eq:qmmp-1}. then
  $ \psi \rho \iota_x = \psi' \gamma \iota_x =
  \psi'' \rho' \gamma \iota_x = \psi'' \psi' \rho
  \iota_x $ and
  $ \psi \rho \iota_z = \psi' \gamma \iota_z =
  \psi'' \rho' \gamma \iota_z = \psi'' \gamma'
  \rho \iota_z $ meaning that both $ \psi $ and
  $ \psi'' \psi' $ satisfy the canonical map
  $ x+_yz \to d $. Hence $ \psi = \psi'' \psi' $.

  The universality of $ \psi'' $ with respect to
  Diagram \ref{eq:qmmp-3} follows from the
  universality of $ \gamma'' $ with respect to
  $ x'+_{y'}z' $.
\end{proof}

\begin{lemma} \label{thm:comp-preserve-monic}
  Horizontal and vertical composition of fine
  rewrites are fine rewrites.
\end{lemma}

\begin{proof}
  We can see that the span of cospan obtained by horizontal
  composition of fine rewrites
    \[
  \begin{tikzpicture}
    \begin{scope}
    \node (a) at (0,4) {$ La $};
    \node (b) at (0,2) {$ Lb $};
    \node (c) at (0,0) {$ Lc $};
    \node (x) at (2,4) {$ v $};
    \node (y) at (2,2) {$ w $};
    \node (z) at (2,0) {$ x $};
    \node (d) at (4,4) {$ La' $};
    \node (e) at (4,2) {$ Lb' $};
    \node (f) at (4,0) {$ Lc' $};
    \draw [cd]  (a) to (x);
    \draw [cd]  (b) to (y);
    \draw [cd]  (c) to (z);
    \draw [cd]  (d) to (x);
    \draw [cd]  (e) to (y);
    \draw [cd]  (f) to (z);
    \draw [cd]  (b) edge node[left]{$ \cong $} (a);
    \draw [>->] (y) to (x);
    \draw [cd]  (e) edge node[right]{$ \cong $} (d);
    \draw [cd]  (b) edge node[left]{$ \cong $} (c);
    \draw [>->] (y) to (z);
    \draw [cd]  (e) edge node[right]{$ \cong $} (f);
    \end{scope}
    \begin{scope}[shift={(6,0)}]
    \node (a) at (0,4) {$ La' $};
    \node (b) at (0,2) {$ Lb' $};
    \node (c) at (0,0) {$ Lc' $};
    \node (x) at (2,4) {$ v' $};
    \node (y) at (2,2) {$ w' $};
    \node (z) at (2,0) {$ x' $};
    \node (d) at (4,4) {$ La'' $};
    \node (e) at (4,2) {$ Lb'' $};
    \node (f) at (4,0) {$ Lc'' $};
    \draw [cd] (a) to (x);
    \draw [cd] (b) to (y);
    \draw [cd] (c) to (z);
    \draw [cd] (d) to (x);
    \draw [cd] (e) to (y);
    \draw [cd] (f) to (z);
    \draw [cd] (b) edge node[left]{$ \cong $} (a);
    \draw [>->] (y) to (x);
    \draw [cd] (e) edge node[right]{$ \cong $} (d);
    \draw [cd] (b) edge node[left]{$ \cong $} (c);
    \draw [>->] (y) to (z);
    \draw [cd] (e) edge node[right]{$ \cong $} (f);
    \end{scope}
    \node () at (5,2) {$ \hcirc $};
    \node () at (11,2) {$ \bydef  $};
  \end{tikzpicture}
\]
\[
  \begin{tikzpicture}
    \begin{scope}[shift={(0,0)}]
    \node (a) at (0,4) {$ La $};
    \node (b) at (0,2) {$ Lb $};
    \node (c) at (0,0) {$ Lc $};
    \node (x) at (2,4) {$ v +_{La'} v' $};
    \node (y) at (2,2) {$ w +_{Lb'} w' $};
    \node (z) at (2,0) {$ x +_{Lc'} x' $};
    \node (d) at (4,4) {$ La'' $};
    \node (e) at (4,2) {$ Lb'' $};
    \node (f) at (4,0) {$ Lc'' $};
    \draw [cd] (a) to (x);
    \draw [cd] (b) to (y);
    \draw [cd] (c) to (z);
    \draw [cd] (d) to (x);
    \draw [cd] (e) to (y);
    \draw [cd] (f) to (z);
    \draw [cd] (b) edge node[left]{$ \cong $} (a);
    \draw [cd] (y) to (x);
    \draw [cd] (e) edge node[right]{$ \cong $} (d);
    \draw [cd] (b) edge node[left]{$ \cong $} (c);
    \draw [cd] (y) to (z);
    \draw [cd] (e) edge node[right]{$ \cong $} (f);
    \end{scope}
  \end{tikzpicture}
  \]
  is again a fine rewrite, that is the arrows
  $ w +_{Le} x \to u+_{Lb} v $ and
  $ w +_{Le} x \to y +_{Lh} z $ are monic, by applying Lemma
  \ref{thm:quotient-map-monic-pushout} to the diagrams
  \[
    \begin{tikzpicture}
      \begin{scope}
        \node (x) at (0,2) {$ v $};
        \node (y) at (0,0) {$ w $};
        \node (d) at (2,2) {$ La' $};
        \node (e) at (2,0) {$ Lb' $};
        \node (x') at (4,2) {$ v' $};
        \node (y') at (4,0) {$ w' $};
        \draw [cd] (d) to (x);
        \draw [cd] (d) to (x');
        \draw [cd] (e) to (y);
        \draw [cd] (e) to (y');
        \draw [>->] (y) to (x);
        \draw [cd] (e) to node [right] {$ \iso $} (d);
        \draw [>->] (y') to (x');
      \end{scope}
      \node at (6,1) {and};
      \begin{scope}[shift={(8,0)}]
        \node (y)  at (0,2) {$ w $};
        \node (z)  at (0,0) {$ x $};
        \node (e)  at (2,2) {$ Lb' $};
        \node (f)  at (2,0) {$ Lc' $};
        \node (y') at (4,2) {$ w' $};
        \node (z') at (4,0) {$ x' $};
        \draw [cd] (e) to (y);
        \draw [cd] (e) to (y');
        \draw [cd] (f) to (z);
        \draw [cd] (f) to (z');
        \draw [>->] (y) to (z);
        \draw [cd] (e) to node[right]{$ \iso $} (f);
        \draw [>->] (y') to  (z');
      \end{scope}
    \end{tikzpicture}
  \]
  The result for vertical composition
  \[
  \begin{tikzpicture}
    \begin{scope}
    \node (a) at (0,4) {$ La $};
    \node (b) at (0,2) {$ Lb $};
    \node (c) at (0,0) {$ Lc $};
    \node (x) at (2,4) {$ v $};
    \node (y) at (2,2) {$ w $};
    \node (z) at (2,0) {$ x $};
    \node (d) at (4,4) {$ La' $};
    \node (e) at (4,2) {$ Lb' $};
    \node (f) at (4,0) {$ Lc' $};
    \draw [cd] (a) to (x);
    \draw [cd] (b) to (y);
    \draw [cd] (c) to (z);
    \draw [cd] (d) to (x);
    \draw [cd] (e) to (y);
    \draw [cd] (f) to (z);
    \draw [cd] (b) to node[left]{$ \iso $}  (a);
    \draw [>->] (y) to (x);
    \draw [cd] (e) to node[right]{$ \iso $}(d);
    \draw [cd] (b) to node[left]{$ \iso $} (c);
    \draw [>->] (y) to (z);
    \draw [cd] (e) to node[right]{$ \iso $} (f);
    \end{scope}
    \begin{scope}[shift={(6,0)}]
    \node (a) at (0,4) {$ Lc $};
    \node (b) at (0,2) {$ Ld $};
    \node (c) at (0,0) {$ Le $};
    \node (x) at (2,4) {$ x $};
    \node (y) at (2,2) {$ y $};
    \node (z) at (2,0) {$ z $};
    \node (d) at (4,4) {$ Lc' $};
    \node (e) at (4,2) {$ Ld' $};
    \node (f) at (4,0) {$ Le' $};
    \draw [cd] (a) to (x);
    \draw [cd] (b) to (y);
    \draw [cd] (c) to (z);
    \draw [cd] (d) to (x);
    \draw [cd] (e) to (y);
    \draw [cd] (f) to (z);
    \draw [cd] (b) to node[left]{$ \iso $} (a);
    \draw [>->] (y) to (x);
    \draw [cd] (e) to node[right]{$ \iso $} (d);
    \draw [cd] (b) to node[left]{$ \iso $} (c);
    \draw [>->] (y) to (z);
    \draw [cd] (e) to node[right]{$ \iso $} (f);
    \end{scope}
    \node () at (5,2) {$ \vcirc $};
    \node () at (11,2) {$ \bydef  $};
  \end{tikzpicture}
\]
\[
  \begin{tikzpicture}
    \begin{scope}[shift={(0,0)}]
    \node (a) at (0,4) {$ La $};
    \node (b) at (0,2) {$ L(b \times_c d) $};
    \node (c) at (0,0) {$ Le $};
    \node (x) at (2,4) {$ v $};
    \node (y) at (2,2) {$ w \times_x y $};
    \node (z) at (2,0) {$ z $};
    \node (d) at (4,4) {$ La' $};
    \node (e) at (4,2) {$ L(b' \times_{c'} d') $};
    \node (f) at (4,0) {$ Le' $};
    \draw [cd] (a) to (x);
    \draw [cd] (b) to (y);
    \draw [cd] (c) to (z);
    \draw [cd] (d) to (x);
    \draw [cd] (e) to (y);
    \draw [cd] (f) to (z);
    \draw [cd] (b) to node[left]{$ \iso $} (a);
    \draw [>->] (y) to (x);
    \draw [cd] (e) to node[right]{$ \iso $} (d);
    \draw [cd] (b) to node[left]{$ \iso $} (c);
    \draw [>->] (y) to (z);
    \draw [cd] (e) to node[right]{$ \iso $} (f);
  \end{scope}
  \end{tikzpicture}
  \]  
  holds because pullback preserves monomorphisms.
\end{proof}

With horizontal and vertical composition in hand, we
construct the double category $ _L\FFFineRewrite $.
Actually, we delay discussing the interchange law until
Section \ref{sec:interchange-law} because it is difficult
enough to warrant its own section.

\begin{proposition}
\label{thm:fine-rewrite-double-cat}
  Let
  \[
    \adjunction{\A}{\X}{L}{R}{2}
  \]
  be a adjunction with $ L $ preserving pullbacks.  There is
  a double category $ _L\FFFineRewrite $ whose objects are
  the $ \A $-objects, horizontal arrows of type $ a \to b $
  are structured cospans $ La \to x \gets Lb $, vertical
  arrows are spans in $ \A $ with invertible arrows, and
  squares are fine rewrites of structured cospans
  \[
    \begin{tikzpicture}
    \node (La)   at (0,4) {$ La $};
    \node (x)    at (2,4) {$ x $};
    \node (Lb)   at (4,4) {$ La' $};
    \node (La')  at (0,2) {$ Lb $};
    \node (x')   at (2,2) {$ y $};
    \node (Lb')  at (4,2) {$ Lb' $};
    \node (La'') at (0,0) {$ Lc $};
    \node (x'')  at (2,0) {$ z $};
    \node (Lb'') at (4,0) {$ Lc' $};
    \draw [cd] (La)   to (x);
    \draw [cd] (Lb)   to (x);
    \draw [cd] (La')  to (x');
    \draw [cd] (Lb')  to (x');
    \draw [cd] (La'') to (x'');
    \draw [cd] (Lb'') to (x'');
    \draw [cd] (La')  to
      node [left] {$ \iso $} (La);
    \draw [>->] (x') to (x);
    \draw [cd] (Lb') to
      node [right] {$ \iso $} (Lb);
    \draw [cd] (La') to
      node [left] {$ \iso $} (La'');
    \draw [>->] (x') to (x'');
    \draw [cd] (Lb') to
      node [right] {$ \iso $} (Lb'');
    \end{tikzpicture}
  \]
\end{proposition}

\begin{proof}
  This proof requires we check the axioms of a double
  category as laid out in Definition \ref{def:dbl-cat}. For
  simplicity, we denote $_L\FFFineRewrite$ by $\RRR $ in
  this proof.

  The object category $\RRR_0$ is given by objects of $\A$
  and isomorphism classes of spans in $\A$ such that each
  leg is an isomorphism.  The arrow category $\RRR_1$ has as
  objects the structured cospans
  \[
    La \to x \gets La'
  \]
  and as morphisms the fine rewrites of structured cospans.
	
  The functor $U \from \RRR_0 \to \RRR_1$ acts
  on objects by mapping $a$ to the identity cospan on $La$
  and on morphisms by mapping $La \gets Lb \to Lc$, whose
  legs are isomorphisms, to the square
  \[
    \begin{tikzpicture}
      \node (A) at (0,4) {$La$};
      \node (A') at (0,2) {$Lb$};
      \node (A'') at (0,0) {$Lc$};
      \node (B) at (2,4) {$La$};
      \node (B') at (2,2) {$Lb$};
      \node (B'') at (2,0) {$Lc$};
      \node (C) at (4,4) {$La$};
      \node (C') at (4,2) {$Lb$};
      \node (C'') at (4,0) {$Lc$};
      \path[cd,font=\scriptsize,>=angle 90]
      (A) edge node{} (B)
      (A') edge node{} (B')
      (A'') edge node{} (B'')
      (C) edge node{} (B)
      (C') edge node{} (B')
      (C'') edge node{} (B'')
      (A') edge node{} (A)
      (A') edge node{} (A'')
      (B') edge node{} (B)
      (B') edge node{} (B'')
      (C') edge node{} (C)
      (C') edge node{} (C'');
    \end{tikzpicture}
  \]
  The functor $S \from \RRR_1 \to \RRR_0$
  acts on objects by sending $La \to x \gets La'$
  to $a$ and on morphisms by sending a square
  \[
    \begin{tikzpicture}
      \node (A) at (0,4) {$La$};
      \node (A') at (0,2) {$Lb$};
      \node (A'') at (0,0) {$Lc$};
      \node (B) at (2,4) {$x$};
      \node (B') at (2,2) {$y$};
      \node (B'') at (2,0) {$z$};
      \node (C) at (4,4) {$La'$};
      \node (C') at (4,2) {$Lb'$};
      \node (C'') at (4,0) {$Lc'$};
      \draw[cd]
      (A) edge node{} (B)
      (A') edge node{} (B')
      (A'') edge node{} (B'')
      (C) edge node{} (B)
      (C') edge node{} (B')
      (C'') edge node{} (B'')
      (A') edge node{} (A)
      (A') edge node{} (A'')
      (B') edge node{} (B)
      (B') edge node{} (B'')
      (C') edge node{} (C)
      (C') edge node{} (C'');
    \end{tikzpicture}
  \]
  to the span $ La \gets Lb \to Lc $. The 
  functor $T$ is defined similarly sends an object \[
  La \to x \gets La' \] of $ \RRR_1 $ to $ a' $ a square
  \[
    \begin{tikzpicture}
      \node (A) at (0,4) {$La$};
      \node (A') at (0,2) {$Lb$};
      \node (A'') at (0,0) {$Lc$};
      \node (B) at (2,4) {$x$};
      \node (B') at (2,2) {$y$};
      \node (B'') at (2,0) {$z$};
      \node (C) at (4,4) {$La'$};
      \node (C') at (4,2) {$Lb'$};
      \node (C'') at (4,0) {$Lc'$};
      \draw[cd]
      (A) edge node{} (B)
      (A') edge node{} (B')
      (A'') edge node{} (B'')
      (C) edge node{} (B)
      (C') edge node{} (B')
      (C'') edge node{} (B'')
      (A') edge node{} (A)
      (A') edge node{} (A'')
      (B') edge node{} (B)
      (B') edge node{} (B'')
      (C') edge node{} (C)
      (C') edge node{} (C'');
    \end{tikzpicture}
  \]
  to the span $ La' \gets Lb' \to Lc' $.
  
  The horizontal composition functor
  \[
    \odot \from \RRR_1 \times_{\RRR_0} \RRR_1 \to
    \RRR_1
  \]
  acts on objects by composing cospans with pushouts in
  the usual way.  It acts on morphisms by
  \[
    \raisebox{-0.5\height}{
      \begin{tikzpicture}
        \node (A) at (0,4) {$La$};
        \node (A') at (0,2) {$Lb$};
        \node (A'') at (0,0) {$Lc$};
        \node (B) at (2,4) {$v$};
        \node (B') at (2,2) {$w$};
        \node (B'') at (2,0) {$x$};
        \node (C) at (4,4) {$La'$};
        \node (C') at (4,2) {$Lb'$};
        \node (C'') at (4,0) {$Lc'$};
        \node (D) at (6,4) {$v'$};
        \node (D') at (6,2) {$w'$};
        \node (D'') at (6,0) {$x'$};
        \node (E) at (8,4) {$La''$};
        \node (E') at (8,2) {$Lb''$};
        \node (E'') at (8,0) {$Lc''$};
        \draw[cd]
        (A) edge node[above]{} (B)
        (A') edge node[above]{} (B')
        (A'') edge node[above]{} (B'')
        (C) edge node[above]{} (B)
        (C') edge node[above]{} (B')
        (C'') edge node[above]{} (B'')
        (C) edge node[above]{} (D)
        (C') edge node[above]{} (D')
        (C'') edge node[above]{} (D'')
        (E) edge node[above]{} (D)
        (E') edge node[above]{} (D')
        (E'') edge node[above]{} (D'')
        (A') edge node[left]{} (A)
        (A') edge node[left]{} (A'')
        (B') edge node[left]{} (B)
        (B') edge node[left]{} (B'')
        (C') edge node[left]{} (C)
        (C') edge node[left]{} (C'')	
        (D') edge node[left]{} (D)
        (D') edge node[left]{} (D'')
        (E') edge node[left]{} (E)
        (E') edge node[left]{} (E'');
      \end{tikzpicture}
    }
    \quad
    \xmapsto{\odot}
    \quad
    \raisebox{-0.5\height}{
      \begin{tikzpicture}
        \node (A) at (0,4) {$La$};
        \node (A') at (0,2) {$Lb$};
        \node (A'') at (0,0) {$Lc$};
        \node (B) at (2,4) {$v +_{La'} v'$};
        \node (B') at (2,2) {$w +_{Lb'} w'$};
        \node (B'') at (2,0) {$x +_{Lc'} x'$};
        \node (C) at (4,4) {$La''$};
        \node (C') at (4,2) {$Lb''$};
        \node (C'') at (4,0) {$Lc''$};
        \draw[cd]
        (A) edge node[above]{} (B)
        (A') edge node[above]{} (B')
        (A'') edge node[above]{} (B'')
        (C) edge node[above]{} (B)
        (C') edge node[above]{} (B')
        (C'') edge node[above]{} (B'')
        (A') edge node[left]{} (A)
        (A') edge node[left]{} (A'')
        (B') edge node[left]{} (B)
        (B') edge node[left]{} (B'')
        (C') edge node[left]{} (C)
        (C') edge node[left]{} (C'');	
      \end{tikzpicture}
    }
  \]
  Section \ref{sec:interchange-law} is devoted to proving
  that $\odot$ is functorial, that is, it preserves
  composition.  It is straightforward to check that the
  required equations are satisfied.  The associator and
  unitors are given by natural isomorphisms that arise from
  universal properties.
\end{proof}

And now, our double category of fine rewrites is defined. It
remains to prove the interchange law, which we do next.

\section{The interchange law}
\label{sec:interchange-law}

Here we prove the most technical part of the proof that
$ _L\FFFineRewrite $ is a double category: the interchange
law. This law relates the horizontal and vertical
composition defined in the previous section.  

\begin{theorem} \label{thm:interchange-law}

  Given four fine rewrites of structured cospans
  \begin{equation}
    \label{eq:interchange}
    \begin{tikzpicture}
      \begin{scope}[shift={(0,6)}]
        \node () at (-1,2) {$ \alpha \bydef $};
        \node (a) at (0,4) {$ La $};
        \node (b) at (0,2) {$ Ld $};
        \node (c) at (0,0) {$ Lg $};
        \node (u) at (2,4) {$ u $};
        \node (v) at (2,2) {$ w $};
        \node (w) at (2,0) {$ y $};
        \node (d) at (4,4) {$ Lb $};
        \node (e) at (4,2) {$ Le $};
        \node (f) at (4,0) {$ Lh $};
        \draw [cd] (a) to (u);
        \draw [cd] (d) to (u);
        \draw [cd] (b) to (v);
        \draw [cd] (e) to (v);
        \draw [cd] (c) to (w);
        \draw [cd] (f) to (w);
        \draw [cd] (b) to node[left]{$\iso $} (a);
        \draw [cd] (b) to node[left]{$\iso$} (c);
        \draw [>->] (v) to (u);
        \draw [>->] (v) to (w);
        \draw [cd] (e) to node[right]{$ \iso $} (d);
        \draw [cd] (e) to node[right]{$ \iso $} (f);
      \end{scope}
      \begin{scope}
        \node () at (-1,2) {$ \beta \bydef $};
        \node (a) at (0,4) {$ Lg $};
        \node (b) at (0,2) {$ Ld' $};
        \node (c) at (0,0) {$ La' $};
        \node (u) at (2,4) {$ y $};
        \node (v) at (2,2) {$ w' $};
        \node (w) at (2,0) {$ x' $};
        \node (d) at (4,4) {$ Lh $};
        \node (e) at (4,2) {$ Le' $};
        \node (f) at (4,0) {$ Lb' $};
        \draw [cd] (a) to (u);
        \draw [cd] (d) to (u);
        \draw [cd] (b) to (v);
        \draw [cd] (e) to (v);
        \draw [cd] (c) to (w);
        \draw [cd] (f) to (w);
        \draw [cd] (b) to node[left]{$ \iso $} (a);
        \draw [cd] (b) to node[left]{$ \iso $} (c);
        \draw [>->] (v) to (u);
        \draw [>->] (v) to (w);
        \draw [cd] (e) to node[right]{$ \iso $} (d);
        \draw [cd] (e) to node[right]{$ \iso $} (f);
      \end{scope}
      \begin{scope}[shift={(7,6)}]
        \node () at (-1,2) {$ \alpha' \bydef $};
        \node (a) at (0,4) {$ Lb $};
        \node (b) at (0,2) {$ Le $};
        \node (c) at (0,0) {$ Lh $};
        \node (u) at (2,4) {$ v $};
        \node (v) at (2,2) {$ x $};
        \node (w) at (2,0) {$ z $};
        \node (d) at (4,4) {$ Lc $};
        \node (e) at (4,2) {$ Lf $};
        \node (f) at (4,0) {$ Li $};
        \draw [cd] (a) to (u);
        \draw [cd] (d) to (u);
        \draw [cd] (b) to (v);
        \draw [cd] (e) to (v);
        \draw [cd] (c) to (w);
        \draw [cd] (f) to (w);
        \draw [cd] (b) to node[left]{$ \iso $} (a);
        \draw [cd] (b) to node[left]{$ \iso $} (c);
        \draw [>->] (v) to (u);
        \draw [>->] (v) to (w);
        \draw [cd] (e) to node[right]{$ \iso $} (d);
        \draw [cd] (e) to node[right]{$ \iso $} (f);
      \end{scope}
      \begin{scope}[shift={(7,0)}]
        \node () at (-1,2) {$ \beta' \bydef $};
        \node (a) at (0,4) {$ Lh $};
        \node (b) at (0,2) {$ Le' $};
        \node (c) at (0,0) {$ Lb' $};
        \node (u) at (2,4) {$ z $};
        \node (v) at (2,2) {$ x' $};
        \node (w) at (2,0) {$ v' $};
        \node (d) at (4,4) {$ Li $};
        \node (e) at (4,2) {$ Lf' $};
        \node (f) at (4,0) {$ Lc' $};
        \draw [cd] (a) to (u);
        \draw [cd] (d) to (u);
        \draw [cd] (b) to (v);
        \draw [cd] (e) to (v);
        \draw [cd] (c) to (w);
        \draw [cd] (f) to (w);
        \draw [cd] (b) to node[left]{$ \iso $} (a);
        \draw [cd] (b) to node[left]{$ \iso $} (c);
        \draw [>->] (v) to (u);
        \draw [>->] (v) to (w);
        \draw [cd] (e) to node[right]{$ \iso $} (d);
        \draw [cd] (e) to node[right]{$ \iso $} (f);
      \end{scope}
    \end{tikzpicture}
  \end{equation}
  it is true that
  \begin{equation}
    \label{eq:interchange-eq}
    ( \alpha \hcirc \alpha' ) \vcirc
    ( \beta \hcirc \beta' ) =
    ( \alpha \vcirc \beta ) \hcirc
    ( \alpha' \vcirc \beta' ).    
  \end{equation}
\end{theorem}

We devote the remainder of this section proving
Theorem \ref{thm:interchange-law}. The first thing
we do is deconstruct Equation
\eqref{eq:interchange-eq}, starting with the left
hand side.

The horizontal compositions
$ \alpha \hcirc \alpha' $ and
$ \beta \hcirc \beta' $ are,
respectively,
\[
  \begin{tikzpicture}
    \begin{scope}
      \node (a) at (0,4) {$ La $};
      \node (b) at (0,2) {$ Ld $};
      \node (c) at (0,0) {$ Lg $};
      \node (udx) at (2,4) {$ u +_{Lb} v $};
      \node (vey) at (2,2) {$ w +_{Le} x $};
      \node (wfz) at (2,0) {$ y +_{Lh} z $};
      \node (g) at (4,4) {$ Lc $};
      \node (h) at (4,2) {$ Lf' $};
      \node (i) at (4,0) {$ Lc' $};
      \draw [cd] (a) to (udx);
      \draw [cd] (g) to (udx);
      \draw [cd] (b) to (vey);
      \draw [cd] (h) to (vey);
      \draw [cd] (c) to (wfz);
      \draw [cd] (i) to (wfz);
      \draw [cd] (b) to node[left]{$ \iso $} (a);
      \draw [cd] (b) to node[left]{$ \iso $} (c);
      \draw [>->] (vey) to (udx);
      \draw [>->] (vey) to (wfz);
      \draw [cd] (h) to node[right]{$ \iso $} (g);
      \draw [cd] (h) to node[right]{$ \iso $} (i);
    \end{scope}
    \begin{scope}[shift={(6,0)}]
      \node (c) at (0,4) {$ Lg $};
      \node (b') at (0,2) {$ Ld' $};
      \node (a') at (0,0) {$ La' $};
      \node (wfz) at (2,4) {$ y +_{Lh} z $};
      \node (v'e'y') at (2,2) {$ w' +_{Le'} x' $};
      \node (u'd'z') at (2,0) {$ x' +_{Lb'} v' $};
      \node (g') at (4,0) {$ Lc' $};
      \node (h') at (4,2) {$ Lf' $};
      \node (i) at (4,4) {$ Li $};
      \draw [cd] (c) to (wfz);
      \draw [cd] (i) to (wfz);
      \draw [cd] (b') to (v'e'y');
      \draw [cd] (h') to (v'e'y');
      \draw [cd] (a') to (u'd'z');
      \draw [cd] (g') to (u'd'z');
      \draw [cd] (b') to node[left]{$ \iso $} (c);
      \draw [cd] (b') to node[left]{$ \iso $} (a');
      \draw [>->] (v'e'y') to (wfz);
      \draw [>->] (v'e'y') to (u'd'z');
      \draw [cd] (h') to node[right]{$ \iso $} (i);
      \draw [cd] (h') to node[right]{$ \iso $} (g');
    \end{scope}
  \end{tikzpicture}
\]
Lemma \ref{thm:comp-preserve-monic} ensures that
the marked arrows above are monic. The vertical
composition of these is
\[
 \begin{tikzpicture}
   \node at (0,2) {
     $ ( \alpha \hcirc \alpha' )
       \vcirc
       ( \beta \hcirc \beta' )
       =  $};
   \begin{scope}[shift={(3.5,0)}]
      \node (a) at (0,4) {$ La $};
      \node (bcb') at (0,2) {$ Ld \times_{Lg} Ld' $};
      \node (a') at (0,0) {$ La' $};
      \node (udx) at (4.5,4) {$ u +_{Lb} v $};
      \node (middle) at (4.5,2)
        { $ ( w +_{Le} x )
          \times_{( y +_{Lh} z )}
          ( w' +_{Le'} x') $ };
      \node (u'd'z') at (4.5,0) {$ x' +_{Lb'} v' $};
      \node (g) at (9,4) {$ Lc $};
      \node (hih') at (9,2) {$ Lf +_{Li} Lf' $};
      \node (g') at (9,0) {$ Lc' $};
      \draw [cd] (a) to (udx);
      \draw [cd] (g) to (udx);
      \draw [cd] (bcb') to (middle);
      \draw [cd] (hih') to (middle);
      \draw [cd] (a') to (u'd'z');
      \draw [cd] (g') to (u'd'z');
      \draw [cd] (bcb') to node[left]{$ \iso $} (a);
      \draw [cd] (bcb') to node[left]{$ \iso $} (a');
      \draw [>->] (middle) to (udx);
      \draw [>->] (middle) to (u'd'z');
      \draw [cd] (hih') to node[right]{$ \iso $} (g);
      \draw [cd] (hih') to node [right] {$ \iso $} (g');
    \end{scope}
  \end{tikzpicture}
\]
Again, the marked arrows are monic due to Lemma
\ref{thm:comp-preserve-monic}. The outside, vertical arrows
are isomorphisms because pullbacks preserve
isomorphism. 

To compute the right hand side of Equation
\eqref{eq:interchange-eq}, we start with the
vertical composites $ \alpha \vcirc
\beta $ and $ \alpha' \vcirc \beta' $,
which are the respective diagrams
\[
  \begin{tikzpicture}
    \begin{scope}
      \node (a) at (0,4) {$ La $};
      \node (b) at (0,2) {$ L(d \times_{g} d') $};
      \node (c) at (0,0) {$ La' $};
      \node (udx) at (3,4) {$ u $};
      \node (vey) at (3,2) {$ w \times_{y} w' $};
      \node (wfz) at (3,0) {$ x' $};
      \node (g) at (6,4) {$ Lb $};
      \node (h) at (6,2) {$ L(e \times_{h} e') $};
      \node (i) at (6,0) {$ Lb' $};
      \draw [cd] (a) to (udx);
      \draw [cd] (g) to (udx);
      \draw [cd] (b) to (vey);
      \draw [cd] (h) to (vey);
      \draw [cd] (c) to (wfz);
      \draw [cd] (i) to (wfz);
      \draw [cd] (b) to node[left]{$ \iso $} (a);
      \draw [cd] (b) to node[left]{$ \iso $} (c);
      \draw [>->] (vey) to (udx);
      \draw [>->] (vey) to (wfz);
      \draw [cd] (h) to node[right]{$ \iso $} (g);
      \draw [cd] (h) to node[right]{$ \iso $} (i);
    \end{scope}
    \begin{scope}[shift={(0,-6)}]
      \node (c) at (0,4) {$ Lb $};
      \node (b') at (0,2) {$ L(e \times_{h} e') $};
      \node (a') at (0,0) {$ Lb' $};
      \node (wfz) at (3,4) {$ v $};
      \node (v'e'y') at (3,2) {$ x \times_{z} x' $};
      \node (u'd'z') at (3,0) {$ v' $};
      \node (i) at (6,4) {$ Lc $};
      \node (h') at (6,2) {$ L(f \times_{i} f') $};
      \node (g') at (6,0) {$ Lc' $};
      \draw [cd] (c) to (wfz);
      \draw [cd] (i) to (wfz);
      \draw [cd] (b') to (v'e'y');
      \draw [cd] (h') to (v'e'y');
      \draw [cd] (a') to (u'd'z');
      \draw [cd] (g') to (u'd'z');
      \draw [cd] (b') to node[left]{$ \iso $} (c);
      \draw [cd] (b') to node[left]{$ \iso $} (a');
      \draw [>->] (v'e'y') to (wfz);
      \draw [>->] (v'e'y') to (u'd'z');
      \draw [cd] (h') to node[right]{$ \iso $} (i);
      \draw [cd] (h') to node[right]{$ \iso $} (g');
    \end{scope}
  \end{tikzpicture}
\]
Lemma \ref{thm:comp-preserve-monic} ensures the
marked arrows are monic.  The horizontal
composition of these is 
\[
 \begin{tikzpicture}
   \node at (0,2) {
     $ ( \alpha \vcirc \beta )
       \hcirc
       ( \alpha' \vcirc \beta' )
       =  $};
   \begin{scope}[shift={(3,0)}]
      \node (a) at (0,4) {$ La $};
      \node (bcb') at (0,2) {$ Ld \times_{Lg} Ld' $};
      \node (a') at (0,0) {$ La' $};
      \node (udx) at (4.5,4) {$ u +_{Lb} v $};
      \node (middle) at (4.5,2)
        { $ ( w \times_{y} w' )
          +_{L(e \times_{h} e')}
          ( x \times_{z} x') $ };
      \node (u'd'z') at (4.5,0) {$ x' +_{Lb'} v' $};
      \node (g) at (9,4) {$ Lc $};
      \node (hih') at (9,2) {$ Lf +_{Li} Lf' $};
      \node (g') at (9,0) {$ Lc' $};
      \draw [cd] (a) to (udx);
      \draw [cd] (g) to (udx);
      \draw [cd] (bcb') to (middle);
      \draw [cd] (hih') to (middle);
      \draw [cd] (a') to (u'd'z');
      \draw [cd] (g') to (u'd'z');
      \draw [cd] (bcb') to node[left]{$ \iso $} (a);
      \draw [cd] (bcb') to node[left]{$ \iso $} (a');
      \draw [>->] (middle) to (udx);
      \draw [>->] (middle) to (u'd'z');
      \draw [cd] (hih') to node[right]{$ \iso $} (g);
      \draw [cd] (hih') to node[right]{$ \iso $} (g');
    \end{scope}
  \end{tikzpicture}
\]
It follows that the proof of Theorem
\ref{thm:interchange-law} comes down to finding an
isomorphism
\[
  (w \times_{y} w') +_{L(e \times_h e')} (x \times_z x')
  \to
  (w +_{Le} x) \times_{(y +_{Lh} z)} (w' +_{Le'} x')
\]  

To simplify our diagrams, we introduce new
notation. We write
\[
  \begin{array}{ll}
    p  \coloneqq
      ( w \times_y w' ) + ( x \times_z x' ), &
    p'  \coloneqq
      ( w \times_y w' ) +_{L(e \times_h e')} ( x' \times_z x' ), \\
    q  \coloneqq
      ( w + x ) \times_{ y + z } ( w' + x' ),  &
    q'  \coloneqq
      ( w +_{Lg} x ) \times_{ y +_{Lh} z } ( w' +_{Li} x' ). \\
  \end{array}
\]
In this notation, the isomorphism we
seek is
\begin{equation}
  \label{eq:interchange-isomorphism}
  \theta' \from p' \to q'
\end{equation}
Also, because $ Lb $, $ Le $, $ Lh $, $ Le' $, $ Lb' $, and
therefore $ L(e \times_{h} e') $ are all isomorphic, we
simply write $ L\ast $ to mean any of these.  Each are
interchangeable in the diagrams below, and adjusting this
notation will not cause any false reasoning. While we do lose
the ability to discern between these objects, context should
help the reader determine this.  Despite losing this ability, we gain
a breezier exposition and a more readable proof.

Apply Lemma \ref{thm:quotient-map-monic-pushout} to
the diagram
\[
  \begin{tikzpicture}
    \node (B) at (-4,2) {$ w \times_y w'$};
    \node (A) at (0,2) {$ L\ast $};
    \node (C) at (4,2) {$ x \times_z x'$};
    
    \node (B') at (-4,0) {$ y $};
    \node (A') at (0,0) {$ L\ast $};
    \node (C') at (4,0) {$ z $};
    \draw [cd] (A) edge (B);
    \draw [cd] (A) edge node[right] {\scriptsize{$=$}} (A');
    \draw [cd] (A) edge (C);
    \draw [cd] (A') edge (B');
    \draw [cd] (A') edge (C');
    \draw [>->] (B) edge (B');
    \draw [>->] (C) edge (C');
  \end{tikzpicture}
\]
to get the pushout
\[
  \begin{tikzpicture}
    \node (BC) at (0,2) {$p$};
    \node (BAC) at (4,2) {$p'$};
    \node (BC') at (0,0) {$ y + z $};
    \node (BAC') at (4,0) {$ y +_{L\ast} z $};
    \draw [cd] (BC) edge (BAC);
    \draw [>->] (BC) edge node[right]{$\psi$} (BC');
    \draw [>->] (BAC) edge node[right]{$\psi'$} (BAC');
    \draw [cd] (BC') edge (BAC');
  \end{tikzpicture}
\]
Similarly, we get pushouts
\[
  \begin{tikzpicture}
    \begin{scope}
    \node (BC) at (0,2) {$p$};
    \node (BAC) at (3,2) {$p'$};
    \node (BC') at (0,0) {$w+x$};
    \node (BAC') at (3,0) {$w+_{L\ast}x$};
    \draw [cd] (BC) edge (BAC);
    \draw [>->] (BC) edge node[right]{$\sigma$} (BC');
    \draw [>->] (BAC) edge node[right]{$\sigma'$} (BAC');
    \draw [cd] (BC') edge (BAC');
  \end{scope}
  \node () at (4,1) {and};
  \begin{scope}[shift={(5,0)}]
    \node (BC) at (0,2) {$ p $};
    \node (BAC) at (3,2) {$ p' $};
    \node (BC') at (0,0) {$ w' + x' $};
    \node (BAC') at (3,0) {$ w'+_{L\ast} x'$};
    \draw [cd] (BC) edge (BAC);
    \draw [>->] (BC) edge node[right]{$ \phi $} (BC');
    \draw [>->] (BAC) edge node[right]{$ \phi' $} (BAC');
    \draw [cd] (BC') edge (BAC');
  \end{scope}
  \end{tikzpicture}
\]
Now, $p$ forms a cone over the cospan
$ w+x \to y+z \gets w'+x' $ via the maps $ \psi $,
$ \sigma $, and $ \phi $.  And so, we get a
canonical map $ \theta \from p \to q $.

\begin{lemma}
  \label{lem:pullback-over-subobject}
  The commuting diagram
  \[
    \begin{tikzpicture}
      \node (Lg) at (0,0) {$ Lg $};
      \node (Lh) at (2,2) {$ L\ast $};
      \node (t) at (2,0) {$ t $};
      \node (Li) at (4,0) {$ Li $};
      \draw [cd] (Lg) to (Lh);
      \draw [cd] (Lg) to (t);
      \draw [cd] (Li) to (Lh);
      \draw [cd] (Li) to (t);
      \draw [>->] (Lh) to (t);
    \end{tikzpicture}
  \]
  induces a canonical isomorphism between $ Lg
  \times_{L\ast} Li $ and $ Lg \times_{t} L\ast $.
\end{lemma}

\begin{proof}
  Via the projection maps, $ Lg \times_{L\ast} Li $ forms a
  cone over the cospan $ Lg \to t \gets Li $
  and, also, $ Lg \times_{t} Li $ forms a cone over the
  cospan $ Lg \to L\ast \gets Li $, though the
  latter requires the monic
  $ L\ast \rightarrowtail t $ to do so. Universality
  implies that the induced maps are mutual
  inverses and they are the only such pair.
\end{proof}

\begin{lemma}
  \label{thm:theta-is-iso}
  The map $\theta \from p \to q $ is an
  isomorphism.
\end{lemma}

\begin{proof}
  Because colimits are stable under pullback
  \cite[Thm.~4.7.2]{maclane-moerdijk_sheaves},
  we get an isomorphism
  \[
    \gamma \from
    ( w \times_{ y + z } w' ) +
    ( w \times_{ y + z } x' ) +
    ( x \times_{ y + z } w' ) +
    ( x \times_{ y + z } x' )
    \to q.
  \]
  But $ w \times_{ y + z } x' $ and
  $ w' \times_{ y + z } x $ are initial. To see
  this, recall that in a topos, all maps to the
  initial object are isomorphisms.  Now, consider
  the diagram
  \[
    \begin{tikzpicture}
      \node (STpb) at (0,4) {$w \times_{y+z} x'$};
      \node (T'') at (4,4) {$z'$};
      \node (0) at (2,2) {$0$};
      \node (T) at (4,2) {$z$};
      \node (S') at (0,0) {$w$};
      \node (S) at (2,0) {$y$};
      \node (ST) at (4,0) {$y+z$};
      \draw [cd] (STpb) edge (S');
      \draw [cd] (STpb) edge (T'');
      \draw [cd] (STpb) edge[dashed] (0);
      \draw [cd] (0) edge (S);
      \draw [cd] (0) edge (T);
      \draw [cd] (S) edge (ST);
      \draw [cd] (T) edge (ST);
      \draw [cd] (S') edge (S);
      \draw [cd] (T'') edge (T);
    \end{tikzpicture}
  \]
  whose lower right square is a pullback because
  coproducts are disjoint in topoi.  Similarly,
  $ x \times_{ y + z } w' $ is initial.  Hence
  we get a canonical isomorphism
  \begin{equation} \label{eq:B second iso}
    \gamma'
    \from ( w \times_{ y + z } w' ) +
    ( x \times_{ y + z } x' ) \to q
  \end{equation}
  that factors through $\gamma$. But Lemma
  \ref{lem:pullback-over-subobject} gives unique
  isomorphisms
  \[
    w \times_{ y } w' \cong
    w \times_{ y + z } w'
    \text{ and }
    x \times_{ z } x' \cong x
    \times_{ y + z } x'.
  \]
  This produces a canonical isomorphism
  \[
    \gamma'' \from p \to
    ( w \times_{ y + z } w' ) +
    ( x \times_{ y + z } x' ).
  \]
  One can show that 
  $ \theta  = \gamma' \circ \gamma'' $ using
  universal properties.
\end{proof}

Having shown that $ \theta \from p \to q $ is an
isomorphism, we can write $ p $ in place of \[ (w+x)
\times_{(y+z)} ( w'+x' )\] in the following diagram 
\begin{equation} \label{diag.the big cube}
  \begin{tikzpicture}
    \node (A) at (2,3) {$p$};
    \node (Ay) at (0,0) {$p'$};
    \node (By) at (2,0) {$q'$};
    \node (ST) at (8,3) {$y+z$};
    \node (ST') at (6,3.75) {$w+x$};
    \node (ST'') at (4,2.25) {$w'+x'$};
    \node (SYT) at (8,0) {$y+_{L\ast}z$};
    \node (SYT') at (6,0.75) {$w+_{Le} x$};
    \node (SYT'') at (4,-0.75) {$w'+_{Le'} x'$};
    \draw [->>] (A) edge (Ay);
    \draw [font=\scriptsize,>->] (A) edge node[below] {$\phi$} (ST'');
    \draw [font=\scriptsize,>->] (A) edge node[above] {$\sigma$} (ST');
    \draw [font=\scriptsize,cd] (Ay) edge[bend left=20,pos=0.65] node[above] {$\sigma'$} (SYT');
    \draw [font=\scriptsize,cd] (Ay) edge[dashed] node[above] {$\theta'$} (By);
    \draw [font=\scriptsize,cd] (Ay) edge[bend right=20] node[below] {$\phi'$} (SYT'');
    \draw [font=\scriptsize,>->] (By) edge node[above] {$\omega$} (SYT');
    \draw [>->] (By) edge (SYT'');
    \draw [->>] (ST') edge (SYT');
    \draw [>->] (ST') edge (ST);
    \draw [>->] (ST'') edge (ST);
    \draw [->>] (ST) edge (SYT);
    \draw [>->] (SYT') edge (SYT);
    \draw [>->] (SYT'') edge (SYT);
    \draw [cd] (A) edge[white,line width=4pt] (By);
    \draw [font=\scriptsize,cd] (A) edge[dashed] node[right] {$\rho$} (By);
    \draw [cd] (ST'') edge[white,line width=6pt] (ST);
    \draw [>->] (ST'') edge (ST);
    \draw [cd] (ST'') edge[white,line width=6pt] (SYT'');
    \draw [->>] (ST'') edge (SYT'');
  \end{tikzpicture}
\end{equation}
where $\theta'$ from Equation
\eqref{eq:interchange-isomorphism} finally appears. It and
$\rho$ are the canonical maps arising from the pullback on
the bottom.  Observe that $\rho$ factors through $\theta'$
in the above diagram.  This follows from the universal
property of pullbacks. 

\begin{lemma}
  \label{lem.theta_Y iso} The map
  $ \theta' \from p' \to q' $ is an
  isomorphism.
\end{lemma}

\begin{proof}
  Because we are working in a topos, it suffices
  to show that $\theta'$ is both monic and epic.
  It is monic because $\sigma'$ is monic.

  To see that $\theta'$ is epic, it suffices to
  show that $\rho$ is epic.  The front and rear
  right faces of \eqref{diag.the big cube} are
  pushouts by Lemma
  \ref{thm:quotient-map-monic-pushout}.  Then
  because the top and bottom squares of
  \eqref{diag.the big cube} are pullbacks
  consisting of only monomorphisms, Lemma
  \ref{lem.vk dual} implies that the front and
  rear left faces are pushouts.  However, as
  pushouts over monomorphisms, Lemma \ref{lem:adhesive-properties} tells us they are pullbacks.  But
  in a topos, regular epimorphisms are stable under
  pullback, and so $\rho$ is epic.
\end{proof}

It remains to show that $\theta'$ serves as an isomorphism
between fine rewrites. This amounts to showing that
\begin{equation} \label{diag.theta 2-cell iso}
  \begin{tikzpicture}
    \node (La) at (-1.5,2) {$ La $};
    \node (X) at (-4,0) {$L(d \times_g d')$};
    \node (La') at (-6.5,-2) {$ La $};
    \node (LyR) at (2.5,2) {$u +_{L\ast} v$};
    \node (Lc) at (6.5,2) {$ Lc $};
    \node (Y) at (4,0) {$L(f \times_i f')$};
    \node (Lc') at (1.5,-2) {$ Lc' $};
    \node (LyR') at (-2.5,-2) {$x'+_{L\ast} v'$};
    \node (Ay) at (0,1) {$p'$};
    \node (By) at (0,-1) {$q'$};
    \draw [cd] (X) edge[] node[above]{$ $} (La);
    \draw [cd] (X) edge[] node[above]{$ $} (La');
    \draw [cd] (X) edge node[above] {$g$} (Ay);
    \draw [cd] (X) edge node[above] {$h$} (By);
    \draw [cd] (La) edge[] node[above]{$j$} (LyR);
    \draw [cd] (La') edge[] node[below]{$ $} (LyR');
    \draw [cd] (Y) edge (Lc);
    \draw [cd] (Y) edge (Lc');
    \draw [cd] (Y) edge[] (By);
    \draw [cd] (Lc) edge[] (LyR);
    \draw [cd] (Lc') edge[] (LyR');
    \draw [cd] (Ay) edge node[above] {$k$} (LyR);
    \draw [cd] (Ay) edge node[right] {$\theta'$} (By);
    \draw [font=\scriptsize,cd] (By) edge[pos=0.4] node[below] {$fp$} (LyR);
    \draw [cd] (By) edge[] (LyR');
    \draw [cd] (Y) edge[white,line width=3.5pt] (Ay);
    \draw [cd] (Y) edge[] (Ay);
    \draw [cd] (Ay) edge[white,line width=3.5pt] (LyR');
    \draw [cd] (Ay) edge[] (LyR');
  \end{tikzpicture}
\end{equation}
commutes. Here $g$ and $k$ are induced from
applying vertical composition before horizontal,
$h$ from applying horizontal composition before
vertical, $j$ is from composing in either order,
$f$ is from horizontal composition as given in
Definition \ref{def:hor-vert-composition} and
$ \omega $ is from \eqref{diag.the big cube}.  The
top and bottom face commute by construction.

\begin{lemma}
  The inner triangles of diagram \eqref{diag.theta
    2-cell iso} commute. That is, we have
  $ k = f \rho \theta' $ and $ h = \theta' g $.
\end{lemma}

\begin{proof}
  To see that $ k = f \omega \theta' $, consider the
  diagram
  \[
    \begin{tikzpicture}
      \node (Yt) at (0,2) {$L\ast$};
      \node (Yb) at (0,0) {$L\ast$};
      \node (Spb) at (2.5,1.25) {$w \times_{y} w'$};
      \node (S') at (2.5,-0.75) {$w$};
      \node (Tpb) at (5,2.75) {$x \times_zx'$};
      \node (T') at (5,0.75) {$x$};
      \node (Ay) at (7.5,2) {$p'$};
      \node (SyT) at (7.5,0) {$y+_{L\ast}z$};
      \node (LyR) at (10,0) {$u+_{L\ast}v$};
      \draw [cd] (Yt) edge[] (Tpb);
      \draw [cd] (Yt) edge[] (Spb);
      \draw [font=\scriptsize,cd] (Yt) edge node[right] {$\iso$} (Yb);
      \draw [cd] (Yb) edge[] (T');
      \draw [cd] (Yb) edge[] (S');
      \draw [cd] (Tpb) edge[] (Ay);
      \draw [>->] (Tpb) edge[] (T');
      \draw [font=\scriptsize,cd] (Ay) edge[dashed] node[above] {$k$} (LyR);
      \draw [cd] (S') edge[] (SyT);
      \draw [font=\scriptsize,cd] (S') edge[bend right=20] node[above] {$\iota_us$} (LyR);
      \draw [font=\scriptsize,cd] (T') edge[bend left=20] node[above,pos=.2] {$\iota_vt$} (LyR);
      \draw [cd] (T') edge[] (SyT);
      \draw [font=\scriptsize,cd] (SyT) edge[dashed] node[above] {$f$} (LyR);
      \draw [cd] (Spb) edge[white,line width=4pt] (Ay);
      \draw [cd] (Spb) edge[] (Ay);
      \draw [>->] (Spb) edge[white,line width=4pt] (S');
      \draw [>->] (Spb) edge[] (S');
      \draw [cd] (Ay) edge[white,line width=4pt] (SyT);
      \draw [font=\scriptsize,cd] (Ay) edge[dashed] node[pos=0.75,left] {$\sigma'$} (SyT);
    \end{tikzpicture}
  \]
  The bottom face is exactly the pushout diagram
  from which $f$ was obtained.  Universality
  implies that $ k = f \sigma' $ and, as seen in
  \eqref{diag.the big cube}, $ \sigma' = \rho \theta' $.

  That $ h = \theta' g $ follows from
  \[
    f \rho h = j = k g = f \rho \theta' g
  \]
  and the fact that $ f \rho $ is monic.
\end{proof}

Of course, we have only shown that two of the four
inner triangles commute, but we can replicate our
arguments to show the remaining two commute as
well.  This lemma was the last step in proving
Theorem \ref{thm:interchange-law}, the interchange
law. 

\section{A symmetric monoidal structure}
\label{sec:sm-structure}

The double category $ _L\FFFineRewrite $ can be equipped
with a symmetric monoidal structure lifted from the
cocartesian structure on $ \A $ and $ \X $. Proving this
amounts to checking the axioms of Definition
\ref{def:mndl-dbl-cats}.

\begin{lemma}
  \label{lem:SpanCospanSM}
  $ _L\FFFineRewrite $ is a symmetric monoidal double category.  
\end{lemma}

\begin{proof}
  We denote $ _L\FFFineRewrite $ by $ \RRR $ for
  convenience.  Let us first show that the category of
  objects $\RRR_0$ and the category of arrows $\RRR_1$ are
  symmetric monoidal categories.

  We obtain the monoidal structure $ (\otimes_0, 0_{\A}) $ on $\RRR_{0}$ by lifting
  the cocartesian structure on $\A$ to the objects and by
  defining
  \[
    (a \xgets{f} b \xto{g} c) \otimes_0 (a' \xgets{f'} b' \xto{g'} c')
    \bydef
    (a+a' \xgets{f+g} b+b' \xto{f'+g'} c+c')
  \]
  on morphisms.  Universal properties provide the associator
  and unitors as well as the coherence axioms.  This
  monoidal structure is clearly symmetric.
	
  Next, we have the category $\RRR_1$ whose objects are the
  structured cospans and morphisms are their fine rewrites.
  We obtain a symmetric monoidal structure
  \[ ( \otimes_1 , L0_{\A} \to L0_{\A} \gets L0_{\A} ) \] on
  the objects via
  \[
    (La \to x \gets La') \otimes_1 (Lb \to y \gets Lb')
    \bydef
    (L(a+b) \to x+y \gets L(a'+b'))
  \]
  and on the morphisms by
  \[
    \raisebox{-0.5\height}{
      \begin{tikzpicture}
        \node (A) at (0,4) {$La$};
        \node (A') at (0,2) {$Lb$};
        \node (A'') at (0,0) {$Lc$};
        \node (B) at (2,4) {$v$};
        \node (B') at (2,2) {$w$};
        \node (B'') at (2,0) {$x$};
        \node (C) at (4,4) {$La'$};
        \node (C') at (4,2) {$Lb'$};
        \node (C'') at (4,0) {$Lc'$};
        \draw[cd]
        (A) edge node[above]{} (B)
        (A') edge node[above]{} (B')
        (A'') edge node[above]{} (B'')
        (C) edge node[above]{} (B)
        (C') edge node[above]{} (B')
        (C'') edge node[above]{} (B'')
        (A') edge node[left]{} (A)
        (A') edge node[left]{} (A'')
        (B') edge[>->] node[left]{} (B)
        (B') edge[>->] node[left]{} (B'')
        (C') edge node[left]{} (C)
        (C') edge node[left]{} (C'');	
      \end{tikzpicture}
    }
    \quad \otimes_1 \quad
    \raisebox{-0.5\height}{
      \begin{tikzpicture}
        \node (A) at (0,4) {$La''$};
        \node (A') at (0,2) {$Lb''$};
        \node (A'') at (0,0) {$Lc''$};
        \node (B) at (2,4) {$v'$};
        \node (B') at (2,2) {$w'$};
        \node (B'') at (2,0) {$x'$};
        \node (C) at (4,4) {$La'''$};
        \node (C') at (4,2) {$Lb'''$};
        \node (C'') at (4,0) {$Lc'''$};
        \draw[cd]
        (A) edge node[above]{} (B)
        (A') edge node[above]{} (B')
        (A'') edge node[above]{} (B'')
        (C) edge node[above]{} (B)
        (C') edge node[above]{} (B')
        (C'') edge node[above]{} (B'')
        (A') edge node[left]{} (A)
        (A') edge node[left]{} (A'')
        (B') edge[>->] node[left]{} (B)
        (B') edge[>->] node[left]{} (B'')
        (C') edge node[left]{} (C)
        (C') edge node[left]{} (C'');	
      \end{tikzpicture}
    }
    \quad \bydef \quad
    \]
    \[
    \raisebox{-0.5\height}{
      \begin{tikzpicture}
        \node (A) at (0,4) {$L(a+a'')$};
        \node (A') at (0,2) {$L(b+b'')$};
        \node (A'') at (0,0) {$L(c+c'')$};
        \node (B) at (3,4) {$v+v'$};
        \node (B') at (3,2) {$w+w'$};
        \node (B'') at (3,0) {$x+x'$};
        \node (C) at (6,4) {$L(a'+a''')$};
        \node (C') at (6,2) {$L(b'+b''')$};
        \node (C'') at (6,0) {$\L(c'+c''')$};
        \draw[cd]
        (A) edge node[above]{} (B)
        (A') edge node[above]{} (B')
        (A'') edge node[above]{} (B'')
        (C) edge node[above]{} (B)
        (C') edge node[above]{} (B')
        (C'') edge node[above]{} (B'')
        (A') edge node[left]{} (A)
        (A') edge node[left]{} (A'')
        (B') edge[>->] node[left]{} (B)
        (B') edge[>->] node[left]{} (B'')
        (C') edge node[left]{} (C)
        (C') edge node[left]{} (C'');	
      \end{tikzpicture}
    }
  \]
  Again, universal properties provide the associator,
  unitors, and coherence axioms.  Hence both $\RRR_0$
  and $\RRR_1$ are symmetric monoidal categories.

  It remains to find globular isomorphisms $\mathfrak{x}$
  and $\mathfrak{u}$ and their coherence. To find
  $\mathfrak{x}$, fix horizontal 1-morphisms
  \begin{align*}
    La & \to x \gets La',   &   La' & \to x' \gets La'', \\
    Lb & \to y \gets Lb',   &   Lb' & \to y' \gets Lb''.
  \end{align*}
  The globular isomorphism $\mathfrak{x}$ is an invertible 2-morphism with domain
  \[
    L(a+b) \to (x+y) +_{L(a'+b')} (x'+y') \gets L(a''+b'')
  \]
  and codomain
  \[
    L(a+b) \to (x+_{La'} y) + (x' +_{Lb'} y') \gets L(a''+b'')
  \]
  This comes down to finding an isomorphism in $\X$
  between the apexes of the above cospans.  Such an
  isomorphism exists, and is unique, because both apexes are
  colimits of the non-connected diagram
  \[
    \begin{tikzpicture}
      \node (a) at (0,0) {$La$};
      \node (b) at (1,.5) {$x$};
      \node (c) at (2,0) {$La'$};
      \node (d) at (3,.5) {$x'$};
      \node (e) at (4,0) {$La''$};
      \node (v) at (5,0) {$Lb$};
      \node (w) at (6,.5) {$y$};
      \node (x) at (7,0) {$Lb'$};
      \node (y) at (8,.5) {$y'$};
      \node (z) at (9,0) {$Lb''$};
      \path[cd,font=\scriptsize,>=angle 90]
      (a) edge node[above]{$$} (b)
      (c) edge node[above]{$$} (b)
      (c) edge node[above]{$$} (d)
      (e) edge node[above]{$$} (d)
      (v) edge node[above]{$$} (w)
      (x) edge node[above]{$$} (w)
      (x) edge node[above]{$$} (y)
      (z) edge node[above]{$$} (y);
    \end{tikzpicture}
  \]
  Moreover, the resulting globular isomorphism is a fine
  rewrite of structured cospans because the universal maps are
  isomorphisms. The globular isomorphism $\mathfrak{u}$ is
  similar.
	
  Finally, we check that the coherence axioms, namely
  (a)-(k) of Definition \ref{def:mndl-dbl-cats},
  hold.  These are straightforward, though tedious, to
  verify.  For instance, if we have
  \[
    \begin{tikzpicture}
      \node (a) at (0,0) {$La$};
      \node (b) at (1,.5) {$x$};
      \node (c) at (2,0) {$La'$};
      \node (M1) at (-1,.25) {$M_1 =$};
      \node (M2) at (3,.25) {$M_2 =$};
      \node (c2) at (4,0) {$La'$};
      \node (d) at (5,.5) {$x'$};
      \node (e) at (6,0) {$La''$};
      \node (M3) at (7,.25) {$M_3 =$};
      \node (e2) at (8,0) {$La''$};
      \node (f) at (9,.5) {$x''$};
      \node (g) at (10,0) {$La'''$};
      \node (t) at (0,-2) {$Lb$};
      \node (u) at (1,-1.5) {$y$};
      \node (v) at (2,-2) {$Lb'$};
      \node (N1) at (-1,-1.75) {$N_1 =$};
      \node (N2) at (3,-1.75) {$N_2 =$};
      \node (v2) at (4,-2) {$Lb'$};
      \node (w) at (5,-1.5) {$y'$};
      \node (x) at (6,-2) {$Lb''$};
      \node (N3) at (7,-1.75) {$N_3 =$};
      \node (x2) at (8,-2) {$Lb''$};
      \node (y) at (9,-1.5) {$y''$};
      \node (z) at (10,-2) {$Lb'''$};
      \path[cd,font=\scriptsize,>=angle 90]
      (a) edge node[above]{$$} (b)
      (c) edge node[above]{$$} (b)
      (c2) edge node[above]{$$} (d)
      (e) edge node[above]{$$} (d)
      (v2) edge node[above]{$$} (w)
      (e2) edge node[above]{$$} (f)
      (g) edge node[above]{$$} (f)
      (t) edge node[above]{$$} (u)
      (v) edge node[above]{$$} (u)
      (x) edge node[above]{$$} (w)
      (x2) edge node[above]{$$} (y)
      (z) edge node[above]{$$} (y);
    \end{tikzpicture}
  \]
  then following Diagram \eqref{diag:MonDblCat} around the
  top right gives the sequence of cospans
  \[
    \begin{tikzpicture}
      \node (a) at (-5,0) {$L(a+b)$};
      \node (b) at (0,0.5) {
        $((x+y) +_{ L( a' + b') } (x'+y'))
        +_{ L (a'' + b'') } (x''+y'') $ };
      \node (c) at (5,0) {$L(a''+b'')$};
      \node (M1) at (-2,1.5) {
        $((M_1 \otimes N_1) \odot
        (M_2 \otimes N_2)) \odot
        (M_3 \otimes N_3) = $ };
      \path[cd,font=\scriptsize]
      (a) edge[in=180,out=45] node[above]{$$} (b)
      (c) edge[in=0,out=135] node[above]{$$} (b);
    \end{tikzpicture}
  \]
  \[
    \begin{tikzpicture}
      \node (a) at (-5,0) {$L(a+b)$};
      \node (b) at (0,0.5) {
        $ ( (x +_{La'} x') +
        ( y +_{Lb'} y')) +_{ L(a''+b'') }
        (x''+y'') $ };
      \node (c) at (5,0) {$ L(a''+b'') $};
      \node (M1) at (-2,1.5) {
        $((M_1 \odot M_2) \otimes
        (N_1 \odot N_2)) \odot
        (M_3 \otimes N_3) = $ };
      \path[cd,font=\scriptsize]
      (a) edge[in=180,out=45] node[above]{$$} (b)
      (c) edge[in=0,out=135] node[above]{$$} (b);
    \end{tikzpicture}
  \]
  \[
    \begin{tikzpicture}
      \node (a) at (-5,0) { $ L(a+b) $ };
      \node (b) at (0,0.5) {
        $ ( ( x +_{La'} x') +_{La''} x'' ) +
        ( ( y +_{Lb'} y' ) +_{Lb''} y'' ) $ };
      \node (c) at (5,0) {$ L(a''+b'') $};
      \node (M1) at (-2,1.5) {
        $((M_1 \odot M_2) \odot M_3) \otimes
        ((N_1 \odot N_2) \odot N_3) = $};
      \path[cd,font=\scriptsize]
      (a) edge[in=180,out=45] node[above]{$$} (b)
      (c) edge[in=0,out=135] node[above]{$$} (b);
    \end{tikzpicture}
  \]
  \[
    \begin{tikzpicture}
      \node (a) at (-5,0) {$L(a+b)$};
      \node (b) at (0,0.5) {
        $ ( x +_{La'} ( x' +_{La''} x'' )) +
        ( y +_{Lb'} ( y' +_{Lb''} y'' ) ) $ };
      \node (c) at (5,0) {$ L(a''+b'') $};
      \node (M1) at (-2,1.5) {
        $( M_1 \odot (M_2 \odot M_3)) \otimes
        (N_1 \odot (N_2 \odot N_3)) = $ };
      \path[cd,font=\scriptsize]
      (a) edge[in=180,out=45] node[above]{$$} (b)
      (c) edge[in=0,out=135] node[above]{$$} (b);
    \end{tikzpicture}
  \]
  Following the diagram \eqref{diag:MonDblCat} around the
  bottom left gives another sequence of cospans
  \[
    \begin{tikzpicture}
      \node (a) at (-5,0) {$ L(a + b) $};
      \node (b) at (0,0.5) {
        $ ( (x + y) +_{L (a'+b')} (x'+y'))
        +_{L(a''+b'')} (x'' + y'' ) $ };
      \node (c) at (5,0) {$ L(a'' + b'') $};
      \node (M1) at (-2,1.5) {
        $ ((M_1 \otimes N_1) \odot
        (M_2 \otimes N_2)) \odot
        (M_3 \otimes N_3) = $ };
      \path[cd,font=\scriptsize]
      (a) edge[in=180,out=45] node[above]{$$} (b)
      (c) edge[in=0,out=135] node[above]{$$} (b);
    \end{tikzpicture}
  \]
  \[
    \begin{tikzpicture}
      \node (a) at (-5,0) {$L(a+b)$};
      \node (b) at (0,0.5) {
        $ (x + y) +_{L(a'+b')}
        ( (x' + y') +_{ L(a''+b'') } (x'' + y'') )$};
      \node (c) at (5,0) {$L(a'''+b''')$};
      \node (M1) at (-2,1.5) {
        $(M_1 \otimes N_1) \odot
        ((M_2 \otimes N_2) \odot
        (M_3 \otimes N_3)) = $ };
      \path[cd,font=\scriptsize]
      (a) edge[in=180,out=45] node[above]{$$} (b)
      (c) edge[in=0,out=135] node[above]{$$} (b);
    \end{tikzpicture}
  \]
  \[
    \begin{tikzpicture}
      \node (a) at (-5,0) {$ L(a+b) $};
      \node (b) at (0,0.5) {
        $ (x+y) +_{L(a'+b')}
        ( (x' +_{La''} x'') + ( y' +_{Lb''} y'') ) $};
      \node (c) at (5,0) {$ L(a'''+b''') $};
      \node (M1) at (-2,1.5) {
        $(M_1 \otimes N_1) \odot
        ((M_2 \odot M_3) \otimes
        (N_2 \odot N_3)) = $};
      \path[cd,font=\scriptsize]
      (a) edge[in=180,out=45] node[above]{$$} (b)
      (c) edge[in=0,out=135] node[above]{$$} (b);
    \end{tikzpicture}
  \]
  \[
    \begin{tikzpicture}
      \node (a) at (-5,0) {$ L(a+b) $};
      \node (b) at (0,.5) {
        $ ( x +_{La'} ( x' +_{La''} x'' )) +
        ( y +_{Lb'} ( y' +_{Lb''} y'' ) )$};
      \node (c) at (5,0) {$ L(a'''+b''') $};
      \node (M1) at (-2,1.5) {
        $(M_1 \odot (M_2 \odot M_3))
        \otimes (N_1 \odot (N_2 \odot N_3)) = $};
      \path[cd,font=\scriptsize]
      (a) edge[in=180,out=45] node[above]{$$} (b)
      (c) edge[in=0,out=135] node[above]{$$} (b);
    \end{tikzpicture}
  \]
  Putting these together gives the following commutative diagram.
  \[
    \begin{tikzpicture}
      \node (a) at (-6,0) {$ L(a+b) $};
      \node (b) at (0,0) {
        $ ( (x+y) +_{L(a'+b')} (x'+y')) +_{L(a''+b'')} (x''+y'')$};
      \node (c) at (6,0) {$L(a'''+b''')$};
      \node (a2) at (-6,1.5) {$L(a+b)$};
      \node (b2) at (0,1.5) {
        $ (( x +_{La'} x')+ ( y+_{Lb'} y')) +_{L(a''+b'')} (x''+y'')$};
      \node (c2) at (6,1.5) {$L(a'''+b''')$};
      \node (a3) at (-6,3) {$L(a+b)$};
      \node (b3) at (0,3) {
        $ (( x +_{La'} x') +_{La''} x'')+ (( y +_{Lb'} y') +_{Lb''} y'')$};
      \node (c3) at (6,3) {$L(a'''+b''')$};
      \node (a4) at (-6,4.5) {$L(a+b)$};
      \node (b4) at (0,4.5) {
        $ ( x +_{La'} (x'+_{La''} x'')) + ( y +_{Lb'} ( y'+_{Lb''} y''))$};
      \node (c4) at (6,4.5) {$L(a'''+b''')$};
      \node (a5) at (-6,-1.5) {$L(a+b)$};
      \node (b5) at (0,-1.5) {
        $ (x+y) +_{L(a'+b')} ( (x'+y') +_{L(a''+b'')} (x''+y''))$};
      \node (c5) at (6,-1.5) {$L(a'''+b''')$};
      \node (a6) at (-6,-3) {$L(a+b)$};
      \node (b6) at (0,-3) {
        $(x+y)+_{L(a'+b')} ( (x' +_{la''} x'') + ( y' +_{La''} y''))$};
      \node (c6) at (6,-3) {$L(a'''+b''')$};
      \node (a7) at (-6,-4.5) {$L(a+b)$};
      \node (b7) at (0,-4.5) {
        $( x +_{La'} (x' +_{La''} x'')) + ( y +_{Lb'} (y'+_{Lb''} y''))$};
      \node (c7) at (6,-4.5) {$L(a'''+b''')$};
      \path[cd,font=\scriptsize,>=angle 90]
      (a) edge node[above]{$$} (b)
      (c) edge node[above]{$$} (b)
      (a2) edge node[above]{$$} (b2)
      (c2) edge node[above]{$$} (b2)
      (a) edge node[above]{$$} (a2)
      (b) edge node[above]{$$} (b2)
      (c) edge node[above]{$$} (c2)
      (a3) edge node[above]{$$} (b3)
      (c3) edge node[above]{$$} (b3)
      (a2) edge node[above]{$$} (a3)
      (b2) edge node[above]{$$} (b3)
      (c2) edge node[above]{$$} (c3)
      (a4) edge node[above]{$$} (b4)
      (c4) edge node[above]{$$} (b4)
      (a3) edge node[above]{$$} (a4)
      (b3) edge node[above]{$$} (b4)
      (c3) edge node[above]{$$} (c4)
      (a5) edge node[above]{$$} (b5)
      (c5) edge node[above]{$$} (b5)
      (a) edge node[above]{$$} (a5)
      (b) edge node[above]{$$} (b5)
      (c) edge node[above]{$$} (c5)
      (a6) edge node[above]{$$} (b6)
      (c6) edge node[above]{$$} (b6)
      (a5) edge node[above]{$$} (a6)
      (b5) edge node[above]{$$} (b6)
      (c5) edge node[above]{$$} (c6)
      (a7) edge node[above]{$$} (b7)
      (c7) edge node[above]{$$} (b7)
      (a6) edge node[above]{$$} (a7)
      (b6) edge node[above]{$$} (b7)
      (c6) edge node[above]{$$} (c7);
    \end{tikzpicture}
  \]
  The vertical 1-morphisms on the left and right are the the
  respective identity spans on $L(a+b)$ and $L(a'''+b''')$.  The
  vertical 1-morphisms in the center are isomorphism classes
  of monic spans where each leg is given by a universal map
  between two colimits of the same diagram.  The horizontal
  1-morphisms are given by universal maps into coproducts
  and pushouts.  The top cospan is the same as the bottom
  cospan, making a bracelet-like figure in which all faces
  commute.  The other diagrams witnessing coherence are
  given in a similar fashion.
\end{proof}


\section{A compact closed bicategory of spans of cospans}
\label{sec:compact-closed-bicategory-spans-of-cospans}

Double categories have many nice features yet are not as
established in the world of higher categories as
bicategories. For those who more comfortable with
bicategories, we write this section to discuss a bicategory
of fine rewrites of structured cospans. Intuitively, it is
straightforward to pass from the double category
$ _L\FFFineRewrite $ to a bicategory of fine rewrites. By
only accepting the squares of $ _L\FFFineRewrite $ that fix
the inputs and outputs, that is disallow permutations, then
the only vertical arrows left are identities.  But a double
category with only identity vertical arrows is virtually a
bicategory. Care is needed, though, because to actually
remove a bicategory of fine rewrites from
$ _L\FFFineRewrite $ requires more rigor than simply picking
out only the vertical arrows that are the identity.

More than a bicategory, we can actually extract a compact closed
bicategory from the symmetric monoidal double category
$ _L\FFFineRewrite $. To obtain a symmetric monoidal
bicategory from $ _L\FFFineRewrite $, we use machinery
developed by Shulman \cite{shulman_contructing}. To show
that this bicategory is also compact closed, we use work by
Stay \cite{stay_cc-bicats}.

First, let us extract the `horizontal bicategory' of
$ _L\FFFineRewrite $, so named because we remove the
vertical arrows.

\begin{definition}
  Define $ _L\FFineRewrite $ to be the bicategory whose
  objects are the objects of $ \A $, 1-arrows are structured
  cospans, and 2-arrows are fine rewrite rules of form
  \begin{center}
    \begin{tikzpicture}
      \node (La) at (0,4) {$ La $};
      \node (x) at (2,4) {$ x $};
      \node (Lb) at (4,4) {$ Lb $};
      \node (La') at (0,2) {$ La $};
      \node (y) at (2,2) {$ y $};
      \node (Lb') at (4,2) {$ Lb $};
      \node (La'') at (0,0) {$ La $};
      \node (z) at (2,0) {$ z $};
      \node (Lb'') at (4,0) {$ Lb $};
      \draw [cd]
      (La) edge[] (x)
      (Lb) edge[] (x)
      (La') edge[] (y)
      (Lb') edge[] (y)
      (La'') edge[] (z)
      (Lb'') edge[] (z)
      (La') edge[] node[left]{$ \id $} (La)
      (La) edge[] node[left]{$ \id $} (La'')
      (y) edge[>->] (x)
      (y) edge[>->] (z)
      (Lb') edge[] node[right]{$ \id $} (Lb)
      (Lb') edge[] node[right]{$ \id $} (Lb''); 
    \end{tikzpicture}
  \end{center}
\end{definition}

That this is a double category follows from Shulman's
construction mentioned in Definition
\ref{def:horiz-bicat}. Had we used the same notation as that
definition, we would let $ _L\FFineRewrite \bydef
\mathcal{H} ( _L\FFFineRewrite ) $.

Shulman's construction continues to be useful, as we use it
to show that the double category $ _L\FFineRewrite $ is
symmetric monoidal.  The first step towards this is showing
that $ _L\FFFineRewrite $ is isofibrant (see Definition
\ref{def:Fibrant}).

\begin{lemma}
  \label{lem:SpanCospanIsofibrant}
  The symmetric monoidal double category $ _L\FFFineRewrite $
  is isofibrant.
\end{lemma}

\begin{proof}
  The companion of a vertical 1-morphism
  \[
    f = (a \xgets{\theta} b \xto{\psi} c)
  \]
  is given by
  \[
    \widehat{f} =
    ( La \xto{ L\theta^{-1} } Lb \xgets{ L\psi^{-1} } Lc )
  \]
  The required 2-arrows are given by
  \[
    \raisebox{-0.5\height}{
      \begin{tikzpicture}
        \node (A) at (0,4) {$La$};
        \node (A') at (0,2) {$Lb$};
        \node (A'') at (0,0) {$Lc$};
        \node (B) at (2,4) {$Lb$};
        \node (B') at (2,2) {$Lc$};
        \node (B'') at (2,0) {$Lc$};
        \node (C) at (4,4) {$Lc$};
        \node (C') at (4,2) {$Lc$};
        \node (C'') at (4,0) {$Lc$};
        \draw[cd]
        (A) edge node[above]{} (B)
        (A') edge node[above]{} (B')
        (A'') edge node[above]{} (B'')
        (C) edge node[above]{} (B)
        (C') edge node[above]{} (B')
        (C'') edge node[above]{} (B'')
        (A') edge node[left]{} (A)
        (A') edge node[left]{} (A'')
        (B') edge node[left]{} (B)
        (B') edge node[left]{} (B'')
        (C') edge node[left]{} (C)
        (C') edge node[left]{} (C'');
      \end{tikzpicture}
    }
    \quad \text{ and } \quad
    \raisebox{-0.5\height}{
      \begin{tikzpicture}
        \node (A) at (0,4) {$La$};
        \node (A') at (0,2) {$La$};
        \node (A'') at (0,0) {$La$};
        \node (B) at (2,4) {$La$};
        \node (B') at (2,2) {$La$};
        \node (B'') at (2,0) {$Lb$};
        \node (C) at (4,4) {$La$};
        \node (C') at (4,2) {$Lb$};
        \node (C'') at (4,0) {$Lc$};
        \path[cd,font=\scriptsize,>=angle 90]
        (A) edge node[above]{} (B)
        (A') edge node[above]{} (B')
        (A'') edge node[above]{} (B'')
        (C) edge node[above]{} (B)
        (C') edge node[above]{} (B')
        (C'') edge node[above]{} (B'')
        (A') edge node[left]{} (A)
        (A') edge node[left]{} (A'')
        (B') edge node[left]{} (B)
        (B') edge node[left]{} (B'')
        (C') edge node[left]{} (C)
        (C') edge node[left]{} (C'');
      \end{tikzpicture}
    }
  \]
  The conjoint of $f$ is given by $\check{f} = \widehat{f}^{\text{op}}$.
\end{proof}

Because the symmetric monoidal double category
$ _L\FFFineRewrite $ is isofibrant, Theorem \ref{thm:horz-bicat}
extracts a symmetric monoidal bicategory $ _L\FFineRewrite $
comprised of the same objects, structured cospans as arrows,
and isomorphism classes of fine rewrites of structured
cospans with form
\[
  \begin{tikzpicture}
    \node (A) at (0,4) {$La$};
    \node (A') at (0,2) {$La$};
    \node (A'') at (0,0) {$La$};
    \node (B) at (2,4) {$v$};
    \node (B') at (2,2) {$w$};
    \node (B'') at (2,0) {$x$};
    \node (C) at (4,4) {$La'$};
    \node (C') at (4,2) {$La'$};
    \node (C'') at (4,0) {$La'$};
    \path[cd,font=\scriptsize,>=angle 90]
    (A) edge node[above]{} (B)
    (A') edge node[above]{} (B')
    (A'') edge node[above]{} (B'')
    (C) edge node[above]{} (B)
    (C') edge node[above]{} (B')
    (C'') edge node[above]{} (B'')
    (A') edge node[left]{\scriptsize{$ \id  $}} (A)
    (A') edge node[left]{\scriptsize{$ \id  $}} (A'')
    (B') edge[>->] node[left]{} (B)
    (B') edge[>->] node[left]{} (B'')
    (C') edge node[right]{\scriptsize{$ \id  $}} (C)
    (C') edge node[right]{\scriptsize{$ \id  $}} (C'');	
  \end{tikzpicture}
\]
The difference between these fine rewrites and the squares
of $ _L\FFFineRewrite $ is that the vertical arrows are
identities. This is necessary given that bicategories have
no vertical arrows.  However, the isofibrancy condition
ensures that information carried by the vertical arrows is
encoded the horizontal arrows.  

\begin{theorem}
  \label{thm:SpansCospasAreSMBicat}
  $ _L\FFineRewrite$ is a symmetric monoidal bicategory.
\end{theorem}

\begin{proof}
  Lemma \ref{lem:SpanCospanIsofibrant} states that $
  _L\FFFineRewrite $ is isofibrant. The result then follows
  from Theorem \ref{thm:horz-bicat}. 
\end{proof}

It remains to show that this bicategory is compact
closed. This structure of $ _L\FFineRewrite $ is another
benefit of bicategories over double categories.  Currently,
there is no notion of compact closedness for double
categories.  However, it is a nice feature to have in a
category that serves as the syntax for open systems with
inputs and outputs.  Here, we mention again that the terms
`inputs' and `outputs' do not imply a causal structure.
Instead, they partition the interface of an open system into
two parts, the purpose of which manifests when composing a
pair of systems. If we connect an open system, considered as
an structured cospan $ La \to x \gets La' $, to another
system, then $ La $ is parts of the connection and $ La' $
is not or vice versa.  That is, partitioning an interface
into inputs and outputs allows a portion of the interface to
be part of a connection and the remain portion to be left
out of the connection. Compact closedness formalizes the
viewpoint that \emph{how} an interface is partitioned is
arbitrary.  Indeed, every possible partition of the
interface exists as an arrow in $ _L\FFineRewrite $. That
is, given a system $ x $ with interface $ i $, then for any
two subobjects $ a $, $ a' $ of $ i $ such that
$ a+a' \cong i $, there is an an arrow
$ La \to x \gets La' $ in $ _L\FFineRewrite $.

\begin{example}
  Denote by $ x $ the graph
  \[
    \begin{tikzpicture}
      \node (1) at (0,0) {$ \bullet_a $};
      \node (2) at (0,2) {$ \bullet_b $};
      \node (3) at (2,0) {$ \bullet_c $};
      \node (4) at (2,2) {$ \bullet_d$};
      \draw[graph]
        (1) edge (2)
        (1) edge (3)
        (3) edge[loop right] (3)
        (4) edge (2)
        (1) edge (4);
    \end{tikzpicture}
  \]
  with interface $ \{a,c,d\} $.  Then $ x $ appears as
  an arrow in $ _L\FFineRewrite $ where $ L $ is from 
  \[
    \adjunction{\Set}{\RGraph}{L}{R}{4}
  \]
  as all of the following
  \begin{align*}
    \{ a,c,d \} & \to x \gets 0     &
    \{ a,c \}   & \to x \gets \{ d \}       \\
    \{ a,d \}   & \to x \gets \{ c \}       &
    \{ c,d \}   & \to x \gets \{ a \}       \\
    \{ a \}     & \to x \gets \{ c,d \}     &
    \{ c \}     & \to x \gets \{ a,d \}     \\
    \{ d \}     & \to x \gets \{ a,c \}     &
    0   & \to x \gets \{ a,b,c\}                                            
  \end{align*}
\end{example}

The ability to change an input to an output and vice versa
comes from the compact closed structure. We take the
remainder of this section to show that $ _L\FFineRewrite $
is compact closed.  

We start with the following lemma. For this lemma, we
introduce the notation $ \nabla \from a+a \to a $ for the
folding map, which arises from the coproduct diagram
\[
  \begin{tikzpicture}
    \node (x1) at (-3,2) {$ a $};
    \node (xx) at (0,2) {$ a+a $};
    \node (x2) at (3,2) {$ a $};
    \node (x)  at (0,0) {$ a $};
    \draw [cd] (x1) to 
      node [above] {\scriptsize{$ \iota $}} (xx);
    \draw [cd] (x2) 
      to node [above] {\scriptsize{$ \iota $}} (xx);
    \draw [cd] (x1) to 
      node [left] {\scriptsize{$ \id $}} (x);
    \draw [cd] (x2) to 
      node [right] {\scriptsize{$ \id $}} (x);
    \draw [cd,dashed] (xx) to node [left] {\scriptsize{$ \nabla $}} (x);
  \end{tikzpicture}
\]

\begin{lemma}
  \label{lem:PushoutDiagram}
  In a category with coproducts, the diagram
  \[
    \begin{tikzpicture}
      \node (UL) at (0,2) {$a+a+a$};
      \node (LL) at (0,0) {$a+a$};
      \node (UR) at (4,2) {$a+a$};
      \node (LR) at (4,0) {$a$};
      \path[cd,font=\scriptsize,>=angle 90]
      (UL) edge node[above] {$\id + \nabla$} (UR)
      (UL) edge node[left] {$\nabla + \id$} (LL)
      (UR) edge node[right] {$\nabla$} (LR)
      (LL) edge node[above] {$\nabla$} (LR);
    \end{tikzpicture}
  \]
  is a pushout square.
\end{lemma}

\begin{proof}
  Suppose that we have two maps $ f,g \from a+a \to b$
  forming a cocone over the span inside the above diagram.
  Let the arrow $\iota_{\text{m}} \from a \to a+a+a$ include
  $a$ into the middle copy.  Observe that
  $\iota_{\text{l}} \coloneqq (\nabla + a) \circ \iota_{\text{m}}
  $ and $\iota_r \coloneqq (a + \nabla) \circ \iota_{\text{m}} $ are,
  respectively, the left and right inclusions $a \to a+a$.
  Then $f \circ \iota_{\text{l}} = g \circ \iota_{\text{r}}$ is a
  map $a \to b$, which we claim is the unique map making
  \begin{center}
    \begin{tikzpicture}
      \node (UL) at (0,2) {$a+a+a$};
      \node (LL) at (0,0) {$a+a$};
      \node (UR) at (4,2) {$a+a$};
      \node (LR) at (4,0) {$a$};
      \node (b) at (6,-2) {$ b $};
      \draw [cd]
      (LR) edge[dashed] (b) 
      (LL) edge[bend right] (b)
      (UR) edge[bend left] (b) 
      (UL) edge node[above] {$\id + \nabla$} (UR)
      (UL) edge node[left] {$\nabla + \id$} (LL)
      (UR) edge node[right] {$\nabla$} (LR)
      (LL) edge node[above] {$\nabla$} (LR);
    \end{tikzpicture}
  \end{center}
  commute.  Indeed, given $h \from a \to b$ such that
  $f = h \circ \nabla = g$, then
  $g \circ \iota_
  { \text{r} } = f \circ \iota_{ \text{l} } = h \circ
  \nabla \circ \iota_{ \text{l} } = h$.
\end{proof}

In the following theorem, we will make a slight abuse of
notation by writing $ \nabla $ to mean
\[
  L(a+a) \to La + La \xto{\nabla} La.
\]
Here, $ L(a+a) \to La+La $ is the structure map which is
invertible because, as a left adjoint, $ L $ preserves coproducts.

\begin{theorem}
  \label{thm:SpansMapsAreCCBicat}
  The symmetric monoidal bicategory $ _L\FFineRewrite $ is compact closed.
\end{theorem}

\begin{proof}
  First we show that each object is its own dual.
  For an object $a$, define the counit
  $\epsilon \from a + a \to 0$ and unit
  $\eta \from 0 \to a+a$ to be the following cospans:
  \[
    \epsilon \bydef (L(a+a) \xto{\nabla} La \gets 0), 
    \quad \quad 
    \eta \bydef (0 \to La \xgets{\nabla} L(a+a)).
  \]  
  Next we define the cusp isomorphisms, $\alpha$
  and $\beta$.  Note that $\alpha$ is a 2-morphism
  whose domain is the composite
  \[
    a \xto{ \iota_{ \text{l} } }
    a+a \xgets{ \id +\nabla }
    a+a+a \xto{ \nabla + \id }
    a+a \xgets{ \iota_{ \text{r} } }
    a
  \]
  and whose codomain is the identity cospan on $a$.  From
  Lemma \ref{lem:PushoutDiagram} we have the equations
  $\nabla+\id = \iota_{\text{l}} \circ \nabla$ and
  $\id + \nabla = \iota_{\text{r}} \circ \nabla$ from which it
  follows that the domain of $\alpha$ is the identity cospan
  on $a$, and the codomain of $\beta$ is also the identity
  cospan on $a$ obtained as the composite
  \[
    a \xto{\iota_{\text{r}}}
    a+a \xgets{\nabla+ \id}
    a+a+a \xto{X+\nabla}
    a+a \xgets{\iota_{\text{l}}}
    a
  \]
  Take $\alpha$ and $\beta$ each to be the
  isomorphism class determined by the identity
  2-morphism on $a$, which in particular is a
  monic span of cospans.  Thus we have a dual pair
  $(a,a,\epsilon,\eta,\alpha,\beta)$.  By Theorem
  \ref{thm:StrictingDualPairs}, there exists a
  cusp isomorphism $\beta'$ such that the tuple
  $(a,a,\epsilon,\eta,\alpha,\beta')$ is a coherent dual
  pair, and thus $  _L \FFineRewrite $ is
  compact closed.
\end{proof}


\chapter{Bold rewriting and structured cospans}
\label{sec:bold-rewriting}

We contrast this section with the previous section on fine
rewriting with an example.  In the fine rewriting of
structured cospans, we ask for rewrite rules with the monic
arrows as in the diagram
\[
  \begin{tikzpicture}
    \node (1) at (0,4) {$ La $};
    \node (2) at (2,4) {$ x $};
    \node (3) at (4,4) {$ La' $};
    \node (4) at (0,2) {$ Lb $};
    \node (5) at (2,2) {$ y $};
    \node (6) at (4,2) {$ Lb' $};
    \node (7) at (0,0) {$ Lc $};
    \node (8) at (2,0) {$ z $};
    \node (9) at (4,0) {$ Lc' $};
    \path[cd,font=\scriptsize]
    (1) edge node[]{$  $} (2)
    (3) edge node[]{$  $} (2)
    (4) edge node[]{$  $} (5)
    (6) edge node[]{$  $} (5)
    (7) edge node[]{$  $} (8)
    (9) edge node[]{$  $} (8)
    (4) edge node[left]{$ \iso $} (1)
    (4) edge node[left]{$ \iso $} (7)
    (6) edge node[right]{$ \iso $} (3)
    (6) edge node[right]{$ \iso $} (9);
    \path[>->,font=\scriptsize]
    (5) edge node[]{$  $} (2)
    (5) edge node[]{$  $} (8);
  \end{tikzpicture}
\]
There are situations, however, where requiring those monic
arrows is untenable. Consider, for instance, the string calculi
so frequently use to reason in monoidal categories. For
this example, we permit ourselves to ignore details and
subtleties so that we do not muddy the point we mean to
illustrate.  For a detailed and complete look at string
calculi, Selinger's survey \cite{selinger} provides an
excellent overview.

Given a monoidal category $ ( \C ,\otimes, I ) $, objects
are represented by certain isotopy classes of strings and
arrows are represented by nodes. This is illustrated in
Figure \ref{fig:string-diagrams}. The diagrams read from
left to right.  Now, to draw a string for an identity arrow,
we do not include the node, giving the diagram
\[
  \begin{tikzpicture}
    \node (1) at (0,0) {$ a $};
    \node (2) at (2,0) {$ a $};
    \draw [-] (1) to (2);
  \end{tikzpicture}
\]
to represent $ \id \from a \to a $. Composing with $ \id $
another arrow should result in nothing changing, as captured
in this equation
\begin{center}
  \begin{tikzpicture}
    \node (1) at (0,0) {$ a $};
    \node[circle,draw=black] (2) at (4,0) {$ f $};
    \node (3) at (6,0) {$ b $};
    \path[-,font=\scriptsize]
    (1) edge[] (2) 
    (2) edge node[]{$  $} (3);
    \node () at (7,0) {$ = $};
    \node (1) at (8,0) {$ a $};
    \node[circle,draw=black] (2) at (10,0) {$ f $};
    \node (3) at (12,0) {$ b $};
    \path[-,font=\scriptsize]
    (1) edge node[]{$ $} (2)
    (2) edge node[]{$ $} (3);
  \end{tikzpicture}
\end{center}
From this, we observe that the length of the string does not
matter.  This accords with defining strings up to
isotopy. In particular, we want to have a string be
equivalent to a point.  In the parlance of this thesis, we
want to be able to rewrite a string, with two distinct
endpoints, into a single point.  Yet, this is not possible
to do with a fine rewrite rule.

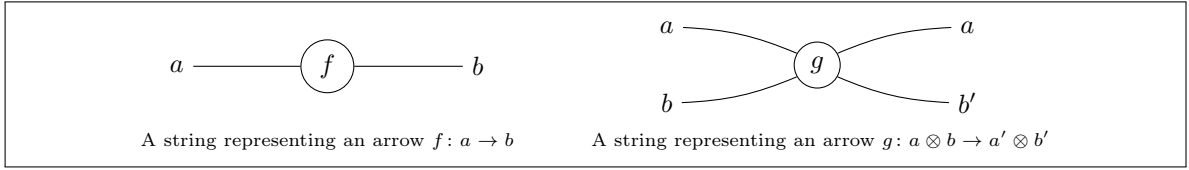
\begin{figure}[h]
  \centering
  \fbox{
    \begin{minipage}{1.0\linewidth}
      \[
        \begin{tikzpicture}
          \node[circle,draw=black] (1) at (0,0) {$ f $};
          \node (2) at (-2,0) {$ a $};
          \node (3) at (2,0) {$ b $};
          \path[-,font=\scriptsize]
          (2) edge node[]{$  $} (1)
          (1) edge node[]{$  $} (3);
          \node () at (0,-1) {\scriptsize{A string representing
              an arrow $ f \from a \to b $}};
        \end{tikzpicture}
        \quad \quad
        \begin{tikzpicture}
          \node[circle,draw=black] (1) at (0,0) {$ g $};
          \node (2) at (-2,0.5) {$ a $};
          \node (3) at (-2,-0.5) {$ b $};
          \node (4) at (2,0.5) {$ a $};
          \node (5) at (2,-0.5) {$ b' $};
          \path[-,font=\scriptsize]
          (2.0) edge [bend left=10] (1.150)
          (3.0) edge [bend right=10] (1.210)
          (1.30) edge [bend left=10] (4.180)
          (1.-30) edge [bend right=10] (5.180);      
          \node () at (0,-1) {\scriptsize{
              A string representing an arrow
              $ g \from a \otimes b \to a' \otimes b' $}};
        \end{tikzpicture}
      \]
    \end{minipage}
  }
  \caption{String diagrams}
  \label{fig:string-diagrams}
\end{figure}

Indeed, suppose we are working with strings in some topos of
spaces and we want to finely rewrite a string into a
point. Such a rewrite rule would be a span
\begin{equation} \label{eq:string-to-point}
  \begin{tikzpicture}
    \begin{scope}
      \node (1) at (0,1) {$ \bullet $};
      \node (2) at (0,0) {$ \bullet $};
      \draw (1.center) -- (2.center);
      \draw [rounded corners]
        (-0.5,-0.5) rectangle (0.5,1.5);
      \node (l) at (0.6,0.5) {};
    \end{scope}
    \begin{scope}[shift={(3,0)}]
      \node () at (0,0.5) {?};
      \draw [rounded corners]
        (-0.5,-0.5) rectangle (0.5,1.5);
      \node (ml) at (-0.6,0.5) {};
      \node (mr) at (0.6,0.5)  {};  
    \end{scope}
    \begin{scope}[shift={(6,0)}]
      \node (1) at (0,0.5) {$ \bullet $};
      \draw [rounded corners]
        (-0.5,-0.5) rectangle (0.5,1.5);
      \node (r) at (-0.6,0.5) {};
    \end{scope}
    \draw[cd]
      (ml) edge[] (l)
      (mr) edge[] (r);
  \end{tikzpicture}
\end{equation}
with `?' replaced by a subobject of both the string on the
left and point on the right.  Thus, `?' must either be empty
or a point. Choosing the empty string does not scale.  A
simple counter example is
\[
  \begin{tikzpicture}
    \begin{scope} 
      \node (a) at (0,0) {$ \bullet $};
      \node (b) at (1,0) {$ \bullet  $};
      \draw (a.center) -- (b.center);
      \draw [rounded corners] (-0.5,-0.5) rectangle (2.5,0.5);
      \node (01b) at (1,-0.6) {$  $};
      \node (01r) at (2.6,0)  {$  $};
    \end{scope}
    \begin{scope}[shift={(4,0)}]
      \node (c) at (1,0) {$ 0 $};
      \draw [rounded corners] (-0.5,-0.5) rectangle (2.5,0.5);
      \node (11l) at (-0.6,0) {$  $};
      \node (11b) at (1,-0.6) {$  $};
      \node (11r) at (2.6,0)  {$  $};
    \end{scope}
    \begin{scope}[shift={(8,0)}]
      \node (c) at (1,0) {$ \bullet $};
      \draw [rounded corners] (-0.5,-0.5) rectangle (2.5,0.5);
      \node (21l) at (-0.6,0) {$  $};
      \node (21b) at (1,-0.6) {$  $};
    \end{scope}
    \begin{scope}[shift={(0,-3)}]
      \node (a) at (0,0)  {$ \bullet $};
      \node (b) at (1,0)  {$ \bullet $};
      \node (c) at (2,1)  {$ \bullet $};
      \node (d) at (2,-1) {$ \bullet $};
      \draw
        (a.center) -- (b.center)
        (b.center) -- (c.center)
        (b.center) -- (d.center);
      \draw [rounded corners]
      (-0.5,-1.5) rectangle (2.5,1.5);
      \node (00l) at (-0.6,0) {$  $};
      \node (00t) at (1,1.6) {$  $};
      \node (00r) at (2.6,0)  {$  $};
    \end{scope}
    \begin{scope}[shift={(4,-3)}]
      \node () at (1,0) {?};
      \draw [rounded corners]
      (-0.5,-1.5) rectangle (2.5,1.5);
      \node (10l) at (-0.6,0) {$  $};
      \node (10t) at (1,1.6) {$  $};
      \node (10r) at (2.6,0)  {$  $};
    \end{scope}
    \begin{scope}[shift={(8,-3)}]
      \node (b) at (1,0)  {$ \bullet $};
      \node (c) at (2,1)  {$ \bullet $};
      \node (d) at (2,-1) {$ \bullet $};
      \draw
        (b.center) -- (c.center)
        (b.center) -- (d.center);
      \draw [rounded corners]
      (-0.5,-1.5) rectangle (2.5,1.5);
      \node (20l) at (-0.6,0) {$  $};
      \node (20t) at (1,1.6) {$  $};
      \node (20r) at (2.6,0)  {$  $};
    \end{scope}
    \draw[cd]
    (11l) edge node[]{$  $} (01r)
    (11r) edge node[]{$  $} (21l)
    (10l) edge node[]{$  $} (00r) 
    (10r) edge node[]{$  $} (20l)
    (01b) edge node[]{$  $} (00t)
    (11b) edge node[]{$  $} (10t)
    (21b) edge node[]{$  $} (20t); 
  \end{tikzpicture}
\]
To see this more clearly, we reframe the question to take
advantage of the fact that pushing out over $ 0 $ is
the same as taking a disjoint union. So we can ask whether 
\begin{center}
  \begin{tikzpicture}
      \node (a) at (0,0)  {$ \bullet $};
      \node (b) at (1,0)  {$ \bullet $};
      \node (c) at (2,1)  {$ \bullet $};
      \node (d) at (2,-1) {$ \bullet $};
      \draw
        (a.center) -- (b.center)
        (b.center) -- (c.center)
        (b.center) -- (d.center);
      \draw [rounded corners]
        (-0.5,-1.5) rectangle (2.5,1.5);
      \node (00l) at (-0.6,0) {$  $};
      \node (00t) at (1,1.6) {$  $};
      \node (00r) at (2.6,0)  {$  $};
  \end{tikzpicture}
\end{center}
is the disjoint union of 
\begin{center}
  \begin{tikzpicture}
    \node (a) at (0,0) {$ \bullet $};
    \node (b) at (1,0) {$ \bullet  $};
    \draw (a.center) -- (b.center);
    \draw [rounded corners] (-0.5,-0.5) rectangle (2.5,0.5);
    \node (01b) at (1,-0.6) {$  $};
    \node (01r) at (2.6,0)  {$  $};
  \end{tikzpicture}
\end{center}
and something else. Of course, it is not.   

But maybe the issue was pushing out over $ 0 $ in
the first place.  What about replacing $ 0 $ with a
point? A simple counter example to illustrate the failure of
this idea is
\[
  \begin{tikzpicture}
    \begin{scope} 
      \node (a) at (0,0) {$ \bullet $};
      \node (b) at (1,0) {$ \bullet  $};
      \draw (a.center) -- (b.center);
      \draw [rounded corners]
        (-1.5,-0.5) rectangle (2.5,0.5);
      \node (01l) at (-1.6,0) {$  $};
      \node (01b) at (0.5,-0.6) {$  $};
      \node (01r) at (2.6,0) {$  $};
    \end{scope}
    \begin{scope}[shift={(5,0)}]
      \node (c) at (0.5,0) {$ \bullet $};
      \draw [rounded corners]
      (-1.5,-0.5) rectangle (2.5,0.5);
      \node (11l) at (-1.6,0) {$  $};
      \node (11b) at (0.5,-0.6) {$  $};
      \node (11r) at (2.6,0) {$  $};
    \end{scope}
    \begin{scope}[shift={(10,0)}]
      \node (c) at (0.5,0) {$ \bullet $};
      \draw [rounded corners]
      (-1.5,-0.5) rectangle (2.5,0.5);
      \node (21l) at (-1.6,0) {$  $};
      \node (21b) at (0.5,-0.6) {$  $};
      \node (21r) at (2.6,0) {$  $};
    \end{scope}
    \begin{scope}[shift={(0,-3)}]
      \node (0) at (-1,1)  {$ \bullet $};
      \node (1) at (-1,-1) {$ \bullet $};
      \node (a) at (0,0)   {$ \bullet $};
      \node (b) at (1,0)   {$ \bullet $};
      \node (c) at (2,1)   {$ \bullet $};
      \node (d) at (2,-1)  {$ \bullet $};
      \draw
        (a.center) -- (b.center)
        (b.center) -- (c.center)
        (b.center) -- (d.center)
        (0.center) -- (a.center)
        (1.center) -- (a.center);
      \draw [rounded corners]
      (-1.5,-1.5) rectangle (2.5,1.5);
      \node (00l) at (-1.6,0) {$  $};
      \node (00t) at (0.5,1.6) {$  $};
      \node (00r) at (2.6,0) {$  $};
    \end{scope}
    \begin{scope}[shift={(5,-3)}]
      \node () at (0.5,0) {?};
      \draw [rounded corners]
      (-1.5,-1.5) rectangle (2.5,1.5);
      \node (10l) at (-1.6,0) {$  $};
      \node (10t) at (0.5,1.6) {$  $};
      \node (10r) at (2.6,0) {$  $};
    \end{scope}
    \begin{scope}[shift={(10,-3)}]
      \node (0) at (-0.5,1)  {$ \bullet $};
      \node (1) at (-0.5,-1) {$ \bullet $};
      \node (b) at (0.5,0)  {$ \bullet $};
      \node (c) at (1.5,1)  {$ \bullet $};
      \node (d) at (1.5,-1) {$ \bullet $};
      \draw (0.center) -- (b.center);
      \draw (1.center) -- (b.center);
      \draw (b.center) -- (c.center);
      \draw (b.center) -- (d.center);
      \draw [rounded corners]
      (-1.5,-1.5) rectangle (2.5,1.5);
      \node (20l) at (-1.6,0)  {$  $};
      \node (20t) at (0.5,1.6) {$  $};
      \node (20r) at (2.6,0)  {$  $};
    \end{scope}
    \draw[cd]
    (11l) edge node[above]{$ \theta $} (01r)
    (11r) edge node[]{$  $} (21l)
    (10l) edge node[]{$  $} (00r)
    (10r) edge node[]{$  $} (20l)
    (01b) edge node[]{$  $} (00t)
    (11b) edge node[]{$  $} (10t)
    (21b) edge node[]{$  $} (20t); 
  \end{tikzpicture}
\]
where we define $ \theta $ to choose the left or the right
point; the failure will occur regardless of the choice.
Again, there is nothing that we can place into the center,
bottom square to give a double pushout diagram. To see why,
we use the fact that if we could fill in `?', we already
know what it must be. The right square must also be a
pushout. This forces us to fill the blank with the graph
\begin{center}
  \begin{tikzpicture}
    \node (0) at (-0.5,1)  {$ \bullet $};
    \node (1) at (-0.5,-1) {$ \bullet $};
    \node (b) at (0.5,0)  {$ \bullet $};
    \node (c) at (1.5,1)  {$ \bullet $};
    \node (d) at (1.5,-1) {$ \bullet $};
    \draw (0.center) -- (b.center);
    \draw (1.center) -- (b.center);
    \draw (b.center) -- (c.center);
    \draw (b.center) -- (d.center);
    \draw [rounded corners]
    (-1.5,-1.5) rectangle (2.5,1.5);
    \node (20l) at (-1.6,0)  {$  $};
    \node (20t) at (0.5,1.6) {$  $};
    \node (20r) at (2.6,0)  {$  $};
  \end{tikzpicture}
\end{center}
But then the left square is not a pushout.

And so, fine rewriting can be insufficient.  In this
chapter, we define bold rewriting of structured cospans
to handle situations like this one found in string
calculi. We see that, though it largely mirrors the fine
rewriting of structured cospans, it has its own character:
the bicategory we extract is a bicategory of relations.  At
the end of the chapter, we illustrate bold rewriting using
the string calculus from quantum computer science known as
the ZX-calculus.

\section{A double category of bold rewrites of structured
  cospans}
\label{sec:bold-rewrites}

In this section, we define a double category
$ _L\BBBoldRewrite $ whose objects are interface types,
whose vertical arrows are spans of interface types with
invertible legs, whose horizontal arrows are structured
cospans, and whose squares are bold rewrites of structured
cospans. The only difference between the definitions of
$ _L\FFFineRewrite $ and $ _L\BBBoldRewrite $ is in the
squares. The objects, horizontal arrows, and vertical arrows
are the same in each case. This winds up having an
interesting effect on the horizontal bicategory of
$ _L\BBBoldRewrite $ which we explore in Section
\ref{sec:cartesian-bicategory-spans-of-cospans}. Before
turning to that, we need to properly define
$ _L\BBBoldRewrite $, Fortunately, most of the work has been
done when constructing $ _L\FFFineRewrite $, so we
begin by defining the squares.

Recall from Definition \ref{def:span-str-cospan} that a
morphism of spans of structured cospans is an arrow
$ \theta $ that fits into a commuting diagram
\[
  \begin{tikzpicture}
    \node (La) at (0,4) {$ La $};
    \node (Lb) at (1,2) {$ Lb $};
    \node (Lc) at (2,0) {$ Lc $};
    \node (x) at (3,4) {$ x $};
    \node (y) at (4,3.5) {$ y $};
    \node (y') at (4,0.5) {$ y' $};
    \node (z) at (5,0) {$ z $};
    \node (La') at (6,4) {$ La' $};
    \node (Lb') at (7,2) {$ Lb' $};
    \node (Lc') at (8,0) {$ Lc' $};
    \draw[cd]
    (Lb) edge node[]{$  $} (La)
    (Lb) edge node[]{$  $} (Lc)
    (Lb') edge node[]{$  $} (La')
    (Lb') edge node[]{$  $} (Lc')
    (y) edge node[]{$  $} (x)
    (y') edge node[]{$  $} (x)
    (y') edge node[]{$  $} (z)
    (La) edge node[]{$  $} (x)
    (La') edge node[]{$  $} (x)
    (Lb) edge node[]{$  $} (y')
    (Lb') edge node[]{$  $} (y');
    \path[line width=0.5em,draw=white]
    (Lb) edge node[]{$  $} (y)
    (Lb') edge node[]{$  $} (y)
    (Lc) edge node[]{$  $} (z)
    (Lc') edge node[]{$  $} (z)
    (y) edge node[]{$  $} (z)
    (y) edge node[]{$  $} (y');
    \path[cd,font=\scriptsize]
    (Lb) edge node[]{$  $} (y)
    (Lb') edge node[]{$  $} (y)
    (Lc) edge node[]{$  $} (z)
    (Lc') edge node[]{$  $} (z)
    (y) edge node[]{$  $} (z)
    (y) edge node[left] {$ \theta $} (y');
  \end{tikzpicture}
\]

Using a morphism of structured cospans, we can define the
connected components of structured cospans.  We first define
a relation $ \sim $ setting
\[
  \begin{tikzpicture}
    \begin{scope}
      \node (La) at (0,4) {$ La $};
      \node (x) at (2,4) {$ x $};
      \node (La') at (4,4) {$ La' $};
      \node (Lb) at (0,2) {$ Lb $};
      \node (y) at (2,2) {$ y $};
      \node (Lb') at (4,2) {$ Lb' $};
      \node (Lc) at (0,0) {$ Lc $};
      \node (z) at (2,0) {$ z $};
      \node (Lc') at (4,0) {$ Lc' $};
      \draw[cd]
      (La) edge node[]{$  $} (x)
      (La') edge node[]{$  $} (x)
      (Lb) edge node[]{$  $} (y)
      (Lb') edge node[]{$  $} (y)
      (Lc) edge node[]{$  $} (z)
      (Lc') edge node[]{$  $} (z)
      (Lb) edge node[]{$  $} (La)
      (Lb) edge node[]{$  $} (Lc)
      (y) edge node[]{$  $} (x)
      (y) edge node[]{$  $} (z)
      (Lb') edge node[]{$  $} (La')
      (Lb') edge node[]{$  $} (Lc');
    \end{scope}
    \node () at (5,2) {$ \sim $};
    \begin{scope}[shift={(6,0)}]
      \node (La) at (0,4) {$ La $};
      \node (x) at (2,4) {$ x $};
      \node (La') at (4,4) {$ La' $};
      \node (Lb) at (0,2) {$ Lb $};
      \node (y) at (2,2) {$ y' $};
      \node (Lb') at (4,2) {$ Lb' $};
      \node (Lc) at (0,0) {$ Lc $};
      \node (z) at (2,0) {$ z $};
      \node (Lc') at (4,0) {$ Lc' $};
      \draw[cd]
      (La) edge node[]{$  $} (x)
      (La') edge node[]{$  $} (x)
      (Lb) edge node[]{$  $} (y)
      (Lb') edge node[]{$  $} (y)
      (Lc) edge node[]{$  $} (z)
      (Lc') edge node[]{$  $} (z)
      (Lb) edge node[]{$  $} (La)
      (Lb) edge node[]{$  $} (Lc)
      (y) edge node[]{$  $} (x)
      (y) edge node[]{$  $} (z)
      (Lb') edge node[]{$  $} (La')
      (Lb') edge node[]{$  $} (Lc');
    \end{scope}    
  \end{tikzpicture}
\]
if there is a morphism from the rewriting on the left side
of $ \sim $ to that on the right.  A \defn{connected component of
  structured cospans} is an equivalence class generated by
$ \sim $. The coarseness of the classes of squares is the
most important distinction between fine and bold rewriting.

\begin{definition}[Bold rewrite]
  \label{def:bold-rewrite-str-cospans}
  A \defn{bold rewrite of structured cospans} is a connected
  component of structured cospans whose representative has
  the form
  \[
    \begin{tikzpicture}
      \begin{scope}
        \node (La)  at (0,4) {$ La $};
        \node (x)   at (2,4) {$ x $};
        \node (La') at (4,4) {$ La' $};
        \node (Lb)  at (0,2) {$ Lb $};
        \node (y)   at (2,2) {$ y $};
        \node (Lb') at (4,2) {$ Lb' $};
        \node (Lc)  at (0,0) {$ Lc $};
        \node (z)   at (2,0) {$ z $};
        \node (Lc') at (4,0) {$ Lc' $};
        \draw[cd]
        (La)  edge node[]{$  $} (x)
        (La') edge node[]{$  $} (x)
        (Lb)  edge node[]{$  $} (y)
        (Lb') edge node[]{$  $} (y)
        (Lc)  edge node[]{$  $} (z)
        (Lc') edge node[]{$  $} (z)
        (Lb)  edge node[left]{$ \iso  $} (La)
        (Lb)  edge node[left]{$ \iso  $} (Lc)
        (y)   edge node[]{$  $} (x)
        (y)   edge node[]{$  $} (z)
        (Lb') edge node[right]{$ \iso $} (La')
        (Lb') edge node[right]{$ \iso $} (Lc');
      \end{scope}
    \end{tikzpicture}
  \]
\end{definition}

The horizontal and vertical compositions for bold rewrites
of structured cospans are defined in the same way as for
fine rewrites. The classes are different, but the operation
on the class representatives work in the same way.

\begin{definition}
  \label{def:composition-bold-rewrites}
  The \defn{horizontal composition} $ \hcirc $ of bold
  rewrites of structured cospans are defined by
  the operation
  \[
    \begin{tikzpicture}
      \begin{scope}
      \node (La)  at (0,4) {$ La $};
      \node (x)   at (2,4) {$ x $};
      \node (La') at (4,4) {$ La' $};
      \node (Lb)  at (0,2) {$ Lb $};
      \node (y)   at (2,2) {$ y $};
      \node (Lb') at (4,2) {$ Lb' $};
      \node (Lc)  at (0,0) {$ Lc $};
      \node (z)   at (2,0) {$ z $};
      \node (Lc') at (4,0) {$ Lc' $};
      \draw[cd]
      (La)  edge node[]{$  $} (x)
      (La') edge node[]{$  $} (x)
      (Lb)  edge node[]{$  $} (y)
      (Lb') edge node[]{$  $} (y)
      (Lc)  edge node[]{$  $} (z)
      (Lc') edge node[]{$  $} (z)
      (Lb)  edge node[]{$  $} (La)
      (Lb)  edge node[]{$  $} (Lc)
      (y)   edge node[]{$  $} (x)
      (y)   edge node[]{$  $} (z)
      (Lb') edge node[]{$  $} (La')
      (Lb') edge node[]{$  $} (Lc');
    \end{scope}
    \node () at (5,2) {$ \hcirc $};
    \begin{scope}[shift={(6,0)}]
      \node (La)  at (0,4) {$ La' $};
      \node (x)   at (2,4) {$ x' $};
      \node (La') at (4,4) {$ La'' $};
      \node (Lb)  at (0,2) {$ Lb' $};
      \node (y)   at (2,2) {$ y' $};
      \node (Lb') at (4,2) {$ Lb'' $};
      \node (Lc)  at (0,0) {$ Lc' $};
      \node (z)   at (2,0) {$ z' $};
      \node (Lc') at (4,0) {$ Lc'' $};
      \draw[cd]
      (La)  edge node[]{$  $} (x)
      (La') edge node[]{$  $} (x)
      (Lb)  edge node[]{$  $} (y)
      (Lb') edge node[]{$  $} (y)
      (Lc)  edge node[]{$  $} (z)
      (Lc') edge node[]{$  $} (z)
      (Lb)  edge node[]{$  $} (La)
      (Lb)  edge node[]{$  $} (Lc)
      (y)   edge node[]{$  $} (x)
      (y)   edge node[]{$  $} (z)
      (Lb') edge node[]{$  $} (La')
      (Lb') edge node[]{$  $} (Lc');
    \end{scope}
    \node () at (11,2) {$ \bydef  $};
    \end{tikzpicture}
  \]
  \[
    \begin{tikzpicture}
    \begin{scope}[shift={(0,0)}]
      \node (La)  at (0,4) {$ La $};
      \node (x)   at (3,4) {$ x +_{La'} x' $};
      \node (La') at (6,4) {$ La'' $};
      \node (Lb)  at (0,2) {$ Lb $};
      \node (y)   at (3,2) {$ y +_{Lb'} y' $};
      \node (Lb') at (6,2) {$ Lb'' $};
      \node (Lc)  at (0,0) {$ Lc $};
      \node (z)   at (3,0) {$ z +_{Lc'} z' $};
      \node (Lc') at (6,0) {$ Lc'' $};
      \draw[cd]
      (La)  edge node[]{$  $} (x)
      (La') edge node[]{$  $} (x)
      (Lb)  edge node[]{$  $} (y)
      (Lb') edge node[]{$  $} (y)
      (Lc)  edge node[]{$  $} (z)
      (Lc') edge node[]{$  $} (z)
      (Lb)  edge node[]{$  $} (La)
      (Lb)  edge node[]{$  $} (Lc)
      (y)   edge node[]{$  $} (x)
      (y)   edge node[]{$  $} (z)
      (Lb') edge node[]{$  $} (La')
      (Lb') edge node[]{$  $} (Lc');
    \end{scope}
    \end{tikzpicture}
  \]
  The \defn{vertical composition} of bold rewrites of structured
  cospans is defined by
  \[
    \begin{tikzpicture}
      \begin{scope}
      \node (La)  at (0,4) {$ La $};
      \node (x)   at (2,4) {$ v $};
      \node (La') at (4,4) {$ La' $};
      \node (Lb)  at (0,2) {$ Lb $};
      \node (y)   at (2,2) {$ w $};
      \node (Lb') at (4,2) {$ Lb' $};
      \node (Lc)  at (0,0) {$ Lc $};
      \node (z)   at (2,0) {$ x $};
      \node (Lc') at (4,0) {$ Lc' $};
      \draw[cd]
      (La)  edge node[]{$  $} (x)
      (La') edge node[]{$  $} (x)
      (Lb)  edge node[]{$  $} (y)
      (Lb') edge node[]{$  $} (y)
      (Lc)  edge node[]{$  $} (z)
      (Lc') edge node[]{$  $} (z)
      (Lb)  edge node[]{$  $} (La)
      (Lb)  edge node[]{$  $} (Lc)
      (y)   edge node[]{$  $} (x)
      (y)   edge node[]{$  $} (z)
      (Lb') edge node[]{$  $} (La')
      (Lb') edge node[]{$  $} (Lc');
    \end{scope}
    \node () at (5,2) {$ \hcirc $};
    \begin{scope}[shift={(6,0)}]
      \node (La)  at (0,4) {$ Lc $};
      \node (x)   at (2,4) {$ x $};
      \node (La') at (4,4) {$ Lc' $};
      \node (Lb)  at (0,2) {$ Ld $};
      \node (y)   at (2,2) {$ y $};
      \node (Lb') at (4,2) {$ Ld' $};
      \node (Lc)  at (0,0) {$ Le $};
      \node (z)   at (2,0) {$ z $};
      \node (Lc') at (4,0) {$ Le' $};
      \draw[cd]
      (La)  edge node[]{$  $} (x)
      (La') edge node[]{$  $} (x)
      (Lb)  edge node[]{$  $} (y)
      (Lb') edge node[]{$  $} (y)
      (Lc)  edge node[]{$  $} (z)
      (Lc') edge node[]{$  $} (z)
      (Lb)  edge node[]{$  $} (La)
      (Lb)  edge node[]{$  $} (Lc)
      (y)   edge node[]{$  $} (x)
      (y)   edge node[]{$  $} (z)
      (Lb') edge node[]{$  $} (La')
      (Lb') edge node[]{$  $} (Lc');
    \end{scope}
    \node () at (11,2) {$ \bydef  $};
    \end{tikzpicture}
  \]
  \[
    \begin{tikzpicture}
    \begin{scope}[shift={(0,0)}]
      \node (La)  at (0,4) {$ La $};
      \node (x)   at (3,4) {$ v $};
      \node (La') at (6,4) {$ La' $};
      \node (Lb)  at (0,2) {$ Lb \times_{Lc} Ld $};
      \node (y)   at (3,2) {$ w \times_{x} y $};
      \node (Lb') at (6,2) {$ Lb' \times_{Lc'} Ld' $};
      \node (Lc)  at (0,0) {$ Le $};
      \node (z)   at (3,0) {$ z $};
      \node (Lc') at (6,0) {$ Le' $};
      \draw[cd]
      (La)  edge node[]{$  $} (x)
      (La') edge node[]{$  $} (x)
      (Lb)  edge node[]{$  $} (y)
      (Lb') edge node[]{$  $} (y)
      (Lc)  edge node[]{$  $} (z)
      (Lc') edge node[]{$  $} (z)
      (Lb)  edge node[]{$  $} (La)
      (Lb)  edge node[]{$  $} (Lc)
      (y)   edge node[]{$  $} (x)
      (y)   edge node[]{$  $} (z)
      (Lb') edge node[]{$  $} (La')
      (Lb') edge node[]{$  $} (Lc');
    \end{scope}
    \end{tikzpicture}
  \]
\end{definition}

Unlike for fine rewrites of structured cospans, the
interchange law is straightforward to prove. The coarser
classes of rewrites of structured cospans vastly simplifies
concocting the isomorphism. 

\begin{lemma}
  \label{thm:bold-interchange}
  Let
  \[
    \begin{tikzpicture}
      \node ()    at (-1,2) {$ \alpha \bydef $};
      \node (La)  at (0,4) {$ La $};
      \node (x)   at (2,4) {$ v $};
      \node (La') at (4,4) {$ La' $};
      \node (Lb)  at (0,2) {$ Lb $};
      \node (y)   at (2,2) {$ w $};
      \node (Lb') at (4,2) {$ Lb' $};
      \node (Lc)  at (0,0) {$ Lc $};
      \node (z)   at (2,0) {$ x $};
      \node (Lc') at (4,0) {$ Lc' $};
      \draw [cd]
      (La)  edge node[]{$  $} (x)
      (La') edge node[]{$  $} (x)
      (Lb)  edge node[]{$  $} (y)
      (Lb') edge node[]{$  $} (y)
      (Lc)  edge node[]{$  $} (z)
      (Lc') edge node[]{$  $} (z)
      (Lb)  edge node[]{$  $} (La)
      (Lb)  edge node[]{$  $} (Lc)
      (y)   edge node[]{$  $} (x)
      (y)   edge node[]{$  $} (z)
      (Lb') edge node[]{$  $} (La')
      (Lb') edge node[]{$  $} (Lc');
    \end{tikzpicture}
    \quad \quad \quad \quad
    \begin{tikzpicture}
      \node ()    at (-1,2) {$ \alpha' \bydef $};
      \node (La)  at (0,4) {$ La' $};
      \node (x)   at (2,4) {$ v' $};
      \node (La') at (4,4) {$ La'' $};
      \node (Lb)  at (0,2) {$ Lb' $};
      \node (y)   at (2,2) {$ w' $};
      \node (Lb') at (4,2) {$ Lb'' $};
      \node (Lc)  at (0,0) {$ Lc' $};
      \node (z)   at (2,0) {$ x' $};
      \node (Lc') at (4,0) {$ Lc'' $};
      \draw[cd]
      (La)  edge node[]{$  $} (x)
      (La') edge node[]{$  $} (x)
      (Lb)  edge node[]{$  $} (y)
      (Lb') edge node[]{$  $} (y)
      (Lc)  edge node[]{$  $} (z)
      (Lc') edge node[]{$  $} (z)
      (Lb)  edge node[]{$  $} (La)
      (Lb)  edge node[]{$  $} (Lc)
      (y)   edge node[]{$  $} (x)
      (y)   edge node[]{$  $} (z)
      (Lb') edge node[]{$  $} (La')
      (Lb') edge node[]{$  $} (Lc');
    \end{tikzpicture}
  \]
  \[
    \begin{tikzpicture}
      \node ()    at (-1,2) {$ \beta \bydef $};
      \node (La)  at (0,4) {$ Lc $};
      \node (x)   at (2,4) {$ x $};
      \node (La') at (4,4) {$ Lc' $};
      \node (Lb)  at (0,2) {$ Ld $};
      \node (y)   at (2,2) {$ y $};
      \node (Lb') at (4,2) {$ Ld' $};
      \node (Lc)  at (0,0) {$ Le $};
      \node (z)   at (2,0) {$ z $};
      \node (Lc') at (4,0) {$ Le' $};
      \draw[cd]
      (La)  edge node[]{$  $} (x)
      (La') edge node[]{$  $} (x)
      (Lb)  edge node[]{$  $} (y)
      (Lb') edge node[]{$  $} (y)
      (Lc)  edge node[]{$  $} (z)
      (Lc') edge node[]{$  $} (z)
      (Lb)  edge node[]{$  $} (La)
      (Lb)  edge node[]{$  $} (Lc)
      (y)   edge node[]{$  $} (x)
      (y)   edge node[]{$  $} (z)
      (Lb') edge node[]{$  $} (La')
      (Lb') edge node[]{$  $} (Lc');
    \end{tikzpicture}
    \quad \quad \quad \quad
    \begin{tikzpicture}
      \node ()    at (-1,2) {$ \beta' \bydef $};
      \node (La)  at (0,4) {$ Lc' $};
      \node (x)   at (2,4) {$ x' $};
      \node (La') at (4,4) {$ Lc'' $};
      \node (Lb)  at (0,2) {$ Ld' $};
      \node (y)   at (2,2) {$ y' $};
      \node (Lb') at (4,2) {$ Ld'' $};
      \node (Lc)  at (0,0) {$ Le' $};
      \node (z)   at (2,0) {$ z' $};
      \node (Lc') at (4,0) {$ Le'' $};
      \draw[cd]
      (La)  edge node[]{$  $} (x)
      (La') edge node[]{$  $} (x)
      (Lb)  edge node[]{$  $} (y)
      (Lb') edge node[]{$  $} (y)
      (Lc)  edge node[]{$  $} (z)
      (Lc') edge node[]{$  $} (z)
      (Lb)  edge node[]{$  $} (La)
      (Lb)  edge node[]{$  $} (Lc)
      (y)   edge node[]{$  $} (x)
      (y)   edge node[]{$  $} (z)
      (Lb') edge node[]{$  $} (La')
      (Lb') edge node[]{$  $} (Lc');
    \end{tikzpicture}
  \]
  be bold rewrites of structured cospans. Then
  \[
    ( \alpha \hcirc \alpha' ) \vcirc ( \beta \hcirc \beta' )
    =
    ( \alpha \vcirc \beta ) \hcirc ( \alpha' \vcirc \beta' ).
  \]
  That is, the interchange law holds.
\end{lemma}

\begin{proof}
  The left hand side of the equation is the bold rewrite of
  structured cospans
  \[
    \begin{tikzpicture}
      \begin{scope}
      \node (La)  at (0,4) {$ La $};
      \node (x)   at (5,4) {$ v +_{La'} v' $};
      \node (La') at (10,4) {$ La'' $};
      \node (Lb)  at (0,2) {$ Lb \times_{Lc} Ld $};
      \node (y)   at (5,2)
        {$ ( w +_{Lb'} w' )
        \times_{( x +_{Lc'} x' )}
        ( y \times_{Ld'} y') $};
      \node (Lb') at (10,2) {$ Lb \times_{Lc'} Ld' $};
      \node (Lc)  at (0,0) {$ Le $};
      \node (z)   at (5,0) {$ z +_{Le'} z' $};
      \node (Lc') at (10,0) {$ Le'' $};
      \draw[cd]
      (La)  edge node[]{$  $} (x)
      (La') edge node[]{$  $} (x)
      (Lb)  edge node[]{$  $} (y)
      (Lb') edge node[]{$  $} (y)
      (Lc)  edge node[]{$  $} (z)
      (Lc') edge node[]{$  $} (z)
      (Lb)  edge node[]{$  $} (La)
      (Lb)  edge node[]{$  $} (Lc)
      (y)   edge node[]{$  $} (x)
      (y)   edge node[]{$  $} (z)
      (Lb') edge node[]{$  $} (La')
      (Lb') edge node[]{$  $} (Lc');
      \end{scope}
    \end{tikzpicture}
  \]
  while the right hand side is
  \[
    \begin{tikzpicture}
      \begin{scope}
      \node (La)  at (0,4) {$ La $};
      \node (x)   at (5,4) {$ v +_{La'} v' $};
      \node (La') at (10,4) {$ La'' $};
      \node (Lb)  at (0,2) {$ Lb \times_{Lc} Ld $};
      \node (y)   at (5,2)
        {$ ( w \times_{x} y)
          +_{(Lb' \times_{Lc'} Ld')}
          ( w' \times_{x'} y' ) $};
      \node (Lb') at (10,2) {$ Lb \times_{Lc'} Ld' $};
      \node (Lc)  at (0,0) {$ Le $};
      \node (z)   at (5,0) {$ z +_{Le'} z' $};
      \node (Lc') at (10,0) {$ Le'' $};
      \draw[cd]
      (La)  edge node[]{$  $} (x)
      (La') edge node[]{$  $} (x)
      (Lb)  edge node[]{$  $} (y)
      (Lb') edge node[]{$  $} (y)
      (Lc)  edge node[]{$  $} (z)
      (Lc') edge node[]{$  $} (z)
      (Lb)  edge node[]{$  $} (La)
      (Lb)  edge node[]{$  $} (Lc)
      (y)   edge node[]{$  $} (x)
      (y)   edge node[]{$  $} (z)
      (Lb') edge node[]{$  $} (La')
      (Lb') edge node[]{$  $} (Lc');
      \end{scope}
    \end{tikzpicture}
  \]
  To show that these are equal as bold rewrites of
  structured cospans, it suffices to find a \emph{morphism} between
  them. Precisely, we need a morphism
  \[
    ( w\times_{x}y )
    +_{( Lb'\times_{Lc'}Ld')}
    (w'\times_{x'}y' )
    \to
    ( w+_{Lb'}w' )
    \times_{(x+_{Lc'}x')}
    ( y\times_{Ld'}y')
  \]

  We can obtain the two objects as follows. Let $ \C $ be
  the walking cospan category $ \{ \bullet \to \bullet \gets
  \bullet \} $ and let $ \S $ be the walking span category $
  \{ \bullet \gets \bullet \to \bullet \} $.  Then $ \C
  \times \S $ is the walking cospan of spans category
  \[
     \begin{tikzpicture}
       \begin{scope}
         \node (1) at (0,4) {$ \bullet $};
         \node (2) at (2,4) {$ \bullet $};
         \node (3) at (4,4) {$ \bullet $};
         \node (4) at (0,2) {$ \bullet $};
         \node (5) at (2,2) {$ \bullet $};
         \node (6) at (4,2) { $ \bullet $};
         \node (7) at (0,0) {$ \bullet $};
         \node (8) at (2,0) {$ \bullet $};
         \node (9) at (4,0) {$ \bullet $};
         \draw[cd]
          (1) edge node[]{$  $} (4)
          (7) edge node[]{$  $} (4)
          (2) edge node[]{$  $} (5)
          (8) edge node[]{$  $} (5)
          (3) edge node[]{$  $} (6)
          (9) edge node[]{$  $} (6)
          (2) edge node[]{$  $} (1)
          (2) edge node[]{$  $} (3)
          (5) edge node[]{$  $} (4)
          (5) edge node[]{$  $} (6)
          (8) edge node[]{$  $} (7)
          (8) edge node[]{$  $} (9);
        \end{scope}
        \draw [rounded corners]
          (-1,-1) rectangle (5,5);
    \end{tikzpicture}
  \]
  Let $ F \from \C \times \S \to \X $ be the functor that
  returns the diagram
  \[
    \begin{tikzpicture}
      \begin{scope}
        \node (1) at (0,4) {$ w $};
        \node (2) at (2,4) {$ Lb' $};
        \node (3) at (4,4) {$ w' $};
        \node (4) at (0,2) {$ x $};
        \node (5) at (2,2) {$ Lc' $};
        \node (6) at (4,2) { $ x' $};
        \node (7) at (0,0) {$ y $};
        \node (8) at (2,0) {$ Ld' $};
        \node (9') at (4,0) {$ y' $};
        \draw[cd]
        (1) edge node[]{$  $} (4)
        (7) edge node[]{$  $} (4)
        (2) edge node[]{$  $} (5)
        (8) edge node[]{$  $} (5)
        (3) edge node[]{$  $} (6)
        (9) edge node[]{$  $} (6)
        (2) edge node[]{$  $} (1)
        (2) edge node[]{$  $} (3)
        (5) edge node[]{$  $} (4)
        (5) edge node[]{$  $} (6)
        (8) edge node[]{$  $} (7)
        (8) edge node[]{$  $} (9);
      \end{scope}
    \end{tikzpicture}
  \]
  which is the middle of the diagram obtained by gluing $
  \alpha$, $ \beta $, $ \alpha' $, and $ \beta' $ together
  along their coinciding edges.  There is a canonical
  morphism of type
  \[
    \colim\limits_\S \lim_\C F \to
    \lim_\C \colim\limits_\S F
  \]
  where the domain is the image of $ F $ under the composite functor
  \[
    \X^{\C \times \S} \xto{\iso}
    ( \X^\C )^\S \xto{\lim_\C}
    \X^\S \xto{\colim_\S} \X
  \]
  and the domain is the image of $ F $ under the composite
  functor
  \[
    \X^{\C \times \S} \xto{\iso}
    ( \X^\S )^\C \xto{\colim_\S}
    \X^\C \xto{\lim_\C} \X.
  \]
  One can check that this canonical morphism gives the
  morphism of bold rewrites of structured cospans we need.  
\end{proof}


\section{A bicategory of relations for bold
  rewriting of structured cospans}
\label{sec:cartesian-bicategory-spans-of-cospans}

There are two philosophies in rewriting. One is that we care
about how one object is rewritten into another, and so we
keep track of certain data to describe the rewriting.  The
other perspective is that we do not care about \emph{how} an
object is rewritten into another, only that the rewriting is
possible.  Bold rewriting of structured cospans belongs to
the latter philosophy. This is realized explicitly through
the fact that the horizontal bicategory forms a bicategory
of relations, specifically that it is locally posetal.
Appendix \ref{sec:cartesian-bicategories} discusses the
theory of such bicategories.

The first goal of this section is to define the bicategory
in question.  We take the same approach as finding the
horizontal bicategory of fine rewrites of structured cospans
in Section
\ref{sec:compact-closed-bicategory-spans-of-cospans}.  After
extracting the bicategory, we show that it is a bicategory
of relations (see Definition \ref{def:bicat-relations}).

This next theorem is proved with virtually the same argument
as Lemma \ref{lem:SpanCospanSM}.

\begin{theorem}
  $ _L \BBBoldRewrite $ is a symmetric monoidal double
  category.
\end{theorem}

From here, we prove a series of lemmas that, when put
together, prove that the horizontal bicategory
$ _L \BBoldRewrite $ of $ _L \BBBoldRewrite $ is a
bicategory of relations. The first lemma in this string is
proved by replicating the proof of Lemma
  \ref{lem:SpanCospanIsofibrant} and the second follows
  from Theorem \ref{thm:horz-bicat}.

\begin{lemma}
  $ _L \BBBoldRewrite $ is isofibrant.
\end{lemma}

\begin{lemma}
  $ _L \BBoldRewrite $ is a symmetric monoidal bicategory.  
\end{lemma}

In the following lemma, we use
$ \nabla \bydef [ \id,\id] \from a + a \to a $ to
denote the codiagonal map and $ ! $ to denote a canonical
arrow from the initial object.

\begin{lemma}
  For each object $ a $ of $ _L \BBoldRewrite $, define
  operations
  \[
    \Delta_a \from a \to a + a
    \quad \text{and} \quad
    \varepsilon_a \from a \to 0
  \]
  to be the structured cospans
  \[
    La \xto{\id} La \xgets{L\nabla_a} L(a+a)
    \quad \text{and} \quad
    La \xto{\id} La \xgets{!} L0
  \]
  respectively. Then $ (a , \Delta_a , \varepsilon_a ) $ is
  a cocommutative comonoid.
\end{lemma}

\begin{proof}
  Proving this amounts to showing that the
  coassociativity, counitality, and
  cocommutativity diagrams commute.
  The coassociativity diagram
  \[
    \begin{tikzpicture}
      \node (01) at (0,2) {$ a $};
      \node (21) at (6,2) {$ a+a $};
      \node (00) at (0,0) {$ a+a $};
      \node (10) at (3,0) {$ a+(a+a) $};
      \node (20) at (6,0) {$ (a+a)+a $};
      \draw[cd]
      (01) edge node[above]{$ \nabla $} (21)
      (01) edge node[left]{$ \nabla $} (00)
      (00) edge node[below]{$ \alpha $} (10)
      (10) edge node[below]{$ \id \otimes \nabla $} (20)
      (21) edge node[right]{$ \nabla \otimes \id $} (20);
    \end{tikzpicture}
  \]
  commutes because the top path, which is the composite
  \[
    \begin{tikzpicture}
      \node (00) at (0,0) {$ La $};
      \node (10) at (4,0) {$ L(a+a) $};
      \node (20) at (8,0) {$ L((a+a)+a) $};
      \node (01) at (2,2) {$ La $};
      \node (11) at (6,2) {$ L((a+a)+a $};
      \node (12) at (4,4) {$ La $};
      \draw[cd]
      (00) edge node[above,left]{$ \id $} (01)
      (10) edge node[above,right]{$ L\nabla $} (01)
      (10) edge node[above,left]{$ \id $} (11)
      (20) edge node[above,right]{$ L(\nabla + \id) $} (11)
      (01) edge node[above,left]{$ \id $} (12)
      (11) edge node[above,right]{$ L\nabla $} (12);
      \draw (3.9,3.5) -- (4,3.4) -- (4.1,3.5);
    \end{tikzpicture}
  \]
  equals the bottom path, which is the composite
  \[
    \begin{tikzpicture}
      \node (00) at (0,0) {$ La $};
      \node (10) at (4,0) {$ L(a+a) $};
      \node (20) at (8,0) {$ L(a+(a+a)) $};
      \node (30) at (12,0) {$ L((a+a)+a $};
      \node (01) at (2,2) {$ La $};
      \node (11) at (6,2) {$ L(a+a) $};
      \node (21) at (10,2) {$ L(a+(a+a)) $};
      \node (02) at (4,4) {$ La $};
      \node (12) at (8,4) {$ L(a+a) $};
      \node (03) at (6,6) {$ La $};
      \draw[cd]
      (00) edge node[above,left]{$ \id $} (01)
      (10) edge node[above,right]{$ L\nabla $} (01)
      (10) edge node[above,left]{$ \id $} (11)
      (20) edge node[above,right]{$ L(\id + \nabla) $} (11)
      (20) edge node[above,left]{$ \id  $} (21)
      (30) edge node[above,right]{$ L(\alpha) $} (21)
      (01) edge node[above,left]{$ \id $} (02)
      (11) edge node[above,right]{$ L\nabla $} (02)
      (11) edge node[above,left]{$ \id $} (12)
      (21) edge node[above,right]{$ L(\id + \nabla) $} (12)
      (02) edge node[above,left]{$ \id $} (03)
      (12) edge node[above,right]{$ L\nabla $} (03);
      \draw (3.9,3.5) -- (4,3.4) -- (4.1,3.5);
      \draw (7.9,3.5) -- (8,3.4) -- (8.1,3.5);
      \draw (5.9,5.5) -- (6,5.4) -- (6.1,5.5);
    \end{tikzpicture}
  \]
  The counitality diagram
  \[
    \begin{tikzpicture}
      \node (11) at (0,0) {$ 0 \otimes a $};
      \node (21) at (3,0) {$ a \otimes a $};
      \node (31) at (6,0) {$ a \otimes 0 $};
      \node (12) at (3,2) {$ a $};
      \draw[cd]
      (12) edge node[left]{$ \lambda $} (11)
      (12) edge node[right]{$ \nabla $} (21)
      (12) edge node[right]{$ \rho $} (31)
      (21) edge node[below]{$ \epsilon \otimes \id $} (11)
      (21) edge node[below]{$ \id \otimes \epsilon $} (31);
    \end{tikzpicture}
  \]
  commutes because the composite
  \[
    \begin{tikzpicture}
      \node (11) at (0,0)   {$ La $};
      \node (21) at (4,0)   {$ L(a+a) $};
      \node (31) at (8,0)   {$ La $};
      \node (12) at (2,2) {$ La $};
      \node (22) at (6,2) {$ L(a+a) $};
      \node (13) at (4,4)   {$ L(0+a) $};
      \draw[cd]
      (11) edge node[above,left]{$ \id $} (12)
      (21) edge node[above,right]{$ L\nabla $} (12)
      (21) edge node[above,left]{$ \id $} (22)
      (31) edge node[above,right]{$ L(!+\id) $} (22)
      (12) edge node[above,left]{$ \id $} (13)
      (22) edge node[above,right]{$ L\nabla $} (13);
      \draw (3.9,3.5) -- (4,3.4) -- (4.1,3.5);
    \end{tikzpicture}
  \]
  is equal to the left unitor and
  \[
    \begin{tikzpicture}
      \node (11) at (0,0)   {$ La $};
      \node (21) at (4,0)   {$ L(a+a) $};
      \node (31) at (8,0)   {$ La $};
      \node (12) at (2,2) {$ La $};
      \node (22) at (6,2) {$ L(a+a) $};
      \node (13) at (4,4)   {$ La $};
      \draw[cd]
      (11) edge node[above,left]{$ \id $} (12)
      (21) edge node[above,right]{$ L\nabla $} (12)
      (21) edge node[above,left]{$ \id $} (22)
      (31) edge node[above,right]{$ L(\id+!) $} (22)
      (12) edge node[above,left]{$ \id $} (13)
      (22) edge node[above,right]{$ L\nabla $} (13);
      \draw (3.9,3.5) -- (4,3.4) -- (4.1,3.5);
    \end{tikzpicture}
  \]
  is the right unitor. Finally, the cocommutative
  diagram
  \[
    \begin{tikzpicture}
      \node (11) at (0,0) {$ a \otimes a $};
      \node (21) at (4,0) {$ a \otimes a $};
      \node (12) at (2,2) {$ a $};
      \draw[cd]
      (12) edge node[above,left]{$ \nabla $} (11)
      (12) edge node[above,right]{$ \nabla $} (21)
      (11) edge node[below]{$ \beta $} (21);
    \end{tikzpicture}
  \]
  commutes because the composite $ \nabla \beta $
  is given by
  \[
    \begin{tikzpicture}
      \node (11) at (0,0) {$ La $};
      \node (21) at (4,0) {$ L(a+a) $};
      \node (31) at (8,0) {$ L(a+a) $};
      \node (12) at (2,2) {$ La $};
      \node (22) at (6,2) {$ L(a+a) $};
      \node (13) at (4,4) {$ La $};
      \draw[cd]
      (11) edge node[above,left]{$ \id $} (12)
      (21) edge node[above,right]{$ L\nabla $} (12)
      (21) edge node[above,left]{$ \beta $} (22)
      (31) edge node[above,right]{$ \id $} (22)
      (12) edge node[above,left]{$ \id $} (13)
      (22) edge node[above,right]{$ L\nabla $}
      (13);
      \draw (3.9,3.5) -- (4,3.4) -- (4.1,3.5);
    \end{tikzpicture}
  \]
  which is exactly the comuliplication.  
\end{proof}

In the following lemma, we follow the convention of writing
$ f \leq g $ to represent a 2-arrow from $ f $ to $ g $ in a
locally posetal bicategory.  This notation is faithful to
the fact that the hom-categories are actually hom-posets.  This
is discussed further in Section
\ref{sec:cartesian-bicategories}.

\begin{lemma}
  Let $ (a,\Delta_a,\varepsilon_a) $ and
  $ ( b,\Delta_b,\varepsilon_b ) $ be cocommutative comonoid
  objects in the double category $ _L \BBoldRewrite $. Every structured cospan $
  La \to x \gets Lb $ in $ _L \BBoldRewrite $ is a lax comonoid
  homomorphism. That is,
  \[
    \Delta_b x \leq (x + x) \Delta_a
    \quad \text{and} \quad
    \varepsilon_b x \leq \varepsilon a
  \]
\end{lemma}

\begin{proof}
  The first 2-arrow is
  \[
    \begin{tikzpicture}
      \node (11) at (3,0) {$ La +_{L(a+a)} (x+x) $};
      \node (12) at (0,2) {$ La $};
      \node (22) at (3,2) {$ La +_{L(a+a)} (x+x) $};
      \node (32) at (9,2) {$ L(b+b) $};
      \node (13) at (3,4) {$ x $};
      \node (3'2) at (6,2) {$ x+x $};
      \draw[cd]
      (12) edge node[above,left]{$ f $} (13)
      (12) edge node[above]{$ \psi $} (22)
      (12) edge node[left]{$ \psi $} (11)
      (32) edge node[above]{$ \langle g,g \rangle $} (13)
      (32) edge node[above]{$ g+g $} (3'2)
      (3'2) edge node[above]{$ \theta $} (22)
      (3'2) edge node[right]{$ \nabla $} (13)
      (32) edge node[right]{$ \theta (g+g) $} (11)
      (22) edge [dashed] node[left]{$  $} (13)
      (22) edge node[left]{$ \id $} (11);
    \end{tikzpicture}
  \]
  where the dashed line is the universal arrow
  formed in reference to $ f $ and $ \nabla
  $. The source of this 2-arrow is the composite
  \[
    \begin{tikzpicture}
      \node (11) at (0,0) {$ La $};
      \node (21) at (4,0) {$ Lb $};
      \node (31) at (8,0) {$ L(b+b) $};
      \node (12) at (2,2) {$ x $};
      \node (22) at (6,2) {$ Lb $};
      \node (13) at (4,4) {$ x $};
      \draw[cd]
      (11) edge node[above,left]{$ f $} (12)
      (21) edge node[above,right]{$ g $} (12)
      (21) edge node[above,left]{$ \id $} (22)
      (31) edge node[above,right]{$ L\nabla $} (22)
      (12) edge node[above,left]{$ \id $} (13)
      (22) edge node[above,right]{$ g $} (13);
      \draw (3.9,3.5) -- (4,3.4) -- (4.1,3.5);
    \end{tikzpicture}
  \]
  and the target is the composite
  \[
    \begin{tikzpicture}
      \node (11) at (0,0)   {$ La $};
      \node (21) at (4,0)   {$ L(a+a) $};
      \node (31) at (8,0)   {$ L(b+b) $};
      \node (12) at (2,2) {$ La $};
      \node (22) at (6,2) {$ x+x $};
      \node (13) at (4,4)   {$ La +_{L(a+a)} (x+x) $};
      \draw[cd]
      (11) edge node[above,left]{$ \id $} (12)
      (21) edge node[above,right]{$ L\nabla $} (12)
      (21) edge node[above,left]{$ f+f $} (22)
      (31) edge node[above,right]{$ g+g $} (22)
      (12) edge node[above,left]{$ \psi $} (13)
      (22) edge node[above,right]{$ \theta $} (13);
      \draw (3.9,3.5) -- (4,3.4) -- (4.1,3.5);
    \end{tikzpicture}
  \]
    The second is witnessed by the 2-arrow
  \[
       \begin{tikzpicture}
      \node (11) at (2,0) {$ x $};
      \node (12) at (0,2) {$ La $};
      \node (22) at (2,2) {$ La $};
      \node (32) at (4,2) {$ L0 $};
      \node (13) at (2,4) {$ La $};
      \draw[cd]
      (12) edge node[above]{$ \id $} (13)
      (12) edge node[above]{$ \id $} (22)
      (12) edge node[below]{$ f $} (11)
      (32) edge node[above]{$ ! $} (13)
      (32) edge node[above]{$ ! $} (22)
      (32) edge node[below]{$ ! $} (11)
      (22) edge node[left]{$ \id $} (13)
      (22) edge node[left]{$ f $} (11);      
    \end{tikzpicture}
  \]
  where target 2-arrow is the composite
  \[
    \begin{tikzpicture}
      \node (11) at (0,0)   {$ La $};
      \node (21) at (4,0)   {$ Lb $};
      \node (31) at (8,0)   {$ L0 $};
      \node (12) at (2,2) {$ x $};
      \node (22) at (6,2) {$ Lb $};
      \node (13) at (4,4)   {$ x $};
      \draw[cd]
      (11) edge node[above]{$ f $} (12)
      (21) edge node[above]{$ g $} (12)
      (21) edge node[above]{$ \id $} (22)
      (31) edge node[above]{$ ! $} (22)
      (12) edge node[above]{$ \id $} (13)
      (22) edge node[above]{$ g $} (13);
      \draw (3.9,3.5) -- (4,3.4) -- (4.1,3.5);
    \end{tikzpicture}
    \qedhere
  \]  
\end{proof}

\begin{lemma}
  For any object $ a $ in $ _L \BBoldRewrite $, each
  cocommutative comonoid structure map
  \[
    \nabla \bydef
    \left(
      La \xto{\id} La \xgets{L\nabla_a} L(a+a)
      \right)
    \quad \text{ and } \quad
    \epsilon \bydef
    \left(
      La \xto{\id} La \xgets{!} L0
      \right)
  \]
  has a right adjoint (see Definition
  \ref{def:adjoint-in-bicat}), respectively,
  \[
    \nabla^\ast \bydef
    \left(
      L(a+a) \xto{L\nabla_a} La \xgets{\id} La
      \right)
    \quad \text{ and } \quad
    \epsilon^\ast \bydef
    \left(
      L0 \xto{!} La \xgets{\id} La \right). 
  \] 
\end{lemma}

\begin{proof}
  The unit of the adjunction
  $ \nabla \dashv \nabla^\ast $ is
  \[
    \begin{tikzpicture}
      \node (11) at (4,0) {$ La $};
      \node (12) at (0,2) {$ La $};
      \node (22) at (4,2) {$ La $};
      \node (32) at (8,2) {$ La $};
      \node (13) at (4,4) {$ La $};
      \draw[cd]
      (12) edge node[above,left]{$ \id $} (13)
      (12) edge node[above]{$ \id $} (22)
      (12) edge node[below,left]{$ \id $} (11)
      (32) edge node[above,right]{$ \id $} (13)
      (32) edge node[above]{$ \id $} (22)
      (32) edge node[below,left]{$ \id $} (11)
      (22) edge node[left]{$ \id $} (13)
      (22) edge node[left]{$ \id $} (11);
    \end{tikzpicture}
  \]
  where the target is the composite 1-arrow
  \[
    \begin{tikzpicture}
      \node (11) at (0,0)   {$ La $};
      \node (21) at (4,0)   {$ L(a+a) $};
      \node (31) at (8,0)   {$ La $};
      \node (12) at (2,2) {$ La $};
      \node (22) at (6,2) {$ La $};
      \node (13) at (4,4)   {$ La $};
      \draw[cd]
      (11) edge node[above,left]{$ \id $} (12)
      (21) edge node[above,right]{$ \nabla $} (12)
      (21) edge node[above,left]{$ \nabla $} (22)
      (31) edge node[above,right]{$ \id $} (22)
      (12) edge node[above,left]{$ \id $} (13)
      (22) edge node[above,right]{$ \id $} (13);
      \draw (3.9,3.5) -- (4,3.4) -- (4.1,3.5);
    \end{tikzpicture}
  \]
  The counit of $ \nabla \dashv \nabla^\ast $ is
  the 2-arrow
  \[
    \begin{tikzpicture}
      \node (11) at (4,0) {$ L(a+a) $};
      \node (12) at (0,2) {$ L(a+a) $};
      \node (22) at (4,2) {$ L(a+a) $};
      \node (32) at (8,2) {$ L(a+a) $};
      \node (13) at (4,4) {$ La $};
      \draw[cd]
      (12) edge node[above,left]{$ \nabla $} (13)
      (12) edge node[above]{$ \id $} (22)
      (12) edge node[below,left]{$ \id $} (11)
      (32) edge node[above,right]{$ \nabla $} (13)
      (32) edge node[above]{$ \id $} (22)
      (32) edge node[below,left]{$ \id $} (11)
      (22) edge node[left]{$ \nabla $} (13)
      (22) edge node[left]{$ \id $} (11);
    \end{tikzpicture}
  \]
  where the source is the composite 1-arrow
  \[
    \begin{tikzpicture}
      \node (11) at (0,0) {$ L(a+a) $};
      \node (21) at (4,0) {$ La $};
      \node (31) at (8,0) {$ L(a+a) $};
      \node (12) at (2,2) {$ La $};
      \node (22) at (6,2) {$ La $};
      \node (13) at (4,4) {$ La $};
      \draw[cd]
      (11) edge node[above,left]{$ \nabla $} (12)
      (21) edge node[above,right]{$ \id $} (12)
      (21) edge node[above,left]{$ \id $} (22)
      (31) edge node[above,right]{$ \nabla $} (22)
      (12) edge node[above,left]{$ \id $} (13)
      (22) edge node[above,right]{$ \id $} (13);
      \draw (3.9,3.5) -- (4,3.4) -- (4.1,3.5);
    \end{tikzpicture}
  \]
  Checking the triangle identities is
  straightforward.

  The unit of the adjunction $ \epsilon \dashv
  \epsilon^\ast $ is the 2-arrow
  \[
    \begin{tikzpicture}
      \node (11) at (4,0) {$ L(a+a) $};
      \node (12) at (0,2) {$ La $};
      \node (22) at (4,2) {$ L(a+a) $};
      \node (32) at (8,2) {$ La $};
      \node (13) at (4,4) {$ La $};
      \draw[cd]
      (12) edge node[above,left]{$ \id $} (13)
      (12) edge node[above]{$ \lambda $} (22)
      (12) edge node[left]{$ \lambda $} (11)
      (32) edge node[right]{$ \id $} (13)
      (32) edge node[above]{$ \rho $} (22)
      (32) edge node[right]{$ \rho $} (11)
      (22) edge node[left]{$ \nabla $} (13)
      (22) edge node[left]{$ \id $} (11);
    \end{tikzpicture}
  \]
  where the target is the composite 1-arrow
  \[
    \begin{tikzpicture}
      \node (11) at (0,0) {$ La $};
      \node (21) at (4,0) {$ L0 $};
      \node (31) at (8,0) {$ La $};
      \node (12) at (2,2) {$ La $};
      \node (22) at (6,2) {$ La $};
      \node (13) at (4,4) {$ L(a+a) $};
      \draw[cd]
      (11) edge node[above,left]{$ \id $} (12)
      (21) edge node[above,right]{$ ! $} (12)
      (21) edge node[above,left]{$ ! $} (22)
      (31) edge node[above,right]{$ \id $} (22)
      (12) edge node[above,left]{$ \lambda $} (13)
      (22) edge node[above,right]{$ \rho $} (13);
      \draw (3.9,3.5) -- (4,3.4) -- (4.1,3.5);
    \end{tikzpicture}
  \]
  the counit of $ \epsilon \dashv \epsilon^\ast $
  is the 2-arrow
  \[
    \begin{tikzpicture}
      \node (11) at (4,0) {$ La $};
      \node (12) at (0,2) {$ L0 $};
      \node (22) at (4,2) {$ La $};
      \node (32) at (8,2) {$ L0 $};
      \node (13) at (4,4) {$ La $};
      \draw[cd]
      (12) edge node[above]{$ ! $} (13)
      (12) edge node[above]{$ ! $} (22)
      (12) edge node[below]{$ ! $} (11)
      (32) edge node[above]{$ ! $} (13)
      (32) edge node[above]{$ ! $} (22)
      (32) edge node[below]{$ ! $} (11)
      (22) edge node[left]{$ \id $} (13)
      (22) edge node[left]{$ \id $} (11);
    \end{tikzpicture}
  \]
  where the source is the composite 1-arrow
  \[
    \begin{tikzpicture}
      \node (11) at (0,0) {$ L0 $};
      \node (21) at (4,0) {$ La $};
      \node (31) at (8,0) {$ L0 $};
      \node (12) at (2,2) {$ La $};
      \node (22) at (6,2) {$ La $};
      \node (13) at (4,4) {$ La $};
      \draw[cd]
      (11) edge node[above,left]{$ ! $} (12)
      (21) edge node[above,right]{$ \id $} (12)
      (21) edge node[above,left]{$ \id $} (22)
      (31) edge node[above,right]{$ ! $} (22)
      (12) edge node[above,left]{$ \id $} (13)
      (22) edge node[above,right]{$ \id $} (13);
      \draw (3.9,3.5) -- (4,3.4) -- (4.1,3.5);
    \end{tikzpicture}
  \]
  Again, the triangle equations are straightforward to
  check.
\end{proof}

The following lemma refers to a `Frobenius monoid', a monoid
and comonoid that satisfy some nice properties that we spell
out in Definition \ref{def:frobenius-monoid}. 

\begin{lemma} \label{thm:frobenius}
  For any object $ a $ of $ _L \BBoldRewrite $,
  $ ( a , \nabla^\ast , \epsilon^\ast , \nabla ,
  \epsilon ) $ is a Frobenius monoid. In particular,
  \begin{equation} \label{eq:frobenius}
    \nabla \nabla^\ast =
    \left( \nabla^\ast \otimes \id \right)
    \left( \id \otimes \nabla \right)
  \end{equation}
\end{lemma}

\begin{proof}
  The left-hand side of Equation
  \ref{eq:frobenius} is given by the composite
  \[
    \begin{tikzpicture}
      \node (11) at (0,0) {$ L(a+a) $};
      \node (21) at (4,0) {$ La $};
      \node (31) at (8,0) {$ L(a+a) $};
      \node (12) at (2,2) {$ La $};
      \node (22) at (6,2) {$ La $};
      \node (13) at (4,4) {$ La $};
      \draw[cd]
      (11) edge node[above]{$ \nabla $} (12)
      (21) edge node[above]{$ \id $} (12)
      (21) edge node[above]{$ \id $} (22)
      (31) edge node[above]{$ \nabla $} (22)
      (12) edge node[above]{$ \id $} (13)
      (22) edge node[above]{$ \id $} (13);
      \draw (3.9,3.5) -- (4,3.4) -- (4.1,3.5);      
    \end{tikzpicture}
  \]
  The right-hand side is given by
  \[
    \begin{tikzpicture}
      \node (11) at (0,0) {$ L(a+a) $};
      \node (21) at (4,0) {$ L(a+a+a) $};
      \node (31) at (8,0) {$ L(a+a) $};
      \node (12) at (2,2) {$ L(a+a) $};
      \node (22) at (6,2) {$ L(a+a) $};
      \node (13) at (4,4) {$ La $};
      \draw[cd]
      (11) edge node[above]{$ \id $} (12)
      (21) edge node[above,right]{$ L(\id +\nabla ) $} (12)
      (21) edge node[below,right]{$ L(\nabla + \id) $} (22)
      (31) edge node[above]{$ \id $} (22)
      (12) edge node[above]{$ L\nabla $} (13)
      (22) edge node[above,right]{$ L\nabla $} (13);
      \draw (3.9,3.5) -- (4,3.4) -- (4.1,3.5);      
    \end{tikzpicture}
  \]
  These both compose to $ L(a+a) \xto{L\nabla} La
  \xgets{L\nabla} L(a+a) $.
\end{proof}

The following structure theorem follows from this
string of lemmas.

\begin{theorem} \label{thm:bold-rewrite-bicat-rels}
  $ _L \BBoldRewrite $ is a bicategory of relations.
\end{theorem}

\section{The ZX-calculus}
\label{sec:zx-calculus}

Perhaps one of the most interesting features of quantum
mechanics is the incompatibility of observables.  Roughly,
an observable is a measurable quantity of some system, for
instance the spin of a photon.  In classical physics,
measureable quantities are comparable, meaning that we can
obtain arbitrarily precise values at the same time.  For
example, given a Porsche speeding down the highway, we can
simultaneously measure its velocity and its mass with
arbitrary precision.  Knowledge about its velocity does not
preclude us from obtain information about its mass. The
situation is quite different in quantum mechanics. Given two
measurable quantities, knowledge of one may prevent us from
obtaining knowledge about the other. This is illustrated by
the famous Heisenberg uncertainty principle which quantifies
the limits of precision to which one can simultaneously
measure the position and momentum of a particle. In general,
the strength of this restriction depends on the
situation. The most extreme case is that knowing one
quantity with total precision implies total uncertainty
about the other quantity. Such a pair of observables are
called \defn{complementary}.

Historically, a quantum physicist would reason about
observables, complementary or otherwise, using Hilbert
spaces. Given the rapid progress of quantum physics in the
twentieth century, this framework seems to have worked quite
well for scientists. Working with Hilbert spaces, however,
is challenging even for skilled researchers. But the
language of quantum physics is now relevant to a wider
audience since the dawn of quantum computing. Given the
challenge of working with Hilbert spaces, perhaps developing
a simpler language is worth pursuing.

Such a high-level language was invented by Coecke and Duncan
\cite{coecke-duncan_quant-full}.  This language, called the
ZX-calculus, was immediately used to generalize both quantum
circuits \cite{nielson-chuage_quantum} and the measurement
calculus \cite{danos-kash-panan_measurement-calc}.  Its
validity was further justified when Duncan and Perdrix
presented a non-trivial method of verifying
measurement-based quantum computations
\cite{duncan-perdrix_rewriting}.  At its core, the
ZX-calculus is an intuitive graphical language in which to
reason about complementary observables.

In this section, we illustrate our framework with the
ZX-calculus.  The backstory of the ZX-calculus dates to
Penrose's tensor networks \cite{penrose} and, more recently,
to the relationship between graphical languages and monoidal
categories \cite{joyal-street,selinger}. Abramsky and Coecke
capitalized on this relationship when inventing a
categorical framework for quantum physics
\cite{abram-coecke}.  Soon after, Coecke and Duncan
introduced a diagrammatic language in which to reason about
complementary quantum observables
\cite{coecke-duncan_quant-initial}.  After a fruitful period
of development \cite{coecke-edwards-spekkens,
  coecke-perdrix_classical-channels,
  duncan-perdix_graph-states, duncan-perdrix_rewriting,
  duncan-evans-panag_classifying,
  pavlovic_nondeterministic}, a full presentation of the
ZX-calculus was published \cite{coecke-duncan_quant-full}.
The completeness of the ZX-calculus for stabilizer quantum
mechanics was later proved by Backens
\cite{backens-completeness}.

The ZX-calculus begins with the five diagrams depicted in
Figure \ref{fig:zx-generators}. On each diagram, the
dangling wires on the left are \defn{inputs} and those on
the right are \defn{outputs}.  By connecting inputs to
outputs, we can form larger diagrams, which we call
\defn{ZX-diagrams}.  These diagrams generate the arrows of a
dagger compact category $\ZX$ whose objects, the
non-negative integers, count the inputs and outputs of a
diagram.  Below, we give a presentation of $\ZX$ along with
a brief discussion on the origins of its generating arrows
(Figure \ref{fig:zx-generators}) and relations (Figure
\ref{fig:zx-equations}).

\begin{figure}
  \fbox{%
  \centering  
  \begin{minipage}{1\textwidth}
    \begin{minipage}{0.2\linewidth}
      \centering
      \includegraphics{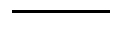} \\
      \vspace{0.5em}
      \textsc{Wire}
    \end{minipage}
    \begin{minipage}{0.2\linewidth}
      \centering
      \includegraphics{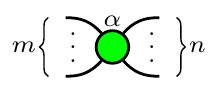} \\
      \vspace{0em}
      \textsc{Green Spider}
    \end{minipage}%
    \begin{minipage}{0.2\linewidth}
    \centering
    \includegraphics{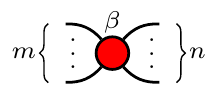} \\
    \vspace{0em}
    \textsc{Red Spider}
    \end{minipage}%
    \begin{minipage}{0.2\linewidth}
    \centering
    \includegraphics{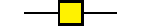} \\
    \vspace{0.5em}
    \textsc{Hadamard}
    \end{minipage}%
    \begin{minipage}{0.2\linewidth}
    \centering
    \includegraphics{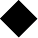} \\
    \vspace{0.5em}
    \textsc{Diamond}
    \end{minipage}%
  \end{minipage}
}
\caption{Generators for the ZX-calculus diagrams}
\label{fig:zx-generators}
\end{figure}

Our goal with this example is to generate, using the
machinery laid out in this chapter, a bicategory of
relations $\ZZX$ to provide a syntax for the ZX-calculus. We
show that $ \ZZX $ extends $ \ZX $ in a way we make precise below.

The five \defn{basic diagrams} in the ZX-calculus
are depicted in Figure \ref{fig:zx-generators} and
are to be read from left to right. They are
\begin{itemize}
\item a \defn{wire} with a single input and
  output,
\item \defn{green spiders} with a non-negative
  integer number of inputs and outputs and paired
  with a phase $\alpha \in [-\pi,\pi)$,
\item \defn{red spiders} with a non-negative
  integer number inputs and outputs and paired
  with a phase $\beta \in [-\pi,\pi)$,
\item the \defn{Hadamard node} with a single input
  and output, and
\item a \defn{diamond node} with no inputs or
  outputs.
\end{itemize}
The wire plays the role of an identity, much like a wire
without resistance in an electrical circuit, or straight
pipe in a plumbing system.  The green and red spiders each
arise from a pair of complementary observables.  In
categorical quantum mechanics \cite{abram-coecke},
observables correspond to certain commutative Frobenius
algebras $A$ living in a dagger symmetric monoidal category
$(\C , \otimes , I)$, the classic example
$ \C \bydef \FinHilb $ being the category of finite
dimensional Hilbert spaces and linear maps.  A pair of
complementary observables gives a pair of Frobenius
algebras whose operations interact via laws like those of a
Hopf algebra
\cite{
  coecke-pavlovic_quantum-measurements,
  coecke-pav-vacary_orth-bases}.
This is particularly nice because Frobenius algebras have
beautiful string diagram representations.  There is an
morphism $\C(I,A) \to \C(A,A)$ of commutative monoids that
gives rise to a group structure on $A$ known as the
\defn{phase group}, which Coecke and Duncan detail
\cite[Def.~7.5]{coecke-duncan_quant-full}.  The phases on
the green and red spider diagrams arise from this group.
The Hadamard node embodies the Hadamard gate.  The diamond
is a scalar obtained when connecting a green and red node
together.  A deeper exploration of these notions goes beyond
the scope of this paper.  For those interested, the original
paper on the topic
\cite{coecke-duncan_quant-full}
is an excellent place learn more.

In the spirit of compositionality, we present a category
$ \ZX $ whose arrows are generated by the five
basic diagrams. We sketched $ \ZX $ at the beginning of this
section, but we now detail the construction.

We start by allowing the basic ZX-diagrams from Figure
\ref{fig:zx-generators} generate the arrows of a free dagger
compact category whose objects are the non-negative
integers. We then subject the arrows (ZX-diagrams) to the
relations given in Figure \ref{fig:zx-equations}, to which
we add equations obtained by exchanging red and green nodes,
daggering, and taking diagrams up to ambient isotopy in
$4$-space.  These listed relations are called \defn{basic}.
Spiders with no phase indicated have a phase of $0$.  

This category, denoted as $ \ZX $, was introduced by Coecke
and Duncan
\cite{coecke-duncan_quant-full}
and further studied by Backens
\cite{backens-completeness}.
To compose in $\ZX$, connect compatible diagrams
along a bijection between inputs and the outputs. For
example
\begin{center}
  \begin{tikzpicture}
    \node            (1) at (0,2) {$  $};
    \node            (2) at (0,0) {$  $};
    \node [zxgreen]  (3) at (2,1) {$  $};
    \node            (4) at (4,1) {$  $};
    \node            ( ) at (5,1) {$ \circ $};
    \node            (5) at (6,1) {$  $};
    \node [zxyellow] (6) at (8,1) {$  $};
    \node            (7) at (10,1) {$  $};
    \node            ( ) at (11,1) {$ = $};
    \draw 
    (1.60) --(3.120)
    (2.-60) -- (3.-120)
    (3) -- (4)
    (5) -- (6)
    (6) -- (7);
  \end{tikzpicture}
\end{center}
\begin{center}
  \begin{tikzpicture}
    \node            (1) at (0,2) {$  $};
    \node            (2) at (0,0) {$  $};
    \node [zxgreen]  (3) at (2,1) {$  $};
    \node            (4) at (4,1) {$  $};
    \node [zxyellow] (5) at (6,1) {$  $};
    \node            (6) at (8,1) {$  $};
    \draw 
    (1.60) -- (3.120)
    (2.-60) -- (3.-120)
    (3) -- (4.center)
    (4.center) -- (5)
    (5) -- (6);
  \end{tikzpicture}
\end{center}
A monoidal structure is given by adding numbers and taking
the disjoint union of ZX-diagrams.  The identity on $n$ is
the disjoint union of $n$ wires:
\begin{center}
  \begin{tikzpicture}
    \node (1) at (0,0) {$  $};
    \node (2) at (0,0.25) {$  $};
    \node ( ) at (1.5,0.6) {$ \vdots $};
    \node (3) at (0,0.75) {$  $};
    \node (1') at (3,0) {$  $};
    \node (2') at (3,0.25) {$  $};
    \node (3') at (3,0.75) {$  $};
    \draw
    (1) -- (1')
    (2) -- (2')
    (3) -- (3');
  \end{tikzpicture}
\end{center}
The symmetry and compactness of the monoidal product provide
a braiding, evaluation, and coevaluation morphisms:
respectively,
\[
  \includegraphics{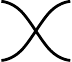}
  \quad \quad \quad \quad 
  \raisebox{-0.25\height}{%
    \includegraphics[scale=0.75]{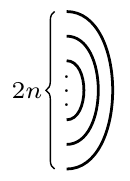}
  }
  \quad \quad \quad \quad 
  \raisebox{-0.25\height}{%
    \includegraphics[scale=0.75]{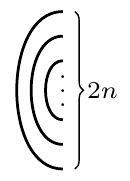}
  }
\]
The evaluation and coevalutation arrows are of type
$2n \to 0$ and $0 \to 2n$ for each object
$n \geq 1$ and the empty diagram for $n=0$.  On
the spider diagrams, the dagger structure swaps
inputs and outputs then multiplies the phase by
$-1$:
\[
  \includegraphics{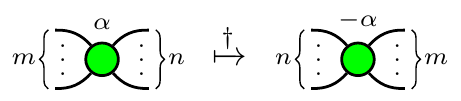}
\]
The dagger acts trivially on the wire, Hadamard,
and diamond elements.

\begin{figure}
  \fbox{
    \begin{minipage}{\textwidth}
      \centering
      \begin{minipage}{0.5\linewidth}
        \centering
        \includegraphics{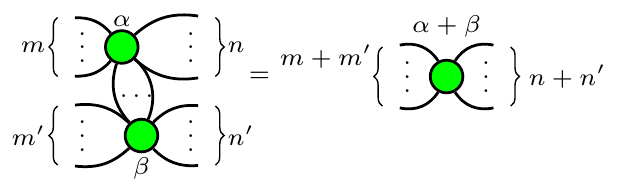} \\
        \textsc{Spider}
      \end{minipage}%
      \begin{minipage}{0.5\linewidth}
        \centering
        \includegraphics{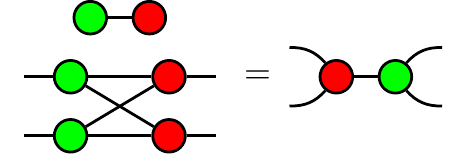} \\
        \textsc{Bialgebra} 
      \end{minipage}
      \vspace{1em} 
      \linebreak
      \begin{minipage}{0.3\linewidth}
        \centering
        \includegraphics{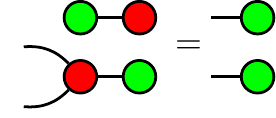} \\
        \textsc{Copy} 
      \end{minipage}%
      \begin{minipage}{0.3\linewidth}
        \centering
        \includegraphics{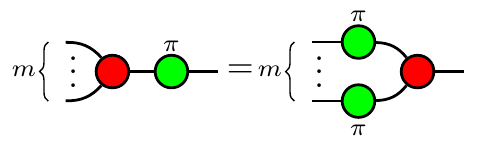} \\
        \textsc{$ \pi $-Copy} 
      \end{minipage}%
      \begin{minipage}{0.3\linewidth}
        \centering
        \includegraphics{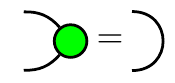} \\
        \textsc{Cup} 
      \end{minipage} 
      \vspace{1em} 
      \linebreak
      \begin{minipage}{0.5\linewidth}
        \centering
        \includegraphics{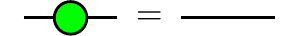} \\
        \textsc{Trivial Spider} 
      \end{minipage}%
      \begin{minipage}{0.5\linewidth}
        \centering
        \includegraphics{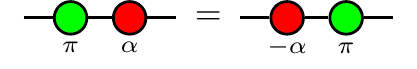} \\
        \textsc{$\pi$-Commutation}
      \end{minipage}%
      \linebreak
      \begin{minipage}{0.3\linewidth}
        \centering
        \includegraphics{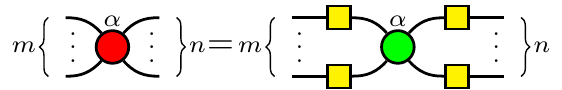} \\
        \textsc{Color Change}
      \end{minipage}%
      \begin{minipage}{0.3\linewidth}
        \centering
        \includegraphics{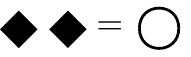} \\
        \textsc{Loop}
      \end{minipage}%
      \begin{minipage}{0.3\linewidth}
        \centering
        \includegraphics{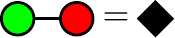} \\
        \textsc{Diamond}
      \end{minipage}%
    \end{minipage}
  }
  \caption{Relations in the category $\ZX$}
  \label{fig:zx-equations}
\end{figure}

A major advantage of using string diagrams, apart
from their intuitive nature, is that computations
are more easily programmed into computers.
Indeed, graphical proof assistants like
Quantomatic
\cite{bar-kiss-vicary_globular,
     dixon-duncun-kiss_quanto-website}
and Globular
\cite{bar-kiss-vicary_globular}
were made for such graphical reasoning.  The logic
of these programs are encapsulated by double
pushout rewrite rules.  However, the algebraic
structure of $\ZX $ and other graphical calculi do
not contain the rewrite rules as explicit
elements. On the other hand, the framework developed in this thesis
explicitly includes the rewrite rules.

To model the ZX-calculus using structured cospans, we need
an appropriate adjunction $ L \from \A \lrto \X \from R $.
Determining the correct pieces to fill in requires some
discussion. Before providing the details, we sketch the
process.  Let $ \A \bydef \FinSet $ be the topos of finite
sets and functions. Let $ \X \bydef \FinGraphGamma $ be the
over-category where we chose a graph $ \Gamma $ to provide
the objects of $ \X \bydef \FinGraph / \Gamma $ with the
same type information as the ZX-diagrams. The functor $ L $
turns a finite set $ a $ into a certain discrete graph over
$ \Gamma $ so that $ La $ can serve as inputs or outputs. To
unpack what this all means, we start with the over-category.

\begin{definition}
  Let $g$ be a graph.  By a \defn{graph over $g$},
  we mean a graph morphism $x \to g$. A morphism
  between graphs over $g$ is a graph morphism
  $x \to y$ such that
  \[
    \begin{tikzpicture}
      \node (1) at (-1,2) {$x$};
      \node (2) at (1,2) {$y$};
      \node (3) at (0,0) {$g$};
      \draw[cd]
        (1) edge (2)
        (1) edge (3)
        (2) edge (3);
    \end{tikzpicture}
  \]
  commutes.
\end{definition} 

One way to think of a graph over $ g $ is as
a $ g $-typed graph. Consider the following simple
example.

\begin{example}
  \label{ex:graph-over-g}
  Let $ g $ be the graph
  \[
    \begin{tikzpicture}
      \node (a) at (0,0) {$ A $};
      \node (b) at (3,0) {$ B  $};
      \draw[graph]
        (a) edge[bend left] node[above]{$F$} (b)
        (b) edge[bend left] node[below]{$G$} (a);
    \end{tikzpicture}
  \]
  Let $ x $ be the graph
  \[
    \begin{tikzpicture}
      \node (a) at (0,0)  {$ a $};
      \node (b) at (4,0)  {$ b $};
      \node (c) at (2,2)  {$ c $};
      \node (d) at (2,-2) {$ d $};
      \draw[graph]
        (a) edge node[above,left]{$ e $} (c)
        (a) edge node[below,left]{$ f $} (d)
        (c) edge node[above,right]{$ g $} (b)
        (d) edge node[below,right]{$ h $} (b);
    \end{tikzpicture}
  \]
  that lies over $ g $ via the map
  \begin{align*}
      a,b & \mapsto A  &    e,f & \mapsto F \\
      c,d & \mapsto B  &    g,h & \mapsto G 
  \end{align*}
  If we think of the nodes and edges of $ g $ as
  types, then these types are transported to $ x $
  along the fibers of this map.  Thus $ x $ is a
  graph with the following type-assignment:
  \begin{center}
    \begin{tabular}{ccccc}
      $ a : A $ && $ b : A $ && $ c : B $ \\
      $ d : B $ && $ e : F $ && $ f : F $ \\
      $ g : G $ && $ h : G $ &&           \\
    \end{tabular}
  \end{center}
  where `$ : $' should be read `is type'.  Any graph over
  $ g $ can have two node types $ A,B $ and two edge types
  $ F,G $.  Edges can only go from an $ A $-type
  node to a $ B $-type node or vice versa.  Edges cannot
  traverse nodes of the same type simply because there are
  no looped edges in $ g $.

  A compact way to draw a graph over $ g $ is to
  label its nodes and edges with their types.
  Thus, the over-graph $ x \to g $ can be drawn as
  \[
    \begin{tikzpicture}
      \node (a) at (0,0)  {$ (a,A) $};
      \node (b) at (4,0)  {$ (b,A) $};
      \node (c) at (2,2)  {$ (c,B) $};
      \node (d) at (2,-2) {$ (d,B) $};
      \draw[graph]
        (a) edge node[above,left]{$ (e,F) $} (c)
        (a) edge node[below,left]{$ (f,F) $} (d)
        (c) edge node[above,right]{$ (g,G) $} (b)
        (d) edge node[below,right]{$ (h,G) $} (b);
    \end{tikzpicture}
  \]
  
  One might recognize the class of graphs over $ g $ as
  something like a bipartite graph.  The difference between
  graphs over $ g $ and bipartite graphs is that bipartite
  graphs are usually defined by graph theorists to satisfy
  \emph{the property} that the nodes can be partitioned into
  two classes and the source and target of each edge must
  belong to different classes.  On the other hand, graphs
  over $ g $ are graphs equipped with extra structure,
  namely the type information. This distinction does not
  appear in the graphs themselves, so we look at their
  morphisms.

  A morphism of graphs over $ g $ must respect the type
  information. So if $ x \to g $ and $ y \to g $ are graphs
  over $ g $, then a morphism between them is a graph
  morphism $ x \to y $ such that the diagram
  \[
    \begin{tikzpicture}
      \node (x) at (0,0) {$ x $};
      \node (y) at (4,0) {$ y $};
      \node (g) at (2,-2) {$ g $};
      \draw [cd]
       (x) edge (y)
       (x) edge (g)
       (y) edge (g);
    \end{tikzpicture}
  \]
  commutes. Suppose that $ x $ is a single node typed $ A $
  and $ y $ is a single node typed $ B $. There is no
  morphism between them because the node in $ x $ must be
  sent to a node of type $ A $.  However, any two bipartite
  graphs with a single node and no edges are isomorphic. The
  moral of this example is by adding the type information, we
  added structure instead of imposing a property.  We denote by
  $ \Graph \downarrow g $ the category of graphs over $ g $
  and their morphisms.
\end{example}

We exploit this method of defining
`typed graphs' to transform typical combinatorial
graphs into ZX-diagrams. The types needed to make
ZX-diagrams from graphs encoded into the graph $ \Gamma $
that we define now.

\begin{definition}
  \label{ex:basic-graph-over-g}
  Let $ \Gamma $ be the graph
  \begin{equation}
    \label{eq:gamma}
    \begin{tikzpicture}
      \node [zxwhite]  (w) at (0,0)    {$ $};
      \node [zxgreen]  (g) at (-1,1)   {};
      \node            ()  at (-1.33,1) {\scriptsize{$\alpha$}};
      \node [zxred]    (r) at (1,1)    {};
      \node            ()  at (1.33,1)  {\scriptsize{$\beta$}};
      \node [zxblack]  (b) at (-1,-1)  {$$};
      \node [zxyellow] (y) at (1,-1)   {$$};
      \path[graph]
      (w) edge [loop below]   (w)
      (w) edge [bend left=10] (g)
      (w) edge [bend left=10] (r)
      (w) edge [bend left=10] (b)
      (w) edge [bend left=10] (y)
      (g) edge [bend left=10] (w)
      (r) edge [bend left=10] (w)
      (b) edge [bend left=10] (w)
      (y) edge [bend left=10] (w);
      \node () at (3,0) {$ \alpha,\beta \in [ -\pi , \pi ) $};
    \end{tikzpicture}.
  \end{equation}

  We have not drawn the entirety of $ \Gamma $. In
  actuality, the green and red nodes run through
  $ [ -\pi , \pi ) $ and each of them have a single
  arrow to and from the white node
\end{definition}

Note that the graphs over $ \Gamma $ are completely
determined by the function's behavior on the nodes.  This is
because there is at most one arrow between any two
nodes. When comparing the $ \Gamma $-types to the types
appearing in the basic ZX-diagrams of Figure
\ref{fig:zx-generators}, there is a clear correlation
except, perhaps, for the white node.  To explain the white
node, first observe that ZX-diagrams have dangling wires on
either end. Dangling edges are not permitted in our
definition of graphs, so the white node anchors them.

To draw graphs over $ \Gamma $, we attach the type
information to the nodes by rendering the nodes as
red, greed, white, black, or yellow.  This manner
of drawing is more economical than drawing a graph
and describing its map to $ \Gamma $.  For
example, consider the graph
\[
  \begin{tikzpicture}
    \node (0) at (0,0) {$ \bullet_a $};
    \node (1) at (2,0) {$ \bullet_b $};
    \node (2) at (4,0) {$ \bullet_c  $};
    \draw[graph]
      (0) edge (1)
      (1) edge (2);
  \end{tikzpicture}
\]
with the map to $ \Gamma $ determined by
\[
  \begin{tikzpicture}
    \node           at (0,1) {$ a,c \mapsto $};
    \node [zxwhite] at (1,1) {$  $};
    \node           at (0,0) {$ b \mapsto $};
    \node [zxgreen] at (1,0) {$ \beta $};
  \end{tikzpicture}
\]
We draw this as
\[
  \begin{tikzpicture}
    \node [zxwhite] (1) at (0,0) {$  $};
    \node [zxgreen] (2) at (2,0) {$ \beta $};
    \node [zxwhite] (3) at (4,0) {$  $};
    \draw[graph]
      (1) edge (2)
      (2) edge (3);
  \end{tikzpicture}
\]

In our adjunction $ L \from \A \lrto \X \from R $, we let
$ \X $ be $ \FinGraphGamma $. This is a topos by the
fundamental theorem of topos theory, which we present in
Theorem \ref{thm:fund-thm-topos}.

The most important objects in $ \FinGraphGamma $ are those
corresponding to the basic ZX-diagrams. These are displaying
in Figure \ref{fig:basic-diagrams-over-gamma}. To choose a
category $ \A $ of interface types, we want to faithfully
represent the fact that ZX-diagrams have a non-negative
integer number of inputs and outputs. Therefore, we let
$ \A $ be the topos $\FinSet $ of finite sets and functions.

\begin{figure}
  \fbox{
    \begin{minipage}{\textwidth}
      \centering
      \begin{minipage}{0.3\linewidth}
        \centering
        \[
          \begin{tikzpicture}
            \node [zxwhite] (a) at (0,0) {};
            \node [zxwhite] (b) at (2,0) {};
            \draw [graph]
              (a) to (b);
          \end{tikzpicture}
        \]
        \textsc{Wire}
      \end{minipage}%
      \begin{minipage}{0.3\linewidth}
        \centering
        \[
          \begin{tikzpicture}
            \node [zxwhite] (a) at (0,0)   {$ $};
            \node           ( ) at (0,0.66) {$ \vdots $};
            \node [zxwhite] (b) at (0,1)   {$  $};
            \node [zxgreen] (c) at (1,0.5) {$ $};
            \node           ( ) at (1,1.25) {$\alpha$};
            \node [zxwhite] (d) at (2,0)   {$ $};
            \node           ( ) at (2,0.66) {$ \vdots $};
            \node [zxwhite] (e) at (2,1)   {$  $};
            \draw[graph]
            (a) edge [bend right=10]  (c)
            (b) edge [bend left=10]   (c)
            (c) edge [bend right=10]  (d)
            (c) edge [bend left=10]   (e);
          \end{tikzpicture}
        \]
        \textsc{Green Spider} 
      \end{minipage}
      \begin{minipage}{0.3\linewidth}
        \centering
        \[
          \begin{tikzpicture}
            \node [zxwhite] (a) at (0,0)   {$ $};
            \node           ( ) at (0,0.66) {$ \vdots $};
            \node [zxwhite] (b) at (0,1)   {$  $};
            \node [zxred] (c) at (1,0.5) {$ $};
            \node           ( ) at (1,1.25) {$\alpha$};
            \node [zxwhite] (d) at (2,0)   {$ $};
            \node           ( ) at (2,0.66) {$ \vdots $};
            \node [zxwhite] (e) at (2,1)   {$  $};
            \draw[graph]
            (a) edge [bend right=10]  (c)
            (b) edge [bend left=10]   (c)
            (c) edge [bend right=10]  (d)
            (c) edge [bend left=10]   (e);
          \end{tikzpicture}
        \]
        \textsc{Red Spider} 
      \end{minipage}
      \vspace{1em} 
      \linebreak
      \begin{minipage}{0.5\linewidth}
        \centering
        \[
          \begin{tikzpicture}
            \node [zxwhite]  (a) at (0,0) {$$};
            \node [zxyellow] (b) at (1,0) {$$};
            \node [zxwhite]  (c) at (2,0) {$$};
            \draw[graph]
              (a) edge (b)
              (b) edge (c);
          \end{tikzpicture}
        \]
        \textsc{Hadamard} 
      \end{minipage}%
      \begin{minipage}{0.5\linewidth}
        \centering
        \[
          \begin{tikzpicture}
            \node [zxblack] () at (0,0) {$$};
          \end{tikzpicture}
        \]
        \textsc{Diamond} 
      \end{minipage}%
    \end{minipage}
  }
  \caption{Basic ZX-diagrams as graphs over $ \Gamma $}
  \label{fig:basic-diagrams-over-gamma}
\end{figure}
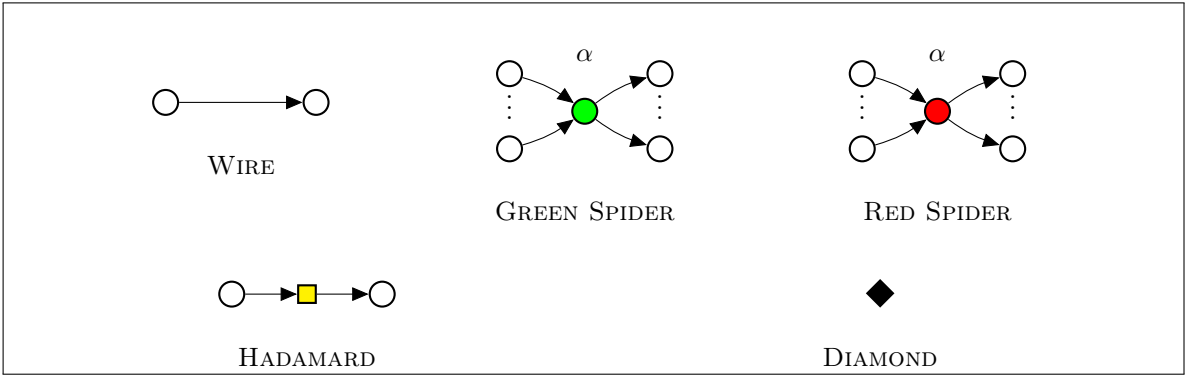

We still need to define $ L $ and $ R $ in the adjunction
\[
  \adjunction{\FinSet}{\FinGraphGamma}{L}{R}{4}
\]
Define
\[
  L \from \FinSet \to \FinGraphGamma
\]
by letting $ La $ be the edgeless graph with node
set $ a $ that is constant over the whites
node in $ \Gamma $. A function $ f \from a \to b $ of finite
sets becomes of morphism $ Lf $ of graphs over $
\Gamma $ that simply reinterprets the action of $
f $ on elements of a set to white nodes in a
graph.  Define
\[
  R \from \FinGraphGamma \to \FinSet
\]
by defining $ R ( x \to \Gamma ) $ as the fiber in
$ x $ of the white node. Given a morphism of
graphs over $ \Gamma $, $ R $ restricts it to the
function on only the white nodes.  

\begin{lemma}
  The functor pair
  \[
    \adjunction{\FinSet}{\FinGraphGamma}{L}{R}{4}
  \]
  forms an adjunction and $ L $ preserves pullbacks.
\end{lemma}

\begin{proof}
  Observe that the composite $ RL $ is the
  identity functor. So the unit
  $ \eta \from a \to RLa $ is the identity which
  is natural in $ a $. The
  counit $ \epsilon \from LRx \to x $ is the
  inclusion of the white nodes of $ x $ into
  $ x $. Given an arrow $ f \from x \to y $ in $
  \FinGraphGamma $, the diagram
  \[
    \begin{tikzpicture}
      \node (LRx) at (0,2) {$ LRx $};
      \node (x) at (2,2) {$ x $};
      \node (LRy) at (0,0) {$ LRy $};
      \node (y) at (2,0) {$ y $};
      \draw[cd]
        (LRx) edge node[above]{$ \epsilon_x $} (x)
        (LRy) edge node[below]{$ \epsilon_y $} (y)
        (LRx) edge node[left]{$ LRf $}         (LRy)
        (x)   edge node[right]{$ f $}          (y);
    \end{tikzpicture}
  \]
  commutes since $ LRf $ is a restriction of $ f $. To show
  that $ L $ preserves pullbacks, take a cospan
  \[ a \to b \gets c \] in $ \Set $ with pullback
  $ a \times_b c $ and apply $ L $ to get the diagram
  \[
    \begin{tikzpicture}
      \node (La) at (-2,0) {$ La $};
      \node (Lb) at (0,0) {$ Lb $};
      \node (Lc) at (0,2) {$ Lc $};
      \node (Gamma) at (2,-2) {$ \Gamma $};
      \path [cd]
        (La) edge (Lb)
        (Lc) edge (Lb)
        (La.-90) edge[bend right] (Gamma.180)
        (Lb) edge (Gamma)
        (Lc.0) edge[bend left] (Gamma.90); 
    \end{tikzpicture}
  \]
  comprised of edgeless graphs $ La $, $ Lb $, and $ Lc $
  that are constant over the white node in $ \Gamma $.  The
  pullback of this diagram is
  $ La \times_{Lb} Lc \to \Gamma $ which is constant over
  the white node. This is isomorphic to $ L(a \times_b c)
  \to \Gamma $ which is constant over the white node.
\end{proof}

With our adjunction established, we can define structured
cospans of graphs over $ \Gamma $ and therefore also the
symmetric monoidal double category of bold rewrites
$ _L \BBBoldRewrite $ for the functor
$ L \from \FinSet \to \FinGraphGamma $ defined above.  This
double category has as objects the finite sets, as
horizontal 1-arrows the structured cospans of graphs over
$ \Gamma $, as vertical 1-arrows the spans of finite sets
with invertible legs, and as squares all possible
bold rewrites of structured cospans.  Clearly,
$ _L \BBBoldRewrite $ is far bigger than the ZX-calculus
because it contains graphs over $ \Gamma $ with no
corresponding ZX-diagram. This does not mean, however, that
$ _L \BBBoldRewrite $ serves no purpose.  It plays the role
of an ambient space in which we chisel out a sub-double
category that \emph{does} correspond to the ZX-calculus.

To begin the process of constructing this sub-double
category of $ _L \BBBoldRewrite $, we identify structured
cospans to capture the basic ZX-diagrams and identify bold
rewrites of structured cospans for the basic
ZX-relations. We also include some additional structured
cospans to give the desired structure.  Figure
\ref{fig:basic-diagrams-as-str-csps}
depicts the basic ZX-diagrams as structured
cospans.

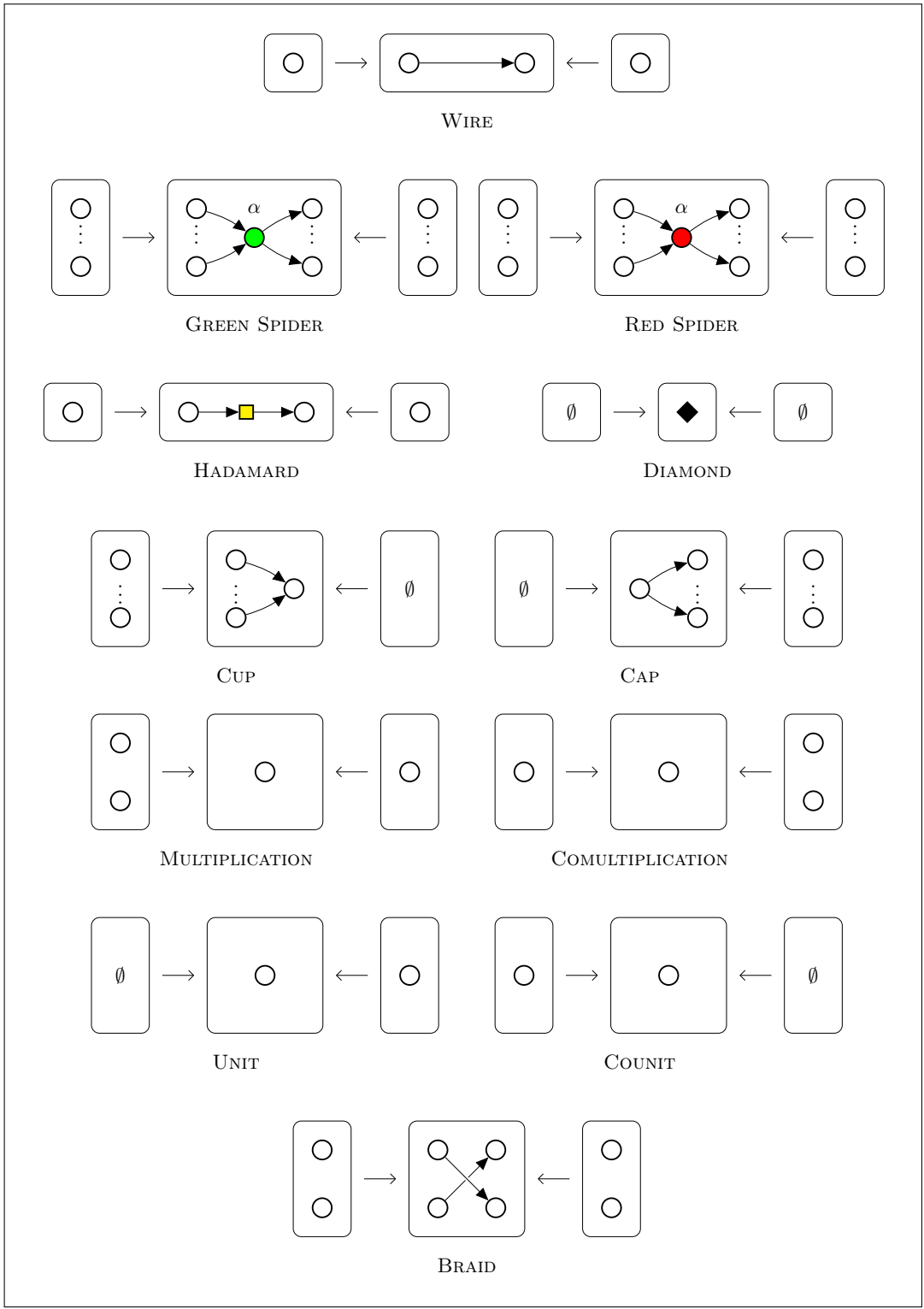
\begin{figure}
  \fbox{
    \scalebox{0.9}{
    \begin{minipage}{\textwidth}
      \centering
      \begin{minipage}{\textwidth}
        \centering
        \[
          \begin{tikzpicture}
            \begin{scope}
              \node [zxwhite] (a) at (0,0) {};
              \draw [rounded corners]
                (-0.5,-0.5) rectangle (0.5,0.5);
              \node (l) at (0.6,0) {};
            \end{scope}
            \begin{scope}[shift={(2,0)}]
              \node [zxwhite] (a) at (0,0) {};
              \node [zxwhite] (b) at (2,0) {};
              \draw [graph]
                (a) to (b);
              \draw [rounded corners]
                (-0.5,-0.5) rectangle (2.5,0.5);
              \node (cr) at (2.6,0) {};
              \node (cl) at (-0.6,0) {};  
              \node () at (1,-1) {\textsc{Wire}};  
            \end{scope}
            \begin{scope}[shift={(6,0)}]
              \node [zxwhite] (a) at (0,0) {};
              \draw [rounded corners]
                (-0.5,-0.5) rectangle (0.5,0.5);
              \node (r) at (-0.6,0) {};
            \end{scope}
            \draw[cd]
              (l) edge (cl)
              (r) edge (cr);
          \end{tikzpicture}
        \]
      \end{minipage}%
      \vspace{1em} 
      \linebreak
      \begin{minipage}{0.40\linewidth}
        \[
          \begin{tikzpicture}
            \begin{scope}
              \node [zxwhite] (a) at (0,0)    {};
              \node           ( ) at (0,0.66) {$ \vdots $};
              \node [zxwhite] (b) at (0,1)    {};
              \draw [rounded corners]
                (-0.5,-0.5) rectangle (0.5,1.5);
              \node (l) at (0.6,0.5) {};
            \end{scope}
            \begin{scope}[shift={(2,0)}]
              \node [zxwhite] (a) at (0,0)   {$ $};
              \node           ( ) at (0,0.66) {$ \vdots $};
              \node [zxwhite] (b) at (0,1)   {$  $};
              \node [zxgreen] (c) at (1,0.5) {$ $};
              \node           ( ) at (1,1) {$\alpha$};
              \node [zxwhite] (d) at (2,0)   {$ $};
              \node           ( ) at (2,0.66) {$ \vdots $};
              \node [zxwhite] (e) at (2,1)   {$  $};
              \draw[graph]
                (a) edge[bend right=10] (c)
                (b) edge[bend left=10]  (c)
                (c) edge[bend right=10] (d)
                (c) edge[bend left=10]  (e);
              \draw [rounded corners]
                (-0.5,-0.5) rectangle (2.5,1.5);  
              \node (cl) at (-0.6,0.5) {};
              \node (cr) at (2.6,0.5)  {};
              \node () at (1,-1) {\textsc{Green Spider}};  
            \end{scope}
            \begin{scope}[shift={(6,0)}]
              \node [zxwhite] (a) at (0,0)    {};
              \node           ( ) at (0,0.66) {$ \vdots $};
              \node [zxwhite] (b) at (0,1)    {};
              \draw [rounded corners]
                (-0.5,-0.5) rectangle (0.5,1.5);
              \node (r) at (-0.6,0.5) {};
            \end{scope}
            \path[cd]
              (l) edge (cl)
              (r) edge (cr);
          \end{tikzpicture}
        \]
      \end{minipage}
      \hspace{3em}
      \begin{minipage}{0.45\linewidth}
        \[
          \begin{tikzpicture}
            \begin{scope}
              \node [zxwhite] (a) at (0,0)    {};
              \node           ( ) at (0,0.66) {$ \vdots $};
              \node [zxwhite] (b) at (0,1)    {};
              \draw [rounded corners]
                (-0.5,-0.5) rectangle (0.5,1.5);
              \node (l) at (0.6,0.5) {};
            \end{scope}
            \begin{scope}[shift={(2,0)}]
              \node [zxwhite] (a) at (0,0)   {$ $};
              \node           ( ) at (0,0.66) {$ \vdots $};
              \node [zxwhite] (b) at (0,1)   {$  $};
              \node [zxred] (c) at (1,0.5) {$ $};
              \node           ( ) at (1,1) {$\alpha$};
              \node [zxwhite] (d) at (2,0)   {$ $};
              \node           ( ) at (2,0.66) {$ \vdots $};
              \node [zxwhite] (e) at (2,1)   {$  $};
              \draw[graph]
                (a) edge[bend right=10] (c)
                (b) edge[bend left=10]  (c)
                (c) edge[bend right=10] (d)
                (c) edge[bend left=10]  (e);
              \draw [rounded corners]
                (-0.5,-0.5) rectangle (2.5,1.5);  
              \node (cl) at (-0.6,0.5) {};
              \node (cr) at (2.6,0.5)  {};
              \node () at (1,-1) {\textsc{Red Spider}};
            \end{scope}
            \begin{scope}[shift={(6,0)}]
              \node [zxwhite] (a) at (0,0)    {};
              \node           ( ) at (0,0.66) {$ \vdots $};
              \node [zxwhite] (b) at (0,1)    {};
              \draw [rounded corners]
                (-0.5,-0.5) rectangle (0.5,1.5);
              \node (r) at (-0.6,0.5) {};
            \end{scope}
            \path[cd]
              (l) edge (cl)
              (r) edge (cr);
          \end{tikzpicture}
        \]
      \end{minipage}
      \vspace{1em} 
      \linebreak
      \begin{minipage}{0.5\linewidth}
        \centering
        \[
          \begin{tikzpicture}
            \begin{scope}
              \node [zxwhite] (a) at (0,0) {};
              \draw [rounded corners]
                (-0.5,-0.5) rectangle (0.5,0.5);
              \node (l) at (0.6,0) {};
            \end{scope}
            \begin{scope}[shift={(2,0)}]
              \node [zxwhite]  (a) at (0,0) {$$};
              \node [zxyellow] (b) at (1,0) {$$};
              \node [zxwhite]  (c) at (2,0) {$$};
              \draw [graph]
                (a) edge (b)
                (b) edge (c);
              \draw [rounded corners]
                (-0.5,-0.5) rectangle (2.5,0.5);
              \node (cl) at (-0.6,0) {};
              \node (cr) at (2.6,0)  {};
              \node () at (1,-1) {\textsc{Hadamard}};
            \end{scope}
            \begin{scope}[shift={(6,0)}]
              \node [zxwhite] (a) at (0,0) {};
              \draw [rounded corners]
                (-0.5,-0.5) rectangle (0.5,0.5);
              \node (r) at (-0.6,0) {};
            \end{scope}
            \path[cd]
              (l) edge (cl)
              (r) edge (cr);
          \end{tikzpicture}
        \]         
      \end{minipage}%
      \begin{minipage}{0.5\linewidth}
        \centering
        \[
          \begin{tikzpicture}
            \begin{scope}[shift={(0,0)}]
              \node () at (0,0) {$ \emptyset $};
              \draw [rounded corners]
                (-0.5,-0.5) rectangle (0.5,0.5);
              \node (l) at (0.6,0) {};
            \end{scope}
            \begin{scope}[shift={(2,0)}]
              \node [zxblack] () at (0,0) {$$};
              \draw [rounded corners]
                (-0.5,-0.5) rectangle (0.5,0.5);
              \node (cl) at (-0.6,0) {};
              \node (cr) at (0.6,0) {};
              \node () at (0,-1) {\textsc{Diamond}};
            \end{scope}
            \begin{scope}[shift={(4,0)}]
              \node (a) at (0,0) {$ \emptyset $};
              \draw [rounded corners]
                (-0.5,-0.5) rectangle (0.5,0.5);
              \node (r) at (-0.6,0) {};
            \end{scope}
            \draw[cd]
              (l) edge (cl)
              (r) edge (cr);
         \end{tikzpicture}            
        \]
      \end{minipage}
      \vspace{1em} 
      \linebreak
      \begin{minipage}{0.45\linewidth}
        \centering
        \[
          \begin{tikzpicture} 
            \begin{scope} 
              \node [zxwhite] (b) at (0,1)    {$  $};
              \node           ( ) at (0,0.5) {$\vdots$};
              \node [zxwhite] (d) at (0,0)    {$  $};
              \node (0r) at (0.6,0.5) {$  $};
              \draw [rounded corners]
                (-0.5,-0.5) rectangle (0.5,1.5);      
            \end{scope}
            \begin{scope}[shift={(2,0)}] 
              \node [zxwhite] (a) at (0,1)   {$  $};
              \node           ( ) at (0,0.5) {$\vdots$};
              \node [zxwhite] (b) at (0,0)   {$  $};
              \node [zxwhite] (c) at (1,0.5) {$  $};
              \draw [graph]
                (a) edge[bend left=10]  (c)
                (b) edge[bend right=10] (c);
              \node (1l) at (-0.6,0.5) {$  $};
              \node (1r) at (1.6,0.5)  {$  $};
              \node () at (0,-1) {\textsc{Cup}};
              \draw [rounded corners]
              (-0.5,-0.5) rectangle (1.5,1.5);    
            \end{scope}
            \begin{scope}[shift={(5,0)}] 
              \node (a) at (0,0.5) {$ \emptyset $};
              \node (2l) at (-0.6,0.5) {$  $};
              \draw [rounded corners]
              (-0.5,-0.5) rectangle (0.5,1.5);      
            \end{scope}
            \draw[cd] 
            (0r) edge (1l)
            (2l) edge (1r);
          \end{tikzpicture}            
        \]
      \end{minipage}
      \begin{minipage}{0.45\linewidth} 
        \centering
        \[
          \begin{tikzpicture}
            \begin{scope} 
              \node (a) at (0,0.5) {$ \emptyset $};
              \node (0r) at (0.6,0.5) {$  $};
              \draw [rounded corners]
                (-0.5,-0.5) rectangle (0.5,1.5);      
            \end{scope}
            \begin{scope}[shift={(2,0)}] 
              \node [zxwhite] (a) at (1,1)   {$  $};
              \node           ( ) at (1,0.5) {$\vdots$};
              \node [zxwhite] (b) at (1,0)   {$  $};
              \node [zxwhite] (c) at (0,0.5) {$  $};
              \draw [graph]
                (c) edge [bend left=10]  (a)
                (c) edge [bend right=10] (b);
              \node (1l) at (-0.6,0.5) {$  $};
              \node (1r) at (1.6,0.5)  {$  $};
              \node () at (0,-1) {\textsc{Cap}};
              \draw [rounded corners]
                (-0.5,-0.5) rectangle (1.5,1.5); 
            \end{scope}
            \begin{scope}[shift={(5,0)}] 
              \node [zxwhite] (b) at (0,1)    {$  $};
              \node           ( ) at (0,0.5) {$\vdots$};
              \node [zxwhite] (d) at (0,0)    {$  $};
              \node (2l) at (-0.6,0.5) {$  $};
              \draw [rounded corners]
                (-0.5,-0.5) rectangle (0.5,1.5);      
            \end{scope}
            \path[cd]
              (0r) edge (1l)
              (2l) edge (1r);
          \end{tikzpicture}            
        \]
      \end{minipage}
      \linebreak
      \vspace{1em}
      \begin{minipage}{0.45\linewidth}
        \centering
        \[
          \begin{tikzpicture} 
            \begin{scope} 
              \node [zxwhite] () at (0,1) {$ $};
              \node [zxwhite] () at (0,0) {$ $};
              \node (0r) at (0.6,0.5) {$  $};
              \draw [rounded corners]
                (-0.5,-0.5) rectangle (0.5,1.5);      
            \end{scope}
            \begin{scope}[shift={(2,0)}] 
              \node [zxwhite] (b) at (0.5,0.5) {$  $};
              \node (1l) at (-0.6,0.5) {$  $};
              \node (1r) at (1.6,0.5)  {$  $};
              \node () at (0,-1) {\textsc{Multiplication}};
              \draw [rounded corners]
              (-0.5,-0.5) rectangle (1.5,1.5);    
            \end{scope}
            \begin{scope}[shift={(5,0)}] 
              \node [zxwhite] (a) at (0,0.5) {};
              \node (2l) at (-0.6,0.5) {};
              \draw [rounded corners]
              (-0.5,-0.5) rectangle (0.5,1.5);      
            \end{scope}
            \draw[cd] 
              (0r) edge (1l)
              (2l) edge (1r);
          \end{tikzpicture}            
        \]
      \end{minipage}
      \begin{minipage}{0.45\linewidth} 
        \centering
        \[
          \begin{tikzpicture}
            \begin{scope} 
              \node [zxwhite] () at (0,0.5) {};
              \node (0r) at (0.6,0.5) {};
              \draw [rounded corners]
                (-0.5,-0.5) rectangle (0.5,1.5);      
            \end{scope}
            \begin{scope}[shift={(2,0)}] 
              \node [zxwhite] (a) at (0.5,0.5) {};
              \node (1l) at (-0.6,0.5) {$  $};
              \node (1r) at (1.6,0.5)  {$  $};
              \node () at (0,-1) {\textsc{Comultiplication}};
              \draw [rounded corners]
                (-0.5,-0.5) rectangle (1.5,1.5); 
            \end{scope}
            \begin{scope}[shift={(5,0)}] 
              \node [zxwhite] () at (0,1)    {$  $};
              \node [zxwhite] () at (0,0)    {$  $};
              \node (2l) at (-0.6,0.5) {$  $};
              \draw [rounded corners]
                (-0.5,-0.5) rectangle (0.5,1.5);      
            \end{scope}
            \path[cd]
              (0r) edge (1l)
              (2l) edge (1r);
          \end{tikzpicture}            
        \]
      \end{minipage}
      \linebreak
      \vspace{1em}
      \begin{minipage}{0.45\linewidth}
        \centering
        \[
          \begin{tikzpicture} 
            \begin{scope} 
              \node () at (0,0.5) {$ \emptyset $};
              \node (0r) at (0.6,0.5) {$  $};
              \draw [rounded corners]
                (-0.5,-0.5) rectangle (0.5,1.5);      
            \end{scope}
            \begin{scope}[shift={(2,0)}] 
              \node [zxwhite] (b) at (0.5,0.5) {$  $};
              \node (1l) at (-0.6,0.5) {$  $};
              \node (1r) at (1.6,0.5)  {$  $};
              \node () at (0,-1) {\textsc{Unit}};
              \draw [rounded corners]
              (-0.5,-0.5) rectangle (1.5,1.5);    
            \end{scope}
            \begin{scope}[shift={(5,0)}] 
              \node [zxwhite] (a) at (0,0.5) {};
              \node (2l) at (-0.6,0.5) {};
              \draw [rounded corners]
              (-0.5,-0.5) rectangle (0.5,1.5);      
            \end{scope}
            \draw[cd] 
              (0r) edge (1l)
              (2l) edge (1r);
          \end{tikzpicture}            
        \]
      \end{minipage}
      \begin{minipage}{0.45\linewidth} 
        \centering
        \[
          \begin{tikzpicture}
            \begin{scope} 
              \node [zxwhite] () at (0,0.5) {};
              \node (0r) at (0.6,0.5) {};
              \draw [rounded corners]
                (-0.5,-0.5) rectangle (0.5,1.5);      
            \end{scope}
            \begin{scope}[shift={(2,0)}] 
              \node [zxwhite] (a) at (0.5,0.5) {};
              \node (1l) at (-0.6,0.5) {$  $};
              \node (1r) at (1.6,0.5)  {$  $};
              \node () at (0,-1) {\textsc{Counit}};
              \draw [rounded corners]
                (-0.5,-0.5) rectangle (1.5,1.5); 
            \end{scope}
            \begin{scope}[shift={(5,0)}] 
              \node () at (0,0.5) {$ \emptyset $};
              \node (2l) at (-0.6,0.5) {$  $};
              \draw [rounded corners]
                (-0.5,-0.5) rectangle (0.5,1.5);      
            \end{scope}
            \path[cd]
              (0r) edge (1l)
              (2l) edge (1r);
          \end{tikzpicture}            
        \]
      \end{minipage}
      \linebreak
      \vspace{1em}
      \begin{minipage}{1.0\linewidth}
        \centering
        \[
          \begin{tikzpicture}
            \begin{scope}[shift={(0,0)}]
              \node [zxwhite] () at (0,0) {};
              \node [zxwhite] () at (0,1) {};
              \draw [rounded corners]
                (-0.5,-0.5) rectangle (0.5,1.5);
              \node (l) at (0.6,0.5) {};
            \end{scope}
            \begin{scope}[shift={(2,0)}]
              \node [zxwhite] (00) at (0,0) {};
              \node [zxwhite] (01) at (0,1) {};
              \node [zxwhite] (10) at (1,0) {};
              \node [zxwhite] (11) at (1,1) {};
              \draw [graph]
                (00) edge (11)
                (01) edge[draw=white,line width=2pt] (10)
                (01) edge (10);
              \draw [rounded corners]
                (-0.5,-0.5) rectangle (1.5,1.5);
              \node (cl) at (-0.6,0.5) {};
              \node (cr) at (1.6,0.5) {};
              \node () at (0.5,-1) {\textsc{Braid}};
            \end{scope}
            \begin{scope}[shift={(5,0)}]
              \node [zxwhite](a) at (0,0) {};
              \node [zxwhite](a) at (0,1) {};
              \draw [rounded corners]
                (-0.5,-0.5) rectangle (0.5,1.5);
              \node (r) at (-0.6,0.5) {};
            \end{scope}
            \draw[cd]
              (l) edge (cl)
              (r) edge (cr);
         \end{tikzpicture}            
        \]
      \end{minipage}
    \end{minipage}
  }}
  \caption{Basic ZX-diagrams as structured cospans}
  \label{fig:basic-diagrams-as-str-csps}
\end{figure}

Translating the relations between ZX-diagrams to
structured cospans is quite straightforward.  We
provide several examples.

\[
  \begin{tikzpicture}
    \node () at (-3,-2.5) {\textsc{Spider Relations}};
    \begin{scope}[shift={(0,0)}]
      \node [zxwhite] at (0,3) {};
      \node        () at (0,2.5) {$ \vdots $};
      \node [zxwhite] at (0,2) {};
      \node [zxwhite] at (0,1) {};
      \node        () at (0,0.5) {$ \vdots $};
      \node [zxwhite] at (0,0) {};
      \node (00b) at (0,-0.6)  {};
      \node (00r) at (0.6,1.5) {};
      \draw [rounded corners]
        (-0.5,-0.5) rectangle (0.5,3.5);
    \end{scope}
    \begin{scope}[shift={(2,0)}] 
      \node [zxwhite] (i3) at (0,3)   {};
      \node           ()   at (0,2.5) {$\vdots$};
      \node [zxwhite] (i2) at (0,2)   {};
      \node [zxwhite] (i1) at (0,1)   {};
      \node           ()   at (0,0.5) {$\vdots$};
      \node [zxwhite] (i0) at (0,0)   {};
      \node [zxgreen] (m1) at (1,2.5) {};
      \node [zxwhite] (w0) at (0.5,1.5) {};
      \node [zxwhite] (w1) at (1.5,1.5) {};
      \node [zxgreen] (m0) at (1,0.5) {};
      \node [zxwhite] (o3) at (2,3)   {};
      \node           ()   at (2,2.5) {$\vdots$};
      \node [zxwhite] (o2) at (2,2)   {};
      \node [zxwhite] (o1) at (2,1)   {};
      \node           ()   at (2,0.5) {$\vdots$};
      \node [zxwhite] (o0) at (2,0)   {};
      \node           ()   at (1,1.5) {$\cdots$};
      \draw [graph]
        (i0) edge[bend right=10] (m0)
        (i1) edge[bend left=10]  (m0)
        (i2) edge[bend right=10] (m1)
        (i3) edge[bend left=10]  (m1)
        (m0) edge[bend right=10] (w1)
        (w1) edge[bend right=10] (m1)
        (m0) edge[bend left=10]  (w0)
        (w0) edge[bend left=10]  (m1)
        (m0) edge[bend right=10] (o0)
        (m0) edge[bend left=10]  (o1)
        (m1) edge[bend right=10] (o2)
        (m1) edge[bend left=10]  (o3);
      \draw [rounded corners]
        (-0.5,-0.5) rectangle (2.5,3.5);
      \node (01l) at (-0.6,1.5) {};
      \node (01r) at (2.6,1.5)  {};
      \node (01b) at (1,-0.6)   {};
    \end{scope}
    \begin{scope}[shift={(6,0)}] 
      \node [zxwhite] at (0,3)   {};
      \node        () at (0,2.5) {$\vdots$};
      \node [zxwhite] at (0,2)   {};
      \node [zxwhite] at (0,1)   {};
      \node        () at (0,0.5) {$\vdots$};
      \node [zxwhite] at (0,0)   {};
      \draw [rounded corners]
        (-0.5,-0.5) rectangle (0.5,3.5);
      \node (02l) at (-0.6,1.5) {};
      \node (02b) at (0,-0.6)   {};
    \end{scope}
    \begin{scope}[shift={(0,-3)}] 
      \node [zxwhite] () at (0,1)   {};
      \node           () at (0,0.5) {$\vdots$};
      \node [zxwhite] () at (0,0)   {};
      \draw [rounded corners]
        (-0.5,-0.5) rectangle (0.5,1.5);
      \node (10t) at (0,1.6)   {};
      \node (10r) at (0.6,0.5) {};
      \node (10b) at (0,-0.6)  {};
    \end{scope}
    \begin{scope}[shift={(2,-3)}] 
      \node [zxwhite] (i1) at (0,1)   {};
      \node           ()   at (0,0.5) {$\vdots$};
      \node [zxwhite] (i0) at (0,0)   {};
      \node [zxwhite] (o1) at (2,1)   {};
      \node           ()   at (2,0.5) {$\vdots$};
      \node [zxwhite] (o0) at (2,0)   {};
      \draw [rounded corners]
        (-0.5,-0.5) rectangle (2.5,1.5);
      \node (11t) at (1,1.6)    {};
      \node (11l) at (-0.6,0.5) {};
      \node (11r) at (2.6,0.5)  {};
      \node (11b) at (1,-0.6)   {};
    \end{scope}
    \begin{scope}[shift={(6,-3)}] 
      \node [zxwhite] () at (0,1)   {};
      \node           () at (0,0.5) {$\vdots$};
      \node [zxwhite] () at (0,0)   {};
      \draw [rounded corners]
        (-0.5,-0.5) rectangle (0.5,1.5);
      \node (12t) at (0,1.6)    {};
      \node (12l) at (-0.6,0.5) {};
      \node (12b) at (0,-0.6)   {};
    \end{scope}
    \begin{scope}[shift={(0,-6)}] 
      \node [zxwhite] () at (0,1)   {};
      \node           () at (0,0.5) {$\vdots$};
      \node [zxwhite] () at (0,0)   {};
      \draw [rounded corners]
        (-0.5,-0.5) rectangle (0.5,1.5);
      \node (20t) at (0,1.6)   {};
      \node (20r) at (0.6,0.5) {};
      \node (20b) at (0,-0.6)  {};
    \end{scope}
    \begin{scope}[shift={(2,-6)}] 
      \node [zxwhite] (i1) at (0,1)   {};
      \node           ()   at (0,0.5) {$\vdots$};
      \node [zxwhite] (i0) at (0,0)   {};
      \node [zxgreen] (m0) at (1,0.5) {};
      \node [zxwhite] (o1) at (2,1)   {};
      \node           ()   at (2,0.5) {$\vdots$};
      \node [zxwhite] (o0) at (2,0)   {};
      \draw[graph]
        (i0) edge[bend right=10] (m0)
        (i1) edge[bend left=10]  (m0)
        (m0) edge[bend right=10] (o0)
        (m0) edge[bend left=10]  (o1);
      \draw [rounded corners]
        (-0.5,-0.5) rectangle (2.5,1.5);
      \node (21t) at (1,1.6)  {};
      \node (21l) at (-0.6,0.5) {};
      \node (21r) at (2.6,0.5)  {};
    \end{scope}
    \begin{scope}[shift={(6,-6)}] 
      \node [zxwhite] at (0,1) {};
      \node [zxwhite] at (0,0) {};
      \draw [rounded corners]
        (-0.5,-0.5) rectangle (0.5,1.5);
      \node (22l) at (-0.6,0.5) {};
      \node (22t) at (0,1.6)  {};
    \end{scope}
    \draw[cd]
      (00r) edge (01l)
      (02l) edge (01r)
      (10r) edge (11l)
      (12l) edge (11r)
      (20r) edge (21l)
      (22l) edge (21r)
      (10t) edge (00b)
      (10b) edge (20t)
      (11t) edge (01b)
      (11b) edge (21t)
      (12t) edge (02b)
      (12b) edge (22t);
  \end{tikzpicture}
\]
\[
  \begin{tikzpicture}
    \node () at (-3,-2.5) {\textsc{Cup Relation}};
    \begin{scope}[shift={(0,0)}] 
      \node [zxwhite] () at (0,0) {};
      \node [zxwhite] () at (0,1) {};
      \draw [rounded corners]
        (-0.5,-0.5) rectangle (0.5,1.5);
      \node (00r) at (0.6,0.5) {};
      \node (00b) at (0,-0.6) {};  
    \end{scope}
    \begin{scope}[shift={(2,0)}] 
      \node [zxwhite] (a) at (0,0) {};
      \node [zxwhite] (b) at (0,1) {};
      \node [zxgreen] (c) at (1,0.5) {};
      \draw [graph]
        (a) edge[bend right=10] (c)
        (b) edge[bend left=10]  (c);
      \draw [rounded corners]
        (-0.5,-0.5) rectangle (1.5,1.5);
      \node (01l) at (-0.6,0.5) {};
      \node (01r) at (1.6,0.5)  {};
      \node (01b) at (0.5,-0.6) {};  
    \end{scope}
    \begin{scope}[shift={(5,0)}] 
      \node () at (0,0.5) {$ 0 $};
      \draw [rounded corners]
        (-0.5,-0.5) rectangle (0.5,1.5);
      \node (02l) at (-0.6,0.5) {};
      \node (02b) at (0,-0.6) {};
    \end{scope}
    \begin{scope}[shift={(0,-3)}] 
      \node [zxwhite] () at (0,0) {};
      \node [zxwhite] () at (0,1) {};
      \draw [rounded corners]
        (-0.5,-0.5) rectangle (0.5,1.5);
      \node (10t) at (0,1.6)   {};
      \node (10r) at (0.6,0.5) {};
      \node (10b) at (0,-0.6)  {};  
    \end{scope}
    \begin{scope}[shift={(2,-3)}] 
      \node [zxwhite] (a) at (0,0) {};
      \node [zxwhite] (b) at (0,1) {};
      \draw [rounded corners]
        (-0.5,-0.5) rectangle (1.5,1.5);
      \node (11l) at (-0.6,0.5) {};
      \node (11r) at (1.6,0.5)  {};
      \node (11t) at (0.5,1.6)  {};
      \node (11b) at (0.5,-0.6) {};  
    \end{scope}
    \begin{scope}[shift={(5,-3)}] 
      \node () at (0,0.5) {$ 0 $};
      \draw [rounded corners]
        (-0.5,-0.5) rectangle (0.5,1.5);
      \node (12l) at (-0.6,0.5) {};
      \node (12t) at (0,1.6)    {};  
      \node (12b) at (0,-0.6)   {};
    \end{scope}
    \begin{scope}[shift={(0,-6)}] 
      \node [zxwhite] () at (0,0) {};
      \node [zxwhite] () at (0,1) {};
      \draw [rounded corners]
        (-0.5,-0.5) rectangle (0.5,1.5);
      \node (20t) at (0,1.6)   {};
      \node (20r) at (0.6,0.5) {};
      \node (20b) at (0,-0.6)  {};  
    \end{scope}
    \begin{scope}[shift={(2,-6)}] 
      \node [zxwhite] (a) at (0,0) {};
      \node [zxwhite] (b) at (0,1) {};
      \node [zxwhite] (c) at (1,0.5) {};
      \draw[graph]
        (a) edge[bend right=10] (c)
        (b) edge[bend left=10]  (c);
      \draw [rounded corners]
        (-0.5,-0.5) rectangle (1.5,1.5);
      \node (21l) at (-0.6,0.5) {};
      \node (21r) at (1.6,0.5)  {};
      \node (21t) at (0.5,1.6)  {};
      \node (21b) at (0.5,-0.6) {};  
    \end{scope}
    \begin{scope}[shift={(5,-6)}] 
      \node () at (0,0.5) {$ 0 $};
      \draw [rounded corners]
        (-0.5,-0.5) rectangle (0.5,1.5);
      \node (22l) at (-0.6,0.5) {};
      \node (22t) at (0,1.6)    {};  
      \node (22b) at (0,-0.6)   {};
    \end{scope}
    \path[cd]
      (00r) edge (01l)
      (02l) edge (01r)
      (10r) edge (11l)
      (12l) edge (11r)
      (20r) edge (21l)
      (22l) edge (21r)
      (00b) edge (10t)
      (20t) edge (10b)
      (01b) edge (11t)
      (21t) edge (11b)
      (02b) edge (12t)
      (22t) edge (12b);
  \end{tikzpicture}
\]

The remaining relations from Figure
\ref{fig:zx-equations} can be translated into
spans of structured cospans in this way.  We include an additional rewrite
\[
  \begin{tikzpicture}
    \node () at (-3,-2) {\textsc{Wire Relation}};
    \begin{scope}[shift={(0,0)}] 
      \node [zxwhite] () at (0,0) {};
      \draw [rounded corners]
        (-0.5,-0.5) rectangle (0.5,0.5);
      \node (00r) at (0.6,0) {};
      \node (00b) at (0,-0.6) {};  
    \end{scope}
    \begin{scope}[shift={(2,0)}] 
      \node [zxwhite] (a) at (0,0) {};
      \node [zxwhite] (b) at (1,0) {};
      \draw [graph] (a) to (b);
      \draw [rounded corners]
        (-0.5,-0.5) rectangle (1.5,0.5);
      \node (01l) at (-0.6,0) {};
      \node (01r) at (1.6,0)  {};
      \node (01b) at (0.5,-0.6) {};  
    \end{scope}
    \begin{scope}[shift={(5,0)}] 
      \node [zxwhite] () at (0,0) {};
      \draw [rounded corners]
        (-0.5,-0.5) rectangle (0.5,0.5);
      \node (02l) at (-0.6,0) {};
      \node (02b) at (0,-0.6) {};
    \end{scope}
    \begin{scope}[shift={(0,-2)}] 
      \node [zxwhite] () at (0,0) {};
      \draw [rounded corners]
        (-0.5,-0.5) rectangle (0.5,0.5);
      \node (10t) at (0,0.6)   {};
      \node (10r) at (0.6,0) {};
      \node (10b) at (0,-0.6)  {};  
    \end{scope}
    \begin{scope}[shift={(2,-2)}] 
      \node [zxwhite] (a) at (0,0) {};
      \node [zxwhite] (b) at (1,0) {};
      \draw [rounded corners]
        (-0.5,-0.5) rectangle (1.5,0.5);
      \node (11l) at (-0.6,0) {};
      \node (11r) at (1.6,0)  {};
      \node (11t) at (0.5,0.6)  {};
      \node (11b) at (0.5,-0.6) {};  
    \end{scope}
    \begin{scope}[shift={(5,-2)}] 
      \node [zxwhite] () at (0,0) {};
      \draw [rounded corners]
        (-0.5,-0.5) rectangle (0.5,0.5);
      \node (12l) at (-0.6,0) {};
      \node (12t) at (0,0.6)  {};  
      \node (12b) at (0,-0.6) {};
    \end{scope}
    \begin{scope}[shift={(0,-4)}] 
      \node [zxwhite] () at (0,0) {};
      \draw [rounded corners]
        (-0.5,-0.5) rectangle (0.5,0.5);
      \node (20t) at (0,0.6)   {};
      \node (20r) at (0.6,0) {};
    \end{scope}
    \begin{scope}[shift={(2,-4)}] 
      \node [zxwhite] (a) at (0.5,0) {};
      \draw [rounded corners]
        (-0.5,-0.5) rectangle (1.5,0.5);
      \node (21l) at (-0.6,0) {};
      \node (21r) at (1.6,0)  {};
      \node (21t) at (0.5,0.6)  {};
    \end{scope}
    \begin{scope}[shift={(5,-4)}] 
      \node [zxwhite] () at (0,0) {};
      \draw [rounded corners]
        (-0.5,-0.5) rectangle (0.5,0.5);
      \node (22l) at (-0.6,0) {};
      \node (22t) at (0,0.6)  {};  
    \end{scope}
    \draw[cd]
      (00r) edge node[]{$  $} (01l)
      (02l) edge node[]{$  $} (01r)
      (10r) edge node[]{$  $} (11l)
      (12l) edge node[]{$  $} (11r)
      (20r) edge node[]{$  $} (21l)
      (22l) edge node[]{$  $} (21r)
      (10t) edge[] (00b)
      (10b) edge[] (20t)
      (11t) edge[] (01b)
      (11b) edge[] (21t)
      (12t) edge[] (02b)
      (12b) edge[] (22t); 
  \end{tikzpicture}
\] 
to account for the fact that the wire structured
cospan in Figure
\ref{fig:basic-diagrams-as-str-csps}
is, a priori, not an identity. This wire relation
ensures that the wire structured cospan is an
identity.

We are now ready to define the double category
$ \ZZZX $.

\begin{definition}
  \label{def:zx-double-category}
  Let
  \[
    \adjunction{\FinSet}{\FinGraphGamma}{L}{R}{4}
  \]
  be the adjunction defined so that $ L $ assigns a set to
  the discrete graph that is constant over the white node on
  that set and where $ R $ returns the set of white nodes of
  a graph over $ \Gamma $.  Define $\ZZZX$ to be the
  isofibrant symmetric monoidal sub-double of
  $_L \BBBoldRewrite$ generated by the basic structured
  cospans and the basic rewrites for ZX-diagrams.
\end{definition}

In this definition, using $ _L \BBBoldRewrite $ as an
ambient double category ensures that generating $ \ZZZX $ is
well-defined. All of the required structure and properties
are in place and $ _L \BBBoldRewrite $ bounds the
generation. Now, because $ \ZZZX $ is an isofibrant
symmetric monoidal category---true by construction---we use
Shulman's work
\cite{shulman_contructing}
to provide the symmetric monoidal bicategory
$ \ZZX $.

\begin{proposition}
  There is a symmetric monoidal bicategory $ \ZZX
  $ whose objects are finite sets, 1-arrows are
  generated by the basic $ L $-structured cospans
  in Figure \ref{fig:basic-diagrams-as-str-csps},
  and 2-arrows are bold rewrites generated by the
  basic rewrites of ZX-diagrams.
\end{proposition}

The ZX-diagrams appear in $ \ZZZX $ as horizontal 1-arrows
and in $ \ZZX $ as 1-arrows. Composing the ZX-diagrams works
as it does in the original ZX-calculus; pushout formalizes
the gluing of dangling edges. Indeed, composing basic
diagrams provides `compound' diagrams. For example,
composing
\[
  \begin{tikzpicture}
    \begin{scope}
      \node [zxwhite] () at (0,0) {};
      \node [zxwhite] () at (0,1) {};
      \draw [rounded corners]
        (-0.5,-0.5) rectangle (0.5,1.5);
      \node (0r) at (0.6,0.5) {};
    \end{scope}
    \begin{scope}[shift={(2,1)}]
      \node [zxwhite]  (a) at (0,0)    {};
      \node [zxwhite]  (b) at (0,1)    {};
      \node [zxgreen]  (c) at (1,0.5)  {};
      \node            ()  at (1,1)    {$ \alpha $};
      \node [zxwhite]  (d) at (2,0.5)  {};
      \draw[graph]
        (a) edge[bend right=10] (c)
        (b) edge[bend left=10]  (c)
        (c) edge                (d);
      \draw [rounded corners]
        (-0.5,-0.5) rectangle (2.5,1.5);
      \node (1l) at (-0.6,0.5) {};
      \node (1r) at (2.6,0.5)  {};
    \end{scope}
    \begin{scope}[shift={(6,0)}]
      \node [zxwhite] () at (0,0.5) {};
      \draw [rounded corners]
        (-0.5,-0.5) rectangle (0.5,1.5);
      \node (2l) at (-0.6,0.5) {};
      \node (2r) at (0.6,0.5) {};
    \end{scope}
    \begin{scope}[shift={(8,1)}]
      \node [zxwhite] (a) at (0,0.5) {};
      \node [zxgreen] (b) at (1,0.5) {};
      \node           ( ) at (1,1)   {$\beta$};
      \node [zxwhite] (c) at (2,0.5) {};
      \draw [graph]
        (a) edge (b)
        (b) edge (c);
      \draw [rounded corners]
        (-0.5,-0.5) rectangle (2.5,1.5);
      \node (3l) at (-0.6,0.5) {};
      \node (3r) at (2.6,0.5)  {};
    \end{scope}
    \begin{scope}[shift={(12,0)}]
      \node [zxwhite] () at (0,0.5) {};
      \draw [rounded corners]
        (-0.5,-0.5) rectangle (0.5,1.5);
      \node (4l) at (-0.6,0.5) {};
    \end{scope}
    \draw[cd]
      (0r) edge node[]{$  $} (1l)
      (2l) edge node[]{$  $} (1r)
      (2r) edge node[]{$  $} (3l)
      (4l) edge node[]{$  $} (3r);
  \end{tikzpicture}
\]
gives
\[
  \begin{tikzpicture}
    \begin{scope}
      \node [zxwhite] () at (0,0) {};
      \node [zxwhite] () at (0,1) {};
      \draw [rounded corners]
        (-0.5,-0.5) rectangle (0.5,1.5);
      \node (0r) at (0.6,0.5) {};
    \end{scope}
    \begin{scope}[shift={(2,1)}]
      \node [zxwhite]  (a) at (0,0)    {};
      \node [zxwhite]  (b) at (0,1)    {};
      \node [zxgreen]  (c) at (1,0.5)  {};
      \node            ( ) at (1,1)    {$\alpha$};
      \node [zxwhite]  (d) at (2,0.5)  {};
      \node [zxgreen]  (e) at (3,0.5)  {};
      \node            ( ) at (3,1)    {$\beta$};
      \node [zxwhite]  (f) at (4,0.5)  {};
      \draw[graph]
        (a) edge[bend right=10] (c)
        (b) edge[bend left=10]  (c)
        (c) edge                (d)
        (d) edge                (e)
        (e) edge                (f);
      \draw [rounded corners]
        (-0.5,-0.5) rectangle (4.5,1.5);
      \node (1l) at (-0.6,0.5) {};
      \node (1r) at (4.6,0.5)  {};
    \end{scope}
    \begin{scope}[shift={(8,0)}]
      \node [zxwhite] () at (0,0.5) {};
      \draw [rounded corners]
        (-0.5,-0.5) rectangle (0.5,1.5);
      \node (2l) at (-0.6,0.5) {};
    \end{scope}
    \draw[cd]
    (0r) edge (1l)
    (2l) edge (1r);
  \end{tikzpicture}
\]
To this, we can apply the Spider Relation
\[
  \begin{tikzpicture}
    \begin{scope}[shift={(0,0)}] 
      \node [zxwhite] at (0,1) {};
      \node [zxwhite] at (0,0) {};
      \draw [rounded corners]
        (-0.5,-0.5) rectangle (0.5,1.5);
      \node (00b) at (0,-0.6)  {};
      \node (00r) at (0.6,0.5) {};
    \end{scope}
    \begin{scope}[shift={(2,0)}] 
      \node [zxwhite]  (a) at (0,0)    {};
      \node [zxwhite]  (b) at (0,1)    {};
      \node [zxgreen]  (c) at (1,0.5)  {};
      \node            ( ) at (1,1)    {$\alpha$};
      \node [zxwhite]  (d) at (2,0.5)  {};
      \node [zxgreen]  (e) at (3,0.5)  {};
      \node            ( ) at (3,1)    {$\beta$};
      \node [zxwhite]  (f) at (4,0.5)  {};
      \draw[graph]
        (a) edge[bend right=10] (c)
        (b) edge[bend left=10]  (c)
        (c) edge                (d)
        (d) edge                (e)
        (e) edge                (f);
      \draw [rounded corners]
        (-0.5,-0.5) rectangle (4.5,1.5);
      \node (01l) at (-0.6,0.5) {};
      \node (01r) at (4.6,0.5)  {};
      \node (01b) at (2,-0.6)   {};
    \end{scope}
    \begin{scope}[shift={(8,0)}] 
      \node [zxwhite] () at (0,0.5) {};
      \draw [rounded corners]
        (-0.5,-0.5) rectangle (0.5,1.5);
      \node (02l) at (-0.6,0.5) {};
      \node (02b) at (0,-0.6)   {};
    \end{scope}
    \begin{scope}[shift={(0,-3)}] 
      \node [zxwhite] () at (0,1)   {};
      \node [zxwhite] () at (0,0)   {};
      \draw [rounded corners]
        (-0.5,-0.5) rectangle (0.5,1.5);
      \node (10t) at (0,1.6)   {};
      \node (10r) at (0.6,0.5) {};
      \node (10b) at (0,-0.6)  {};
    \end{scope}
    \begin{scope}[shift={(2,-3)}] 
      \node [zxwhite] (i1) at (0,1)   {};
      \node [zxwhite] (i0) at (0,0)   {};
      \node [zxwhite] (o0) at (4,0.5) {};
      \draw [rounded corners]
        (-0.5,-0.5) rectangle (4.5,1.5);
      \node (11t) at (2,1.6)    {};
      \node (11l) at (-0.6,0.5) {};
      \node (11r) at (4.6,0.5)  {};
      \node (11b) at (2,-0.6)   {};
    \end{scope}
    \begin{scope}[shift={(8,-3)}] 
      \node [zxwhite] () at (0,0.5) {};
      \draw [rounded corners]
        (-0.5,-0.5) rectangle (0.5,1.5);
      \node (12t) at (0,1.6)    {};
      \node (12l) at (-0.6,0.5) {};
      \node (12b) at (0,-0.6)   {};
    \end{scope}
    \begin{scope}[shift={(0,-6)}] 
      \node [zxwhite] () at (0,1)   {};
      \node [zxwhite] () at (0,0)   {};
      \draw [rounded corners]
        (-0.5,-0.5) rectangle (0.5,1.5);
      \node (20t) at (0,1.6)   {};
      \node (20r) at (0.6,0.5) {};
    \end{scope}
    \begin{scope}[shift={(2,-6)}] 
      \node [zxwhite]  (a) at (0,0)    {};
      \node [zxwhite]  (b) at (0,1)    {};
      \node [zxgreen]  (c) at (2,0.5)  {};
      \node            ( ) at (2,1)    {$\alpha+\beta$};
      \node [zxwhite]  (d) at (4,0.5)  {};
      \draw[graph]
        (a) edge[bend right=10] (c)
        (b) edge[bend left=10]  (c)
        (c) edge                (d);
      \draw [rounded corners]
        (-0.5,-0.5) rectangle (4.5,1.5);
      \node (21t) at (2,1.6)  {};
      \node (21l) at (-0.6,0.5) {};
      \node (21r) at (4.6,0.5)  {};
    \end{scope}
    \begin{scope}[shift={(8,-6)}] 
      \node [zxwhite] at (0,0.5) {};
      \draw [rounded corners]
        (-0.5,-0.5) rectangle (0.5,1.5);
      \node (22l) at (-0.6,0.5) {};
      \node (22t) at (0,1.6)  {};
    \end{scope}
    \draw[cd]
      (00r) edge (01l)
      (02l) edge (01r)
      (10r) edge (11l)
      (12l) edge (11r)
      (20r) edge (21l)
      (22l) edge (21r)
      (10t) edge (00b)
      (10b) edge (20t)
      (11t) edge (01b)
      (11b) edge (21t)
      (12t) edge (02b)
      (12b) edge (22t);
  \end{tikzpicture}
\]
Because the vertical 1-arrows are identities, this 2-arrow
exists in both $ \ZZZX $ and $ \ZZX $. The spider relation
simplifies the ZX-diagram in the top row to that in the
bottom row.

\begin{theorem}
  The bicategory $ \ZZX $ is a bicategory of relations.
\end{theorem}

\begin{proof}
  Because $ \ZZX $ includes the structure maps to give every
  object a Frobenius monoid structure, every requirement
  descends from ambient category $ _L \BBoldRewrite $ being
  a bicategory of relations (see Theorem
  \ref{thm:bold-rewrite-bicat-rels}).
\end{proof}

This bicategory extends the original category $ \ZX $. To
show this, we will show the `decategorification' of $ \ZZX $
is $ \ZX $. The process of decategorification essentially
turns an $ n $-category into an $ n-1 $-category. For us, we
turn a (weak) 2-category into a 1-category by identifying any
1-arrows connected by a zig-zag of 2-arrows. 

\begin{definition} \label{def:decat-zx}

  Define $\decat (\ZZX)$ to be the category whose objects
  are those of $\ZZX$ and whose arrows the 1-arrows of
  $\ZZX $ modulo the equivalence relation $\sim$ generated
  by $f \sim g$ if and only if there is a 2-arrow
  $f \Rightarrow g$ in $\ZZX$.
  
\end{definition}

\begin{theorem} \label{thm:decatzx-is-dagger-compact}

  The category $\decat ( \ZZX )$ is dagger compact
  via the identity on objects functor described by
  \[
  \begin{tikzpicture}
    \begin{scope}[shift={(0,0)}] 
      \node [zxwhite] ()  at (0,1)    {};
      \node           ()  at (0,0.75) {$\vdots$};
      \node [zxwhite] ()  at (0,0)    {};
      \draw [rounded corners]
      (-0.5,-0.5) rectangle (0.5,1.5);
      \node (0r) at (0.6,0.5) {$  $};
    \end{scope}
    \begin{scope}[shift={(2,1)}] 
      \node [zxwhite] (a) at (0,1) {};
      \node           ()  at (0.25,0.75) {$\vdots$};
      \node [zxwhite] (b) at (0,0) {};
      \node [zxgreen] (c) at (1,0.5) {};
      \node           ( ) at (1,1) {$\alpha$};
      \node [zxwhite] (d) at (2,1) {};
      \node           ()  at (1.75,0.75) {$\vdots$};
      \node [zxwhite] (e) at (2,0) {};
      \draw[graph]
        (a) edge[bend left=10]  (c)
        (b) edge[bend right=10] (c)
        (c) edge[bend left=10]  (d)
        (c) edge[bend right=10] (e);
      \draw [rounded corners]
        (-0.5,-0.5) rectangle (2.5,1.5);
      \node (1l) at (-0.6,0.5) {$  $};  
      \node (1r) at (2.6,0.5)  {$  $};
    \end{scope}
    \begin{scope}[shift={(6,0)}] 
      \node [zxwhite] () at  (0,1)     {};
      \node           ()  at (0,0.75) {$\vdots$};
      \node [zxwhite] () at  (0,0)     {};
      \draw [rounded corners]
        (-0.5,-0.5) rectangle (0.5,1.5);
      \node (2l) at (-0.6,0.5) {$  $};
    \end{scope}
    \path[cd] 
      (0r) edge node[]{$  $} (1l)
      (2l) edge node[]{$  $} (1r);
      \node () at (7,1) {$ \xmapsto{\dagger} $};
    \begin{scope}[shift={(8,0)}] 
      \node [zxwhite] () at (0,1) {};
      \node           () at (0,0.75) {$\vdots$};
      \node [zxwhite] () at (0,0) {};
      \draw [rounded corners]
        (-0.5,-0.5) rectangle (0.5,1.5);
      \node (0'r) at (0.6,0.5) {$  $};
    \end{scope}
    \begin{scope}[shift={(10,1)}] 
      \node [zxwhite] (a) at (0,1) {};
      \node           ()  at (0.25,0.75) {$\vdots$};
      \node [zxwhite] (b) at (0,0) {};
      \node [zxgreen] (c) at (1,0.5) {};
      \node           ( ) at (1,1)       {$-\alpha$};
      \node [zxwhite] (d) at (2,1) {};
      \node           ()  at (1.75,0.75) {$\vdots$};
      \node [zxwhite] (e) at (2,0) {};
      \draw[graph]
        (a) edge[bend left=10]  (c)
        (b) edge[bend right=10] (c)
        (c) edge[bend left=10]  (d)
        (c) edge[bend right=10] (e);
      \draw [rounded corners]
        (-0.5,-0.5) rectangle (2.5,1.5);                     
      \node (1'l) at (-0.6,0.5) {$  $};  
      \node (1'r) at (2.6,0.5)  {$  $};
    \end{scope}
    \begin{scope}[shift={(14,0)}] 
      \node [zxwhite] () at (0,1) {};
      \node           () at (0,0.75) {$\vdots$};
      \node [zxwhite] () at (0,0) {};
      \draw [rounded corners]
        (-0.5,-0.5) rectangle (0.5,1.5);
      \node (2'l) at (-0.6,0.5) {$  $};
    \end{scope}
    \path[cd] 
      (0'r) edge node[]{$  $} (1'l)
      (2'l) edge node[]{$  $} (1'r);
  \end{tikzpicture}
\]
\[
  \begin{tikzpicture}
    \begin{scope}[shift={(0,0)}] 
      \node [zxwhite] () at (0,1)     {};
      \node           ()  at (0,0.75) {$\vdots$};
      \node [zxwhite] () at (0,0)     {};
      \draw [rounded corners]
        (-0.5,-0.5) rectangle (0.5,1.5);                    
      \node (0r) at (0.6,0.5) {$  $};
    \end{scope}
    \begin{scope}[shift={(2,1)}] 
      \node [zxwhite] (a) at (0,1) {};
      \node           ()  at (0.25,0.75) {$\vdots$};
      \node [zxwhite] (b) at (0,0) {};
      \node [zxred] (c) at (1,0.5) {};
      \node           ( ) at (1,1) {$\alpha$};
      \node [zxwhite] (d) at (2,1) {};
      \node           ()  at (1.75,0.75) {$\vdots$};
      \node [zxwhite] (e) at (2,0) {};
      \draw[graph]
        (a) edge[bend left=10]  (c)
        (b) edge[bend right=10] (c)
        (c) edge[bend left=10]  (d)
        (c) edge[bend right=10] (e);
      \draw [rounded corners]
        (-0.5,-0.5) rectangle (2.5,1.5);
      \node (1l) at (-0.6,0.5) {$  $};  
      \node (1r) at (2.6,0.5)  {$  $};
    \end{scope}
    \begin{scope}[shift={(6,0)}] 
      \node [zxwhite] () at (0,1)     {};
      \node           ()  at (0,0.75) {$\vdots$};
      \node [zxwhite] () at (0,0)     {};
      \draw [rounded corners]
        (-0.5,-0.5) rectangle (0.5,1.5);
      \node (2l) at (-0.6,0.5) {$  $};
    \end{scope}
    \path[cd] 
      (0r) edge node[]{$  $} (1l)
      (2l) edge node[]{$  $} (1r);
      \node () at (7,1) {$ \xmapsto{\dagger} $};
    \begin{scope}[shift={(8,0)}] 
      \node [zxwhite] () at (0,1) {};
      \node           () at (0,0.75) {$\vdots$};
      \node [zxwhite] () at (0,0) {};
      \draw [rounded corners]
        (-0.5,-0.5) rectangle (0.5,1.5);
      \node (0'r) at (0.6,0.5) {$  $};
    \end{scope}
    \begin{scope}[shift={(10,1)}] 
      \node [zxwhite] (a) at (0,1) {};
      \node           ()  at (0.25,0.75) {$\vdots$};
      \node [zxwhite] (b) at (0,0) {};
      \node [zxred] (c) at (1,0.5) {};
      \node           ( ) at (1,1)       {$-\alpha$};
      \node [zxwhite] (d) at (2,1) {};
      \node           ()  at (1.75,0.75) {$\vdots$};
      \node [zxwhite] (e) at (2,0) {};
      \draw[graph]
        (a) edge[bend left=10]  (c)
        (b) edge[bend right=10] (c)
        (c) edge[bend left=10]  (d)
        (c) edge[bend right=10] (e);
      \draw [rounded corners]
        (-0.5,-0.5) rectangle (2.5,1.5);
      \node (1'l) at (-0.6,0.5) {$  $};  
      \node (1'r) at (2.6,0.5)  {$  $};
    \end{scope}
    \begin{scope}[shift={(14,0)}] 
      \node [zxwhite] () at (0,1) {};
      \node           () at (0,0.75) {$\vdots$};
      \node [zxwhite] () at (0,0) {};
      \draw [rounded corners]
        (-0.5,-0.5) rectangle (0.5,1.5);
      \node (2'l) at (-0.6,0.5) {$  $};
    \end{scope}
    \draw[cd] 
      (0'r) edge node[]{$  $} (1'l)
      (2'l) edge node[]{$  $} (1'r);
  \end{tikzpicture}
\]
as well as by identity on the wire, Hadamard, and
diamond morphisms.
\end{theorem}
      
\begin{proof}
  Compact closedness follows from the self duality of
  objects via the evaluation 
  \[
    \begin{tikzpicture}
      \begin{scope}
        \node [zxwhite] (a) at  (0,2)   {$  $};
        \node [zxwhite] (b) at  (0,1.5) {$  $};
        \node           ( )  at (0,1)   {$\vdots$};
        \node [zxwhite] (c) at  (0,0.5) {$  $};
        \node [zxwhite] (d) at  (0,0)   {$  $};
        \node (0r) at (0.6,1) {$  $};
        \draw [rounded corners]
          (-0.5,-0.5) rectangle (0.5,2.5);      
      \end{scope}
      \begin{scope}[shift={(2,1)}]
        \node [zxwhite] (a) at (0,2)      {$  $};
        \node [zxwhite] (b) at (0,1.5)    {$  $};
        \node           ( ) at (0,1)      {$\vdots$};
        \node [zxwhite] (c) at (0,0.5)    {$  $};
        \node [zxwhite] (d) at (0,0)      {$  $};
        \node [zxwhite] (e) at (0.5,1)    {$  $};
        \node [zxwhite] (f) at (1,1)      {$  $};
        \draw[graph]
          (a) edge[bend left=10]  (f)
          (d) edge[bend right=10] (f)
          (b) edge[bend left=10]  (e)
          (c) edge[bend right=10] (e);
        \node (1l) at (-0.6,1) {$  $};
        \node (1r) at (1.6,1)  {$  $};
        \draw [rounded corners]
          (-0.5,-0.5) rectangle (1.5,2.5);    
      \end{scope}
      \begin{scope}[shift={(5,0)}]
        \node (a) at (0,1) {$ \emptyset $};
        \node (2l) at (-0.6,1) {$  $};
        \draw [rounded corners]
          (-0.5,-0.5) rectangle (0.5,2.5);      
      \end{scope}
      \path[cd]
        (0r) edge node[]{$  $} (1l)
        (2l) edge node[]{$  $} (1r);      
    \end{tikzpicture}
  \]
  and coevaluation arrows
  \[
    \begin{tikzpicture}
      \begin{scope}
        \node (a) at (0,1) {$ \emptyset $};
        \node (0r) at (0.6,1) {$  $};
        \draw [rounded corners]
          (-0.5,-0.5) rectangle (0.5,2.5);      
      \end{scope}
      \begin{scope}[shift={(2,1)}]
        \node [zxwhite] (a) at (1,2)      {$  $};
        \node [zxwhite] (b) at (1,1.5)    {$  $};
        \node           ( ) at (1,1)      {$\vdots$};
        \node [zxwhite] (c) at (1,0.5)    {$  $};
        \node [zxwhite] (d) at (1,0)      {$  $};
        \node [zxwhite] (e) at (0.5,1)    {$  $};
        \node [zxwhite] (f) at (0,1)      {$  $};
        \draw[graph]
          (f) edge[bend left=10]  (a)
          (f) edge[bend right=10] (d)
          (e) edge[bend left=10]  (b)
          (e) edge[bend right=10] (c);
        \node (1l) at (-0.6,1) {$  $};
        \node (1r) at (1.6,1)  {$  $};
        \draw [rounded corners]
          (-0.5,-0.5) rectangle (1.5,2.5);    
      \end{scope}
      \begin{scope}[shift={(5,0)}]
        \node [zxwhite] (a) at  (0,2)   {$  $};
        \node [zxwhite] (b) at  (0,1.5) {$  $};
        \node           ( )  at (0,1)   {$\vdots$};
        \node [zxwhite] (c) at  (0,0.5) {$  $};
        \node [zxwhite] (d) at  (0,0)   {$  $};
        \node (2l) at (-0.6,1) {$  $};
        \draw [rounded corners]
          (-0.5,-0.5) rectangle (0.5,2.5);      
      \end{scope}
      \draw[cd]
        (0r) edge (1l)
        (2l) edge (1r);      
    \end{tikzpicture}
  \]
  obtained by applying the braiding maps to the disjoint
  union of cups and caps. Moreover, we can derive the snake
  equation as follows. Decompose the arrow
  \[
    \begin{tikzpicture} 
      \begin{scope} 
        \node [zxwhite] () at (0,0) {$  $};
        \node (0r) at (0.6,1) {$  $};
        \draw [rounded corners]
          (-0.5,-0.5) rectangle (0.5,2.5);
      \end{scope}
      \begin{scope}[shift={(2,1)}] 
        \node [zxwhite] (a) at (0,0) {$  $};
        \node [zxwhite] (b) at (0.5,1.5) {$  $};
        \node [zxwhite] (c) at (1,0) {$  $};
        \node [zxwhite] (d) at (1,1) {$  $};
        \node [zxwhite] (e) at (1,2) {$  $};
        \node [zxwhite] (f) at (1.5,0.5) {$  $};
        \node [zxwhite] (g) at (2,2) {$  $};
        \draw[graph]
         (a) edge[bend right=0]   (c)
         (c) edge[bend right=10]  (f)
         (d) edge[bend right=-10] (f)
         (b) edge[bend right=10]  (d)
         (b) edge[bend right=-10] (e)
         (e) edge[bend right=0]   (g);
        \node (1l) at (-0.6,1) {$  $};
        \node (1r) at (2.6,1) {$  $};
        \draw [rounded corners]
          (-0.5,-0.5) rectangle (2.5,2.5);
        \end{scope}
        \begin{scope}[shift={(6,0)}] 
        \node [zxwhite] () at (0,2) {$  $};
        \node (2l) at (-0.6,1) {$  $};
        \draw [rounded corners]
          (-0.5,-0.5) rectangle (0.5,2.5);
        \end{scope}
        \path[cd]
        (0r) edge node[]{$  $} (1l)
        (2l) edge node[]{$  $} (1r);
    \end{tikzpicture}
  \]
  into
  \[
    \begin{tikzpicture}
      \begin{scope} 
        \node [zxwhite] () at (0,0) {$  $};
        \node (0r) at (0.6,1) {$  $};
        \draw [rounded corners]
          (-0.5,-0.5) rectangle (0.5,2.5);
      \end{scope}
      \begin{scope}[shift={(2,1)}] 
        \node [zxwhite] (a) at (0,0) {$  $};
        \node [zxwhite] (b) at (0.5,1.5) {$  $};
        \node [zxwhite] (c) at (1,0) {$  $};
        \node [zxwhite] (d) at (1,1) {$  $};
        \node [zxwhite] (e) at (1,2) {$  $};
        \draw[graph]
          (a) edge                (c)
          (b) edge[bend right=10] (d)
          (b) edge[bend left=10]  (e);
        \node (1l) at (-0.6,1) {$  $};
        \node (1r) at (1.6,1) {$  $};
        \draw [rounded corners]
          (-0.5,-0.5) rectangle (1.5,2.5);
        \end{scope}
        \begin{scope}[shift={(5,0)}] 
          \node [zxwhite] () at (0,2) {$  $};
          \node [zxwhite] () at (0,1) {$  $};
          \node [zxwhite] () at (0,0) {$  $};
          \node (2l) at (-0.6,1) {$  $};
          \node (2r) at (0.6,1)  {$  $};
          \draw [rounded corners]
            (-0.5,-0.5) rectangle (0.5,2.5);
        \end{scope}
        \begin{scope}[shift={(7,1)}] 
          \node [zxwhite] (a) at (1,2) {$  $};
          \node [zxwhite] (b) at (0.5,0.5) {$  $};
          \node [zxwhite] (c) at (0,0) {$  $};
          \node [zxwhite] (d) at (0,1) {$  $};
          \node [zxwhite] (e) at (0,2) {$  $};
          \draw[graph]
            (a) edge                (e)
            (c) edge[bend right=10] (b)
            (d) edge[bend left=10]  (b);
          \node (3l) at (-0.6,1) {$  $};
          \node (3r) at (1.6,1) {$  $};
          \draw [rounded corners]
            (-0.5,-0.5) rectangle (1.5,2.5);
        \end{scope}
        \begin{scope}[shift={(10,0)}] 
          \node [zxwhite] () at (0,2) {$  $};
          \node (4l) at (-0.6,1) {$  $};
          \draw [rounded corners]
            (-0.5,-0.5) rectangle (0.5,2.5);
        \end{scope}
        \draw[cd]
          (0r) edge node[]{$  $} (1l)
          (2l) edge node[]{$  $} (1r)
          (2r) edge node[]{$  $} (3l)
          (4l) edge node[]{$  $} (3r);
    \end{tikzpicture}
  \]
  which by the cup relation, illustrated in Figure
  \ref{fig:zx-equations}, equals
  \[
    \begin{tikzpicture}
      \begin{scope} 
        \node [zxwhite] () at (0,0) {$  $};
        \node (0r) at (0.6,1) {$  $};
        \draw [rounded corners]
          (-0.5,-0.5) rectangle (0.5,2.5);
      \end{scope}
      \begin{scope}[shift={(2,1)}] 
        \node [zxwhite] (a) at (0,0) {$  $};
        \node [zxgreen] (b) at (0.5,1.5) {$  $};
        \node [zxwhite] (c) at (1,0) {$  $};
        \node [zxwhite] (d) at (1,1) {$  $};
        \node [zxwhite] (e) at (1,2) {$  $};
        \draw[graph]
         (a) edge                (c)
         (b) edge[bend right=10] (d)
         (b) edge[bend left=10]  (e);
        \node (1l) at (-0.6,1) {$  $};
        \node (1r) at (1.6,1) {$  $};
        \draw [rounded corners]
          (-0.5,-0.5) rectangle (1.5,2.5);
        \end{scope}
        \begin{scope}[shift={(5,0)}] 
          \node [zxwhite] () at (0,2) {$  $};
          \node [zxwhite] () at (0,1) {$  $};
          \node [zxwhite] () at (0,0) {$  $};
          \node (2l) at (-0.6,1) {$  $};
          \node (2r) at (0.6,1)  {$  $};
          \draw [rounded corners]
            (-0.5,-0.5) rectangle (0.5,2.5);
        \end{scope}
        \begin{scope}[shift={(7,1)}] 
          \node [zxwhite] (a) at (1,2) {$  $};
          \node [zxgreen] (b) at (0.5,0.5) {$  $};
          \node [zxwhite] (c) at (0,0) {$  $};
          \node [zxwhite] (d) at (0,1) {$  $};
          \node [zxwhite] (e) at (0,2) {$  $};
          \draw[graph]
            (a) edge                (e)
            (c) edge[bend right=10] (b)
            (d) edge[bend left=10]  (b);
          \node (3l) at (-0.6,1) {$  $};
          \node (3r) at (1.6,1) {$  $};
          \draw [rounded corners]
          (-0.5,-0.5) rectangle (1.5,2.5);
        \end{scope}
        \begin{scope}[shift={(10,0)}] 
          \node [zxwhite] () at (0,2) {$  $};
          \node (4l) at (-0.6,1) {$  $};
          \draw [rounded corners]
          (-0.5,-0.5) rectangle (0.5,2.5);
        \end{scope}
        \draw[cd]
          (0r) edge (1l)
          (2l) edge (1r)
          (2r) edge (3l)
          (4l) edge (3r);
    \end{tikzpicture}
  \]
  This can be composed to get
    \[
    \begin{tikzpicture} 
      \begin{scope} 
        \node [zxwhite] () at (0,0) {$  $};
        \node (0r) at (0.6,1) {$  $};
        \draw [rounded corners]
          (-0.5,-0.5) rectangle (0.5,2.5);
      \end{scope}
      \begin{scope}[shift={(2,1)}] 
        \node [zxwhite] (a) at (0,0) {$  $};
        \node [zxgreen] (b) at (0.5,1.5) {$  $};
        \node [zxwhite] (c) at (1,0) {$  $};
        \node [zxwhite] (d) at (1,1) {$  $};
        \node [zxwhite] (e) at (1,2) {$  $};
        \node [zxgreen] (f) at (1.5,0.5) {$  $};
        \node [zxwhite] (g) at (2,2) {$  $};
        \draw[graph]
          (a) edge                (c)
          (c) edge[bend right=10] (f)
          (d) edge[bend left=10]  (f)
          (b) edge[bend right=10] (d)
          (b) edge[bend left=10]  (e)
          (e) edge                (g);
        \node (1l) at (-0.6,1) {$  $};
        \node (1r) at (2.6,1) {$  $};
        \draw [rounded corners]
          (-0.5,-0.5) rectangle (2.5,2.5);
        \end{scope}
        \begin{scope}[shift={(6,0)}] 
        \node [zxwhite] () at (0,2) {$  $};
        \node (2l) at (-0.6,1) {$  $};
        \draw [rounded corners]
          (-0.5,-0.5) rectangle (0.5,2.5);
        \end{scope}
        \path[cd]
          (0r) edge (1l)
          (2l) edge (1r);
    \end{tikzpicture}
  \]
  which equals
    \[
    \begin{tikzpicture} 
      \begin{scope} 
        \node [zxwhite] () at (0,0) {$  $};
        \node (0r) at (0.6,0) {$  $};
        \draw [rounded corners]
          (-0.5,-0.5) rectangle (0.5,0.5);
      \end{scope}
      \begin{scope}[shift={(2,1)}] 
        \node [zxwhite] (a) at (0,0) {$  $};
        \node [zxwhite] (b) at (1,0) {$  $};
        \node [zxgreen] (c) at (2,0) {$  $};
        \node [zxwhite] (d) at (3,0) {$  $};
        \node [zxwhite] (e) at (4,0) {$  $};
        \draw [graph]
          (a) edge (b)
          (b) edge (c)
          (c) edge (d)
          (d) edge (e);
        \node (1l) at (-0.6,0) {$  $};
        \node (1r) at (4.6,0) {$  $};
        \draw [rounded corners]
          (-0.5,-0.5) rectangle (4.5,0.5);
        \end{scope}
        \begin{scope}[shift={(8,0)}] 
        \node [zxwhite] () at (0,0) {$  $};
        \node (2l) at (-0.6,0) {$  $};
        \draw [rounded corners]
          (-0.5,-0.5) rectangle (0.5,0.5);
        \end{scope}
        \path[cd]
          (0r) edge node[]{$  $} (1l)
          (2l) edge node[]{$  $} (1r);
    \end{tikzpicture}
  \]
  because of the spider relation. Finally, this
  equals the identity because of the trivial
  spider and wire relations.  Showing that the
  described functor is a dagger functor is a
  matter of checking some easy to verify details.
\end{proof}

We now show that $ \ZZZX $ is an extension of
$ \ZX $ in the sense that the category
$ \decat ( \ZZX ) $ obtained from $ \ZZZX $ is
equivalent to $ \ZX $.

\begin{theorem}
\label{thm:equiv of zx cats}
The identity on objects, dagger compact functor
$E \colon \ZX \to
\operatorname{decat}(\ZZX)$
given by
\begin{align*}
  \raisebox{
    0.75\height}{
    \includegraphics{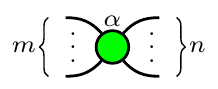}}
  &
  \begin{tikzpicture}
    \node () at (-1,0.5) {$ \mapsto  $};
    \begin{scope}
      \node [zxwhite] (a) at (0,0)    {};
      \node           ( ) at (0,0.66) {$ \vdots $};
      \node [zxwhite] (b) at (0,1)    {};
      \draw [rounded corners]
        (-0.5,-0.5) rectangle (0.5,1.5);
      \node (l) at (0.6,0.5) {};
    \end{scope}
    \begin{scope}[shift={(2,0)}]
      \node [zxwhite] (a) at (0,0)   {$ $};
      \node           ( ) at (0,0.66) {$ \vdots $};
      \node [zxwhite] (b) at (0,1)   {$  $};
      \node [zxgreen] (c) at (1,0.5) {$ $};
      \node           ( ) at (1,1) {$\alpha$};
      \node [zxwhite] (d) at (2,0)   {$ $};
      \node           ( ) at (2,0.66) {$ \vdots $};
      \node [zxwhite] (e) at (2,1)   {$  $};
      \draw[graph]
        (a) edge [bend right=10] (c)
        (b) edge [bend left=10]  (c)
        (c) edge [bend right=10] (d)
        (c) edge [bend left=10]  (e);
      \draw [rounded corners]
        (-0.5,-0.5) rectangle (2.5,1.5);  
      \node (cl) at (-0.6,0.5) {};
      \node (cr) at (2.6,0.5)  {};
    \end{scope}
    \begin{scope}[shift={(6,0)}]
      \node [zxwhite] (a) at (0,0)    {};
      \node           ( ) at (0,0.66) {$ \vdots $};
      \node [zxwhite] (b) at (0,1)    {};
      \draw [rounded corners]
        (-0.5,-0.5) rectangle (0.5,1.5);
      \node (r) at (-0.6,0.5) {};
    \end{scope}
    \draw[cd,>=angle 90]
      (l) edge node[]{$  $} (cl)
      (r) edge node[]{$  $} (cr);
  \end{tikzpicture}
    \\
  \raisebox{
    0.75\height}{
    \includegraphics{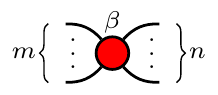}}
  &
    \begin{tikzpicture}
      \node () at (-1,0.5) {$ \mapsto  $};
      \begin{scope}
        \node [zxwhite] (a) at (0,0)    {};
        \node           ( ) at (0,0.66) {$ \vdots $};
        \node [zxwhite] (b) at (0,1)    {};
        \draw [rounded corners]
        (-0.5,-0.5) rectangle (0.5,1.5);
        \node (l) at (0.6,0.5) {};
      \end{scope}
      \begin{scope}[shift={(2,0)}]
        \node [zxwhite] (a) at (0,0)    {$ $};
        \node           ( ) at (0,0.66) {$ \vdots $};
        \node [zxwhite] (b) at (0,1)    {$  $};
        \node [zxred]   (c) at (1,0.5)  {$ $};
        \node           ( ) at (1,1)    {$\alpha$};
        \node [zxwhite] (d) at (2,0)    {$ $};
        \node           ( ) at (2,0.66) {$ \vdots $};
        \node [zxwhite] (e) at (2,1)    {$  $};
        \draw[graph]
        (a) edge [bend right=10] (c)
        (b) edge [bend left=10]  (c)
        (c) edge [bend right=10] (d)
        (c) edge [bend left=10]  (e);
        \draw [rounded corners]
        (-0.5,-0.5) rectangle (2.5,1.5);  
        \node (cl) at (-0.6,0.5) {};
        \node (cr) at (2.6,0.5)  {};
      \end{scope}
      \begin{scope}[shift={(6,0)}]
        \node [zxwhite] (a) at (0,0)    {};
        \node           ( ) at (0,0.66) {$ \vdots $};
        \node [zxwhite] (b) at (0,1)    {};
        \draw [rounded corners]
        (-0.5,-0.5) rectangle (0.5,1.5);
        \node (r) at (-0.6,0.5) {};
      \end{scope}
      \draw[cd]
      (l) edge node[]{$  $} (cl)
      (r) edge node[]{$  $} (cr);
    \end{tikzpicture}
    \\
  \raisebox{
    0.7\height}{
    \includegraphics{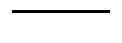}}
  &
  \begin{tikzpicture}
    \node () at (-1,0) {$ \mapsto  $};
    \begin{scope}
      \node [zxwhite] (a) at (0,0) {};
      \draw [rounded corners]
        (-0.5,-0.5) rectangle (0.5,0.5);
      \node (l) at (0.6,0) {};
    \end{scope}
    \begin{scope}[shift={(2,0)}]
      \node [zxwhite] (a) at (0,0) {};
      \node [zxwhite] (b) at (2,0) {};
      \draw [graph] (a) to (b);
      \draw [rounded corners]
        (-0.5,-0.5) rectangle (2.5,0.5);
      \node (cr) at (2.6,0) {};
      \node (cl) at (-0.6,0) {};  
    \end{scope}
    \begin{scope}[shift={(6,0)}]
      \node [zxwhite] (a) at (0,0) {};
      \draw [rounded corners]
        (-0.5,-0.5) rectangle (0.5,0.5);
      \node (r) at (-0.6,0) {};
    \end{scope}
    \draw [cd]
      (l) edge (cl)
      (r) edge (cr);
  \end{tikzpicture}
    \\
  \raisebox{
    0.75\height}{
    \includegraphics{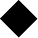}}
  &
  \begin{tikzpicture}
    \node () at (-1,0) {$ \mapsto  $};
    \begin{scope}[shift={(0,0)}]
      \node () at (0,0) {$ 0 $};
      \draw [rounded corners]
        (-0.5,-0.5) rectangle (0.5,0.5);
      \node (l) at (0.6,0) {};
    \end{scope}
    \begin{scope}[shift={(2,0)}]
      \node [zxblack] () at (1,0) {$$};
      \draw [rounded corners]
        (-0.5,-0.5) rectangle (2.5,0.5);
      \node (cl) at (-0.6,0) {};
      \node (cr) at (2.6,0) {};
    \end{scope}
    \begin{scope}[shift={(6,0)}]
      \node (a) at (0,0) {$ 0 $};
      \draw [rounded corners]
        (-0.5,-0.5) rectangle (0.5,0.5);
      \node (r) at (-0.6,0) {};
    \end{scope}
    \draw[cd]
      (l) edge node[]{$  $} (cl)
      (r) edge node[]{$  $} (cr);
  \end{tikzpicture}            
  \\
  \raisebox{
    0.8\height}{
    \includegraphics{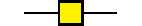}}
  &
  \begin{tikzpicture}
    \node () at (-1,0) {$ \mapsto  $};
    \begin{scope}
      \node [zxwhite] (a) at (0,0) {};
      \draw [rounded corners]
        (-0.5,-0.5) rectangle (0.5,0.5);
      \node (l) at (0.6,0) {};
    \end{scope}
    \begin{scope}[shift={(2,0)}]
      \node [zxwhite]  (a) at (0,0) {$$};
      \node [zxyellow] (b) at (1,0) {$$};
      \node [zxwhite]  (c) at (2,0) {$$};
      \draw [graph]
        (a) edge (b)
        (b) edge (c);
      \draw [rounded corners]
        (-0.5,-0.5) rectangle (2.5,0.5);
      \node (cl) at (-0.6,0) {};
      \node (cr) at (2.6,0)  {};
    \end{scope}
    \begin{scope}[shift={(6,0)}]
      \node [zxwhite] (a) at (0,0) {};
      \draw [rounded corners]
        (-0.5,-0.5) rectangle (0.5,0.5);
      \node (r) at (-0.6,0) {};
    \end{scope}
    \draw[cd]
      (l) edge node[]{$  $} (cl)
      (r) edge node[]{$  $} (cr);
  \end{tikzpicture}
\end{align*}
is an equivalence of categories.
\end{theorem}

\begin{proof}
  Essential surjectivity follows immediately from
  $E$ being identity on objects.  Fullness follows
  from the fact that the morphism generators for
  $\decat (\ZZX)$ are all in the image of $E$.
  
  Faithfulness is more involved. Let $f,g$ be
  $\ZX$-morphisms. Let $\widetilde{Ef}$,
  $\widetilde{Eg}$ be the representatives of $Ef$,
  $Eg$ obtained by directly translating the
  graphical representation of $f,g$ to structured
  cospans of graphs of $\Gamma$. For faithfulness,
  it suffices to show that the existence of a
  2-arrow
  $\widetilde{Ef} \Rightarrow \widetilde{Eg}$ in
  $\ZZX$ implies that $f=g$.
	
  Observe that any 2-arrow $\alpha$ in $\ZZX$ can
  be written, not necessarily uniquely, as
  sequence
  $\alpha_1 \square \dotsm \square \alpha_n$ of
  length $n$ where each $\alpha_i$ is a basic
  $2$-cell and each box is filled in with
  `$\hcirc$', `$\vcirc$', or `$+$'. By `$\hcirc$'
  and `$\vcirc$', we mean horizontal and vertical
  composition. We will induct on sequence length.
  If
  $\alpha \colon \widetilde{Ef} \Rightarrow
  \widetilde{Eg}$ is a basic 2-arrow, then there
  is clearly a corresponding basic relation
  equating $f$ and $g$.  Suppose we have a
  sequence of length $n+1$ such that the left-most
  square is a `$+$'. When we have a 2-arrow
  $\alpha_1 + \alpha_2 \colon Ef \Rightarrow Eg$
  where $\alpha_1$ is a basic 2-arrow and
  $\alpha_2$ can be written with length $n$.  By
  fullness, we can write
  $\alpha_1 + \alpha_2 \colon Ef_1 + EF_2
  \Rightarrow Eg_1 + Eg_2$ where
  $\alpha_i \colon Ef_i \Rightarrow Eg_i$.  This
  gives that $f_i = g_i$ and the result follows.
  A similar argument handles the cases when the
  left-most operation is vertical or horizontal
  composition.
\end{proof}


\chapter{Decomposing systems}
\label{sec:structural-induction}

The idea of decomposing a whole into parts has long been
useful. It exists across so many human disciplines, be it
academic, artistic, or artisanal. A biologist decomposes
life-forms into genuses and species. A literary critic
decomposed a play into acts and scenes. A sommelier
decomposes a wine into color, viscosity, aroma, and
taste. In this chapter, as do the biologist, critic, and
sommelier, we decompose. Though for us, we decompose a
closed system into open sub-systems.

This may seem to conflict with the aim of this thesis, which
is to advance a theory of open systems. However, we still
recognize the value of closed systems. We just believe that
our ideas on open systems are useful for closed systems.

As mathematicians, we must bring rigor to our
decomposition. In this chapter, we do just that. We start
by formalizing closed systems as structured cospans with an
empty interface $ 0 \to x \gets 0 $. Then,
using the fine rewriting paradigm from Chapter
\ref{sec:fine-rewriting}, we place structured cospans into
the double category $ _L \FFFineRewrite $ as horizontal
1-arrows.  To decompose a closed system
\[
  L0 \to x \gets L0
\]
is to write an arrow as a composite of arrows
\[
  L0 \to x_1 \gets La_1 \to x_2 \gets La_2
  \dotsm La_{n-1} \to x_n \gets L0
\]
We use such decompositions to prove our main result which states that two structured cospans
\[
  L0 \to x \gets L0
  \quad \text{and} \quad
  L0 \to x' \gets L0
\]
are equivalent precisely when there is a square between
them.  We interpret this result in three ways.

\begin{enumerate}
\item It shows that the rewriting relation for a closed system
  is functorial and is characterized using squares in a double
  category.
\item A closed system decomposes into open systems, and
  simplifying each open system simplifies the composite
  closed system.  
\item Open systems provide a local perspective on the closed
  perspective via this decomposition.
\end{enumerate}

There are two main thrusts to this proof.  The first
generalizes a classification of formal graph grammars given
by Ehrig, et.~al. \cite{ehrig_graph-grammars}. This is
Theorem \ref{thm:production-same-rewrite-relation-as-discrete}.
Gadducci and Heckel proved this in the case of graphs
\cite{gadducci_ind-graph-transf}, but our result generalizes
this to structured cospans. Our proof mirrors theirs.

\section{Expressiveness of underlying discrete grammars}
\label{sec:gen-result-graph-rewriting}

As mentioned above, we want to decompose closed systems into
open systems.  We did not yet mention which open systems are
available to use. This depends on context.  That is,
whatever type of system one has, there is an appropriate
grammar stipulated by a theory that describes that
system. To illustrate, for an electrical system, a
corresponding grammar would have rules for adding resistors
in series, or adding the reciprocal of resistors in
parallel. Therefore, our starting data is a grammar
$ ( \X , P ) $---a topos $ \X $ and a set of fine rewrite
rules $ P \bydef \{ \ell_j \gets k_j \to r_j \} $---plus a
closed system $ x $ in $ \X $.  Eventually entering the
story is a topos $ \A $ of input types and an adjunction
between $ \A $ and $ \X $.  For now, however, we focus on the
set of rewrite rules $ P $.

We can prove the main result of this section, Theorem
\ref{thm:production-same-rewrite-relation-as-discrete}, by
controlling the form of the rewrite rules.  In particular,
we want the intermediary of the rules, the $ k_j $'s, to be
`discrete'. In what follows, we discuss what we mean by
`discrete' and show that the grammar obtained by
discretizing $ ( \X,P ) $ is just as expressive as
$ ( \X,P ) $, by which we mean that the induced rewriting
relations are equal. This result generalizes a
characterization of discrete \emph{graph grammars} given by
Ehrig, et.~al. \cite[Prop.~3.3]{ehrig_graph-grammars}.

Our concept of `discreteness' is borrowed from the flat
modality on a local topos. However, we avoid the lengthy
detour required to discuss the `flat modality' and a `local
topos'. The background does not add to our story, so we
point curious readers elsewhere
\cite[Ch.~C3.6]{johnstone_elephant}. By avoiding that
detour, we instead require the concept of a comonad, which
we present in Definition \ref{def:(co)monad}.

To start our discussion on discreteness, we define a
`discrete comonad'. The definition is straightforward enough,
but its purpose may seem alien at first. After the
definition, we explain its role in rewriting structured
cospans.

\begin{definition}[Discrete comonad]
  \label{def:discrete-comonad}
  A comonad on a topos is called \defn{discrete} if its
  counit is monic. We use $ \flat $ to denote a discrete comonad.  
\end{definition}

Secretly, we have been working with a discrete comonad all
along. The adjunction
\[
  \adjunction{\Set}{\RGraph}{L}{R}{4}
\]
induces the comonad $ LR $ on $ \RGraph $.  Applying $ LR $ to
a graph $ x $ returns the edgeless graph underlying $
x $, hence the term `discrete'. For example
\begin{center}
  \begin{tikzpicture}
    \begin{scope}
      \node (a) at (0,2) {$ \bullet $};
      \node (b) at (2,1) {$ \bullet $};
      \node (c) at (0,0) {$ \bullet $};
      \draw [graph] 
        (a) edge[] (b)
        (a) edge[] (c)
        (b) edge[] (c);
      \draw [rounded corners] (-1,-1) rectangle (3,3);
    \end{scope}
    \begin{scope}[shift={(6,0)}]
      \node (a) at (0,2) {$ \bullet $};
      \node (b) at (2,1) {$ \bullet $};
      \node (c) at (0,0) {$ \bullet $};
      \draw [rounded corners] (-1,-1) rectangle (3,3);
    \end{scope}
    \draw [|->] (3.5,1) to node[above]{$ LR $} (4.5,1);
  \end{tikzpicture}
\end{center}
The counit $ \epsilon_x \from LRx \to x $ of the comonad
$ LR $ includes the underlying edgeless graph $ LRx $ into
the original graph $ x $. For example
\begin{center}
  \begin{tikzpicture}
    \begin{scope}
      \node (a) at (0,2) {$ \bullet $};
      \node (b) at (2,1) {$ \bullet $};
      \node (c) at (0,0) {$ \bullet $};
      \draw [rounded corners] (-1,-1) rectangle (3,3);
    \end{scope}
    \begin{scope}[shift={(6,0)}]
      \node (a) at (0,2) {$ \bullet $};
      \node (b) at (2,1) {$ \bullet $};
      \node (c) at (0,0) {$ \bullet $};
      \draw [graph] 
        (a) edge[] (b)
        (a) edge[] (c)
        (b) edge[] (c);
      \draw [rounded corners] (-1,-1) rectangle (3,3);
    \end{scope}
    \draw [->]
      (3.5,1) to node[above]{$ \epsilon $} (4.5,1);
  \end{tikzpicture}
\end{center}

Abstractly, this inclusion is why we ask for the counit to be monic. The
property we capture with a discrete comonad comes from the
systems interpretation of the adjunctions
\[
  \adjunction{\A}{\X}{L}{R}{2}
\]
between topoi. That is, $ R $ takes a system $ x $,
identifies the largest sub-system that can serve as an
interface and turns that sub-system into an interface type
$ Rx $. Then $ L $ takes that interface type and turns it
back into a system $ LRx $. This process effectively strips
away every part of a system leaving only those parts that
can connect to the outside world. That means $ LRx $ is a
part of $ x $ or, in the parlance of category theory,
$ LRx $ is a subobject of $ x $.  Hence, we ask for a monic
counit.

How do we plan to use discrete comonads?  We use them to
control the form of our grammars.  In general, a rewrite
rule has form
\[
  \ell \gets k \to r
\]
where there are no restrictions on what $ k $ can be.
However, recall that $ k $ identifies the part of $ \ell $
that is fixed throughout the rewrite. It does not direct how
the rewrite is performed.  Therefore, we can deform it a bit
without changing the outcome of the applying the rewrite.
In particular, we can discretize it by replacing $ k $ with
$ \flat k $.  And because $ \flat $ has a monic counit, we
can insert $ \flat k $ right into the middle of the
fine rewrite rule.

\begin{definition}[Discrete grammar]
  Given a grammar $ ( \X , P ) $, define the set $ P_\flat $
  as consisting of the rules
  \[
    \ell \gets k \gets \flat k \to k \to r
  \]
  for each rule $ \ell \gets k \to r $ in $ P $. We call $
  ( X , P_\flat ) $ the \defn{discrete grammar} underlying $
  ( \X, P ) $.
\end{definition}

Discrete grammars are easier to work with than arbitrary
grammars. So when given an opportunity to work with a
discrete grammar instead of a non-discrete grammar, we
should take it. Theorem
\ref{thm:production-same-rewrite-relation-as-discrete} gives
a sufficient condition that allows us to swap $ ( \X,P ) $
for $ ( \X,P_\flat ) $ without consequence. To prove this,
however, we borrow from lattice theory which requires that
we make a brief turn to fill in some required background.

\begin{definition}[Lattice]
  A lattice is a poset $ ( S, \leq ) $ equipped with all
  finite joins $ \bigvee $ and all finite meets
  $ \bigwedge $. It follows that there is a minimal element
  and maximal element, realized as the empty meet and join
  respectively, which we denote by $ 0 $ and $ 1 $.
\end{definition}

Joins and meets are also known as suprema and infima. We are
using the definition of a lattice common in the category
theory literature. This leaves out objects that some
mathematicians might consider lattices. Below we give one
counter-example and several examples of lattices, the last
one being the most relevant.

\begin{example}[Integer Lattice]
  The integers with the usual ordering $ \leq $ do not form
  a lattice because there is no minimal or maximal element.
\end{example}

\begin{example}[Lattice of power sets]
  For any set $ S $, its powerset $ \mathcal{P}S $ is a
  poset via subset inclusion. The powerset becomes a lattice
  by taking join to be union $ a \vee b \bydef a \cup b $,
  and meet to be intersection $ a \wedge b \bydef a \cap b
  $. In general, union and intersection are defined over
  arbitrary sets, thus realizing arbitrary joins
  $ \bigvee a_\alpha $ and arbitrary meets
  $ \bigwedge a_\alpha $. 
\end{example}

Those few examples provide intuition about lattices, but the
next example is the most important lattice for us. It is the
mechanism by which the power set is generalized into topos
theory. It is called the subobject lattice.

\begin{example}[Subobject lattice]
  Let $ \T $ be a topos and $ t $ be an object. There is a
  lattice $ \Sub (t) $ called the subobject lattice of
  $ t $.  The elements of $ \Sub (t) $ are called subobjects.
  They are isomorphism classes of monomorphisms into $ t
  $. Here, two monomorphisms $ f,g $ into $ t $ are isomorphic if
  there is a commuting diagram
  \begin{center}
    \begin{tikzpicture}
      \node (a) at (-1,2) {$ a $};
      \node (b) at (1,2) {$ b $};
      \node (t) at (0,0) {$ t $};
      \draw [cd] 
      (a) edge[] node[left]{$ f $} (t)
      (b) edge[] node[right]{$ g $} (t)
      (a) edge[] node[above]{$ \iso $} (b); 
    \end{tikzpicture}
  \end{center}
  The order on $ \Sub (t) $ is given by $ f \leq g $ if
  $ f $ factors through $ g $, meaning there is an arrow
  $ h \from a \to b $ such that $ f = gh $. Note that $ h $
  is necessarily monic. The meet operation in $ \Sub (t) $ is given by
  pullback
  \begin{center}
    \begin{tikzpicture}
      \node (ab) at (0,2) {$ a \vee b $};
      \node (a)  at (0,0) {$ a $};
      \node (b)  at (2,2) {$ b $};
      \node (t)  at (2,0) {$ t $};
      \draw [cd] 
      (ab) edge[] (a)
      (ab) edge[] (b)
      (a)  edge[] (t)
      (b)  edge[] (t);
      \draw [->,dashed] (ab) edge (t) ;
      \draw (0.3,1.6) -- (0.4,1.6) -- (0.4,1.7);
    \end{tikzpicture}
  \end{center}
  and join is given by pushout over the meet
  \begin{center}
    \begin{tikzpicture}
      \node (ab) at (0,2) {$ a \vee b $};
      \node (a)  at (0,0) {$ a $};
      \node (b)  at (2,2) {$ b $};
      \node (t)  at (2,0) {$ a \wedge b $};
      \node (tt) at (4,-2) {$ t $};
      \draw [cd] 
      (ab) edge[]           (a)
      (ab) edge[]           (b)
      (a)  edge[]           (t)
      (b)  edge[]           (t)
      (a)  edge[bend right] (tt)
      (b)  edge[bend left]  (tt);
      \draw [->,dashed] (t) edge (tt) ;
      \draw (1.7,0.4) -- (1.6,0.4) -- (1.6,0.3);
    \end{tikzpicture}
  \end{center}
\end{example}

We use subobject lattices to characterize which grammars are
as expressive as their underlying discrete grammars. To do
this, we require subobject lattices with arbitrary meets.  The
powerset lattice mentioned above has this property, but
when do subobject lattices have this property?  Here are
several sufficient conditions, starting with a well-known
result coming from the domain of order theory.

\begin{proposition} \label{thm:lattice-alljoins-allmeets}
  Any lattice that has all joins also has all meets.
\end{proposition}

\begin{proof}
  Consider a subset $ S $ of a lattice. Define the meet of $
  S $ to be the join of the set of all lower bounds of $ S $.
\end{proof}

\begin{proposition} \label{thm:subob-arbitrary-meets}
  Consider a topos $ \T $ and object $ t $.  The subobject
  lattice $ \Sub (t) $ has arbitrary meets when the over
  category $ T \downarrow t $ has all products.
\end{proposition}

\begin{proof}
  Because $ T \downarrow t $ is a topos, it has
  equalizers. Thus giving it all products ensures the
  existence of all limits, hence meets. 
\end{proof}

\begin{corollary}
  Consider a topos $ \T $ and object $ t $.  The subobject
  lattice $ \Sub (t) $ has arbitrary meets when the over
  category $ T \downarrow t $ has all coproducts.
\end{corollary}

\begin{proof}
  Combine Propositions \ref{thm:lattice-alljoins-allmeets}
  and \ref{thm:subob-arbitrary-meets}.
\end{proof}

\begin{corollary}
  Consider a presheaf category $ \Set^{ \C^{\op} } $ on a
  small category $ \C $. For any presheaf $ x $,
  $ \Sub (x) $ has all meets.
\end{corollary}

\begin{proof}
  The category $ \Set^{ \C^{\op} } \downarrow x $ of
  presheaves over $ x $ is again a presheaf category by
  Theorem \ref{thm:presheaf-slice-is-presheaf} so has all
  products.
\end{proof}

At last, we combine the discrete comonad, the discrete
grammar, and the complete subobject lattice into a result on
the expressiveness on discrete grammars.

\begin{theorem}
  \label{thm:production-same-rewrite-relation-as-discrete}  
  Let $ \T $ be a topos and $ \flat \from \T \to \T $ be a
  discrete comonad.  Let $ ( \T , P ) $ be a grammar such
  that for every rule $ \ell \gets k \to r $ in $ P $, the
  subobject lattice $ \Sub (k) $ has all meets. Then the
  rewriting relation for $ ( \T , P ) $ equals the
  rewriting relation for the underlying discrete grammar
  $ ( \T,P_\flat ) $.
\end{theorem}

\begin{proof}
  Suppose that $ ( \T,P ) $ induces $ \dderiv{g}{h} $. That
  means there exists a rule $ \spn{\ell}{k}{r} $ in $ P $
  and a derivation
  \begin{equation} \label{eq:prod-rewrite-1}
  \begin{tikzpicture}
    \node (1t) at (0,2) {$ \ell $};
    \node (2t) at (2,2) {$ k $};
    \node (3t) at (4,2) {$ r $};
    \node (1b) at (0,0) {$ g $};
    \node (2b) at (2,0) {$ d $};
    \node (3b) at (4,0) {$ h $};
    \draw [cd]
      (2t) edge (1t)
      (2t) edge (3t)
      (2b) edge (1b)
      (2b) edge (3b)
      (1t) edge (1b)
      (2t) edge (2b)
      (3t) edge (3b);
      \draw (0.3,0.4) -- (0.4,0.4) -- (0.4,0.3);
      \draw (3.7,0.4) -- (3.6,0.4) -- (3.6,0.3);
    \end{tikzpicture}
  \end{equation}
  we can achieve that same derivation using rules in
  $ P_\flat $. This requires we build a pushout complement
  $ w $ of the diagram
  \begin{center}
    \begin{tikzpicture}
      \node (k) at (0,2) {$ k $};
      \node (bk) at (2,2) {$ \flat k $};
      \node (d) at (0,0) {$ d $};
      \draw [cd]
        (bk) edge[] node[above]{$ \epsilon $} (k)
        (k) edge[] (d); 
    \end{tikzpicture}
  \end{center}
  Define
  \[
    w \coloneqq
    \bigwedge \{ z \colon z \vee k = d \} \vee
    \flat k,
  \]
  This comes with inclusions $ \flat k \to w $ and
  $ w \to d $. This $ w $ exists because $ \Sub (k) $ has
  all meets.  Note that $ w \vee k = d $ and
  $ w \wedge k = \flat k $ which means that
  \begin{center}
    \begin{tikzpicture}
      \node (k) at (0,2) {$ k $};
      \node (bk) at (2,2) {$ \flat k $};
      \node (d) at (0,0) {$ d $};
      \node (w) at (2,0) {$ w $};
      \draw [cd] 
      (bk) edge[] (k)
      (bk) edge[] (w)
      (k) edge[] (d)
      (w) edge[] (d);
      \draw (0.3,0.4) -- (0.4,0.4) -- (0.4,0.3);
    \end{tikzpicture}
  \end{center}
  is a pushout. It follows that there is a derivation
  \begin{equation} \label{eq:prod-rewrite-2}
  \begin{tikzpicture}
    \node (01) at (0,2) {$ \ell $};
    \node (11) at (2,2) {$ k $};
    \node (21) at (4,2) {$ \flat k $};
    \node (31) at (6,2) {$ k $};
    \node (41) at (8,2) {$ r $};
    \node (00) at (0,0) {$ g $};
    \node (10) at (2,0) {$ d $};
    \node (20) at (4,0) {$ w $};
    \node (30) at (6,0) {$ d $};
    \node (40) at (8,0) {$ h $};
    \draw [cd] (11) to node [] {\scriptsize{$  $}} (01);
    \draw [cd] (21) to node [] {\scriptsize{$  $}} (11);
    \draw [cd] (21) to node [] {\scriptsize{$  $}} (31);
    \draw [cd] (31) to node [] {\scriptsize{$  $}} (41);
    \draw [cd] (10) to node [] {\scriptsize{$  $}} (00);
    \draw [cd] (20) to node [] {\scriptsize{$  $}} (10);
    \draw [cd] (20) to node [] {\scriptsize{$  $}} (30);
    \draw [cd] (30) to node [] {\scriptsize{$  $}} (40);
    \draw [cd] (01) to node [] {\scriptsize{$  $}} (00);
    \draw [cd] (11) to node [] {\scriptsize{$  $}} (10);
    \draw [cd] (21) to node [] {\scriptsize{$  $}} (20);
    \draw [cd] (31) to node [] {\scriptsize{$  $}} (30);
    \draw [cd] (41) to node [] {\scriptsize{$  $}} (40);
    \draw (0.3,0.4) -- (0.4,0.4) -- (0.4,0.3);
    \draw (2.3,0.4) -- (2.4,0.4) -- (2.4,0.3);
    \draw (5.7,0.4) -- (5.6,0.4) -- (5.6,0.3);
    \draw (7.7,0.4) -- (7.6,0.4) -- (7.6,0.3);
  \end{tikzpicture}
\end{equation}
with respect to $ P_\flat $ because, the top row is a rule in
$ P_\flat $.  Therefore, $ \dderiv{g}{h} $ via $ P $ in
Diagram \eqref{eq:prod-rewrite-1} implies that
$ \deriv{g}{h} $ via $ P_\flat $ as shown in Diagram
\eqref{eq:prod-rewrite-2}.

For the other direction, suppose $ \dderiv{g}{h} $ via $
P_\flat $, giving a derivation
\begin{equation}
  \label{eq:prod-rewrite-3}
  \begin{tikzpicture}
    \node (1t) at (0,2) {$ \ell $};
    \node (2t) at (2,2) {$ \flat k $};
    \node (3t) at (4,2) {$ r $};
    \node (1b) at (0,0) {$ g $};
    \node (2b) at (2,0) {$ d $};
    \node (3b) at (4,0) {$ h $};
    \draw [cd]
      (2t) edge                          (1t)
      (2t) edge                          (3t)
      (2b) edge node[below]{$ \psi $}    (1b)
      (2b) edge                          (3b)
      (1t) edge node[left]{$ m $}        (1b)
      (2t) edge node[left]{$ \theta $}   (2b)
      (3t) edge node[right]{$ m' $}      (3b);
      \draw (0.3,0.4) -- (0.4,0.4) -- (0.4,0.3);
      \draw (3.7,0.4) -- (3.6,0.4) -- (3.6,0.3);
  \end{tikzpicture}
\end{equation}
By construction of $ P_\flat $, the rule
$ \spn{\ell}{\flat k}{r} $ in $ P_\flat $ was induced from a
rule
\[
  \ell \xgets{\tau} k \to r
\]
in $ P $, meaning that the map $ \flat k \to \ell $ factors
through $ \tau $. Next, define $ d' $ to be the pushout of the diagram
\begin{center}
  \begin{tikzpicture}
    \node (flatk) at (0,2) {$ \flat k $};
    \node (k)     at (2,2) {$ k $};
    \node (d)     at (0,0) {$ d $};
    \node (d')    at (2,0) {$ d' $};
    \draw [cd] 
    (flatk) edge[] node[above]{$ \epsilon $}       (k)
    (flatk) edge[] node[left]{$ \theta $}          (d)
    (d)     edge[] node[below]{$ \hat{\epsilon} $} (d')
    (k)     edge[] node[right]{$ \hat{\theta} $}   (d');
    \draw (1.7,0.4) -- (1.6,0.4) -- (1.6,0.3);
  \end{tikzpicture}
\end{center}
By invoking the universal property of this pushout with the
maps
\[
  \psi \from d \to g
  \quad \text{and} \quad
  m \tau \from k \to \ell \to g,
\]
we get a canonical map $ d' \to g $ that we can fit into a
commuting diagram
\begin{center}
  \begin{tikzpicture}
    \node (l) at (0,2) {$ \ell $};
    \node (g) at (0,0) {$ g $};
    \node (k) at (2,2) {$ k $};
    \node (d') at (2,0) {$ d' $};
    \node (flatk) at (1,3) {$ \flat k $};
    \node (d) at (1,1) {$ d $};
    \draw [cd] 
      (flatk) edge[] (l)
      (flatk) edge[] node[right]{$ \epsilon $} (k)
      (flatk) edge[] node[left,pos=0.7]{$ \theta $} (d)
      (d) edge[] node[below,pos=0.3]{$ \psi $}     (g)
      (d) edge[] node[below,pos=0.3]{$ \hat{\epsilon} $} (d')
      (l) edge[] node[left]{$ m $} (g)
      (k) edge[] node[right]{$ \hat{\theta} $} (d')
      (d') edge[] (g);
    \draw [white,line width=0.5em] (k) -- (l);
    \draw [cd] (k) to node[above,pos=0.3]{$ \tau $} (l);
    \draw (0.2,0.6) -- (0.3,0.7) -- (0.3,0.6);
    \draw (1.8,0.6) -- (1.7,0.7) -- (1.7,0.6);
  \end{tikzpicture}
\end{center}
whose back faces are pushouts. Using a standard diagram
chasing argument, we can show that the front face is also a
pushout.  Similarly, the square
\begin{center}
  \begin{tikzpicture}
    \node (k) at (0,2) {$ k $};
    \node (r) at (2,2) {$ r $};
    \node (d') at (0,0) {$ d' $};
    \node (h) at (2,0) {$ h $};
    \draw [cd] 
    (k) edge[] (r)
    (k) edge[] (d')
    (d') edge[] (h)
    (r) edge[] (h);
    \draw (1.7,0.4) -- (1.6,0.4) -- (1.6,0.3);
  \end{tikzpicture}
\end{center}
is a pushout.  Sticking these two pushouts together
\begin{center}
  \begin{tikzpicture}
    \node (1t) at (0,2) {$ \ell $};
    \node (2t) at (2,2) {$ k $};
    \node (3t) at (4,2) {$ r $};
    \node (1b) at (0,0) {$ g $};
    \node (2b) at (2,0) {$ d' $};
    \node (3b) at (4,0) {$ h $};
    \draw [cd]
      (2t) edge                     (1t)
      (2t) edge                     (3t)
      (2b) edge                     (1b)
      (2b) edge                     (3b)
      (1t) edge node[left]{$ m $}   (1b)
      (2t) edge node[left]{$ f $}   (2b)
      (3t) edge node[right]{$ m' $} (3b);
    \draw (0.3,0.4) -- (0.4,0.4) -- (0.4,0.3);
    \draw (3.7,0.4) -- (3.6,0.4) -- (3.6,0.3);
  \end{tikzpicture}
\end{center}
shows that $ \dderiv{g}{h} $ arises from $ P $.

Because the relation $ \dderiv{}{} $ is the same for
$ P $ and $ P_\flat $, it follows that $ \deriv{}{} $ is
also the same as claimed.
\end{proof}

\section{Rewriting structured cospans}
\label{sec:Rewriting-StrCsp}

Equipped with knowledge about when grammars and their
underlying discrete grammars generate the same rewriting
relation, we continue towards goal of decomposing closed
systems. First, we revisit Section \ref{sec:rewriting-topoi}
to get some facts about grammars. We then obtain the
language associated to a grammar in a functorial
way. Finally, we show how to decompose into open subsystems a given system
equipped with a grammar. 

Recall the category $ \Gram $.  The objects of $ \Gram $ are
pairs $ ( \T , P ) $ where $ \T $ is a topos and $ P $ is a
set of rewrite rules in $ \T $.  The arrows
$ (\T , P) \to ( \T' , P' )$ of $ \Gram $ are
rule-preserving functors $ \T \to \T' $.  Our interest now lies
in the full subcategory of structured cospan grammars
$ \StrCspGram $ whose objects are the grammars of form
$ ( _L \StrCsp , P ) $ where $ P $ consists of fine rewrites
of structured cospans, meaning they have the form
 \begin{center}
   \begin{tikzpicture}
    \begin{scope}
        \node (1) at (0,4) {\( La \)};
        \node (2) at (2,4) {\( x \)};
        \node (3) at (4,4) {\( La' \)};
        \node (4) at (0,2) {\( Lb \)};
        \node (5) at (2,2) {\( y \)};
        \node (6) at (4,2) {\( Lb' \)};
        \node (7) at (0,0) {\( Lc \)};
        \node (8) at (2,0) {\( z \)};
        \node (9) at (4,0) {\( Lc' \)};
        \draw [cd] (1) to (2);
        \draw [cd] (3) to (2);
        \draw [cd] (4) to (5);
        \draw [cd] (6) to (5);
        \draw [cd] (7) to (8);
        \draw [cd] (9) to (8);
        \draw [cd] (4) to node[left]{\( \iso \)} (1);
        \draw [cd] (4) to node[left]{\( \iso \)} (7);
        \draw [cd,>->] (5) to (2);
        \draw [cd,>->] (5) to (8);
        \draw [cd] (6) to node[right]{\( \iso \)} (3);
        \draw [cd] (6) to node[right]{\( \iso \)} (9);
    \end{scope}
  \end{tikzpicture}
\end{center}
and the left adjoint $ L $ has a monic counit.

It is on this category $ \StrCspGram $ that we define a
functor encoding the rewrite relation to each grammar. We
denote this functor
\[
  \Lang \from \StrCspGram \to \DblCat
\]
where $ \Lang $ is short for `language'. This is an
appropriate term as this functor provides \emph{(i)} the
terms formed by connecting together open systems (instead
of, in linguistics, concatenating units of syntax) and
\emph{(ii)} the rules governing how to interchange open
systems (instead of parts of speech). To help visualize
this, we sketch a simple example.

\begin{example}

  Start with the, by now familiar, adjunction
  \[
    \adjunction{\Set}{\RGraph}{L}{R}{4}
  \]
  For this $ L $, $ _L \StrCsp $ is the category of open
  graphs.  Make a grammar from $ _L \StrCsp $ by defining a
  $ P $ to have the single rule
  \begin{center}
    \begin{tikzpicture}
    \begin{scope}[shift={(0,0)}] 
      \node () at (0,0) {$ \bullet $};
      \draw [rounded corners] (-1,-1) rectangle (1,1);
      \node (00r) at (1.1,0) {};
      \node (00b) at (0,-1.1) {};  
    \end{scope}
    \begin{scope}[shift={(3,0)}] 
      \node (a) at (0,0) {$ \bullet $};
      \draw [graph]
        (a) edge[loop above] (a);
      \draw [rounded corners] (-1,-1) rectangle (1,1);
      \node (01l) at (-1.1,0) {};
      \node (01r) at (1.1,0)  {};
      \node (01b) at (0,-1.1) {};  
    \end{scope}
    \begin{scope}[shift={(6,0)}] 
      \node () at (0,0) {$ \bullet $};
      \draw [rounded corners] (-1,-1) rectangle (1,1);
      \node (02l) at (-1.1,0) {};
      \node (02b) at (0,-1.1) {};
    \end{scope}
    \begin{scope}[shift={(0,-3)}] 
      \node () at (0,0) {$ \bullet $};
      \draw [rounded corners] (-1,-1) rectangle (1,1);
      \node (10t) at (0,1.1)   {};
      \node (10r) at (1.1,0) {};
      \node (10b) at (0,-1.1)  {};  
    \end{scope}
    \begin{scope}[shift={(3,-3)}] 
      \node (a) at (0,0) {$ \bullet $};
      \draw [rounded corners] (-1,-1) rectangle (1,1);
      \node (11l) at (-1.1,0) {};
      \node (11r) at (1.1,0)  {};
      \node (11t) at (0,1.1)  {};
      \node (11b) at (0,-1.1) {};  
    \end{scope}
    \begin{scope}[shift={(6,-3)}] 
      \node () at (0,0) {$ \bullet $};
      \draw [rounded corners] (-1,-1) rectangle (1,1);
      \node (12l) at (-1.1,0) {};
      \node (12t) at (0,1.1)    {};  
      \node (12b) at (0,-1.1)   {};
    \end{scope}
    \begin{scope}[shift={(0,-6)}] 
      \node () at (0,0) {$ \bullet $};
      \draw [rounded corners] (-1,-1) rectangle (1,1);
      \node (20t) at (0,1.1)   {};
      \node (20r) at (1.1,0) {};
      \node (20b) at (0,-1.1)  {};  
    \end{scope}
    \begin{scope}[shift={(3,-6)}] 
      \node (a) at (0,0) {$ \bullet $};
      \draw [rounded corners] (-1,-1) rectangle (1,1);
      \node (21l) at (-1.1,0) {};
      \node (21r) at (1.1,0)  {};
      \node (21t) at (0,1.1)  {};
      \node (21b) at (0,-1.1) {};  
    \end{scope}
    \begin{scope}[shift={(6,-6)}] 
      \node () at (0,0) {$ \bullet $};
      \draw [rounded corners]
        (-1,-1) rectangle (1,1);
      \node (22l) at (-1.1,0) {};
      \node (22r) at (1.1,0)  {};
      \node (22t) at (0,1.1)  {};
      \node (22b) at (0,-1.1) {};
    \end{scope}
    \path[cd]
      (00r) edge (01l)
      (02l) edge (01r)
      (10r) edge (11l)
      (12l) edge (11r)
      (20r) edge (21l)
      (22l) edge (21r)
      (00b) edge (10t)
      (20t) edge (10b)
      (01b) edge (11t)
      (21t) edge (11b)
      (02b) edge (12t)
      (22t) edge (12b);
    \end{tikzpicture}
  \end{center}
  The language associated to this grammar consists of all
  open graphs. The rewrite relation says $ \deriv{g}{h} $ if
  we obtain $ h $ be removing loops from $ g $. We
  illustrate this with the following square in the double
  category $ \Lang ( _L \StrCsp , P ) $.
  \begin{center}
    \begin{tikzpicture}
    \begin{scope}[shift={(0,0)}] 
      \node (a) at (0,-1) {$ \bullet $};
      \node (b) at (0,1) {$ \bullet $};
      \draw [rounded corners] (-2,-2) rectangle (2,2);
      \node (00r) at (2.1,0) {};
      \node (00b) at (0,-2.1) {};  
    \end{scope}
    \begin{scope}[shift={(5,0)}] 
      \node (a) at (-1,-1) {$ \bullet $};
      \node (b) at (-1,1) {$ \bullet $};
      \node (c) at (0,0) {$ \bullet $};
      \node (d) at (1,0) {$ \bullet $};
      \draw [graph]
        (b) edge[loop above] (b)
        (c) edge[loop above] (c)
        (a) edge[] (c)
        (b) edge[] (c)
        (c) edge[] (d); 
      \draw [rounded corners] (-2,-2) rectangle (2,2);
      \node (01l) at (-2.1,0) {};
      \node (01r) at (2.1,0)  {};
      \node (01b) at (0,-2.1) {};  
    \end{scope}
    \begin{scope}[shift={(10,0)}] 
      \node () at (0,0) {$ \bullet $};
      \draw [rounded corners] (-2,-2) rectangle (2,2);
      \node (02l) at (-2.1,0) {};
      \node (02b) at (0,-2.1) {};
    \end{scope}
    \begin{scope}[shift={(0,-5)}] 
      \node () at (0,-1) {$ \bullet $};
      \node () at (0,1) {$ \bullet $};
      \draw [rounded corners] (-2,-2) rectangle (2,2);
      \node (10t) at (0,2.1)   {};
      \node (10r) at (2.1,0) {};
      \node (10b) at (0,-2.1)  {};  
    \end{scope}
    \begin{scope}[shift={(5,-5)}] 
      \node (a) at (-1,-1) {$ \bullet $};
      \node (b) at (-1,1) {$ \bullet $};
      \node (d) at (1,0) {$ \bullet $};
      \draw [rounded corners] (-2,-2) rectangle (2,2);
      \node (11l) at (-2.1,0) {};
      \node (11r) at (2.1,0)  {};
      \node (11t) at (0,2.1)  {};
      \node (11b) at (0,-2.1) {};  
    \end{scope}
    \begin{scope}[shift={(10,-5)}] 
      \node () at (0,0) {$ \bullet $};
      \draw [rounded corners] (-2,-2) rectangle (2,2);
      \node (12l) at (-2.1,0) {};
      \node (12t) at (0,2.1)    {};  
      \node (12b) at (0,-2.1)   {};
    \end{scope}
    \begin{scope}[shift={(0,-10)}] 
      \node () at (0,-1) {$ \bullet $};
      \node () at (0,1) {$ \bullet $};
      \draw [rounded corners] (-2,-2) rectangle (2,2);
      \node (20t) at (0,2.1)   {};
      \node (20r) at (2.1,0) {};
      \node (20b) at (0,-2.1)  {};  
    \end{scope}
    \begin{scope}[shift={(5,-10)}] 
      \node (a) at (-1,-1) {$ \bullet $};
      \node (b) at (-1,1) {$ \bullet $};
      \node (c) at (0,0) {$ \bullet $};
      \node (d) at (1,0) {$ \bullet $};
      \draw [graph]
        (a) edge[] (c)
        (b) edge[] (c)
        (c) edge[] (d); 
      \draw [rounded corners] (-2,-2) rectangle (2,2);
      \node (21l) at (-2.1,0) {};
      \node (21r) at (2.1,0)  {};
      \node (21t) at (0,2.1)  {};
      \node (21b) at (0,-2.1) {};  
    \end{scope}
    \begin{scope}[shift={(10,-10)}] 
      \node () at (0,0) {$ \bullet $};
      \draw [rounded corners] (-2,-2) rectangle (2,2);
      \node (22l) at (-2.1,0) {};
      \node (22r) at (2.1,0)  {};
      \node (22t) at (0,2.1)  {};
      \node (22b) at (0,-2.1) {};
    \end{scope}
    \path[cd]
      (00r) edge (01l)
      (02l) edge (01r)
      (10r) edge (11l)
      (12l) edge (11r)
      (20r) edge (21l)
      (22l) edge (21r)
      (00b) edge (10t)
      (20t) edge (10b)
      (01b) edge (11t)
      (21t) edge (11b)
      (02b) edge (12t)
      (22t) edge (12b);
    \end{tikzpicture}
  \end{center}
\end{example}

To actually construct $ \Lang $, we use functors
$ D \from \StrCspGram \to \StrCspGram $ and
$ S \from \StrCspGram \to \DblCat $. Roughly, $ D $ sends a
grammar $ ( _L \StrCsp , P ) $ to all of the rewrite rules
derived from $ P $ and $ S $ generates a double category on
the squares obtained from the rewrite rules of a grammar
$ ( _L \StrCsp , P ) $. In this way, we get the language of
a grammar as a double category where the squares are the
rewrite rules.  The next lemma defines $ D $ and gives some
of its properties.

\begin{lemma}
  There is an idempotent functor
  $ D \from \StrCspGram \to \StrCspGram $ defined as
  follows. On objects define $ D ( _{L}\StrCsp , P ) $ to be
  the grammar $ ( _{L} \StrCsp , P_D) $, where $ P_D $
  consists of all rules $ \spn{g}{h}{d} $ witnessing the
  relation $ \dderiv{g}{h} $ with respect to
  $ ( _{L}\StrCsp , P ) $. On arrows, define
  $ DF \from D( _{L}\StrCsp , P ) \to D( _{L'}\StrCsp , Q )
  $ to be $ F $.  Moreover, the identity on $ \StrCspGram $
  is a subfunctor of $ D $.
\end{lemma}

\begin{proof}
  That $ D ( _{L}\StrCsp , P ) $ actually gives a grammar
  follows from the fact that pushouts respect monics in a
  topos \cite[Lem.~12]{lack-sobo_adhesive-cats}.
  
  To show that $ \D $ is idempotent, we show that for any
  grammar $ ( _{L}\StrCsp , P ) $, we have
  $ D ( _{L}\StrCsp , P ) = DD ( _{L}\StrCsp , P ) $.  Rules
  in $ DD ( _{L}\StrCsp , P ) $ appear in the bottom row of a
  double pushout diagram whose top row is a rule in
  $ D ( _{L}\StrCsp , P ) $, which in turn is the bottom row
  of a double pushout diagram whose top row is in
  $ ( _{L}\StrCsp , P ) $. Thus, a rule in
  $ DD ( _{L}\StrCsp , P ) $ is the bottom row of a double
  pushout diagram whose top row is in
  $ ( _{L}\StrCsp , P ) $. See Figure \ref{fig:idempotentD}.

  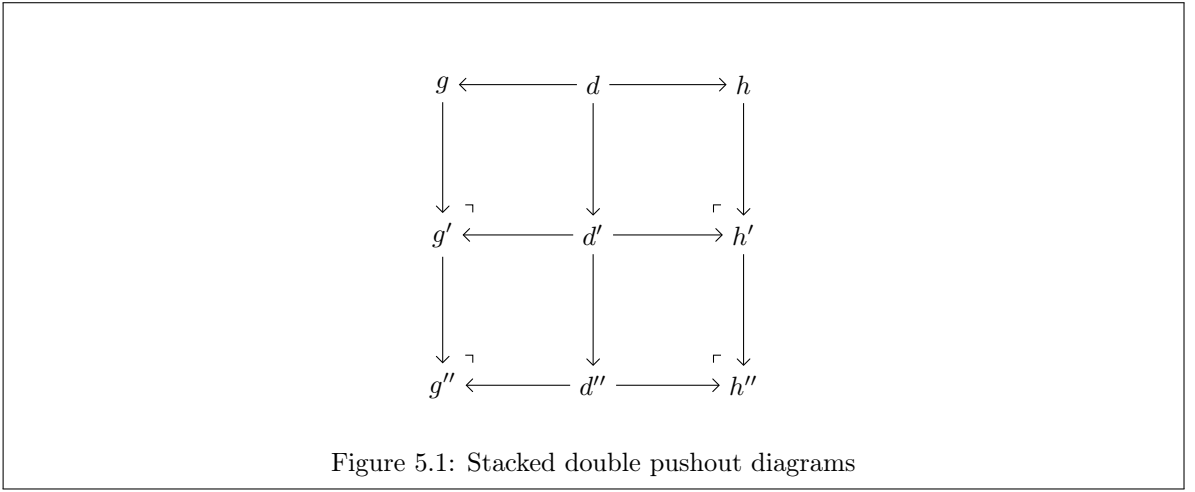
\begin{figure}[h]
    \centering
    \fbox{
    \begin{minipage}{\linewidth}
    \vspace{2em}    
    \[
    \begin{tikzpicture}
      \node (1) at (0,4) {$ g $};
      \node (2) at (2,4) {$ d $};
      \node (3) at (4,4) {$ h $};
      \node (4) at (0,2) {$ g' $};
      \node (5) at (2,2) {$ d' $};
      \node (6) at (4,2) {$ h' $};
      \node (7) at (0,0) {$ g'' $};
      \node (8) at (2,0) {$ d'' $};
      \node (9) at (4,0) {$ h'' $};
      \draw [cd] (2) to node [] {\scriptsize{$  $}} (1);
      \draw [cd] (2) to node [] {\scriptsize{$  $}} (3);
      \draw [cd] (5) to node [] {\scriptsize{$  $}} (4);
      \draw [cd] (5) to node [] {\scriptsize{$  $}} (6);
      \draw [cd] (8) to node [] {\scriptsize{$  $}} (7);
      \draw [cd] (8) to node [] {\scriptsize{$  $}} (9);
      \draw [cd] (1) to node [] {\scriptsize{$  $}} (4);
      \draw [cd] (2) to node [] {\scriptsize{$  $}} (5);
      \draw [cd] (3) to node [] {\scriptsize{$  $}} (6);
      \draw [cd] (4) to node [] {\scriptsize{$  $}} (7);
      \draw [cd] (5) to node [] {\scriptsize{$  $}} (8);
      \draw [cd] (6) to node [] {\scriptsize{$  $}} (9);
      \draw (0.3,0.4) -- (0.4,0.4) -- (0.4,0.3);
      \draw (3.7,0.4) -- (3.6,0.4) -- (3.6,0.3);   
      \draw (0.3,2.4) -- (0.4,2.4) -- (0.4,2.3);
      \draw (3.7,2.4) -- (3.6,2.4) -- (3.6,2.3);   
    \end{tikzpicture}
    \]
    \caption{Stacked double pushout diagrams}
    \label{fig:idempotentD}
  \end{minipage}
}
\end{figure}

    The identity is a subfunctor of $ D $ because
    $ \dderiv{\ell}{r} $ for any production
    $ \spn{\ell}{k}{r} $ in $ ( _{L}\StrCsp , P ) $ via a
    triple of identity arrows. Hence there is a monomorphism
    \[
      ( _L \StrCsp , P ) \to
      D ( _L \StrCsp , P )
    \]
    induced from the identity functor on $ _L\StrCsp $.
\end{proof}

In this lemma, we have created a functor $ D $ that sends a
grammar to a new grammar consisting of
all derived rules.  That $ D $ is idempotent means that all
rules derived from $ P $ can be derived directly; multiple
applications of $ D $ are unnecessary.  That the identity is
a subfunctor of $ D $ means that set of the derived rules
$ P_D $ contains the set of initial rules $ P $.

The next stage in defining $ \Lang $ is to define
$ S \from \StrCspGram \to \DblCat $. On objects, let
$ S ( _L \StrCsp , P ) $ be the sub-double category of
$ _L \SSStrCsp $ generated by the rules in $ P $ considered
as squares.  On arrows, $ S $ sends
\[
  F \from ( _{L}\StrCsp , P ) \to ( _{L'} \StrCsp , P' )
\]
to the double functor defined that extends the mapping
between the generators of $ S ( _{L}\StrCsp , P ) $ and
$ S ( _{L'}\StrCsp , P' ) $.  This preserves composition because
$ F $ preserves pullbacks and pushouts. 

\begin{definition}(Language of a grammar)
  The \defn{language functor} is defined to be
  $ \Lang \coloneqq SD $. 
\end{definition}

To witness the rewriting relation on a closed system as a
square in a double category, we require this next lemma
that formalizes the analogy between rewriting the disjoint
union of systems and tensoring squares.

\begin{lemma} \label{thm:rewrite-rel-is-additive}
  If $ \deriv{x}{y} $ and $ \deriv{x'}{y'} $, then
  $ \deriv{x+x'}{y+y'} $
\end{lemma}

\begin{proof}
  If the derivation $ \deriv{x}{y} $ comes from a string of
  double pushout diagrams
  \begin{center}
    \begin{tikzpicture}[scale=0.8]
      \node (1t) at (0,2) {$ \ell_1 $};
      \node (2t) at (2,2) {$ k_1 $};
      \node (3t) at (4,2) {$ r_1 $};
      \node (4t) at (6,2) {$ \ell_2 $};
      \node (5t) at (8,2) {$ k_2 $};
      \node (6t) at (10,2) {$ r_2 $};
      \node (7t) at (11,2) {$ \ell_n $};
      \node (8t) at (13,2) {$ k_n $};
      \node (9t) at (15,2) {$ r_n $};
      \node (1b) at (0,0) {$ x $};
      \node (2b) at (2,0) {$ d_1 $};
      \node (3b) at (5,0) {$ w_1 $};
      \node (4b) at (8,0) {$ d_2 $};
      \node (5b) at (10,0) {$ w_2 $};
      \node (6b) at (11,0) {$ w_{n-1} $};
      \node (7b) at (13,0) {$ d_n $};
      \node (8b) at (15,0) {$ y $};
      \draw [cd] (2t) to node [] {\scriptsize{$  $}} (1t);
      \draw [cd] (2t) to node [] {\scriptsize{$  $}} (3t);
      \draw [cd] (5t) to node [] {\scriptsize{$  $}} (4t);
      \draw [cd] (5t) to node [] {\scriptsize{$  $}} (6t);
      \draw [cd] (8t) to node [] {\scriptsize{$  $}} (7t);
      \draw [cd] (8t) to node [] {\scriptsize{$  $}} (9t);
      \draw [cd] (2b) to node [] {\scriptsize{$  $}} (1b);
      \draw [cd] (2b) to node [] {\scriptsize{$  $}} (3b);
      \draw [cd] (4b) to node [] {\scriptsize{$  $}} (3b);
      \draw [cd] (4b) to node [] {\scriptsize{$  $}} (5b);
      \draw [cd] (7b) to node [] {\scriptsize{$  $}} (6b);
      \draw [cd] (7b) to node [] {\scriptsize{$  $}} (8b);
      \draw [cd] (1t) to node [] {\scriptsize{$  $}} (1b);
      \draw [cd] (2t) to node [] {\scriptsize{$  $}} (2b);
      \draw [cd] (3t) to node [] {\scriptsize{$  $}} (3b);
      \draw [cd] (4t) to node [] {\scriptsize{$  $}} (3b);
      \draw [cd] (5t) to node [] {\scriptsize{$  $}} (4b);
      \draw [cd] (6t) to node [] {\scriptsize{$  $}} (5b);
      \draw [cd] (7t) to node [] {\scriptsize{$  $}} (6b);
      \draw [cd] (8t) to node [] {\scriptsize{$  $}} (7b);
      \draw [cd] (9t) to node [] {\scriptsize{$  $}} (8b);
      \node () at (10.5,1) {$ \dotsm $};
      \draw (0.3,0.4) -- (0.4,0.4) -- (0.4,0.3);
      \draw (4.7,0.4) -- (4.6,0.4) -- (4.6,0.3);
      \draw (5.3,0.4) -- (5.4,0.4) -- (5.4,0.3);
      \draw (9.7,0.4) -- (9.6,0.4) -- (9.6,0.3);
      \draw (11.3,0.4) -- (11.4,0.4) -- (11.4,0.3);
      \draw (14.7,0.4) -- (14.6,0.4) -- (14.6,0.3);
    \end{tikzpicture}
  \end{center}
  and the derivation $ \deriv{x'}{y'} $ comes from a string
  of double pushout diagrams
    \[
    \begin{tikzpicture}[scale=0.8]
      \node (1t) at (0,2) {$ \ell'_1 $};
      \node (2t) at (2,2) {$ k'_1 $};
      \node (3t) at (4,2) {$ r'_1 $};
      \node (4t) at (6,2) {$ \ell'_2 $};
      \node (5t) at (8,2) {$ k'_2 $};
      \node (6t) at (10,2) {$ r'_2 $};
      \node (7t) at (11,2) {$ \ell'_m $};
      \node (8t) at (13,2) {$ k'_m $};
      \node (9t) at (15,2) {$ r'_m $};
      \node (1b) at (0,0) {$ x' $};
      \node (2b) at (2,0) {$ d'_1 $};
      \node (3b) at (5,0) {$ w'_1 $};
      \node (4b) at (8,0) {$ d'_2 $};
      \node (5b) at (10,0) {$ w'_2 $};
      \node (6b) at (11,0) {$ w'_{m-1} $};
      \node (7b) at (13,0) {$ d'_m $};
      \node (8b) at (15,0) {$ y' $};
      \draw [cd] (2t) to node [] {\scriptsize{$  $}} (1t);
      \draw [cd] (2t) to node [] {\scriptsize{$  $}} (3t);
      \draw [cd] (5t) to node [] {\scriptsize{$  $}} (4t);
      \draw [cd] (5t) to node [] {\scriptsize{$  $}} (6t);
      \draw [cd] (8t) to node [] {\scriptsize{$  $}} (7t);
      \draw [cd] (8t) to node [] {\scriptsize{$  $}} (9t);
      \draw [cd] (2b) to node [] {\scriptsize{$  $}} (1b);
      \draw [cd] (2b) to node [] {\scriptsize{$  $}} (3b);
      \draw [cd] (4b) to node [] {\scriptsize{$  $}} (3b);
      \draw [cd] (4b) to node [] {\scriptsize{$  $}} (5b);
      \draw [cd] (7b) to node [] {\scriptsize{$  $}} (6b);
      \draw [cd] (7b) to node [] {\scriptsize{$  $}} (8b);
      \draw [cd] (1t) to node [] {\scriptsize{$  $}} (1b);
      \draw [cd] (2t) to node [] {\scriptsize{$  $}} (2b);
      \draw [cd] (3t) to node [] {\scriptsize{$  $}} (3b);
      \draw [cd] (4t) to node [] {\scriptsize{$  $}} (3b);
      \draw [cd] (5t) to node [] {\scriptsize{$  $}} (4b);
      \draw [cd] (6t) to node [] {\scriptsize{$  $}} (5b);
      \draw [cd] (7t) to node [] {\scriptsize{$  $}} (6b);
      \draw [cd] (8t) to node [] {\scriptsize{$  $}} (7b);
      \draw [cd] (9t) to node [] {\scriptsize{$  $}} (8b);
      \node () at (10.5,1) {$ \dotsm $};
      \draw (0.3,0.4) -- (0.4,0.4) -- (0.4,0.3);
      \draw (4.7,0.4) -- (4.6,0.4) -- (4.6,0.3);
      \draw (5.3,0.4) -- (5.4,0.4) -- (5.4,0.3);
      \draw (9.7,0.4) -- (9.6,0.4) -- (9.6,0.3);
      \draw (11.3,0.4) -- (11.4,0.4) -- (11.4,0.3);
      \draw (14.7,0.4) -- (14.6,0.4) -- (14.6,0.3);
    \end{tikzpicture}
  \]
  realize $ \deriv{x+x'}{y+y'} $ by
  \begin{center}
    \begin{tikzpicture}[scale=0.8]
      \node (1t) at (0,2) {$ \ell_1 $};
      \node (2t) at (2,2) {$ k_1 $};
      \node (3t) at (4,2) {$ r1 $};
      \node () at (5,1) {$ \cdots $};
      \node (4t) at (6,2) {$ r_n $};
      \node (5t) at (8,2) {$ \ell'_1 $};
      \node (6t) at (10,2) {$ k'_1 $};
      \node (7t) at (12,2) {$ r'_1 $};
      \node () at (13,1) {$ \cdots $};
      \node (8t) at (14,2) {$ k'_m $};
      \node (9t) at (16,2) {$ r'_m $};
      \node (1b) at (0,0) {$ x+x' $};
      \node (2b) at (2,0) {$ d_1+x' $};
      \node (3b) at (4,0) {$ w_1+x' $};
      \node (5b) at (7,0) {$ y+x' $};
      \node (6b) at (10,0) {$ y+d'_1 $};
      \node (7b) at (12,0) {$ y+w'_1 $};
      \node (8b) at (14,0) {$ y+d'_m $};
      \node (9b) at (16,0) {$ y+y' $};
      \draw [cd]
      (2t) edge[] (1t)
      (2t) edge[] (3t)
      (6t) edge[] (5t)
      (6t) edge[] (7t)
      (8t) edge[] (9t) 
      (2b) edge[] (1b)
      (2b) edge[] (3b)
      (6b) edge[] (5b)
      (6b) edge[] (7b)
      (8b) edge[] (9b) 
      (1t) edge[] (1b)
      (2t) edge[] (2b)
      (3t) edge[] (3b)
      (4t) edge[] (5b)
      (5t) edge[] (5b)
      (6t) edge[] (6b)
      (7t) edge[] (7b)
      (8t) edge[] (8b)
      (9t) edge[] (9b);      
    \end{tikzpicture}
  \end{center}
\end{proof}

As promised, we can now decompose closed systems into open
systems. For this, we need a topos of closed systems $ \X $
equipped with a grammar $ ( \X , P ) $. The closed systems
need interfaces, meaning we need to introduce an adjunction
\[
  \adjunction{\A}{\X}{L}{R}{2}
\]
where $ L $ preserves pullbacks and has a monic counit. At
this point, the material from the previous section
returns. This adjunction gives a discrete comonad $ \flat
\bydef LR $ from which we form the discrete
grammar $ ( \X, P_\flat ) $. Now define the structured cospan
grammar $ ( _L \StrCsp , \hat{P_\flat} ) $ where $
\hat{P_\flat} $ contains the rule
\begin{equation} \label{eq:decomposition-square}
  \begin{tikzpicture}
    \node (1) at (0,4) {$ L 0 $};
    \node (2) at (2,4) {$ \ell $};
    \node (3) at (4,4) {$ LRk $};
    \node (4) at (0,2) {$ L 0 $};
    \node (5) at (2,2) {$ LRk $};
    \node (6) at (4,2) {$ LRk $};
    \node (7) at (0,0) {$ L 0 $};
    \node (8) at (2,0) {$ r $};
    \node (9) at (4,0) {$ LRk $};
    \draw [cd] (1) to (2);
    \draw [cd] (3) to (2);
    \draw [cd] (4) to (5);
    \draw [cd] (6) to (5);
    \draw [cd] (7) to (8);
    \draw [cd] (9) to (8);
    \draw [cd] (4) to (1);
    \draw [cd] (4) to (7);
    \draw [cd] (5) to (2);
    \draw [cd] (5) to (8);
    \draw [cd] (6) to (3);
    \draw [cd] (6) to (9);
  \end{tikzpicture}
\end{equation}
for each rule $ \spn{\ell}{LRk}{r} $ of $ P_{\flat} $. We
use $ ( _L \StrCsp , \hat{P_\flat} ) $ to prove our main
theorem.

Before stating the theorem, we note that this theorem
generalizes work by Gadducci and Heckel
\cite{gadducci_ind-graph-transf} whose domain of
inquiry was graph rewriting. The arc of our
proof follows theirs.

\begin{theorem} \label{thm:inductive-rewriting}

  Fix an adjunction $ (L \dashv R) \from \X \lrto \A $ with monic
  counit. Let $ ( \X , P ) $ be a grammar such that for
  every $ \X $-object $ x $ in the apex of a production of
  $ P $, the lattice $ \Sub (x) $ has all meets. Given
  $ g $, $ h \in \X $, then $ \deriv{g}{h} $ in the
  rewriting relation for a grammar $ ( \X , P ) $ if and
  only if there is a square
  \[
    \begin{tikzpicture}
      \node (1t) at (0,4) {$ LR 0 $};
      \node (2t) at (2,4) {$ g $};
      \node (3t) at (4,4) {$ LR 0 $};
      \node (1m) at (0,2) {$ LR 0 $};
      \node (2m) at (2,2) {$ d $};
      \node (3m) at (4,2) {$ LR 0 $};
      \node (1b) at (0,0) {$ LR 0 $};
      \node (2b) at (2,0) {$ h $};
      \node (3b) at (4,0) {$ LR 0 $};
      \draw [cd] (1t) to node [] {\scriptsize{$  $}} (2t);
      \draw [cd] (3t) to node [] {\scriptsize{$  $}} (2t);
      \draw [cd] (1m) to node [] {\scriptsize{$  $}} (2m);
      \draw [cd] (3m) to node [] {\scriptsize{$  $}} (2m);
      \draw [cd] (1b) to node [] {\scriptsize{$  $}} (2b);
      \draw [cd] (3b) to node [] {\scriptsize{$  $}} (2b);
      \draw [cd] (1m) to node [] {\scriptsize{$  $}} (1t);
      \draw [cd] (1m) to node [] {\scriptsize{$  $}} (1b);
      \draw [cd] (2m) to node [] {\scriptsize{$  $}} (2t);
      \draw [cd] (2m) to node [] {\scriptsize{$  $}} (2b);
      \draw [cd] (3m) to node [] {\scriptsize{$  $}} (3t);
      \draw [cd] (3m) to node [] {\scriptsize{$  $}} (3b);
    \end{tikzpicture}
  \]
  in the double category $ \Lang ( _{L}\StrCsp , \hat{P_\flat} ) $.
\end{theorem}

\begin{proof}
  We show sufficiency by inducting on the length of the
  derivation. If $ \deriv{g}{h} $ in a single step, meaning
  that there is a diagram
  \[
  \begin{tikzpicture}
    \node (1t) at (0,2) {$ \ell $};
    \node (2t) at (2,2) {$ LRk $};
    \node (3t) at (4,2) {$ r $};
    \node (1b) at (0,0) {$ g $};
    \node (2b) at (2,0) {$ d $};
    \node (3b) at (4,0) {$ h $};
    \draw [cd] (2t) to (1t);
    \draw [cd] (2t) to (3t);
    \draw [cd] (2b) to (1b);
    \draw [cd] (2b) to (3b);
    \draw [cd] (1t) to (1b);
    \draw [cd] (2t) to (2b);
    \draw [cd] (3t) to (3b);
    \draw (0.3,0.4) -- (0.4,0.4) -- (0.4,0.3);
    \draw (3.7,0.4) -- (3.6,0.4) -- (3.6,0.3);
  \end{tikzpicture}
  \]
  then the desired square is the horizontal composition of
  \[
    \begin{tikzpicture}
    \begin{scope}
      \node (1t) at (0,4) {$ L0 $};
      \node (2t) at (2,4) {$ \ell $};
      \node (3t) at (4,4) {$ LRk $};
      \node (4t) at (6,4) {$ d $};
      \node (5t) at (8,4) {$ L0 $};
      \node (1m) at (0,2) {$ L 0 $};
      \node (2m) at (2,2) {$ LRk $};
      \node (3m) at (4,2) {$ LRk $};
      \node (4m) at (6,2) {$ d $};
      \node (5m) at (8,2) {$ L 0 $};
      \node (1b) at (0,0) {$ L 0 $};
      \node (2b) at (2,0) {$ r $};
      \node (3b) at (4,0) {$ LRk $};
      \node (4b) at (6,0) {$ d $};
      \node (5b) at (8,0) {$ L 0 $};
      \draw [cd] (1t) to node [] {\scriptsize{$  $}} (2t);
      \draw [cd] (3t) to node [] {\scriptsize{$  $}} (2t);
      \draw [cd] (3t) to node [] {\scriptsize{$  $}} (4t);
      \draw [cd] (5t) to node [] {\scriptsize{$  $}} (4t);
      \draw [cd] (1m) to node [] {\scriptsize{$  $}} (2m);
      \draw [cd] (3m) to node [] {\scriptsize{$  $}} (2m);
      \draw [cd] (3m) to node [] {\scriptsize{$  $}} (4m);
      \draw [cd] (5m) to node [] {\scriptsize{$  $}} (4m);
      \draw [cd] (1b) to node [] {\scriptsize{$  $}} (2b);
      \draw [cd] (3b) to node [] {\scriptsize{$  $}} (2b);
      \draw [cd] (3b) to node [] {\scriptsize{$  $}} (4b);
      \draw [cd] (5b) to node [] {\scriptsize{$  $}} (4b);
      \draw [cd] (1m) to node [] {\scriptsize{$  $}} (1t);
      \draw [cd] (1m) to node [] {\scriptsize{$  $}} (1b);
      \draw [cd] (2m) to node [] {\scriptsize{$  $}} (2t);
      \draw [cd] (2m) to node [] {\scriptsize{$  $}} (2b);
      \draw [cd] (3m) to node [] {\scriptsize{$  $}} (3t);
      \draw [cd] (3m) to node [] {\scriptsize{$  $}} (3b);
      \draw [cd] (4m) to node [] {\scriptsize{$  $}} (4t);
      \draw [cd] (4m) to node [] {\scriptsize{$  $}} (4b);
      \draw [cd] (5m) to node [] {\scriptsize{$  $}} (5t);
      \draw [cd] (5m) to node [] {\scriptsize{$  $}} (5b);
    \end{scope}
    \end{tikzpicture}
  \]
  The left square is a generator and the right square is the
  identity on the horizontal arrow $ \csp{LRk}{d}{L0}
  $. The square for a derivation
  $ \dderiv{\deriv{g}{h}}{j} $ is the vertical composition
  of
  \[
    \begin{tikzpicture}
      \node (1t) at (0,8) {$ L 0 $};
      \node (2t) at (2,8) {$ g $};
      \node (3t) at (4,8) {$ L 0 $};
      \node (1m) at (0,6) {$ L 0 $};
      \node (2m) at (2,6) {$ d $};
      \node (3m) at (4,6) {$ L 0 $};
      \node (1b) at (0,4) {$ L 0 $};
      \node (2b) at (2,4) {$ h $};
      \node (3b) at (4,4) {$ L 0 $};
      \node (1bb) at (0,2) {$ L 0 $};
      \node (2bb) at (2,2) {$ e $};
      \node (3bb) at (4,2) {$ L 0 $};
      \node (1bbb) at (0,0) {$ L 0 $};
      \node (2bbb) at (2,0) {$ j $};
      \node (3bbb) at (4,0) {$ L 0 $};
      \draw [cd] (1t) to node [] {\scriptsize{$  $}} (2t);
      \draw [cd] (3t) to node [] {\scriptsize{$  $}} (2t);
      \draw [cd] (1m) to node [] {\scriptsize{$  $}} (2m);
      \draw [cd] (3m) to node [] {\scriptsize{$  $}} (2m);
      \draw [cd] (1b) to node [] {\scriptsize{$  $}} (2b);
      \draw [cd] (3b) to node [] {\scriptsize{$  $}} (2b);
      \draw [cd] (1m) to node [] {\scriptsize{$  $}} (1t);
      \draw [cd] (1m) to node [] {\scriptsize{$  $}} (1b);
      \draw [cd] (2m) to node [] {\scriptsize{$  $}} (2t);
      \draw [cd] (2m) to node [] {\scriptsize{$  $}} (2b);
      \draw [cd] (3m) to node [] {\scriptsize{$  $}} (3t);
      \draw [cd] (3m) to node [] {\scriptsize{$  $}} (3b);
      \draw [cd] (1bb) to node [] {\scriptsize{$  $}} (2bb);
      \draw [cd] (3bb) to node [] {\scriptsize{$  $}} (2bb);
      \draw [cd] (1bbb) to node [] {\scriptsize{$  $}} (2bbb);
      \draw [cd] (3bbb) to node [] {\scriptsize{$  $}} (2bbb);
      \draw [cd] (1bb) to node [] {\scriptsize{$  $}} (1b);
      \draw [cd] (1bb) to node [] {\scriptsize{$  $}} (1bbb);
      \draw [cd] (2bb) to node [] {\scriptsize{$  $}} (2b);
      \draw [cd] (2bb) to node [] {\scriptsize{$  $}} (2bbb);
      \draw [cd] (3bb) to node [] {\scriptsize{$  $}} (3b);
      \draw [cd] (3bb) to node [] {\scriptsize{$  $}} (3bbb);
    \end{tikzpicture}
  \]
  The top square is from $ \deriv{g}{h} $ and the second
  from $ \dderiv{h}{j} $.

  Conversely, proceed by structural induction on the
  generating squares of
  $ \Lang ( _{L}\StrCsp , \hat{P_\flat} ) $.  It suffices to
  show that the rewrite relation is preserved by vertical
  and horizontal composition by generating squares.  Suppose
  we have a square
  \[
    \begin{tikzpicture}
      \node (1t) at (0,4) {$ L 0 $};
      \node (2t) at (2,4) {$ w $};
      \node (3t) at (4,4) {$ L 0 $};
      \node (1m) at (0,2) {$ L 0 $};
      \node (2m) at (2,2) {$ x $};
      \node (3m) at (4,2) {$ L 0 $};
      \node (1b) at (0,0) {$ L 0 $};
      \node (2b) at (2,0) {$ y $};
      \node (3b) at (4,0) {$ L 0 $};
      \draw [cd] (2t) to node [] {\scriptsize{$  $}} (1t);
      \draw [cd] (2t) to node [] {\scriptsize{$  $}} (3t);
      \draw [cd] (2m) to node [] {\scriptsize{$  $}} (1m);
      \draw [cd] (2m) to node [] {\scriptsize{$  $}} (3m);
      \draw [cd] (2b) to node [] {\scriptsize{$  $}} (1b);
      \draw [cd] (2b) to node [] {\scriptsize{$  $}} (3b);
      \draw [cd] (1m) to node [] {\scriptsize{$  $}} (1t);
      \draw [cd] (2m) to node [] {\scriptsize{$  $}} (2t);
      \draw [cd] (3m) to node [] {\scriptsize{$  $}} (3t);
      \draw [cd] (1m) to node [] {\scriptsize{$  $}} (1b);
      \draw [cd] (2m) to node [] {\scriptsize{$  $}} (2b);
      \draw [cd] (3m) to node [] {\scriptsize{$  $}} (3b);
    \end{tikzpicture}
  \]
  corresponding to a derivation $ \deriv{w}{y} $. Composing
  this vertically with a generating square, which must have
  form
  \[
    \begin{tikzpicture}
      \node (1t) at (0,4) {$ L 0 $};
      \node (2t) at (2,4) {$ y $};
      \node (3t) at (4,4) {$ L 0 $};
      \node (1m) at (0,2) {$ L 0 $};
      \node (2m) at (2,2) {$ L 0 $};
      \node (3m) at (4,2) {$ L 0 $};
      \node (1b) at (0,0) {$ L 0 $};
      \node (2b) at (2,0) {$ z $};
      \node (3b) at (4,0) {$ L 0 $};
      \draw [cd] (2t) to node [] {\scriptsize{$  $}} (1t);
      \draw [cd] (2t) to node [] {\scriptsize{$  $}} (3t);
      \draw [cd] (2m) to node [] {\scriptsize{$  $}} (1m);
      \draw [cd] (2m) to node [] {\scriptsize{$  $}} (3m);
      \draw [cd] (2b) to node [] {\scriptsize{$  $}} (1b);
      \draw [cd] (2b) to node [] {\scriptsize{$  $}} (3b);
      \draw [cd] (1m) to node [] {\scriptsize{$  $}} (1t);
      \draw [cd] (2m) to node [] {\scriptsize{$  $}} (2t);
      \draw [cd] (3m) to node [] {\scriptsize{$  $}} (3t);
      \draw [cd] (1m) to node [] {\scriptsize{$  $}} (1b);
      \draw [cd] (2m) to node [] {\scriptsize{$  $}} (2b);
      \draw [cd] (3m) to node [] {\scriptsize{$  $}} (3b);
    \end{tikzpicture}
  \]
  corresponding to a production $ \spn{y}{L0}{z} $ gives
  \[
    \begin{tikzpicture}
      \node (1t) at (0,4) {$ L 0 $};
      \node (2t) at (2,4) {$ w $};
      \node (3t) at (4,4) {$ L 0 $};
      \node (1m) at (0,2) {$ L 0 $};
      \node (2m) at (2,2) {$ L 0 $};
      \node (3m) at (4,2) {$ L 0 $};
      \node (1b) at (0,0) {$ L 0 $};
      \node (2b) at (2,0) {$ z $};
      \node (3b) at (4,0) {$ L 0 $};
      \draw [cd] (2t) to node [] {\scriptsize{$  $}} (1t);
      \draw [cd] (2t) to node [] {\scriptsize{$  $}} (3t);
      \draw [cd] (2m) to node [] {\scriptsize{$  $}} (1m);
      \draw [cd] (2m) to node [] {\scriptsize{$  $}} (3m);
      \draw [cd] (2b) to node [] {\scriptsize{$  $}} (1b);
      \draw [cd] (2b) to node [] {\scriptsize{$  $}} (3b);
      \draw [cd] (1m) to node [] {\scriptsize{$  $}} (1t);
      \draw [cd] (2m) to node [] {\scriptsize{$  $}} (2t);
      \draw [cd] (3m) to node [] {\scriptsize{$  $}} (3t);
      \draw [cd] (1m) to node [] {\scriptsize{$  $}} (1b);
      \draw [cd] (2m) to node [] {\scriptsize{$  $}} (2b);
      \draw [cd] (3m) to node [] {\scriptsize{$  $}} (3b);
    \end{tikzpicture}
  \]
  which corresponds to a derivation
  $ \dderiv{\deriv{w}{y}}{z} $.  Composing horizontally with
  a generating square
  \[
    \begin{tikzpicture}
      \node (1t) at (0,4) {$ L 0 $};
      \node (2t) at (2,4) {$ \ell $};
      \node (3t) at (4,4) {$ L 0 $};
      \node (1m) at (0,2) {$ L 0 $};
      \node (2m) at (2,2) {$ LRk $};
      \node (3m) at (4,2) {$ L 0 $};
      \node (1b) at (0,0) {$ L 0 $};
      \node (2b) at (2,0) {$ r $};
      \node (3b) at (4,0) {$ L 0 $};
      \draw [cd] (2t) to node [] {\scriptsize{$  $}} (1t);
      \draw [cd] (2t) to node [] {\scriptsize{$  $}} (3t);
      \draw [cd] (2m) to node [] {\scriptsize{$  $}} (1m);
      \draw [cd] (2m) to node [] {\scriptsize{$  $}} (3m);
      \draw [cd] (2b) to node [] {\scriptsize{$  $}} (1b);
      \draw [cd] (2b) to node [] {\scriptsize{$  $}} (3b);
      \draw [cd] (1m) to node [] {\scriptsize{$  $}} (1t);
      \draw [cd] (2m) to node [] {\scriptsize{$  $}} (2t);
      \draw [cd] (3m) to node [] {\scriptsize{$  $}} (3t);
      \draw [cd] (1m) to node [] {\scriptsize{$  $}} (1b);
      \draw [cd] (2m) to node [] {\scriptsize{$  $}} (2b);
      \draw [cd] (3m) to node [] {\scriptsize{$  $}} (3b);
    \end{tikzpicture}
  \]
  corresponding with a production $ \spn{\ell}{LRk}{r} $
  results in the square
  \[
    \begin{tikzpicture}
      \node (1t) at (0,4) {$ L 0 $};
      \node (2t) at (2,4) {$ w + \ell $};
      \node (3t) at (4,4) {$ L 0 $};
      \node (1m) at (0,2) {$ L 0 $};
      \node (2m) at (2,2) {$ x + LRk $};
      \node (3m) at (4,2) {$ L 0 $};
      \node (1b) at (0,0) {$ L 0 $};
      \node (2b) at (2,0) {$ y + r $};
      \node (3b) at (4,0) {$ L 0 $};
      \draw [cd] (2t) to node [] {\scriptsize{$  $}} (1t);
      \draw [cd] (2t) to node [] {\scriptsize{$  $}} (3t);
      \draw [cd] (2m) to node [] {\scriptsize{$  $}} (1m);
      \draw [cd] (2m) to node [] {\scriptsize{$  $}} (3m);
      \draw [cd] (2b) to node [] {\scriptsize{$  $}} (1b);
      \draw [cd] (2b) to node [] {\scriptsize{$  $}} (3b);
      \draw [cd] (1m) to node [] {\scriptsize{$  $}} (1t);
      \draw [cd] (2m) to node [] {\scriptsize{$  $}} (2t);
      \draw [cd] (3m) to node [] {\scriptsize{$  $}} (3t);
      \draw [cd] (1m) to node [] {\scriptsize{$  $}} (1b);
      \draw [cd] (2m) to node [] {\scriptsize{$  $}} (2b);
      \draw [cd] (3m) to node [] {\scriptsize{$  $}} (3b);
    \end{tikzpicture}
  \]
  But $ \deriv{w+\ell}{y+r} $ as seen in Lemma
  \ref{thm:rewrite-rel-is-additive}. 
\end{proof}

With this result, we have completely described the rewrite
relation for a grammar $ ( \X , P ) $ with squares in
$ \Lang ( _{L}\StrCsp, \hat{P_\flat} ) $ framed by the
initial object of $ \X $.  These squares are rewrites of a
closed system in the sense that the interface is empty.  We
can instead begin with a closed system $ x $ in $ \X $ as
represented by a horizontal arrow $ \csp{L0}{x}{L0} $ in
$ \Lang ( _{L}\StrCsp , \hat{P_\flat} ) $ and decompose it
into a composite of sub-systems, that is a sequence of
composable horizontal arrows
\begin{center}
  \begin{tikzpicture}
    \node (L0) at (0,0) {$ L0 $};
    \node (x1) at (1,1) {$ x_1 $};
    \node (La1) at (2,0) {$ La_1 $};
    \node (x2) at (3,1) {$ x_2 $};
    \node (La2) at (4,0) {$ La_2 $};
    \node () at (5,0) {$ \dotsm $};
    \node (Lan1) at (6,0) {$ La_{n-1} $};
    \node (xn) at (7,1) {$ x_n $};
    \node (L0') at (8,0) {$ L0 $};
    \draw [cd] 
    (L0) edge[] (x1)
    (La1) edge[] (x1)
    (La1) edge[] (x2)
    (La2) edge[] (x2)
    (Lan1) edge[] (xn)
    (L0') edge[] (xn);
  \end{tikzpicture}
\end{center}

Rewriting can be performed on each of these sub-systems
\[
  \begin{tikzpicture}
    \begin{scope}[shift={(0,6)}]
      \node (1) at (0,4) {$ L0 $};
      \node (2) at (2,4) {$ x_1 $};
      \node (3) at (4,4) {$ La_1 $};
      \node (4) at (0,2) {$ L0 $};
      \node (5) at (2,2) {$ x'_1 $};
      \node (6) at (4,2) {$ La'_1 $};
      \node (7) at (0,0) {$ L0 $};
      \node (8) at (2,0) {$ x''_1 $};
      \node (9) at (4,0) {$ La''_1 $};
      \draw [cd] (1) to (2);
      \draw [cd] (3) to (2);
      \draw [cd] (4) to (5);
      \draw [cd] (6) to (5);
      \draw [cd] (7) to (8);
      \draw [cd] (9) to (8);
      \draw [cd] (4) to node [left]
        {\scriptsize{$ \cong $}} (1);
      \draw [cd] (4) to node [left]
        {\scriptsize{$ \cong $}} (7);
      \draw [>->] (5) to (2);
      \draw [>->] (5) to (8);
      \draw [cd] (6) to node [right]
        {\scriptsize{$ \cong  $}} (3);
      \draw [cd] (6) to node [right]
        {\scriptsize{$ \cong $}} (9);
    \end{scope}
    \begin{scope}[shift={(6,6)}]
      \node (1) at (0,4) {$ La_{n-1} $};
      \node (2) at (2,4) {$ x_n $};
      \node (3) at (4,4) {$ L0 $};
      \node (4) at (0,2) {$ La_{n-1} $};
      \node (5) at (2,2) {$ x'_n $};
      \node (6) at (4,2) {$ L0 $};
      \node (7) at (0,0) {$ La_{n-1} $};
      \node (8) at (2,0) {$ x''_n $};
      \node (9) at (4,0) {$ L0 $};
      \draw [cd] (1) to (2);
      \draw [cd] (3) to (2);
      \draw [cd] (4) to (5);
      \draw [cd] (6) to (5);
      \draw [cd] (7) to (8);
      \draw [cd] (9) to (8);
      \draw [cd] (4) to node [left]
        {\scriptsize{$ \cong $}} (1);
      \draw [cd] (4) to node [left]
        {\scriptsize{$ \cong $}} (7);
      \draw [>->] (5) to (2);
      \draw [>->] (5) to (8);
      \draw [cd] (6) to node [right]
        {\scriptsize{$ \cong  $}} (3);
      \draw [cd] (6) to node [right]
        {\scriptsize{$ \cong $}} (9);
    \end{scope}
    \begin{scope}
      \node (1) at (0,4) {$ L0 $};
      \node (2) at (2,4) {$ y_1 $};
      \node (3) at (4,4) {$ La_1 $};
      \node (4) at (0,2) {$ L0 $};
      \node (5) at (2,2) {$ y'_1 $};
      \node (6) at (4,2) {$ La_1 $};
      \node (7) at (0,0) {$ L0 $};
      \node (8) at (2,0) {$ y''_1 $};
      \node (9) at (4,0) {$ La_1 $};
      \draw [cd] (1) to (2);
      \draw [cd] (3) to (2);
      \draw [cd] (4) to (5);
      \draw [cd] (6) to (5);
      \draw [cd] (7) to (8);
      \draw [cd] (9) to (8);
      \draw [cd] (4) to node [left]
        {\scriptsize{$ \cong $}} (1);
      \draw [cd] (4) to node [left]
        {\scriptsize{$ \cong $}} (7);
      \draw [>->] (5) to (2);
      \draw [>->] (5) to (8);
      \draw [cd] (6) to node [right]
        {\scriptsize{$ \cong  $}} (3);
      \draw [cd] (6) to node [right]
        {\scriptsize{$ \cong $}} (9);
    \end{scope}
    \begin{scope}[shift={(6,0)}]
      \node (1) at (0,4) {$ La_{n-1} $};
      \node (2) at (2,4) {$ y_n $};
      \node (3) at (4,4) {$ L0 $};
      \node (4) at (0,2) {$ La_{n-1} $};
      \node (5) at (2,2) {$ y'_n $};
      \node (6) at (4,2) {$ L0 $};
      \node (7) at (0,0) {$ La_{n-1} $};
      \node (8) at (2,0) {$y''_n $};
      \node (9) at (4,0) {$ L0 $};
      \draw [cd] (1) to (2);
      \draw [cd] (3) to (2);
      \draw [cd] (4) to (5);
      \draw [cd] (6) to (5);
      \draw [cd] (7) to (8);
      \draw [cd] (9) to (8);
      \draw [cd] (4) to node [left]
        {\scriptsize{$ \cong $}} (1);
      \draw [cd] (4) to node [left]
        {\scriptsize{$ \cong $}} (7);
      \draw [>->] (5) to (2);
      \draw [>->] (5) to (8);
      \draw [cd] (6) to node [right]
        {\scriptsize{$ \cong  $}} (3);
      \draw [cd] (6) to node [right]
        {\scriptsize{$ \cong $}} (9);
    \end{scope}
    \node () at (5,2) {$ \dotsm $};
    \node () at (2,5) {$ \vdots $};
    \node () at (5,8) {$ \dotsm $};
    \node () at (8,5) {$ \vdots $};
  \end{tikzpicture}
\]
The composite of these squares is a rewriting of the
original system.


\chapter{Conclusions}
\label{sec:conclusions}

Our work here demarcates a starting line on the path
towards a fully general mathematical theory of systems. We now
have a syntax to reason with.  Built into this syntax is a
mechanism to identify when distinct systems behave
similarly. That is, our syntax reflects semantics.

The semantics side requires attention.  We can conjecture
that the category $ \Rel $ of sets and relations will be the
most appropriate category to serve as our semantic
universe.  The naive idea behind this belief is that
semantics should describe the relationship between inputs
and outputs possible for a particular system.  If not
$ \Rel $, then something structurally similar such as the
category $ \cat{Hilb} $ of Hilbert spaces and linear
maps. This would be appropriate semantics for the
ZX-calculus.

Given a more robust theory of semantics to work with, we can
fill in the larger picture of a general language for
systems.  To do this, Lawvere's `functorial semantics'
\cite{lawvere_func-semantics} is a promising area from which
to pull. Functorial semantics has been successful in
developing universal algebra, and the author believes that
we can leverage Lawvere's thinking in the systems context. To
what extent, however, remains an open question.


\appendix


\chapter{An account of some category theory topics}
\label{sec:double-cats}

Category theory has been in mainstream mathematical
discourse for decades now.  This section does not seek to
add to an already crowded literature on category
theory. Instead, we give just enough background for those
readers coming to this thesis without much knowledge about
category theory. For a more in depth study of category
theory, there are many excellent resources
\cite{
  awody,
  lawvere-schanuel_conceptual-math,
  maclane_cats-working,
  riehl_cats-context}.

As a baseline, we assume basic knowledge of category
theory.  This includes the definitions of categories, functors,
natural transformations, limits, colimits, adjunctions,
monoidal categories, and symmetric monoidal categories.  
But our needs extend beyond these basic concepts, so we
provide the reader with a brief account of some more
advanced topics.

\section{Enrichment and bicategories}
\label{sec:extern-bicats}

The most familiar examples of categories are built from
mathematical widgets and their homomorphisms.  For example,
the category $ \Vect_F $ whose objects are vector spaces
over a fixed field $ F $ and arrows are linear maps.  Yet,
as a category, $ \Vect_F $ does not truly capture everything
we like about vector spaces. We are missing the fact that,
for any two vector spaces $ V $ and $ W $, the space of
linear maps from $ V $ to $ W $ form a vector space by
pointwise addition and scaling.  Yet the hom-set
$ \Vect_F (V,W) $ is merely a collection of linear maps
without additional structure.  The theory of enriched
categories fixes this drawback.

Many familiar categories are actually enriched. For example,
the category $ \Set $ of sets has that, for any two sets
$ x,y $, the collection of arrows $ \Set (x,y) $ is actually
a set.  We say that $ \Set $ is enriched over $ \Set $.
Given the category $ \Mod_R $ whose objects are modules over
an arbitrary ring $ R $ and any two such modules $ x $,
$ y $, the collections of arrows $ \Mod_R (x,y) $ is
actually a $ \mathbb{Z} $-module.  Thus we say that
$ \Mod_R $ is enriched over $ \Mod_{\mathbb{Z}} $.  However,
to be an enriched category, it is not enough for the
collections of arrows to simply have additional structure.
Cohesion is needed.

\begin{definition}[Enriched category]
  Let $ (\M,\otimes,I,\alpha,\lambda,\rho) $ be a
  monoidal category. A category $ \C $ is \defn{enriched}
  over $ \M $ consists of
  \begin{itemize}
  \item a class $ \ob (\C) $ of objects,
  \item an object $ \C (a,b) $ of $ \M $ for each
    pair $ a,b \in \ob (\C) $ that collects the
    arrows of type $ a \to b $
  \item an arrow $ 1_a \from I \to \C (a,a) $ in
    $ \M $ that chooses an identity arrow on $ a $
  \item an arrow
    \[
      \circ_{abc} \from \C (b,c) \otimes \C (a,b)
      \to \C (a,c)
    \]
    for each triple of objects $ a,b,c \in \ob
    (\C) $ that defines the composition
  \end{itemize}
  together with a commuting diagram expressing associativity
  \[
    \begin{tikzpicture}
      \node (1) at (0,4)
        {$ (\C (c,d) \otimes \C (b,c) )
        \otimes \C (a,b) $};
      \node (2) at (6,4)
        {$ \C (b,d) \otimes \C (a,d) $};
      \node (3) at (6,2)
        {$ \C (a,d) $};
      \node (4) at (0,0)
        {$ \C (c,d) \otimes (\C (b,c)
        \otimes \C ( a,b) ) $};  
      \node (5) at (6,0)
        {$ \C (c,d) \otimes \C (a,c) $};
      \draw[cd]
        (1) edge node[above]{$\circ \otimes \id$} (2)
        (2) edge node[right]{$\circ$} (3)
        (1) edge node[left]{$\alpha$} (4)
        (4) edge node[below]{$\id \otimes \circ$} (5)
        (5) edge node[right]{$\circ$} (3);
    \end{tikzpicture}
  \]
  and commuting diagrams expressing left and right unity
  \[
    \begin{tikzpicture}
      \begin{scope}
        \node (1) at (0,2)
          {$ I \otimes \C (a,b) $};
        \node (2) at (2,0)
          {$ \C (a,b) $};
        \node (3) at (4,2)
          {$ \C (b,b) \otimes \C (a,b) $};
        \draw[cd]
        (1) edge node[left]{$ \lambda $} (2)
        (1) edge node[above]{$ 1 \otimes \id $} (3)
        (3) edge node[right]{$ \circ $} (2);
      \end{scope}
      \begin{scope}[shift={(8,0)}]
        \node (1) at (0,2)
          {$ \C (a,b) \otimes I $};
        \node (2) at (2,0)
          {$ \C (a,b) $};
        \node (3) at (4,2)
          {$ \C (a,b) \otimes \C (a,a) $};
        \draw[cd]
        (1) edge node[left]{$ \rho $} (2)
        (1) edge node[above]{$ \id \otimes 1 $} (3)
        (3) edge node[right]{$ \circ $} (2);
      \end{scope}
    \end{tikzpicture}
  \]
  When $ \M $ is actually a 2-category and the
  above diagrams only commute up to natural
  isomorphism, then we say that $ \C $ is
  \defn{weakly enriched} over $ \M $.  
\end{definition}

In this thesis, the we are interested in one example of an
weakly enriched category: a bicategory. In short, a
bicategory is a category weakly enriched in the 2-category $
\Cat $.  Thus a bicategory has a \emph{category} of arrows between
objects, not merely a collection of arrows.

Defining a bicategory to be a category weakly enriched in
$ \Cat $ is elegant but hardly illuminating. Thus, the
definition is worth unpacking but, for clarity's sake, we
only approximate the definition by providing the important
information to know and ignoring technical details.

\begin{definition}[Bicategory]
  A bicategory $ \CC $ consists of
  \begin{itemize}
  \item a collection of objects $ \ob (\CC) $
  \item for each pair of objects $ x,y $, a collection of
    arrows of type $ x \to y $ which compose, that is
    \[
      ( x \xto{f} y \xto{g} z ) \mapsto (x \xto{gf} z)
    \]
  \item for each pair of arrows $ f,g \from x \to y $ of the
    same type, a collection of 2-arrows
    \[
      \begin{tikzpicture}
        \node (x) at (0,0) {$ x $};
        \node (y) at (2,0) {$ y $};
        \draw[cd]
        (x) edge[bend left=60]  node[above]{$ f $} (y)
        (x) edge[bend right=60] node[below]{$ g $} (y);
        \node at (1,0) {$\Downarrow$};
      \end{tikzpicture}
    \]
    together with operations expressing a horizontal composition
    \[
      \begin{tikzpicture}
        \begin{scope}
          \node (x) at (0,0) {$ x $};
          \node (y) at (2,0) {$ y $};
          \node (z) at (4,0) {$ z $};
          \draw [cd]
          (x) edge[bend left=60] node[above]{$ f $} (y)
          (x) edge[bend right=60] node[below]{$ g $} (y)
          (y) edge[bend left=60] node[above]{$ f' $} (z)
          (y) edge[bend right=60] node[below]{$ g' $} (z);
          \node at (1,0) {{$\Downarrow_\alpha$}};
          \node at (3,0) {{$\Downarrow_\beta$}};
        \end{scope}
        \begin{scope}[shift={(6,0)}]
          \node (x) at (0,0) {$ x $};
          \node (z) at (2,0) {$ z $};
          \draw [cd]
          (x) edge[bend left=60] node[above]{$f'f$} (z)
          (x) edge[bend right=60] node[below]{$g'g$} (z);
          \node at (1,0) {{$\Downarrow_{\alpha\hcirc\beta}$}};
        \end{scope}
        \node () at (5,0) {$\xmapsto{\hcirc}$};
      \end{tikzpicture}
    \]
    and vertical composition
    \[
      \begin{tikzpicture}
        \begin{scope}
          \node (x) at (0,0) {$ x $};
          \node (y) at (3,0) {$ y $};
          \draw [cd]
          (x) edge[bend left=60] (y)
          (x) edge[bend right=60] (y)
          (x) edge (y);
          \node at (1.5,0.5)
          {\scriptsize{$\Downarrow_\alpha $}};
          \node at (1.5,-0.5)
          {\scriptsize{$\Downarrow_\beta $}};
        \end{scope}
        \begin{scope}[shift={(5,0)}]
          \node (x) at (0,0) {$ x $};
          \node (y) at (3,0) {$ y $};
          \draw [cd]
          (x) edge[bend left=60] (y)
          (x) edge[bend right=60] (y);
          \node () at (1.5,0) {$\Downarrow_{\beta\vcirc\alpha}$};
        \end{scope}
        \node () at (4,0) {$\xmapsto{\vcirc}$};
      \end{tikzpicture}
    \]
    that satisfy the interchange law
    \[
      (\alpha  \hcirc \beta   ) \vcirc
      (\alpha' \hcirc \beta'  ) =
      (\alpha  \vcirc \alpha' ) \hcirc
      (\beta   \vcirc \beta'  )
    \]
  \end{itemize}
  
  The interchange law states that given an array of 2-arrows
  \[
    \begin{tikzpicture}
      \node (x) at (0,0) {$ x $};
      \node (y) at (3,0) {$ y $};
      \node (z) at (6,0) {$ z $};
      \node () at (1.5,0.5)  {$ \Downarrow_\alpha $};
      \node () at (1.5,-0.5) {$ \Downarrow_{\alpha'} $};
      \node () at (4.5,0.5)  {$ \Downarrow_\beta $};
      \node () at (4.5,-0.5) {$ \Downarrow_{\beta'} $};
      \draw[cd]
      (x) edge[bend left=60]  (y)
      (x) edge                (y)
      (x) edge[bend right=60] (y)
      (y) edge[bend left=60]  (z)
      (y) edge                (z)
      (y) edge[bend right=60] (z);
    \end{tikzpicture}
  \]
  performing the two horizontal compositions
  \[
    \begin{tikzpicture}
      \node (x) at (0,0) {$ x $};
      \node (z) at (4,0) {$ z $};
      \node at (2,0.5)  {$\Downarrow_{\alpha \hcirc \beta}$};
      \node at (2,-0.5) {$\Downarrow_{\alpha' \hcirc \beta}$};
      \path[cd,font=\scriptsize]
      (x) edge[bend left=60] (z)
      (x) edge               (z)
      (x) edge[bend right=60] (z);
    \end{tikzpicture}
  \]
  followed by the vertical composition
  \[
    \begin{tikzpicture}
      \node (x) at (0,0) {$ x $};
      \node (z) at (4,0) {$ z $};
      \node at (2,0)
      { $\Downarrow_{(\alpha \hcirc \beta)
          \vcirc (\alpha' \hcirc \beta')} $};
      \draw[cd]
      (x) edge[bend left=60]  (z)
      (x) edge[bend right=60] (z);
    \end{tikzpicture}
  \]
  gives exactly the same 2-arrow as first performing
  the two vertical compositions
  \[
    \begin{tikzpicture}
      \node (x) at (0,0) {$ x $};
      \node (y) at (2,0) {$ y $};
      \node (z) at (4,0) {$ z $};
      \node at (1,0) {$\Downarrow_{\alpha \vcirc \beta}$};
      \node at (3,0) {$\Downarrow_{\alpha' \vcirc \beta'}$};
      \draw[cd]
      (x) edge[bend left=60]  (y)
      (x) edge[bend right=60] (y)
      (y) edge[bend left=60]  (z)
      (y) edge[bend right=60] (z);
    \end{tikzpicture}
  \]
  followed by the horizontal composition
  \[
    \begin{tikzpicture}
      \node (x) at (0,0) {$ x $};
      \node (z) at (4,0) {$ z $};
      \node () at (2,0) {
        $ \Downarrow_{(\alpha \vcirc \alpha')
          \hcirc (\beta \vcirc \beta'}) $ };
      \draw[cd]
      (x) edge[bend left=60] (z)
      (x) edge[bend right=60] (z);
    \end{tikzpicture}
  \]
\end{definition}

That definition deconstructs a bicategory, laying out all of
the components. Next, we give a definition in the spirit of
enrichment.

\begin{definition}[Bicategory]
  Consider the monoidal 2-category $ (\CCat, \times,
  \cat{1}) $. A bicategory $ \CC $ has
  \begin{itemize}
  \item a collection of objects $ \ob (\CC) $
  \item for each pair of objects $ a,b \in \ob (\CC) $, a
    category $ \CC (a,b) $ of arrows
  \item for each object $ a \in \ob (\CC) $, a functor $ \id_a \from
    \cat{1} \to \CC (a,a) $ that chooses the identity element
  \item for each triple of objects $ a,b,c \in \ob (\CC) $,
    a functor
    \[
      \circ_{a,b,c} \from
      \CC (b,c) \times \CC (a,b) \to \CC (a,c)
    \]
    expressing composition    
  \end{itemize}
  such that, for all $ a,b,c,d \in \ob (\CC) $, the associativity diagram
  \[
    \begin{tikzpicture}
      \node (1) at (0,2)
        {$ (\CC (c,d) \times \CC (b,c)) \times \CC (a,b) $};
      \node (2) at (8,2)
        {$ \CC (c,d) \times (\CC (b,c) \times \CC (a,b) $};
      \node (3) at (0,0) {$ \CC (b,d) \times \CC (a,b) $};
      \node (4) at (8,0) {$ \CC (c,d) \times \CC (a,c) $};
      \node (5) at (4,-2) {$ \CC (a,d) $};
      \node () at (4,0) {$ \Downarrow \iso $};
      \draw[cd]
      (1) edge node[above]{$ \alpha $} (2)
      (1) edge node[left]{$ \circ \times \id $} (3)
      (2) edge node[right]{$ \id \times \circ $} (4)
      (3) edge node[left]{$ \circ $} (5)
      (4) edge node[right]{$ \circ $} (5);
      \end{tikzpicture}
  \]
  and the left and right unitor diagrams
  \[
    \begin{tikzpicture}
    \begin{scope}
      \node (1) at (0,2) {$ \bicat{1} \times \CC (a,b) $};
      \node (2) at (4,2) {$ \CC (b,b) \times \CC (a,b) $};
      \node (3) at (4,0) {$ \CC (a,b) $};
      \node () at (3,1.25)  {$ \Downarrow \iso $};
      \draw[cd]
      (1) edge node[above]{$ i_b \times \id $} (2)
      (1) edge node[below]{$ \lambda $} (3)
      (2) edge node[right]{$ \circ $} (3);
    \end{scope}
    \begin{scope}[shift={(8,0)}]
      \node (1) at (0,2) {$ \CC (a,b) \times \bicat{1} $};
      \node (2) at (4,2) {$ \CC (a,b) \times \CC (a,a) $};
      \node (3) at (4,0) {$ \CC (a,b) $};
      \node ()  at (3,1.25) {$ \Downarrow \iso $};
      \draw[cd]
      (1) edge node[above]{$ i_a \times \id $} (2)
      (1) edge node[below]{$ \rho $}           (3)
      (2) edge node[right]{$ \circ $}          (3);
    \end{scope}
    \end{tikzpicture}
  \]
  commute up to a natural isomorphism.  
\end{definition}

We observe that the objects of the hom-category
$ \CC (a,b) $ are arrows in $ \CC $ and the arrows of
$ \CC (a,b) $ are 2-arrows in $ \CC $. Composition in
$ \CC (a,b) $ is the vertical composition in $ \CC $. The
composition of the arrows and horizontal composition of
2-arrows in $ \CC $ is given by the functor $ \circ $ which,
by light of it preserving composition, gives the interchange
law.

\section{Internalization and double categories}
\label{sec:internal-double-cats}

Most treatments of mathematics base definitions on set
theory. The definitions for a monoid, topological space,
poset, and so on all begin by establishing a set. An
alternative viewpoint is to internalize such gadgets in a
category. 

For example, a monoid is traditionally defined to be a set
$ M $ together with an identity element $ e \in M $ equipped
with a binary operation $ M \times M \to M $ such that for
all $ x,y,z \in M $, we have $ ex=x=xe $ and $ (xy)z=x(yz)
$. However, we can also define a monoid \emph{internal} to a
category. 

\begin{definition}[Internal monoid]
  \label{def:intern-monoid}
  Let $ (\C , \otimes , I ) $ be a
  monoidal category.  A \defn{monoid internal to $ \C $}
  consists of an object $ m \in \ob (\C) $ and two arrows in
  $ \C $
  \begin{itemize}
  \item (multiplication) $ \mu \from m \otimes m \to m $,
  \item (unit) $ \eta \from I \to m $
  \end{itemize}
  such that the associator diagram
  \[
  \begin{tikzpicture}
    \node (1) at (0,2) {$ (m \otimes m) \otimes m $};
    \node (2) at (4,2) {$ m \otimes (m \otimes m) $};
    \node (3) at (8,2) {$ m \otimes m $};
    \node (4) at (0,0) {$ m \otimes m $};
    \node (5) at (8,0) {$ m $};
    \draw [cd] (1) to node [above]
      {$ \alpha $} (2);
    \draw [cd] (2) to node [above]
      {$ \id \otimes \mu $} (3);
    \draw [cd] (3) to node [right]
      {$ \mu $} (5);
    \draw [cd] (1) to node [left]
      {$ \mu \otimes \id$} (4);
    \draw [cd] (4) to node [below] {$ \mu $} (5);
  \end{tikzpicture}
\]
and unitor diagram
\[
  \begin{tikzpicture}
    \node (1) at (-4,2) {$ I \otimes m $};
    \node (2) at (0,2)  {$ m \otimes m $};
    \node (3) at (4,2)  {$ m \otimes I $};
    \node (4) at (0,0)  {$ m $};
    \draw [cd] (1) to node[above]{$ \eta \otimes \id $} (2);
    \draw [cd] (1) to node[below]{$ \lambda $} (4);
    \draw [cd] (2) to node[left]{$ \mu $} (4);
    \draw [cd] (3) to node[above]{$ \id \otimes \eta $} (2);
    \draw [cd] (3) to node[below]{$ \rho $} (4);
  \end{tikzpicture}
\]
commute. 

A \defn{morphism of monoids} is an arrow
$ f \from m \to m' $ in $ \C $ between two monoid objects $
( m,\mu,\eta ) $ and $ ( m',\mu',\eta' ) $ that
preserve multiplication and the unit as expressed by the
following commuting diagrams
\begin{center}
  \begin{tikzpicture}
    \begin{scope}
      \node (mm) at (0,2) {$ m \otimes m $};
      \node (m'm') at (3,2) {$ m' \otimes m' $};
      \node (m) at (0,0) {$ m $};
      \node (m') at (3,0) {$ m' $};
      \draw [cd]
        (mm) edge[] node[above]{$ f \otimes f $} (m'm')
        (mm) edge[] node[left]{$ \mu $} (m)
        (m) edge[] node[below]{$ f $} (m')
        (m'm') edge[] node[right]{$ \mu' $} (m');
      \end{scope}
      \begin{scope}[shift={(6,0)}]
      \node (I) at (0,2) {$ I $};
      \node (m) at (2,2) {$ m $};
      \node (m') at (2,0) {$ m' $};
      \draw [cd]
      (I) edge[] node[above]{$ \eta $} (m)
      (I) edge[] node[below]{$ \eta' $} (m')
      (m) edge[] node[right]{$ f $} (m'); 
    \end{scope}
    \node () at (4,1) {and};
  \end{tikzpicture}
\end{center}
\end{definition}

We can also provide an internal monoid with a commutative
structure.

\begin{definition}[Internal commutative monoid]
  \label{def:intern-comm-monoid}
  Given a symmetric monoidal category $ ( \C,\otimes,I,\tau ) $
  where $ \tau $ is the twist map, a \defn{commutative
    monoid internal to $ \C $} is, first, a monoid internal
  to $ \C $ with the additional property that the diagram
  \begin{center}
    \begin{tikzpicture}
      \node (mm1) at (0,2) {$ m \otimes m $};
      \node (mm2) at (2,2) {$ m \otimes m $};
      \node (m) at (1,0) {$ m $};
      \draw [cd] 
        (mm1) edge[] node[above]{$ \tau $} (mm2)
        (mm1) edge[] node[below]{$ \mu $} (m)
        (mm2) edge[] node[below]{$ \mu $} (m);      
    \end{tikzpicture}
  \end{center}
  commutes
\end{definition}

Algebraic structures often have dual counterparts, and
internal monoids are no exception.

\begin{definition}[Internal comonoid]
  Given a monoidal category $ ( \C , \otimes , I ) $, a
  \defn{comonoid internal to $ \C $} is a monoid internal to
  $ \C\op $. If $ ( \C , \otimes , I , \tau ) $ is a
  symmetric monoidal category, then a \defn{cocommutative
    comonoid internal to $ \C $} is a cocommutative
  comonoid internal to $ \C\op $.   
\end{definition}

In other words, we define comonoids exactly as we did
monoids in Definitions \ref{def:intern-monoid} and
\ref{def:intern-comm-monoid} except we turn the arrows
around.  Many familiar algebraic objects can be exhibited as
monoids internal to select categories.

\begin{example}

  A monoid internal to $ \Set $ is an ordinary monoid. A
  monoid internal to the category $ \Ab $ of abelian groups
  is a ring. A monoid internal to a category $ [\C,\C] $ of
  endofunctors is a monad on $ \C $.
\end{example}

As in algebra, objects can have multiple structures
simultaneously. The most important for us is the Frobenius
monoid.

\begin{definition}[Frobenius monoid]
  \label{def:frobenius-monoid}
  An object $ (m,\mu,\eta,\delta,\epsilon) $ in a monoidal
  category $ ( \C, \otimes , I ) $ is called a
  \defn{Frobenius monoid} if $ ( m,\mu,\eta ) $ is a monoid
  object, $ ( m,\delta,\epsilon ) $ is a comonoid structure
  and the equation
  \[
    (\id \otimes \mu)
      ( \delta \otimes \id )
    = \delta \mu
    = ( \mu \otimes \id )
      ( \id \otimes \delta ).
   \]
    holds.
\end{definition}

Internalization can be extended to constructions beyond
monoids and their variants.  The most important construction
for us is the internalization of a category.

\begin{definition}[Internal category]

  Let $ \D $ be a category. A \defn{category $ \CCC $
    internal to $ \D $} consists of the data
  \begin{itemize}
  \item an object $ \CCC_0 \in \ob (\D) $ of \emph{objects} of $ \CCC $
  \item an object $ \CCC_1 \in \ob (\D) $ of \emph{arrows} of $ \CCC $
  \item source and target arrows $ s,t \from \CCC_1 \to \CCC_0
    $ in $ \D $
  \item an identity arrow $ e \from \CCC_0 \to \CCC_1 $ in $ \D $
  \item a composition arrow $ \circ \from \CCC_1 \times_{\CCC_0}
    \CCC_1 \to \CCC_1 $    
  \end{itemize}
  together with commuting diagrams
  \begin{itemize}
  \item that specify the source and target of the identity
    arrow
    \[
      \begin{tikzpicture}
        \begin{scope}
          \node (1) at (0,2) {$ \CCC_0 $};
          \node (2) at (2,2) {$ \CCC_1 $};
          \node (3) at (2,0) {$ \CCC_0 $};
          \draw [cd] (1) to node[above]{$ e $}  (2);
          \draw [cd] (1) to node[left]{$ \id $} (3);
          \draw [cd] (2) to node[right]{$ s $}  (3);
        \end{scope}
        \begin{scope}[shift={(6,0)}]
          \node (1) at (0,2) {$ \CCC_0 $};
          \node (2) at (2,2) {$ \CCC_1 $};
          \node (3) at (2,0) {$ \CCC_0 $};
          \draw [cd] (1) to node[above]{$ e $}  (2);
          \draw [cd] (1) to node[left]{$ \id $} (3);
          \draw [cd] (2) to node[right]{$ t $}  (3);
        \end{scope}
      \end{tikzpicture}
    \]
  \item that specify the source and target of composite
    arrows
    \[
      \begin{tikzpicture}
        \begin{scope}
          \node (1) at (0,2) {$ \CCC_1 \times_{\CCC_0} \CCC_1 $};
          \node (2) at (3,2) {$ \CCC_1 $};
          \node (3) at (0,0) {$ \CCC_1 $};
          \node (4) at (3,0) {$ \CCC_0 $};
          \draw [cd] (1) to node[above]{$ \circ $} (2);
          \draw [cd] (1) to node[left]{$ p_1 $}    (3);
          \draw [cd] (2) to node[right]{$ s $} (4);
          \draw [cd] (3) to node[below]{$ s $} (4);
        \end{scope}
        \begin{scope}[shift={(6,0)}]
          \node (1) at (0,2) {$ \CCC_1 \times_{\CCC_0} \CCC_1 $};
          \node (2) at (3,2) {$ \CCC_1 $};
          \node (3) at (0,0) {$ \CCC_1 $};
          \node (4) at (3,0) {$ \CCC_0 $};
          \draw [cd] (1) to node[above]{$ \circ $} (2);
          \draw [cd] (1) to node[left]{$ p_2 $} (3);
          \draw [cd] (2) to node[right]{$ t $} (4);
          \draw [cd] (3) to node[below]{$ t $} (4);
        \end{scope}
      \end{tikzpicture}
    \]
  \item that specify associativity
    \[
      \begin{tikzpicture}
        \node (1) at (0,2)
          {$ \CCC_1 \times_{\CCC_0} \CCC_1 \times_{\CCC_0} \CCC_1 $};
        \node (2) at (4,2) {$ \CCC_1 \times_{\CCC_0} \CCC_1 $};
        \node (3) at (0,0) {$ \CCC_1 \times_{\CCC_0} \CCC_1 $};
        \node (4) at (4,0) {$ \CCC_1 $};
        \draw[cd]
          (1) to node[above]{$ \circ \times_{\CCC_0} \id $} (2)
          (1) to node[left]{$ \id \times_{\CCC_0} \circ $} (3)
          (2) to node[right]{$ \circ $} (4)
          (3) to node[below]{$ \circ $} (4);
      \end{tikzpicture}
    \]
  \item that specify unit laws
    \[
      \begin{tikzpicture}
        \node (1) at (-4,2) {$ \CCC_0 \times_{\CCC_0} \CCC_1 $};
        \node (2) at (0,2) {$ \CCC_1 \times_{\CCC_0} \CCC_1 $};
        \node (3) at (4,2) {$ \CCC_1 \times_{\CCC_0} \CCC_0 $};
        \node (4) at (0,0) {$ \CCC_0 $};
        \draw [cd]
          (1) to node[above]{$ e \times_{\CCC_0} \id $} (2)
          (3) to node[above]{$ \id \times_{\CCC_0} e $} (2)
          (1) to node[below]{$ p_2 $} (4)
          (3) to node[below]{$ p_1 $} (4);
      \end{tikzpicture}
    \]
  \end{itemize}
  If we are instead working in an ambient 2-category $ \D $
  and the diagrams only commute up to natural isomorphism,
  we say that $ \CCC $ is \defn{weakly internal} to $ \D
  $. 
\end{definition}

The most important example of an internal category for us is
a (pseudo) double category. A \defn{(pseudo) double
  category} $ \CCC $ is a category weakly internal to
$ \Cat $. This can be unpacked.

Roughly, a double category consists of two categories $
\CCC_0 $ and $ \CCC_1 $ that we consider as follows.
\begin{itemize}
\item The $ \CCC_0 $-objects are called the
  objects of $ \CCC $.
\item The $ \CCC_0 $-arrows are called the
  vertical arrows in $ \CCC $.
\item The $ \CCC_1 $-objects are called the the
  horizontal arrows in $ \CCC $.
\item The $ \CCC_1 $-arrows are called the squares
  of $ \CCC $. 
\end{itemize}

This data is depicted in Figure \ref{fig:square}. When the
vertical arrows are both identities, we call the square
\defn{globular}.

Double categories often arise when a mathematical object
has two different sorts of morphisms. One morphism type
becomes the horizontal arrows, which we denote by $ \hto $,
and the other morphism type becomes the vertical arrows,
which we denote by $ \to $. 

\begin{example}
  \label{ex:sets-rel-dbl-cat}
  There is a double category whose objects are sets,
  vertical arrows $ f \from x \to y $ are functions, horizontal
  arrows $r \from x \hto y $ are relations $ r \subseteq x
  \times y $, and squares
  \[
    \begin{tikzpicture}
      \node (1) at (0,2) {$ x $};
      \node (2) at (2,2) {$ y $};
      \node (3) at (0,0) {$ x' $};
      \node (4) at (2,0) {$ y' $};
      \draw [cd] (1) to node[above]{$ r $} (2);
      \draw [cd] (1) to node[left]{$ f $} (3);
      \draw [cd] (2) to node[right]{$ g $} (4);
      \draw [cd] (3) to node[below]{$ s $} (4);
      \node () at (1,1) {$ \Downarrow $};
    \end{tikzpicture}
  \]
  are inclusions of relations $ gr \subseteq sf  $.
\end{example}

\begin{figure}[h]
  \centering
  \fbox{
  \begin{minipage}{\linewidth}
  \[
  \begin{tikzpicture}
    \node (1) at (0,2) {\( c \)};
    \node (2) at (2,2) {\( d \)};
    \node (3) at (0,0) {\( c' \)};
    \node (4) at (2,0) {\( d' \)};
    \draw [cd,-|->] (1) to node[above]{\( m \)} (2);
    \draw [cd] (1) to node[left]{\( f \)} (3);
    \draw [cd] (2) to node[right]{\( g \)} (4);
    \draw [-|->] (3) to node[below]{\( n \)} (4);
    \node (5) at (6,2) {\( c,c',d,d' \in \ob ( \C_0 ) \)};
    \node (6) at (6,1.33) {\( f,g \in \arr ( \C_0 ) \)};
    \node (7) at (6,.66) {\( m,n \in \ob ( \C_1 ) \)};
    \node (8) at (1,1) {\( \Downarrow  \theta \)};
    \node (9) at (6,0) {\( \theta \in \arr ( \C_1 ) \)};
  \end{tikzpicture}
\]
\caption{A square in a double category}
\label{fig:square}
\end{minipage}
}
\end{figure}

The first definition for a double category we gave---a
category weakly internal to $ \Cat $---is too terse to
provide much meaningful interpretation. So we unpack it.

\begin{definition}[Double category]
  \label{def:dbl-cat}
  A \defn{pseudo double category} $\CCC$, or simply
  \defn{double category}, consists of a category of objects
  $\CCC_{0}$ and a category of arrows $\CCC_{1}$
  together with the following functors
  \begin{align*}
    U     & \from \CCC_{0} \to \CCC_{1}, \\
    S,T   & \from \CCC_{1} \to \CCC_{0},  \\
    \odot & \from \CCC_{1} \times_{\CCC_{0}} \CCC_{1} \to \CCC_{1}
  \end{align*}
  where the pullback $\CCC_{1} \times_{\CCC_{0}} \CCC_{1}$
  is taken over $S$ and $T$.  These functors satisfy the
  equations
  \begin{align}
    SUa = a & = TUa \\
    S(x \odot y) & = Sy \\
    T(x \odot y) & = Tx.
  \end{align}
  This also comes equipped with natural
  isomorphisms
  \begin{align}
    \alpha  & \from (x \odot y) \odot z \to
              x \odot (y \odot z) \\
    \lambda & \from Ua \odot x \to x \\
    \rho    & \from x \odot Ua \to x
  \end{align}
  such that $S(\alpha)$, $S(\lambda)$, $S(\rho)$,
  $T(\alpha)$, $T(\lambda)$, and $T(\rho)$ are
  each identities and that the coherence axioms of
  a monoidal category are satisfied.\footnote{
    Sometimes the term \defn{horizontal
      1-cell} is used for these
    \cite{shulman_contructing}, and for good
    reason. A $(n \times 1)$-category consists of
    categories $\mathbf{D_i}$ for
    $0 \leq i \leq n$ where the objects of
    $\mathbf{D_i}$ are $i$-cells and the morphisms
    of $\mathbf{D_i}$ are vertical
    $i+1$-morphisms. A double category is then
    just a $(1 \times 1)$-category. From this
    perspective, `cells' are always objects with
    morphisms going between them.}

  As for notation, we write vertical and
  horizontal morphisms with the arrows $\to$ and
  $\hto$, respectively, and 2-morphisms we draw as
  in Figure \ref{fig:square}.

  One can define double functors and double transformations,
  but we refrain having no need of them in this thesis.
  Double categories, double functors, and double
  transformations form a 2-category $ \DblCat $.
\end{definition}

Like categories, we can equip double categories with
additional structure. We focus on adding a monoidal
structure. As is typical in category theory, we can provide
definitions at various levels of abstraction.  As such, a
symmetric monoidal double category is a monoid weakly
internal to $ \DblCat $.  This uses the same definition of a
monoid internal to a category $ \D $ as above, though the
diagrams commute up to invertible transformation. It is
worth unpacking this definition.

\begin{definition}[Monoidal double category] \label{def:mndl-dbl-cats}
  A \defn{monoidal double category} $ ( \CCC, \otimes ) $ is a double category
  $\CCC$ equipped with a functor $ \otimes \from \CCC \times
  \CCC \to \CCC $ such that
  \begin{enumerate}
  \item $\CCC_{0}$ and $\CCC_{1}$ are
    both monoidal categories.
  \item If $I$ is the monoidal unit of
    $\CCC_{0}$, then $U_I$ is the monoidal
    unit of $\CCC_{1}$.
  \item The functors $S$ and $T$ are strict
    monoidal and preserve the associativity and
    unit constraints.
  \item There are globular 2-isomorphisms
    \[ 
      \mathfrak{x} \from 
      (x \otimes y) \odot (x' \otimes y') 
      \to 
      (x \odot x') \otimes (y \odot y')
    \]
    and
    \[
      \mathfrak{u} \from 
      U(a \otimes b) 
      \to 
      Ua \otimes Ub
    \]
  
  \item \label{diag:MonDblCat}

    The following diagrams that express the constraint data
    for the double functor $\otimes$ commute
    \[
      \begin{tikzpicture}[scale=0.9]
        \node (A) at (0,3) {\scriptsize{
          $((x \otimes y )\odot (x'\otimes y')) \odot (x'' \otimes y'')$}
          };
        \node (B) at (9,3) {\scriptsize{
          $((x \odot x')\otimes (y \odot y')) \odot (x'' \otimes y'') $}
          };
        \node (A') at (0,1.5) {\scriptsize{
          $(x \otimes y) \odot ((x' \otimes y') \odot (x''\otimes y'')) $}
          };
        \node (B') at (9,1.5) {\scriptsize{
          $((x \odot x')\odot x'') \otimes ((y\odot y')\odot y'')$}
          };
        \node (A'') at (0,0) {\scriptsize{
          $(x\otimes y) \odot ((x'\odot x'') \otimes (y'\odot y''))$}
          };
        \node (B'') at (9,0) {\scriptsize{
          $(x\odot (x'\odot x'')) \otimes (y\odot (y' \odot y''))$}
          };
        \path[cd,font=\scriptsize]
        (A) edge node[left]{$\alpha$} (A')
        (A') edge node[left]{$1 \odot \mathfrak{x}$} (A'')
        (B) edge node[right]{$\mathfrak{x}$} (B')
        (B') edge node[right]{$\alpha \otimes \alpha$} (B'')
        (A) edge node[above]{$\mathfrak{x} \odot 1$} (B)
        (A'') edge node[above]{$\mathfrak{x}$} (B'');
      \end{tikzpicture}
    \]
    \[
      \begin{tikzpicture}[scale=0.9]
        \node (UL) at (0,1.5) {\scriptsize{
          $(x\otimes y) \odot U(a\otimes b)$}
          };
        \node (LL) at (0,0) {\scriptsize{
          $ x \otimes y$}
          };
        \node (UR) at (4.5,1.5) {\scriptsize{
          $(x\otimes y)\odot (Ua\otimes Ub)$}
          };
        \node (LR) at (4.5,0) {\scriptsize{
          $(x\odot Ua) \otimes (y\odot Ub)$}
          };
        \path[cd,font=\scriptsize]
        (UL) edge node[above]{$1 \odot \mathfrak{u}$} (UR) 
        (UL) edge node[left]{$\rho$} (LL)
        (LR) edge node[above]{$\rho \otimes \rho$} (LL)
        (UR) edge node[right]{$\mathfrak{x}$} (LR);
      \end{tikzpicture}
      \begin{tikzpicture}[scale=0.9]
        \node (UL) at (0,1.5) {\scriptsize{$U(a\otimes b)\odot (x\otimes y)$}};
        \node (LL) at (0,0) {\scriptsize{$x\otimes y$}};
        \node (UR) at (4.5,1.5) {\scriptsize{$(Ua \otimes Ub)\odot (x\otimes y)$}};
        \node (LR) at (4.5,0) {\scriptsize{$(Ua \odot x) \otimes (Ub \odot y)$}};
        \path[cd,font=\scriptsize]
        (UL) edge node[above]{$\mathfrak{u} \odot 1$} (UR) 
        (UL) edge node[left]{$\lambda$} (LL)
        (LR) edge node[above]{$\lambda \otimes \lambda$} (LL)
        (UR) edge node[right]{$\mathfrak{x}$} (LR);
      \end{tikzpicture}
    \]
  \item The following diagrams commute expressing
    the associativity isomorphism for $\otimes$ is
    a transformation of double categories.
    \[
      \begin{tikzpicture}[scale=0.9]
        \node (A) at (0,3) {\scriptsize{
          $((x\otimes y) \otimes z) \odot ((x' \otimes y')\otimes z')$}
          };
        \node (B) at (9,3) {\scriptsize{
            $(x\otimes (y\otimes z)) \odot (x'\otimes (y'\otimes z'))$}
        };
        \node (A') at (0,1.5) {\scriptsize{
            $((x\otimes y) \odot (x'\otimes y')) \otimes (z\odot z')$}
        };
        \node (B') at (9,1.5) {\scriptsize{
            $(x\odot x') \otimes ((y\otimes z)\odot (y'\otimes z'))$}
        };
        \node (A'') at (0,0) {\scriptsize{
            $((x\odot x') \otimes(y\odot y')) \otimes (z\odot z')$}
        };
        \node (B'') at (9,0) {\scriptsize{
            $(x\odot x') \otimes ((y\odot y')\otimes (z\odot z'))$}
        };
        \path[cd,font=\scriptsize]
        (A) edge node[left]{$\mathfrak{x}$} (A')
        (A') edge node[left]{$\mathfrak{x} \otimes 1$} (A'')
        (B) edge node[right]{$\mathfrak{x}$} (B')
        (B') edge node[right]{$1 \otimes \mathfrak{x}$} (B'')
        (A) edge node[above]{$a \odot a$} (B)
        (A'') edge node[above]{$a$} (B'');
      \end{tikzpicture}
    \]
    \[
      \begin{tikzpicture}
        \node (A) at (0,3)
          {\scriptsize{ $ U ( (a \otimes b) \otimes c ) $ }};
        \node (B) at (4,3)
          {\scriptsize{ $ U ( a \otimes (b \otimes c) ) $ }};
        \node (A') at (0,1.5) {\scriptsize{$U(a\otimes b) \otimes Uc $}};
        \node (B') at (4,1.5) {\scriptsize{$Ua \otimes U(b\otimes c)$}};
        \node (A'') at (0,0) {\scriptsize{$(Ua \otimes Ub)\otimes Uc$}};
        \node (B'') at (4,0) {\scriptsize{$Ua\otimes (Ub\otimes Uc) $}};
        \path[cd,font=\scriptsize]
        (A) edge node[left]{$\mathfrak{u}$} (A')
        (A') edge node[left]{$\mathfrak{u} \otimes 1$} (A'')
        (B) edge node[right]{$\mathfrak{u}$} (B')
        (B') edge node[right]{$\id \otimes \mathfrak{u}$} (B'')
        (A) edge node[above]{$Ua$} (B)
        (A'') edge node[above]{$a$} (B'');
      \end{tikzpicture}
    \]
  \item The following diagrams commute expressing
    that the unit isomorphisms for $\otimes$ are
    transformations of double categories.
    \[
      \begin{tikzpicture}
        \node (A) at (0,1.5) {\scriptsize{$(x\otimes UI)\odot (y\otimes UI)$}};
        \node (A') at (0,0) {\scriptsize{$x\odot y$}};
        \node (B) at (4,1.5) {\scriptsize{$(x\odot y)\otimes (UI \odot UI) $}};
        \node (B') at (4,0) {\scriptsize{$(x\odot y)\otimes UI $}};
        \path[cd,font=\scriptsize]
        (A) edge node[left]{$r \odot r$} (A')
        (A) edge node[above]{$\mathfrak{x}$} (B)
        (B) edge node[right]{$1 \otimes \rho$} (B')
        (B') edge node[above]{$r$} (A');
      \end{tikzpicture}
      \quad
      \begin{tikzpicture}
        \node (A) at (0,0.75) {\scriptsize{$U(a \otimes I) $}};
        \node (B) at (1.5,1.5) {\scriptsize{$Ua \otimes UI $}};
        \node (B') at (1.5,0) {\scriptsize{$Ua$}};
        \path[cd,font=\scriptsize]
        (A) edge node[above]{$\mathfrak{u}$} (B)
        (A) edge node[below]{$U_{r}$} (B')
        (B) edge node[right]{$r$} (B');
      \end{tikzpicture}
    \]
    \[
      \begin{tikzpicture}
        \node (A) at (0,1.5) {\scriptsize{$(UI \otimes x)\odot (UI\otimes y)$}};
        \node (A') at (0,0) {\scriptsize{$x\odot y$}};
        \node (B) at (4,1.5) {\scriptsize{$(UI \odot UI) \otimes (x\odot y)$}};
        \node (B') at (4,0) {\scriptsize{$UI\otimes (x \odot y) $}};
        \path[cd,font=\scriptsize]
        (A) edge node[left]{$\ell \odot \ell$} (A')
        (A) edge node[above]{$\mathfrak{x}$} (B)
        (B) edge node[right]{$\lambda \otimes 1$} (B')
        (B') edge node[above]{$\ell$} (A');
      \end{tikzpicture}
      \quad
      \begin{tikzpicture}
        \node (A) at (0,0.75) {\scriptsize{$U(I\otimes a)$}};
        \node (B) at (1.5,1.5) {\scriptsize{$UI \otimes Ua$}};
        \node (B') at (1.5,0) {\scriptsize{$Ua$}};
        \path[cd,font=\scriptsize]
        (A) edge node[above]{$\mathfrak{u}$} (B)
        (A) edge node[below]{$U_{\ell}$} (B')
        (B) edge node[right]{$\ell$} (B');
      \end{tikzpicture}
    \]
    \newcounter{mondbl}
    \setcounter{mondbl}{\value{enumi}}
  \end{enumerate}
  A \textbf{braided monoidal double category} is a
  monoidal double category such that:
  \begin{enumerate}
    \setcounter{enumi}{\value{mondbl}}
  \item $\CCC_{0}$ and $\CCC_{1}$ are braided monoidal categories.
  \item The functors $S$ and $T$ are strict braided monoidal functors.
  \item The following diagrams commute expressing that the braiding is a transformation of double categories.
    \[
      \begin{tikzpicture}
        \node (A) at (0,1.5) {\scriptsize{$(x \odot x') \otimes (y \odot y)$}};
        \node (A') at (0,0) {\scriptsize{$(x\otimes y) \odot (x'\otimes y')$}};
        \node (B) at (5,1.5) {\scriptsize{$(y\odot y') \otimes (x \odot x')$}};
        \node (B') at (5,0) {\scriptsize{$(y \otimes x) \odot (y' \otimes x')$}};
        \path[cd,font=\scriptsize]
        (A) edge node[left]{$\mathfrak{x}$} (A')
        (A) edge node[above]{$\mathfrak{s}$} (B)
        (B) edge node[right]{$\mathfrak{x}$} (B')
        (A') edge node[above]{$\mathfrak{s} \odot \mathfrak{s}$} (B');
      \end{tikzpicture}
      \quad
      \begin{tikzpicture}
        \node (A) at (0,1.5) {\scriptsize{$Ua \otimes Ub$}};
        \node (A') at (0,0) {\scriptsize{$Ub \otimes Ua$}};
        \node (B) at (2,1.5) {\scriptsize{$U(a \otimes b) $}};
        \node (B') at (2,0) {\scriptsize{$U(b \otimes a)$}};
        \path[cd,font=\scriptsize]
        (A) edge node[left]{$\mathfrak{s}$} (A')
        (B) edge node[above]{$\mathfrak{u}$} (A)
        (B) edge node[right]{$U_\mathfrak{s}$} (B')
        (B') edge node[above]{$\mathfrak{u}$} (A');
      \end{tikzpicture}
    \]
    \setcounter{mondbl}{\value{enumi}}
  \end{enumerate}
  Finally, a \textbf{symmetric monoidal double category} 
  is a braided monoidal double category $\CCC$ such that
  \begin{enumerate}
    \setcounter{enumi}{\value{mondbl}}
  \item $\CCC_{0}$ and $\CCC_{1}$ are symmetric monoidal.
  \end{enumerate}
\end{definition}
 
In Example \ref{ex:sets-rel-dbl-cat}, we saw a double
category whose vertical arrows are functions and horizontal
arrows are relations. But, functions are examples of
relations. So in a sense, the vertical arrows are redundant
because that information is contained in the horizontal
arrows.  The next definitions formalizes this observation.

\begin{definition}[Companion and conjoint]
  \label{def:CompanionConjoint}
  Let $\CCC$ be a double category and $f \from a \to b$
  a vertical arrow.  A \textbf{companion} of $f$ is a
  horizontal arrow $\widehat{f} \from a \hto b$
  together with squares
  \[
    \begin{tikzpicture}
      \begin{scope}
      \node (A) at (0,2) {$a$};
      \node (B) at (2,2) {$b$};
      \node (A') at (0,0) {$b$};
      \node (B') at (2,0) {$b$};
      \draw[cd]
      (A) edge node[above]{$\widehat{f}$} (B)
      (A) edge node[left]{$f$} (A')
      (B) edge node[right]{$\id$} (B')
      (A') edge node[below]{$Ub$} (B');
      \draw (1,1.925) -- (1,2.075);
      \draw (1,-0.075) -- (1,0.075);
      \node at (1,1) {\scriptsize{$\Downarrow$}};
      \end{scope}
      \node at (4,1) {and};
      \begin{scope}[shift={(6,0)}]
        \node (A) at (0,2) {$a$};
        \node (B) at (2,2) {$a$};
        \node (A') at (0,0) {$a$};
        \node (B') at (2,0) {$b$};
        \draw[cd]
        (A) edge node[above]{$Ua$} (B)
        (A) edge node[left]{$a$} (A')
        (B) edge node[right]{$f$} (B')
        (A') edge node[below]{$\widehat{f}$} (B');
        \draw (1,1.925) -- (1,2.075);
        \draw (1,-0.075) -- (1,0.075);
        \node at (1,1) {\scriptsize{$\Downarrow$}};
      \end{scope}
    \end{tikzpicture}
  \]
  such that the following equations hold:
  \begin{align} \label{eq:CompanionEq}
    &
    \begin{tikzpicture}
    \begin{scope}
      \node (A) at (0,4) {$a$};
      \node (B) at (2,4) {$a$};
      \node (A') at (0,2) {$a$};
      \node (B') at (2,2) {$b$};
      \node (A'') at (0,0) {$b$};
      \node (B'') at (2,0) {$b$};
      \draw[cd]
      (A)   edge node[left]{$\id$} (A')
      (A')  edge node[left]{$f$} (A'')
      (B)   edge node[right]{$f$} (B')
      (B')  edge node[right]{$\id$} (B'')
      (A)   edge node[above]{$Ua$} (B)
      (A')  edge node[above]{$\hat{f}$}  (B')
      (A'') edge node[below]{$Ub$} (B'');
      \draw (1,3.925) -- (1,4.075);
      \draw (1,1.925) -- (1,2.075);
      \draw (1,-.075) -- (1,.075);
      \node () at (1,1) {{$\Downarrow$}};
      \node () at (1,3) {{$\Downarrow$}};
    \end{scope}
    \node () at (3,2) {$ = $};
    \begin{scope}[shift={(4,1)}]
      \node (A) at (0,2) {$a$};
      \node (B) at (2,2) {$a$};
      \node (A') at (0,0) {$b$};
      \node (B') at (2,0) {$b$};
      \draw[cd]
        (A) edge node[left]{$f$} (A')
        (B) edge node[right]{$f$} (B')
        (A) edge node[above]{$Ua$} (B)
        (A') edge node[below]{$Ub$} (B');
      \draw (1,1.925) -- (1,2.075);
      \draw (1,-.075) -- (1,.075);
      \node at (1,1) {$\Downarrow Uf$};
    \end{scope}
    \end{tikzpicture}
    \\
    &
    \begin{tikzpicture}
    \begin{scope}[shift={(0,0)}]
        \node (A) at (0,2) {$a$};
        \node (A') at (0,0) {$a$};
        \node (B) at (2,2) {$a$};
        \node (B') at (2,0) {$b$};
        \node (C) at (4,2) {$b$};
        \node (C') at (4,0) {$b$};
        \draw[cd]
        (A) edge node[left]{$\id$} (A')
        (B) edge node[left]{$f$} (B')
        (C) edge node[right]{$\id$} (C')
        (A) edge node[above]{$Ua$} (B)
        (B) edge node[above]{$\widehat{f}$} (C)
        (A') edge node[below]{$\widehat{f}$} (B')
        (B') edge node[below]{$Ub$} (C');
        \draw (3,1.925) -- (3,2.075);
        \draw (1,1.925) -- (1,2.075);
        \draw (3,-0.075) -- (3,0.075);
        \draw (1,-0.075) -- (1,0.075);
        \node () at (1,1) {{$\Downarrow$}};
        \node () at (3,1) {{$\Downarrow$}};
      \end{scope}
      \node () at (5,1) {$ = $};
      \begin{scope}[shift={(6,0)}]
        \node (A) at (0,2) {$a$};
        \node (B) at (2,2) {$b$};
        \node (A') at (0,0) {$a$};
        \node (B') at (2,0) {$b$};
        \draw[cd]
        (A) edge node[left]{$\id$} (A')
        (B) edge node[right]{$b$} (B')
        (A) edge node[above]{$\widehat{f}$} (B)
        (A') edge node[below]{$\widehat{f}$} (B');
        \draw (0.5,.925) -- (0.5,1.075);
        \draw (0.5,.925) -- (0.5,1.075);
        \draw (0.5,-.075) -- (0.5,.075);
        \node () at (1,1) {$\Downarrow \id_{\hat{f}}$};
        \end{scope}
      \end{tikzpicture}
  \end{align}
  A \textbf{conjoint} of $f$, denoted 
  $\check{f} \from b \hto a$, 
  is a companion of $f$ in the double category 
  $\CCC^{h\cdot\mathrm{op}}$ 
  obtained by reversing the horizontal 1-morphisms, 
  but \emph{not} the vertical 1-morphisms.
\end{definition}

%
\begin{definition}[Fibrant double category]
  \label{def:Fibrant}
  We say that a double category is \defn{fibrant} 
  if every vertical 1-morphism has 
  both a companion and a conjoint. 
  If every \emph{invertible} vertical 1-morphism 
  has both a companion and a conjoint, 
  then we say the double category is \defn{isofibrant}.
\end{definition}

In some sense, a double category is more than a
bicategory. One might believe that there is some way to
extract a bicategory from a double category. In fact you can.

\begin{definition}[Horizontal edge bicategory]
  \label{def:horiz-bicat}
  Given a double category $\CCC$, the \defn{horizontal
    edge bicategory} $H(\CCC)$ of $\CCC$ is the bicategory
  whose objects are those of $\CCC$, arrows are horizontal
  arrows of $\CCC$, and $2$-arrows are the globular
  squares.
\end{definition}

Even though we can turn any double category into a
bicategory by throwing out the vertical arrows, what becomes of
double categories with additional structure? The next
theorem partially answers this puzzle.

\begin{theorem}[{\cite[Theorem 5.1]{shulman_contructing}}]
  \label{thm:horz-bicat}
    Let $\CCC$ be an isofibrant symmetric
    monoidal double category.  Then $H(\CCC)$
    is a symmetric monoidal bicategory.
\end{theorem}

The wonderful thing about this theorem is that the axioms
for the symmetric monoidal bicategory definition are
typically much harder to check than the axioms for symmetric
monoidal double category, and so it provides a streamlined
way to construct a symmetric monoidal bicategory.

\section{Bicategories of relations}
\label{sec:cartesian-bicategories}

In the early days of bicategory theory, when concerned
mathematicians were exploring additional structures placed
on bicategories, they discovered that the coherence involved
tended to be convoluted. And so they did what mathematicians
typically do, restrict their considerations to a more
manageable case.

Looking at the definition of a monoidal bicategory, one is
confronted with many diagrams commuting.  By
placing certain restrictions on the type of 2-arrows in your
monoidal bicategory, this coherence is greatly simplified. The
particular case we are interested in comes when the tensor
$ \otimes $ behaves like a product in the sense that there
is a diagonal arrow $ \Delta_x \from x \to x \otimes x $ and a
terminal object $ I $ (the empty product a.k.a.~the unit for
product). A motivating example comes from studying relations.

Relations are pervasive throughout mathematics.
They play an central role in the theory of
rewriting as evidenced through the importance of the rewriting
relation. Classically, a relation is thought of as a
subset of a product of sets $ R \subseteq A \times
B $. This set-theoretic perspective on
relations has a category-theoretic
counterpart. Given any category $ \C $, we can
talk about relations \emph{internal} to $ \C $.
To foster our intuition, we first look at
relations internal to $ \Set $.

\begin{example}
  A relation internal to $ \Set $ from $ x $ to $ y $ is a
  subobject $ r \monicto x \times y $. Set-theoretically
  speaking, $ r $ is a subset of $ x \times y $. This
  matches the classical notion of relation.
\end{example}

However, defining a relation internal to a category $ \C $
as a subobject of a binary product is poor form.  Not all
categories have products. Hence the following definition is
given.

\begin{definition}[Internal relation]
  \label{def:internal-relation}
  A \defn{relation internal to a category $ \C $},
  denoted $ x \rel y $ for $ x,y \in \ob ( \C ) $, is a
  jointly monic span
  \[
    x \xgets{f} r \xto{g} y.
  \]
  That is, for any pair of arrows $ f',g' \from u \to r $
  such that $ ff'=fg' $ and $ gf'=gg' $, then $ f'=g'
  $. When $ \C $ has binary products, this is equivalent to
  the pairing $ \langle f,g \rangle \from r \to x \times y $
  being a monomorphism.
\end{definition}

The categorical minded mathematician might see this and ask
if we can construct category from the objects of $ \C $ and
its internal relations.  If $ \C $ is a topos, then the
answer is yes. This is not the broadest class of categories
for which this construction works, but the class of topoi is
as broad as we can go without writing another section of this
appendix.  Given a topos $ \C $, there is a
category $ \Rel (\C) $ called \defn{the category of
  relations internal to $ \C $}. Its objects are those of
$ \C $ and arrows $ \Rel (\C) (x,y) $ are internal relations
$ x \gets r \to y $.  Composition is given by pullback
\[
  \begin{tikzpicture}
    \node () at (-1.5,1)
      {$ ( x \rel y \rel z ) \xmapsto{\circ} $};
    \node (x) at (0,0) {$ x $};
    \node (y) at (2,0) {$ y $};
    \node (z) at (4,0) {$ z $};
    \node (r) at (1,1) {$ r $};
    \node (s) at (3,1) {$ s $};
    \node (rys) at (2,2) {$ r \times_y s $};
    \draw[cd]
    (r)   edge node[]{$  $} (x)
    (r)   edge node[]{$  $} (y)
    (s)   edge node[]{$  $} (y)
    (s)   edge node[]{$  $} (z)
    (rys) edge node[]{$  $} (r)
    (rys) edge node[]{$  $} (s);
    \draw (1.9,1.7) -- (2,1.6) -- (2.1,1.7);
  \end{tikzpicture} 
\]

In fact, $ \Rel (\C) $ can be promoted to a bicategory $
\RRel (\C) $ by taking as 2-arrows maps of
spans. Specifically, a 2-arrow between internal relations $
x \gets r \to y $ to $ x \gets s \to y $ is an arrow $ f \from r
\to s $ of $ \C $ fitting into the commuting diagram
\[
  \begin{tikzpicture}
    \node (x) at (0,0) {$ x $};
    \node (y) at (4,0) {$ y $};
    \node (s) at (2,1) {$ s $};
    \node (r) at (2,2) {$ r $};
    \draw[cd]
    (r) edge node[]{$  $} (x)
    (r) edge node[]{$  $} (y)
    (s) edge node[]{$  $} (x)
    (s) edge node[]{$  $} (y)
    (r) edge node[]{$  $} (s);
  \end{tikzpicture}
\]
It follows from the jointly monic condition that given any other
arrow $ g \from r \to s $ fitting into the above diagram, it
follows that $ f=g $. The parallel between relations in
$ \Set $ is clear: a morphism of relations is like a subset
inclusion.

\begin{remark}
  There is a name to the property of $ \RRel (\C) $ that
  between parallel arrows, either a single 2-arrow exists or
  none does. It is called being \defn{locally posetal}.  Another
  way of saying this is that $ \RRel (\C) $ is a category
  enriched in $ \Pos $, the category of posets and order
  preserving functions. This means that for any objects
  $ x,y $ of $ \RRel (\C) $, there is a poset
  $ \RRel ( \C ) (x,y) $ whose elements are the relations
  from $ x \rel y $ that are internal to $ \C $ and the
  ordering is defined by setting $ r \leq s $ whenever there
  is an arrow $ r \to s $ in $ \C $ such that the diagram
  \[
  \begin{tikzpicture}
    \node (x) at (0,0) {$ x $};
    \node (y) at (4,0) {$ y $};
    \node (s) at (2,1) {$ s $};
    \node (r) at (2,2) {$ r $};
    \draw[cd]
    (r) edge node[]{$  $} (x)
    (r) edge node[]{$  $} (y)
    (s) edge node[]{$  $} (x)
    (s) edge node[]{$  $} (y)
    (r) edge node[]{$  $} (s);
  \end{tikzpicture}
  \]
  commutes. Because of this, we denote 2-arrows in locally
  posetal bicategories by $ \leq $ instead of $ \To $. We
  explain enriched category theory basics in Appendix
  \ref{sec:extern-bicats}.
\end{remark}

Fix a cartesian category $ ( \T , \times , 1 ) $ with $ \T $
a topos. This cartesian structure provides $ \RRel (\T) $
with some nice structure of its own. First, there is a
tensor product in the form of a pseudo-functor
\[
  \otimes \from \RRel (\T) \times \RRel (\T) \to \RRel (\T)
\]
defined by $ (x,y) \mapsto x \times y $ where $ \times $ is
the product in $ \T $, and pointwise application of $ \times
$ on the jointly monic spans.  We also have natural
isomorphisms
\begin{itemize}
\item $ x \to x \otimes 1 $ given by the internal
relation
\[
\begin{tikzpicture}
  \node (1) at (0,0) {$ x $};
  \node (3) at (4,0) {$ x \times 1 $};
  \node (2) at (2,2) {$ x $};
  \draw[cd]
  (2) edge node[above]{$ \id $} (1)
  (2) edge node[above]{$ \iso $} (3);
\end{tikzpicture}
\]

\item $ x \otimes y \to y \otimes x $ given by the internal
  relation
  \[
\begin{tikzpicture}
  \node (1) at (0,0) {$ x \times y $};
  \node (2) at (2,2) {$ x \times y $};
  \node (3) at (4,0) {$ y \times x $};
  \draw[cd]
  (2) edge node[above]{$ \id $} (1)
  (2) edge node[above]{$ \iso $} (3);
\end{tikzpicture}
\]

\item $ (x \otimes y) \otimes z \to (x \otimes y) \otimes z
  $ given by the internal relation
  \[
    \begin{tikzpicture}
      \node (1) at (0,0) {$ (x \times y) \times z $};
      \node (2) at (2,2) {$ (x \times y) \times z $};
      \node (3) at (4,0) {$ (x \times y) \times z $};
      \draw[cd]
      (2) edge node[above]{$ \id $} (1)
      (2) edge node[above]{$ \iso $} (3);
    \end{tikzpicture}
  \]
\end{itemize}
that satisfy the required coherence conditions.  Because
$ \RRel (\T) $ is locally posetal, the 1-category coherence
laws for unity, symmetry, and associativity suffice.

Because the definition of $ \otimes $ uses the cartesian
structure on $ \T $, there is a cartesian-like quality to
$ \otimes $ in $ \RRel (\T) $.  However, 2-limits are
difficult, so we characterize this quality via
comonoids. Before talking about comonoids in $ \RRel (\T) $,
we look at comonoids in $ \T $.  Observe that by taking $ \T
$ to be cartesian, every object in $ \T $ has a
comonoid structure: the comultiplication
$ \Delta_x \from x \to x \times x $ is given by the diagonal
map and the counit $ \epsilon_x \from x \to 1 $ is the
unique map to the terminal object $ 1 $.  We lift this to
define a comonoid structure on $ \RRel (\T) $ by setting the
comultiplication $ \Delta \from x \to x \otimes x $ as the
internal relation
\[
  \begin{tikzpicture}
    \node (1) at (0,0) {$ x $};
    \node (2) at (2,2) {$ x $};
    \node (3) at (4,0) {$ x \times x $};
    \draw[cd]
    (2) edge node[above]{$ \id $} (1)
    (2) edge node[above,right]{$\langle \id,\id \rangle$} (3);
  \end{tikzpicture}
\]
and the counit to be the internal relation
\[
  \begin{tikzpicture}
    \node (1) at (0,0) {$ x $};
    \node (2) at (2,2) {$ x $};
    \node (3) at (4,0) {$ 1 $};
    \draw[cd]
    (2) edge node[above]{$ \id $} (1)
    (2) edge node[above] {$ ! $} (3);
  \end{tikzpicture}
\]

Every arrow in $ \RRel (\T) $ plays nicely with the
comonoid structure. Suppose we have an arrow $ r \from x \rel y
$, hence a jointly monic span
\[
  \begin{tikzpicture}
    \node (1) at (0,0) {$ x $};
    \node (2) at (2,2) {$ r $};
    \node (3) at (4,0) {$ y $};
    \draw[cd]
    (2) edge (1)
    (2) edge (3);
  \end{tikzpicture}
\]
Then $ r $ is a lax comonoid homomorphism in that there are
2-arrows $ \Delta_y r \leq (r \otimes r) \Delta_x $ and
$ \epsilon_y r \leq t_x $.   The lax
preservation of comultiplication is the 2-arrow
\[
  \begin{tikzpicture}
    \node (1) at (0,0) {$ x $};
    \node (2) at (2,2) {$ r \times_y y $};
    \node (3) at (4,0) {$ y $};
    \node (4) at (2,-2)
      {$ x \times_{x \times x} (r \times r) $};
    \draw[cd]
    (2) edge node[]{$  $} (1)
    (2) edge node[]{$  $} (3)
    (2) edge node[]{$  $} (4)
    (4) edge node[]{$  $} (1)
    (4) edge node[]{$  $} (3);    
  \end{tikzpicture}
\]
where $ r \times_y y \iso r $ and one can determine that $ r
\times r $ is a subobject of $ x \times_{x \times x}(r
\times r) $. The 2-arrow then is the composite
\[
  r \times_y y \xto{\iso} r \xto{\Delta} r \times r \monicto
  x \times_{x \times x}(r \times r).
\]
The lax preservation of unit is the 2-arrow
\[
  \begin{tikzpicture}
    \node (1) at (0,0) {$ x $};
    \node (3) at (4,0) {$ y $};
    \node (2) at (2,2) {$ r \times_y y $};
    \node (4) at (2,-2) {$ x $};
    \draw[cd]
    (2) edge node[]{$  $} (1)
    (2) edge node[]{$  $} (3)
    (2) edge node[]{$  $} (4) 
    (4) edge node[]{$  $} (1)
    (4) edge node[]{$  $} (3);
  \end{tikzpicture}
\]
obtained as the composite
\[
  r \times_y y \xto{\iso} r \to x.
\]

Also, because we are working with spans, we can turn them
around to give a monoid structure $ \Delta^\ast_x \from x
\otimes x \rel x $ and $ \epsilon^\ast_x \from 1 \rel x $
given by the respective spans
\[
  \begin{tikzpicture}
    \begin{scope}
      \node (1) at (0,0) {$ x \times x $};
      \node (2) at (2,2) {$ x $};
      \node(3) at (4,0) {$ x $};
      \draw[cd]
      (2) edge node[above]{$ \Delta $} (1)
      (2) edge node[above]{$ \id $} (3);
    \end{scope}
    \begin{scope}[shift={(6,0)}]
      \node (1) at (0,0) {$ 1 $};
      \node (2) at (2,2) {$ x $};
      \node(3) at (4,0) {$ x $};
      \draw[cd]
      (2) edge node[above]{$ ! $} (1)
      (2) edge node[above]{$ \id $} (3);
    \end{scope}
  \end{tikzpicture}
\]

What Carboni and Walters did was to take this structure as
primitive to define a Cartesian bicategory. Though they went
farther by axiomatizing another important property of
$ \RRel (\T) $. Namely that any object $ x $ of
$ \RRel (\T) $ is a Frobenius monoid (see Definition
\ref{def:frobenius-monoid}) which, recall, requires the
equation
\[
    (\id \otimes \mu)
      ( \delta \otimes \id )
    = \delta \mu
    = ( \mu \otimes \id )
      ( \id \otimes \delta ).
 \]
 to hold.  The left hand side of this equation is given
 by the composite
\[
  \begin{tikzpicture}
    \node (1) at (0,0) {$ x \times x $};
    \node (2) at (3,0) {$ x $};
    \node (3) at (6,0) {$ x \times x $};
    \node (4) at (1.5,2) {$ x $};
    \node (5) at (4.5,2) {$ x $};
    \node (6) at (3,4) {$ x $};
    \draw [cd] 
    (4) edge node[above]{$ \delta $} (1)
    (4) edge node[above]{$ \id $} (2)
    (5) edge node[above]{$ \id $} (2)
    (5) edge node[above]{$ \delta $} (3)
    (6) edge node[above]{$ \id $} (4)
    (6) edge node[above]{$ \id $} (5);
    \draw (2.9,3.6) -- (3,3.5) -- (3.1,3.6);
  \end{tikzpicture}
\]
and the right-hand side of the equation is given by the
composite
\begin{center}
  \begin{tikzpicture}
    \node (1) at (0,0) {$ x \times x $};
    \node (2) at (3,0) {$ x \times x \times x $};
    \node (3) at (6,0) {$ x \times x $};
    \node (4) at (1.5,2) {$ x \times x $};
    \node (5) at (4.5,2) {$ x \times x $};
    \node (6) at (3,4) {$ x $};
    \draw [cd] 
    (4) edge node[above]{$ \id $} (1)
    (4) edge[pos=0.25] node[right]{$ \id \times \delta $} (2)
    (5) edge node[left]{$ \delta \times \id $} (2)
    (5) edge node[above]{$ \id $} (3)
    (6) edge node[above]{$ \delta $} (4)
    (6) edge node[above]{$ \delta $} (5);
    \draw (2.9,3.6) -- (3,3.5) -- (3.1,3.6);
    \end{tikzpicture}
\end{center}
Hence, the equality of the composite spans. In Section
\ref{sec:duality-in-bicategories}, we axiomatize the
structures and properties found in a category of relations
internal to a topos.

Having though about $ \RRel (\T) $, we can now axiomatize some
important structures. The first structure needed is a
tensor product for a bicategory.  In general, the coherence can be
quite complicated but simplifies significantly when restricting our attention to
locally posetal bicategories.

\begin{definition}
  A tensor product $ \otimes \from \BB \times \BB \to \BB $
  on a locally posetal bicategory $ \BB $ is a pseudo-functor
  equipped with an unit object $ I $ and natural
  isomorphisms
  \begin{align*}
    \rho   & \from x \to x \otimes I             &
    \lambda & \from x \to I \otimes x            \\
    \sigma & \from x \otimes y \to y \otimes x   &
    \alpha  & \from (x \otimes y) \otimes z \to
                     x \otimes (y \otimes z)                                                 
  \end{align*}
  that satisfy the classical coherence conditions.  
\end{definition}

We also need to place the concept of adjoint functors into a
general bicategory. The data of an adjoint pair---two
functors and two natural transformations---are merely
1-arrows and 2-arrows in $ \CCat $. However, this structure
can be supported by bicategories other than $ \CCat $.

\begin{definition}[Adjunction] \label{def:adjoint-in-bicat}
  Let $ \BB $ be a bicategory.  We say the 1-arrows
  \[
    \ell \from x \to y
    \quad \text{and} \quad
    r \from y \to x
  \]
  form an \defn{adjunction}, with $ \ell $ the left adjoint
  and $ r $ the right adjoint if there exist 2-arrows
    \[
    \begin{tikzpicture}
      \node (1) at (0,0) {$ y $};
      \node (2) at (2,0) {$ y $};
      \node () at (1,0) {$ \Downarrow \eta $};
      \draw [cd]
        (1) edge[bend left=60]  node[above]{$ \id $} (2)
        (1) edge[bend right=60] node[below]{$ \ell r $} (2);
    \end{tikzpicture}
    \quad \quad \quad \quad
    \begin{tikzpicture}
      \node (1) at (0,0) {$ x $};
      \node (2) at (2,0) {$ x $};
      \node () at (1,0) {$ \Downarrow \epsilon $};
       \draw [cd]
        (1) edge[bend left=60]  node[above]{$ r \ell $} (2)
        (1) edge[bend right=60] node[below]{$ \id $} (2);
    \end{tikzpicture}  
  \]
  respectively named the unit and the counit such that each
  composite
  \[
    \begin{tikzpicture}
      \node (x) at (0,0) {$ x $};
      \node (y) at (4,0) {$ y $};
      \node at (2,1)  {$ \Downarrow \id \otimes \eta $};
      \node at (2,-1) {$ \Downarrow \epsilon \otimes \id $};
      \draw [cd]
      (x) edge[bend left=100] node[above]{$ \ell $} (y) 
      (x) edge[] node[above]{$ \ell r \ell $} (y) 
      (x) edge[bend right=100] node[below]{$ \ell $} (y); 
    \end{tikzpicture}
    \quad \quad \quad \quad
    \begin{tikzpicture}
      \node (y) at (4,0) {$ y $};
      \node (x) at (0,0) {$ x $};
      \node at (2,1)  {$ \Downarrow\eta\otimes\id $};
      \node at (2,-1) {$ \Downarrow\id\otimes\epsilon $};
      \draw [cd] 
      (x) edge[bend left=100] node[above]{$ r $} (y) 
      (x) edge[] node[above]{$ r \ell r $} (y) 
      (x) edge[bend right=100] node[below]{$ r $} (y); 
    \end{tikzpicture}
  \]
  is an identity.
\end{definition}

Closely related to adjoint arrows are the dual concepts of
monad and comonad.  Also like adjunctions, the most common
monads and comonads are internal to the 2-category
$ \CCat $.  Comonads in particular are relevant for us in
Section \ref{sec:gen-result-graph-rewriting}.

\begin{definition}[(Co)monad] \label{def:(co)monad}
  In a bicategory $ \BB $, an arrow $ m \from b \to b $ is
  called a \defn{monad} if there are 2-arrows
  $ \mu \from mm \to m $ and $ \eta \from \id_b \to m $ such
  that
  \begin{center}
    \begin{tikzpicture}
      \begin{scope}
        \node (1) at (0,0) {$ b $};
        \node (2) at (2,0) {$ b $};
        \node (3) at (4,0) {$ b $};
        \node () at (1,0) {$ \Downarrow \id $};
        \node () at (3,0) {$ \Downarrow \mu $};
        \node () at (2,-0.75) {$ \Downarrow \mu $};
        \draw [cd] 
        (1) edge[bend left]  node[above]{$ m $}  (2)
        (2) edge[bend left]  node[above]{$ mm $} (3)
        (1) edge[bend right] node[below]{$ m $} (2)
        (2) edge[bend right] node[below]{$ m $} (3)
        (1) edge[bend right=60] node[below]{$ m $} (3);
      \end{scope}
      \begin{scope}[shift={(6,0)}]
        \node (1) at (0,0) {$ b $};
        \node (2) at (2,0) {$ b $};
        \node (3) at (4,0) {$ b $};
        \node () at  (1,0) {$ \Downarrow \mu $};
        \node () at  (3,0) {$ \Downarrow \id $};
        \node () at  (2,-0.75) {$ \Downarrow \mu $};
        \draw [cd] 
        (1) edge[bend left] node[above]{$ mm $}  (2)
        (2) edge[bend left] node[above]{$ m $} (3)
        (1) edge[bend right] node[below]{$ m $} (2)
        (2) edge[bend right] node[below]{$ m $} (3)
        (1) edge[bend right=60] node[below]{$ m $} (3);
      \end{scope}
      \node () at (5,-0.5) {$ = $};
    \end{tikzpicture}
  \end{center}
  and also
  \begin{center}
    \begin{tikzpicture}
      \begin{scope}
        \node (1) at (0,0) {$ b $};
        \node (2) at (2,0) {$ b $};
        \node (3) at (4,0) {$ b $};
        \node () at (1,0) {$ \Downarrow \eta $};
        \node () at (3,0) {$ \Downarrow \id $};
        \node () at (2,-0.75) {$ \Downarrow \mu $};
        \draw [cd] 
        (1) edge[bend left]  node[above]{$ \id $}  (2)
        (2) edge[bend left]  node[above]{$ m $} (3)
        (1) edge[bend right] node[below]{$ m $} (2)
        (2) edge[bend right] node[below]{$ m $} (3)
        (1) edge[bend right=60] node[below]{$ m $} (3);
      \end{scope}
      \begin{scope}[shift={(6,0)}]
        \node (1) at (0,0) {$ b $};
        \node (2) at (2,0) {$ b $};
        \node (3) at (4,0) {$ b $};
        \node () at (1,0) {$ \Downarrow \id $};
        \node () at (3,0) {$ \Downarrow \eta $};
        \node () at (2,-0.75) {$ \Downarrow \mu $};
        \draw [cd] 
        (1) edge[bend left] node[above]{$ m $}  (2)
        (2) edge[bend left] node[above]{$ \id $} (3)
        (1) edge[bend right] node[below]{$ m $} (2)
        (2) edge[bend right] node[below]{$ m $} (3)
        (1) edge[bend right=60] node[below]{$ m $} (3);
      \end{scope}
      \node () at (5,-0.5) {$ = $};
    \end{tikzpicture}
    \end{center}
    When the 2-arrows are reversed, we get a \defn{comonad}
\end{definition}

There is a close relationship between adjunctions, monads,
and comonads. Instead of exploring this relationship in its
full generality, we restrict our attention to adjunctions,
monads, and comonads in $ \Cat $.

For any adjunction 
\[
  \adjunction{\A}{\X}{L}{R}{2}
\]
with unit $ \eta $ and counit $ \epsilon $, we define a
monad $ RL \from \A \to \A $ with unit
\begin{center}
  \begin{tikzpicture}
    \node (1) at (0,0) {$ \A $};
    \node (2) at (2,0) {$ \A $};
    \draw [cd] 
      (1) edge[bend left] node[above]{$ \id $} (2)
      (1) edge[bend right] node[below]{$ RL $} (2);
    \node () at (1,0) {$ \Downarrow \eta $};
  \end{tikzpicture}
\end{center}
and multiplication $ R \epsilon L \from RLRL \To RL $ given
by the horizontal composite
\begin{center}
  \begin{tikzpicture}
    \node (1) at (0,0) {$ \A $};
    \node (2) at (2,0) {$ \X $};
    \node (3) at (4,0) {$ \X $};
    \node (4) at (6,0) {$ \A $};
    \draw [cd] 
      (1) edge[bend left] node[above]{$ L $} (2)
      (2) edge[bend left] node[above]{$ RL $} (3)
      (3) edge[bend left] node[above]{$ R $} (4);
    \node () at (1,0) {$ \Downarrow \id_L $};
    \node () at (3,0) {$ \Downarrow \epsilon $};
    \node () at (5,0) {$ \Downarrow \id_R $};
    \draw [cd] 
      (1) edge[bend right] node[below]{$ L $} (2) 
      (2) edge[bend right] node[below]{$ \id $} (3) 
      (3) edge[bend right] node[below]{$ R $} (4); 
  \end{tikzpicture}
\end{center}
The adjunction also induces a comonad with counit
\begin{center}
  \begin{tikzpicture}
    \node (1) at (0,0) {$ \X $};
    \node (2) at (2,0) {$ \X $};
    \draw [cd] 
      (1) edge[bend left] node[above]{$ LR $} (2)
      (1) edge[bend right] node[below]{$ \id $} (2);
    \node () at (1,0) {$ \Downarrow \epsilon $};
  \end{tikzpicture}
\end{center}
and comultiplication $ L \eta R \from LR \To LRLR $ given by
the composite
\begin{center}
  \begin{tikzpicture}
    \node (1) at (0,0) {$ \X $};
    \node (2) at (2,0) {$ \A $};
    \node (3) at (4,0) {$ \A $};
    \node (4) at (6,0) {$ \X $};
    \draw [cd] 
      (1) edge[bend left] node[above]{$ R $} (2)
      (2) edge[bend left] node[above]{$ \id $} (3)
      (3) edge[bend left] node[above]{$ L $} (4);
    \node () at (1,0) {$ \Downarrow \id_R $};
    \node () at (3,0) {$ \Downarrow \eta $};
    \node () at (5,0) {$ \Downarrow \id_L $};
    \draw [cd] 
      (1) edge[bend right] node[below]{$ R $} (2) 
      (2) edge[bend right] node[below]{$ LR $} (3) 
      (3) edge[bend right] node[below]{$ L $} (4); 
  \end{tikzpicture}
\end{center}
We use this latter fact in Section
\ref{sec:gen-result-graph-rewriting}.

The opposite direction, from monads to adjunctions, is a
more subtle issue because to each monad is associated a
family of adjunctions.  This is not used in this thesis,
however, so we point the reader to a standard reference
\cite{maclane_cats-working} to learn more.

We now have all of the background needed to define a cartesian bicategory.

\begin{definition}[Cartesian bicategory]
  A cartesian bicategory consists of the following data:
  \begin{itemize}
  \item a locally posetal bicategory $ \BB $
  \item a tensor product $ \otimes \from \BB \times \BB \to
    \BB $
  \item for every object $ x $ of $ \BB $, a cocommutative
    monoid structure $ \Delta_x \from x \to x \otimes x $
    and $ \epsilon_x \from x \to x \otimes I $
  \end{itemize}
  such that
  \begin{itemize}
  \item every arrow $ r \from x \to r $ is a lax
  comonoid homomorphism, that is
  \[
    \Delta_y r \leq (r \otimes r) \Delta_x
    \quad \text{and} \quad
    \epsilon_y r \leq \epsilon_x
  \]
  \item for each object $ x $, comultiplication $ \Delta_x $
    and counit $ \epsilon_x $ have right adjoints $
    \Delta^\ast_x $ and $ \epsilon^\ast_x $ that give a
    commutative monoid structure to $ x $.  
  \end{itemize}
\end{definition}

Such a bicategory is called cartesian because of its
similarities to a cartesian category.

Another nice feature we saw in our favorite cartesian
bicategory $ \RRel (\C) $ is that each object $ x $ is a
Frobenius monoid.  When we append this axiom to those for a
cartesian bicategory, we obtain a more complete
axiomatization of $ \RRel (\C) $. Because of this we call
such a gadget a \defn{bicategory of relations}.

\begin{definition}[Bicategory of relations]
\label{def:bicat-relations}
  A \defn{bicategory of relations} is a cartesian bicategory
  $ \BB $ such that for all objects $ x $, the structure
  maps $ \Delta_x, \epsilon_x, \Delta^\ast_x,
  \epsilon^\ast_x $ satisfy the Frobenius law
  \[
    \Delta_x \Delta^\ast_x =
    ( \id \otimes \Delta_x )( \Delta^\ast_x \otimes \id).
  \]
\end{definition}

It follows from the Frobenius law that in a bicategory of
relations, every object is its own dual. This brings us to
our next section on duality in bicategories.

\section{Duality in bicategories}
\label{sec:duality-in-bicategories}

One's first encounter with the term `dual' is typically in
linear algebra.  Recall that given a $ K $-vector space
$ V $ and its dual $ V^\ast $, there is a linear map
$ V^\ast \otimes_K V \to K $. Also, $ K $ is the identity
with respect to $ \otimes_K $, that is
$ K \otimes_K V \iso V $.  The fact every object in the
monoidal category $ ( \Vect_K, \otimes_K, K) $ of
$ K $-vector spaces and $ K $-linear maps has such a dual
can be generalized to other monoidal categories.  Such
categories are called \defn{compact closed}.  

Briefly returning to the previous section, we left off
saying that in a bicategory of relations every object is its
own dual. And though the coherence is more complicated for
bicategories in general, locally posetal bicategories, such
as bicategories of relations, skirt this issue. Due to the
restriction on 2-arrows, showing that a locally posetal
bicategory is compact closed is exactly the same as showing a
categories is compact closed. Hence our next theorem, that a
bicategory of relations is necessarily compact closed, holds
true and it is the Frobenius law that provides this
structure.

\begin{theorem}
  A bicategory of relations is compact closed.
\end{theorem}

\begin{proof}
  See Theorem 2.4 in Carboni and Walters \cite{carboni-walters_bicats-relations}.
\end{proof}

For the remainder of this section, we move beyond locally
posetal bicategories to discuss compact closure for generic monoidal
bicategories. 

To define `compact closed bicategories' as conceived by
Stay \cite{stay_cc-bicats}, we discuss a notion of duality
suitable for bicategories.  We write $LR$ for the tensor
product of objects $L$ and $R$ and $fg$ for the tensor
product of morphisms $f$ and $g$. 

\begin{definition}[Dual pair, category]
  \label{def:DualPairCat}
  A \defn{dual pair} in a symmetric monoidal category
  $ ( \C , \otimes, I ) $ is a tuple $(L,R,e,c)$ with
  objects $L$ and $R$, called the \textbf{left} and
  \textbf{right} duals, and morphisms
  \[
    e \from LR \to I 
    \quad \quad 
    c \from I \to RL,
  \]
  called the \textbf{counit} and \textbf{unit},
  respectively, such that the following diagrams commute.
  \[
    \begin{tikzpicture}
      \node (L1) at (0,2) {$L$};
      \node (L2) at (0,0) {$L$};
      \node (LRL) at (2,1) {$LRL$};
      \draw[cd]
      (L1) edge node[left]{$L$} (L2)
      (L1) edge node[above]{$Lc$} (LRL)
      (LRL) edge node[below]{$eL$} (L2);
    \end{tikzpicture}
    \quad \quad
    \begin{tikzpicture}
      \node (R1) at (0,2) {$R$};
      \node (R2) at (0,0) {$R$};
      \node (RLR) at (2,1) {$RLR$};
      \draw[cd]
      (R1) edge node[left]{$R$} (R2)
      (R1) edge node[above]{$cR$} (RLR)
      (RLR) edge node[below]{$Re$} (R2);
    \end{tikzpicture}	
  \]
  A category such such that every object has a dual is
  called \defn{compact closed}.  
\end{definition}

\begin{definition}[Dual pair, bicategory]
  \label{def:DualPairBicat}
  Inside a monoidal bicategory, a \defn{dual pair} is a
  tuple $(L,R,e,c,\alpha,\beta)$ with objects $L$ and $R$,
  morphisms
  \[
    e \from LR \to I 
    \quad \quad 
    c \from I \to RL,
  \]
  and invertible 2-morphisms
  \[
    \begin{tikzpicture}
      \node (L1)   at (0,4) {$L$};
      \node (LI)   at (0,2) {$LI$};
      \node (LRL1) at (0,0) {$L(RL)$};
      \node (LRL2) at (3,0) {$(LR)L$};
      \node (IL)   at (3,2) {$IL$};
      \node (L2)   at (3,4) {$L$};
      \draw[cd]
      (L1)   edge                   (LI)
      (LI)   edge node[left]{$Lc$}  (LRL1)
      (LRL1) edge                   (LRL2)
      (LRL2) edge node[right]{$eL$} (IL)
      (IL)   edge                   (L2)
      (L1)   edge node[above]{$L$}  (L2);
      \node () at (1.5,2) {$\Downarrow \alpha$};
    \end{tikzpicture}
    \quad \quad
    \begin{tikzpicture}
      \node (L1)   at (0,4) {$R$};
      \node (LI)   at (0,2) {$RI$};
      \node (LRL1) at (0,0) {$(RL)R$};
      \node (LRL2) at (3,0) {$R(LR)$};
      \node (IL)   at (3,2) {$RI$};
      \node (L2)   at (3,4) {$R$};
      \draw [cd]
      (L1)   edge                    (LI)
      (LI)   edge node[left] {$cR$}  (LRL1)
      (LRL1) edge                    (LRL2)
      (LRL2) edge node[right] {$Re$} (IL)
      (IL)   edge                    (L2)
      (L1)   edge node[above] {$R$} (L2);
      \node () at (1.5,2) {$\Downarrow \beta$};
    \end{tikzpicture}
  \]
  called \defn{cusp isomorphisms}.  If this data satisfies
  the swallowtail equations in the sense that the diagrams
  in Figure \ref{fig:swallowtail} are identities, then we
  call the dual pair \defn{coherent}.
\end{definition}

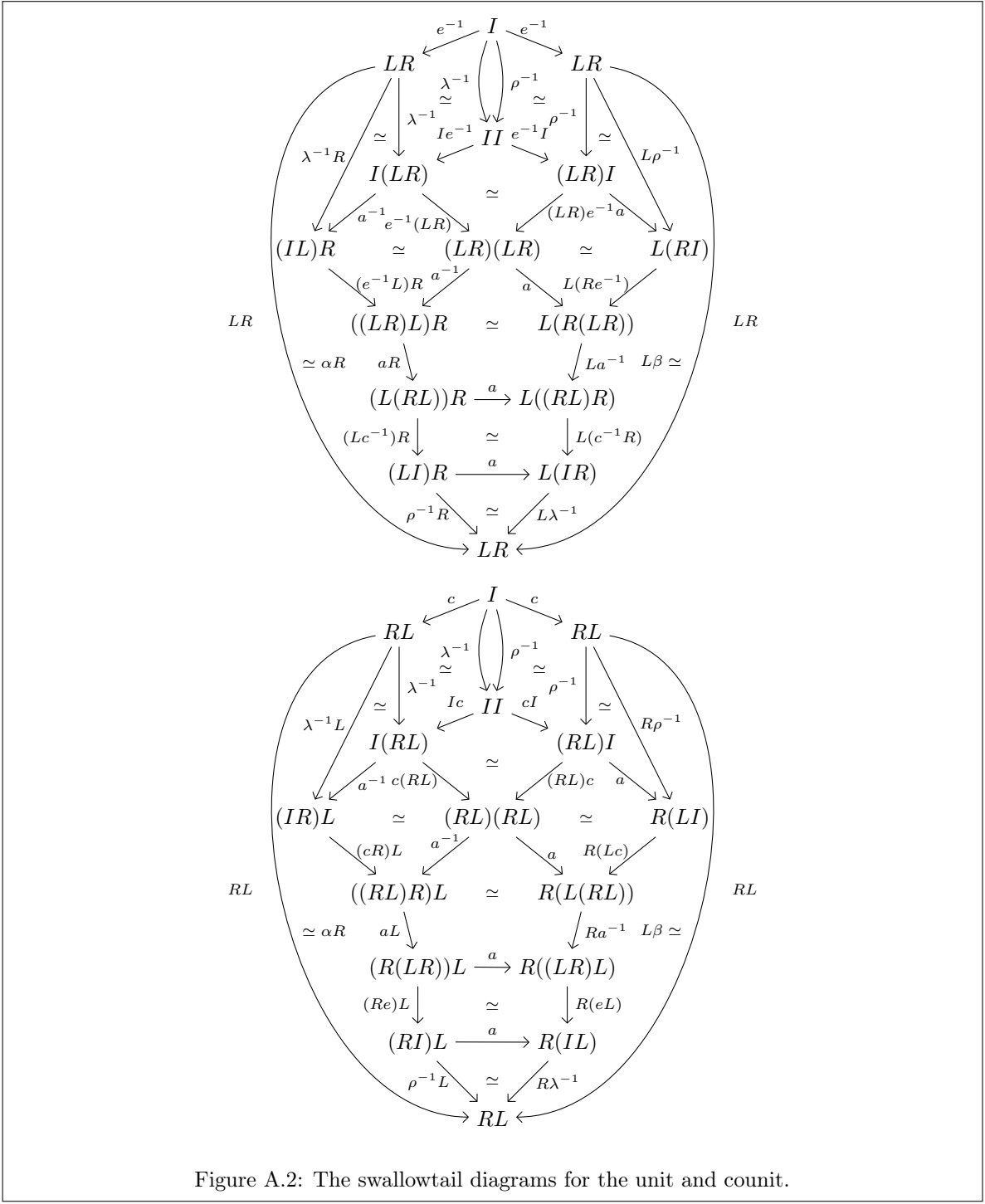
\begin{figure}
  \centering
  \fbox{
  \begin{minipage}{\linewidth}
  %
  %
  \[
    \begin{tikzpicture}[scale=0.60]
      \node (A) at (0,0) {$I$};
      \node (B) at (-2.5,-1) {$L R$};
      \node (C) at (2.5,-1) {$L R$};
      \node (D) at (0,-3) {$I I$};
      \node (E) at (-2.5,-4) {$I (L R)$};
      \node (F) at (2.5,-4) {$(L R) I$};
      \node (G) at (0,-6) {$(L R) (L R)$};
      \node (H) at (-5,-6) {$(I L) R$};
      \node (I) at (5,-6) {$L (R I)$};
      \node (J) at (-2.5,-8) {$((L R) L)  R$};
      \node (K) at (2.5,-8) {$L (R (L R))$};
      \node (L) at (-2,-10) {$(L (R L)) R$};
      \node (M) at (2,-10) {$L  ((R L) R)$};
      \node (N) at (-2,-12) {$(L I) R$};
      \node (O) at (2,-12) {$L (I R)$};
      \node (P) at (0,-14) {$L  R$};
      %
      %
      \node (Q) at (-1.25,-2) {\scriptsize{$\simeq$}};
      \node (R) at (1.25,-2) {\scriptsize{$\simeq$}};
      \node (S) at (0,-4.5) {\scriptsize{$\simeq$}};
      \node (T) at (-3,-3) {\scriptsize{$\simeq $}};
      \node (U) at (3,-3) {\scriptsize{$\simeq$}};
      \node (V) at (-2.5,-6) {\scriptsize{$\simeq$}};
      \node (W) at (2.5,-6) {\scriptsize{$\simeq$}};
      \node (X) at (0,-8) {\scriptsize{$\simeq$}};
      \node (Y) at (0,-11) {\scriptsize{$\simeq$}};
      \node (Z) at (0,-13) {\scriptsize{$\simeq$}};
      \node (A1) at (-4.5,-9) {\scriptsize{$\simeq \alpha R$}};
      \node (A2) at (4.5,-9) {\scriptsize{$L \beta \simeq$}};
      \draw[cd]
      (A) edge node[above]{$e^{-1}$} (B)
      (A) edge node[above]{$e^{-1}$} (C)
      (A) edge[out=-110,in=110] node[left]{$\lambda^{-1}$} (D)
      (A) edge[out=-75,in=75] node[right]{$\rho^{-1}$} (D)
      (B) edge node[right]{$\lambda^{-1}$}(E)
      (C) edge node[left]{$\rho^{-1}$} (F)
      (D) edge node[left=0.2cm,above=0.1cm]{$I e^{-1}$} (E)
      (D) edge node[left=0.2cm,above=0.1cm]{$e^{-1} I$} (F)
      (E) edge node[below,left,pos=0.8]{$e^{-1} (L R)$} (G)
      (F) edge node[below,right]{$(L R) e^{-1}$} (G)
      (B) edge node[above,left]{$\lambda^{-1} R$} (H)
      (C) edge node[above,right]{$L \rho^{-1}$} (I)
      (E) edge node[below,right,pos=0.55]{$a^{-1}$} (H)
      (F) edge node[below,left]{$a$} (I)
      (H) edge node[above,right,pos=0.4]{$(e^{-1} L) R$} (J)
      (G) edge node[above]{$a^{-1}$} (J)
      (I) edge node[above,left,pos=0.4]{$L (R e^{-1})$} (K)
      (G) edge node[below,left]{$a$} (K)
      (J) edge node[left]{$a R$} (L)
      (K) edge node[right]{$L a^{-1}$} (M)
      (L) edge node[above]{$a$} (M)
      (L) edge node[left]{$(L  c^{-1}) R$} (N)
      (M) edge node[right]{$L (c^{-1} R)$} (O)
      (N) edge node[above]{$a$} (O)
      (N) edge node[below,left]{$\rho^{-1} R$} (P)
      (O) edge node[below,right]{$L \lambda^{-1}$} (P)
      (B) edge [out=-170,in=-180] node [below=0.3cm,left=0.3cm] {$L R$} (P)
      (C) edge [in=-360,out=-370,] node [below=0.3cm,right=0.3cm] {$LR$} (P);
    \end{tikzpicture}
  \]
  %
  %
  \[
    \begin{tikzpicture}[scale=0.60]
      \node (A) at (0,0) {$I$};
      \node (B) at (-2.5,-1) {$RL$};
      \node (C) at (2.5,-1) {$RL$};
      \node (D) at (0,-3) {$I I$};
      \node (E) at (-2.5,-4) {$I (RL)$};
      \node (F) at (2.5,-4) {$(RL) I$};
      \node (G) at (0,-6) {$(RL) (RL)$};
      \node (H) at (-5,-6) {$(I R) L$};
      \node (I) at (5,-6) {$R (L I)$};
      \node (J) at (-2.5,-8) {$((RL) R)  L$};
      \node (K) at (2.5,-8) {$R (L (R L))$};
      \node (L) at (-2,-10) {$(R (L R)) L$};
      \node (M) at (2,-10) {$R  ((L R) L)$};
      \node (N) at (-2,-12) {$(R I) L$};
      \node (O) at (2,-12) {$R (I L)$};
      \node (P) at (0,-14) {$R  L$};
      %
      %
      \node (Q) at (-1.25,-2) {\scriptsize{$\simeq$}};
      \node (R) at (1.25,-2) {\scriptsize{$\simeq$}};
      \node (S) at (0,-4.5) {\scriptsize{$\simeq$}};
      \node (T) at (-3,-3) {\scriptsize{$\simeq $}};
      \node (U) at (3,-3) {\scriptsize{$\simeq$}};
      \node (V) at (-2.5,-6) {\scriptsize{$\simeq$}};
      \node (W) at (2.5,-6) {\scriptsize{$\simeq$}};
      \node (X) at (0,-8) {\scriptsize{$\simeq$}};
      \node (Y) at (0,-11) {\scriptsize{$\simeq$}};
      \node (Z) at (0,-13) {\scriptsize{$\simeq$}};
      \node (A1) at (-4.5,-9) {\scriptsize{$\simeq \alpha R$}};
      \node (A2) at (4.5,-9) {\scriptsize{$L \beta \simeq$}};
      \draw[cd]
      (A) edge[cd] node[above]{$c$} (B)
      (A) edge[cd] node[above]{$c$} (C)
      (A) edge[cd,out=-110,in=110] node[left]{$\lambda^{-1}$} (D)
      (A) edge[cd,out=-75,in=75] node[right]{$\rho^{-1}$} (D)
      (B) edge[cd] node[right]{$\lambda^{-1}$}(E)
      (C) edge[cd] node[left]{$\rho^{-1}$} (F)
      (D) edge[cd] node[left=0.2cm,above=0.1cm]{$I c$} (E)
      (D) edge[cd] node[left=0.2cm,above=0.1cm]{$c I$} (F)
      (E) edge[cd] node[below,left,pos=0.5]{$c (R L)$} (G)
      (F) edge[cd] node[below,right]{$(R L) c$} (G)
      (B) edge[cd] node[above,left]{$\lambda^{-1} L$} (H)
      (C) edge[cd] node[above,right]{$R \rho^{-1}$} (I)
      (E) edge[cd] node[below,right,pos=0.55]{$a^{-1}$} (H)
      (F) edge[cd] node[below,left]{$a$} (I)
      (H) edge[cd] node[above,right,pos=0.4]{$(c R) L$} (J)
      (G) edge[cd] node[above]{$a^{-1}$} (J)
      (I) edge[cd] node[above,left,pos=0.4]{$R (L c)$} (K)
      (G) edge[cd] node[above,right]{$a$} (K)
      (J) edge[cd] node[left]{$a L$} (L)
      (K) edge[cd] node[right]{$R a^{-1}$} (M)
      (L) edge[cd] node[above]{$a$} (M)
      (L) edge[cd] node[left]{$(R  e) L$} (N)
      (M) edge[cd] node[right]{$R (e L)$} (O)
      (N) edge[cd] node[above]{$a$} (O)
      (N) edge[cd] node[below,left]{$\rho^{-1} L$} (P)
      (O) edge[cd] node[below,right]{$R \lambda^{-1}$} (P)
      (B) edge[cd,out=-170,in=-180] node [below=0.3cm,left=0.3cm] {$R L$} (P)
      (C) edge[cd,in=-360,out=-370] node [below=0.3cm,right=0.3cm] {$RL$} (P);
    \end{tikzpicture}
  \]
  \caption{The swallowtail diagrams for the
    unit and counit.}
  \label{fig:swallowtail}
  \end{minipage}
  }
\end{figure}

Recall that a symmetric monoidal category is called
\defn{compact closed} if every object is part of a dual
pair.  We can generalize this idea to bicategories by
introducing 2-morphisms and some coherence axioms.  The
following definition is due to Stay \cite{stay_cc-bicats}.

\begin{definition}[Compact closed bicategory]
  \label{def:CompClosdBicat}
  A \defn{compact closed} bicategory is a symmetric
  monoidal bicategory for which every object $R$ is part of
  a coherent dual pair.
\end{definition}

The difference between showing compact
closedness in categories versus bicategories
might seem quite large because of the swallowtail
equations.  Looking at Figure
\ref{fig:swallowtail}, it is no surprise that
these can be incredibly tedious to work with.
Fortunately, Pstr\'{a}gowski
\cite{pstragowski_dual-mon-bicats} proved a
wonderful strictification theorem that effectively
circumvents the need to consider the swallowtail
equations.

\begin{theorem}[{\cite[p.~22]{pstragowski_dual-mon-bicats}}]
  \label{thm:StrictingDualPairs}
  Given a dual pair $(L,R,e,c,\alpha,\beta)$, we can find a
  cusp isomorphism $\beta'$ such that
  $(L,R,e,c,\alpha,\beta')$ is a coherent dual pair.
\end{theorem}

\section{Adhesive categories}
\label{sec:adhesive-categories}

After Ehrig, et.~al. introduced double pushout
graph rewriting \cite{ehrig_graph-grammars}, there
were several attempts at axiomatizing it. The
first successful attempt is called High-Level
Replacement Systems (HLRS)
\cite{ehrig_graph-to-hlr,ehrig_parallel-concurrency}. To
be thorough, we include the axioms of an HLRS.

\begin{definition}[High level replacement system]
  \label{def:hlrs}
  A category $ \C $ is called a \defn{High Level Replacement
    System} if
  \begin{enumerate}
  \item pushouts exist for all spans $ a \gets b
    \to c $ such that one arrow is monic;
  \item pullbacks exist for all cospans $ a \to b
    \gets c $ where both arrows are monic;
  \item pushouts and pullbacks respect
    monomorphisms;
  \item for any diagram
    \[
      \begin{tikzpicture}
        \node (12) at (0,2) {$ a $};
        \node (22) at (2,2) {$ b $};
        \node (32) at (4,2) {$ c $};
        \node (11) at (0,0) {$ d $};
        \node (21) at (2,0) {$ e $};
        \node (31) at (4,0) {$ f $};
        \draw[cd]
        (12) edge       node[]{$  $} (22)
        (22) edge [>->] node[]{$  $} (32)
        (11) edge       node[]{$  $} (21)
        (21) edge [>->] node[]{$  $} (31)
        (12) edge [>->] node[]{$  $} (11)
        (22) edge [>->] node[]{$  $} (21)
        (32) edge [>->] node[]{$  $} (31);
      \end{tikzpicture}
    \]
    such that the marked arrows are monic, the
    outside rectangle is a pushout, and the right
    square is a pullback, then the left square is
    a pushout;
  \item binary coproducts exist;
  \item any pushout of a span with a monic arrow is
    also a pullback.
  \end{enumerate}  
\end{definition}

This collection of axioms was curated to prove
theorems such as the local Church--Rosser and concurrency, the
presence of which provide a rich rewriting theory. Lack and
Soboci\'{n}ski later provided a more compact set of axioms
that also allowed local Church--Rosser and concurrency
theorems \cite{lack-sobo_adhesive-cats}. To earn the shorter
list of axioms, they packed quite a bit of information into
an axiom by using a `Van Kampen square'.  

  A \defn{Van Kampen square} is a pushout
  \[
    \begin{tikzpicture}
      \node (12) at (0,2) {$ a $};
      \node (22) at (2,2) {$ b $};
      \node (11) at (0,0) {$ c $};
      \node (21) at (2,0) {$ d $};
      \draw[cd]
      (12) edge node[]{$  $} (11)
      (12) edge node[]{$  $} (22)
      (11) edge node[]{$  $} (21)
      (22) edge node[]{$  $} (21);
      \draw (0.3,1.6) -- (0.4,1.6) -- (0.4,1.7);
    \end{tikzpicture}
  \]
  that, when placed on the bottom of a cube
  \[
    \begin{tikzpicture}
      \node (c)  at (6,2) {$ c $};
      \node (a)  at (4,3) {$ a $};
      \node (d)  at (2,1) {$ d $};
      \node (b)  at (0,2) {$ b $};
      \node (c') at (6,6) {$ c' $};
      \node (a') at (4,7) {$ a' $};
      \node (d') at (2,5) {$ d' $};
      \node (b') at (0,6) {$ b' $};
      \draw[cd]
      (a)  edge node[]{$  $} (b)
      (a)  edge node[]{$  $} (c)
      (b)  edge node[]{$  $} (d)
      (c)  edge node[]{$  $} (d);
      \path[-,draw=white,line width=0.2cm]
      (a') edge node[]{$  $} (a)
      (b') edge node[]{$  $} (b)
      (c') edge node[]{$  $} (c)
      (d') edge node[]{$  $} (d);
      \path[cd,font=\scriptsize,>=angle 90]
      (a') edge node[]{$  $} (a)
      (b') edge node[]{$  $} (b)
      (c') edge node[]{$  $} (c)
      (d') edge node[]{$  $} (d);
      \path[-,draw=white,line width=0.2cm]
      (a')  edge node[]{$  $} (b')
      (a')  edge node[]{$  $} (c')
      (b')  edge node[]{$  $} (d')
      (c')  edge node[]{$  $} (d');
      \draw[cd]
      (a')  edge node[]{$  $} (b')
      (a')  edge node[]{$  $} (c')
      (b')  edge node[]{$  $} (d')
      (c')  edge node[]{$  $} (d');
    \end{tikzpicture}
  \]
  such that the back faces are pullbacks, then the
  front faces are pullbacks if and only if the top
  face is a pushout.

\begin{definition}[Adhesive category] \label{def:adhesive-category}
  An \defn{adhesive category}
  \begin{enumerate}
  \item has pushouts along monomorphisms;
  \item has pullbacks;
  \item pushouts along monomorphisms are Van Kampen squares.
  \end{enumerate}
\end{definition}

Roughly, the Van Kampen condition places adhesive
categories in the company of distributive
categories and extensive categories in the sense
of a compatibility between certain finite limits
and finite colimits.  In the case of distributive
categories, there is a compatibility between
products and coproducts. For extensive categories,
pullbacks and coproducts play nicely together. The
Van Kampen condition stipulates the compatibility
between pullback and pushout.  

Certainly, the definition of an adhesive category
is more elegant than that of an HLRS.  The price
of elegance is the dense Van Kampen
condition. While adhesive categories are not
exactly HRLS's, they are closely related as one
might expect.

\begin{proposition}[{\cite[Lem.~29]{lack-sobo_adhesive-cats}}]
  An adhesive category with an initial object is an HLRS.
\end{proposition}

Though fewer in number, the axioms for an adhesive category
are non-trivial. Also, adhesive categories are not so
well-known outside of rewriting theory.  Therefore, instead
of working with adhesive category, we work with a much more
well-known class of category: a topos.  Fortunately, every
elementary topos is adhesive. This result is the subject of
a paper by Lack and Soboci\'{n}ski
\cite{lack-sobo_topoi-adhesive}.

\begin{theorem}
  \label{thm:topos-are-adhesive}
  Every elementary topos is adhesive.
\end{theorem}

Because topoi are our categories of choice for the present
work and in light of Theorem \ref{thm:topos-are-adhesive},
we leave our discussion of adhesive categories here. In the
next section, we cover topos theory, but just enough for our needs.
This includes facts that morally belong to adhesive category
theory and also hold true for topoi.

\section{Topoi}
\label{sec:topoi}

When searching the literature on topos theory, one finds
myriad descriptions of what a topos is like.  Suffice to
say, any topos has a geometric aspect and a logical
aspect. With regards to the geometric aspect, a topos is
like a generalized space, where the objects are subspaces
and the arrows describe how the various subspaces relate to
one another.  But to each topos, there is an internal logic
from which we can recover various logics by using the arrows
to and from the subobject classifier which we define now\footnote{
  For a full account of logic via topos theory, see Part D
  of Johnstone's \emph{Sketches of an Elephant}
  \cite{johnstone_elephant}. }.

\begin{definition}[Subobject classifier]

  A \defn{subobject classifier} is a monomorphism
  \[
    \mathtt{true} \from 1 \to \Omega
  \]
  from the terminal object with the property that, for every
  objects $ t \in \T $ and subobject $ s \to t $, there
  exists a unique arrow $ \chi_s $ fitting into the pullback
  diagram
    \begin{center}
      \begin{tikzpicture}
        \node (s) at (0,2) {$ s $};
        \node (t) at (0,0) {$ t $};
        \node (1) at (2,2) {$ 1 $};
        \node (omega) at (2,0) {$ \Omega $};
        \draw[cd]
          (s) edge node[]{$  $} (t)
          (s) edge node[]{$  $} (1)
          (t) edge node[below]{$ \chi_s $} (omega)
          (1) edge node[]{$  $} (omega);
        \draw (0.3,1.6) -- (0.4,1.6) -- (0.4,1.7);
      \end{tikzpicture}
    \end{center}
\end{definition}

In the category $ \Set $, any two element set is a subobject
classifier.  Take the set $ \{0,1\} $.  Then any function
into that set determines a subobject, here just a subset, by
taking the fiber of $ 1 $.  Similarly, any subobject $
s \to t $ determines a map $ \chi_s \from t \to \{0,1\} $ by
sending an element of $ t $ to $ 1 $ if it belongs to $ s $
and sending an element of $ t $ to $ 0 $ if it does not
belong to $ s $.  

\begin{definition}[Topos]
  A \defn{topos} $ \T $ is a category with finite limits, is
  cartesian closed, and has a subobject classifier. 
\end{definition}

The examples we give below cover our needs.

\begin{example}
  \begin{enumerate}
  \item

    The archetypal topos is the category $ \Set $.  The
    subobject classifier is the two-element set
    $ \{ 0,1 \} $ where we interpret $ 0 $ as `false' and $
    1 $ as `true'.
    
  \item

    Presheaf categories $ \Set^{ \C^{\op} } $ are topoi when
    $ \C $ is a small category.  The subobject classifier is
    the functor $ \C^{\op} \to \Set $ that sends any object
    $ c $ in $ \C $ to the set of subfunctors of
    $ \C (-,c) $. This is called a `sieve' of $ c $.

  \item 

    Finite presheaf categories are topoi. These are functor
    categories of the type $ \FinSet^{ \C^{\op} } $ for $ \C $ finite.
    
  \end{enumerate}
  
\end{example}

Of these classes of examples, the presheaf topoi are the most
pertinent. There is one specific presheaf topos that we
particularly like.

  \begin{example}
    
    Our favorite example of a presheaf topos is $
    \RGraph $, the category of reflexive directed
    multi-graphs.  This is the category of
    presheaves on 
    \[
      \begin{tikzpicture}
        \node ( ) at (-1.5,0) {$ \C^{\op} \bydef $};
        \node (0) at (0,0)  {$ e $};
        \node (1) at (2,0)  {$ n $};
        \draw [cd]
          (0) edge[bend left=60] node[above]{$s$} (1)
          (1) edge[] node[above]{$t$} (0)
          (0) edge[bend right=60] node[below]{$ t $} (1); 
        \draw [rounded corners]
         (-0.5,-1.5) rectangle (2.5,1.5);
      \end{tikzpicture}
    \]
    such that all arrows $ n \to n $ are the identity.  A
    presheaf $ g \from \C^{\op} \to \Set $ then consists of
    two sets $ g(e) $ and $ g(n) $ considered as sets of
    edges and nodes.  Then there are two arrows of type
    $ g(e) \to g(n) $ assigning each edge its source and
    target and one arrow of type $ g(n) \to g(e) $ assigning
    a reflexive edge to each node. This is exactly a
    reflexive graph.  A natural transformation $ \theta $
    between presheaves $ g,h \from \C^{\op} \to \Set $ is a pair
    of functions $ \theta_e \from g(e) \to h(e) $ and
    $ \theta_n \from g(n) \to h(n) $ such that the squares
    \[
      \begin{tikzpicture}
        \begin{scope}
          \node (11) at (0,0) {$ h(e) $};
          \node (12) at (0,2) {$ g(e) $};
          \node (21) at (2,0) {$ h(n) $};
          \node (22) at (2,2) {$ g(n) $};
          \path[cd,font=\scriptsize,>=angle 90]
          (12) edge node[left]{$ \theta_e $} (11)
          (22) edge node[right]{$ \theta_n $} (21)
          (12) edge node[above]{$ g(t) $} (22)
          (11) edge node[below]{$ h(t) $} (21); 
        \end{scope}
        \begin{scope}[shift={(4,0)}]
          \node (11) at (0,0) {$ h(e) $};
          \node (12) at (0,2) {$ g(e) $};
          \node (21) at (2,0) {$ h(n) $};
          \node (22) at (2,2) {$ g(n) $};
          \path[cd,font=\scriptsize,>=angle 90]
          (12) edge node[left]{$ \theta_e $} (11)
          (22) edge node[right]{$ \theta_n $} (21)
          (12) edge node[above]{$ g(t) $} (22)
          (11) edge node[below]{$ h(t) $} (21); 
        \end{scope}
        \begin{scope}[shift={(8,0)}]
          \node (11) at (0,0) {$ h(e) $};
          \node (12) at (0,2) {$ g(e) $};
          \node (21) at (2,0) {$ h(n) $};
          \node (22) at (2,2) {$ g(n) $};
          \path[cd,font=\scriptsize,>=angle 90]
          (12) edge node[left]{$ \theta_e $} (11)
          (22) edge node[right]{$ \theta_n $} (21)
          (12) edge node[above]{$ g(t) $} (22)
          (11) edge node[below]{$ h(t) $} (21); 
        \end{scope}
      \end{tikzpicture}
    \]
    commute. These squares assert that the natural
    transformations preserve source, targets, and reflexive
    nodes. Hence, this is precisely the data of
    a reflexive graph morphism.
  \end{example}

Because topoi have both geometric and logical
aspects, there are morphisms of topos for each.  
  
  \begin{definition}[Geometric morphism]
    \label{def:geometric-morphism}
    A \defn{geometric morphism} between topoi
    $ \X \to \A $ is an adjunction
    \[
      \adjunction{\A}{\X}{L}{R}{2}
    \]
    such that $ L $ preserves finite limits. We
    call $ L $ the inverse image functor and $ R $
    the direct image functor.    
  \end{definition}

  Geometric morphisms abstract from continuous maps between
  spaces $ f \from S \to T $.  Denote by $ \mathcal{O}S $
  and $ \mathcal{O}T $ the open sets of $ S $ and $ T $.
  Then $ f $ induces the direct image map
  $ f_\ast \from \mathcal{O}S \to \mathcal{O}T$ that sends a
  set $ A \subseteq S $ to its image
  $ \{ t \in T \vert \exists a \in A.fa = t \} $. But $ f $
  also induces an inverse image map
  $ f^\ast \from \mathcal{O}T \to \mathcal{O}S $ that sends
  a set $ B \subseteq T $ to its preimage
  $ \{ s \in S \vert \exists b \in B . fs=b \} $. Observe
  that $ f_\ast $ preserves finite intersections and
  $ f^\ast $ preserves finite intersection and unions.  This
  mirrors the fact that, in a geometric morphism the right
  adjoint preserves finite limits and the left adjoint
  preserves finite limits and colimits.

  Now that the basic definition of a topos are
  given, we provide just enough theory to
  develop the ideas in this thesis.

  The first result we give is often called the
  fundamental theorem of topos theory \cite[{A.2.3.2}]{johnstone_elephant}. 
  
  \begin{theorem} \label{thm:fund-thm-topos}

    Given a topos $ \T $ and an object $ t $ of
    $ \T $, then the over-category
    $ \T \downarrow t $ is also a topos. 

  \end{theorem}

  The operation of `slicing over an object' is stable
  in presheaf topoi. This result uses a construction called
  the category of elements. Given a functor
  $ f \from \C \to \Set $, its category of elements, denoted
  $ \int^f \C $, has for objects pairs $ ( c, x ) $ where
  $ c $ is an object of $ \C $ and $ x $ is an element of
  the set $ fc $.  The arrows $ (c,x) \to ( d,y ) $ are the
  set functions $ fc \to fd $ such that $ x \mapsto y $. The
  category of elements is a first foray into the much larger
  topic called `the Grothendieck construction'. However,
  it is not useful for us to pursue this topic.
  
  \begin{theorem} \label{thm:presheaf-slice-is-presheaf}
    Let $ \C $ be a small category and
    $ F \from \C^{\op} \to \Set $ a presheaf.
    Then the over-category
    $ \Set^{\C^{\op}} \downarrow F $ is equivalent
    to the topos of presheaves on the category of
    elements $ \int^F \C $.
  \end{theorem}

  This result is used in Section
  \ref{sec:zx-calculus}. We illustrate it here
  with graphs.

  \begin{example}
    \label{ex:graph-over-graph}
    In this example, we illustrate the equivalence
    of Theorem
    \ref{thm:presheaf-slice-is-presheaf} by
    translating an object from
    $ \Set^{\C^{\op}} \downarrow F $ to a presheaf
    in the category $ \Set^{\int^F \C} $ for a
    specific choice of $ F $ and $ \C $.

    Let $ \C^{\op} $ be the walking graph category. That
    is,
    \[
      \begin{tikzpicture}
        \node [ob] (e) at (0,0) {$ e $};
        \node [ob] (n) at (2,0) {$ n $};
        \draw[cd]
        (e) edge[bend left=20] node[above]{$ s $} (n)
        (e) edge[bend right=20] node[below]{$ t $} (n);
        \draw [rounded corners]
          (-1,-1) rectangle (3,1);
      \end{tikzpicture}
    \]
    We call this the walking graph category to
    suggest that the presheaves on $ \C^{\op} $
    are exactly graphs and natural transformations
    between these functors are exactly the graph
    morphisms. Let $ F $ be the graph
    \[
      \begin{tikzpicture}
        \node (a) at (0,0.25) {$ _{b}\bullet $};
        \node (b) at (2,0.25) {$ \bullet_{b'} $};
        \draw [graph] 
        (a) edge[bend left=30] node[above]{$ \beta $} (b)
        (a) edge[bend right=30] node[below]{$ \beta' $} (b) 
        (b) edge[loop above] node[above]{$ \beta'' $} (b); 
        \draw [rounded corners]
          (-1,-1) rectangle (3,2);
      \end{tikzpicture}
    \]
    As a functor, $ F \from \C^{\op} \to \Set
    $ returns the edge set $ Fe \bydef \{ \beta,
    \beta', \beta'' \}$, the node set $ Fn \bydef
    \{b,b'\} $, the source map $ Fs \from Fe \to
    Fn $ defined by
    \[
      Fs(\beta)   \bydef b, \quad
      Fs(\beta')  \bydef b, \quad
      Fs(\beta'') \bydef b'
    \]
    and the target map $ Ft \from Fe \to Fn $
    defined by
    \[
      Ft(\beta)   \bydef b', \quad
      Ft(\beta')  \bydef b', \quad
      Ft(\beta'') \bydef b'.
    \]
    
    The graph morphism $ G \to F $, depicted by
    \[
      \begin{tikzpicture}
        \begin{scope}
          \node (a) at (0,0) {$ _a \bullet $};
          \node (b) at (2,0) {$ \bullet_{a'} $};
          \node (c) at (2,2) {$ \bullet^{a''} $};
          \draw [graph]
            (a) edge node[below]{$\alpha'$} (b)
            (b) edge node[right]{$\alpha''$} (c)
            (a) edge node[above,left]{$\alpha$} (c);
          \draw [rounded corners]
            (-1,-1) rectangle (3,3);
          \node (tb) at (3.1,1) {$  $};
        \end{scope}
        \begin{scope}[shift={(6,0)}]
          \node (a) at (0,1) {$ _{b}\bullet $};
        \node (b) at (2,1) {$ \bullet_{b'} $};
        \draw [graph] 
        (a) edge[bend left=30] node[above]{$ \beta $} (b)
        (a) edge[bend right=30] node[below]{$ \beta' $} (b) 
        (b) edge[loop above] node[above]{$ \beta'' $} (b); 
        \draw [rounded corners]
          (-1,-1) rectangle (3,3);
        \node (bt) at (-1.1,1) {$  $};
        \end{scope}
        \draw [cd] (tb) to node[above]{$\theta$} (bt);
      \end{tikzpicture}
    \]
    where $ \theta $ is given by,
    \[
      \theta (a) \bydef b \quad \quad
      \theta (a'), \theta ( a'') \bydef b' \quad \quad
      \theta (\alpha) \bydef \beta \quad \quad
      \theta (\alpha') \bydef \beta' \quad \quad
      \theta (\alpha'') \bydef \beta''
    \]
    is an object in
    $ \Set^{\C^{\op}} \downarrow F $
    
    According to Theorem
    \ref{thm:presheaf-slice-is-presheaf}, we can
    translate $ G \to F
    $ to a presheaf on the category of elements $
    \int^F \C $, which we depict as
    \[
      \begin{tikzpicture}
        \node [ob] (eb)   at (0,0) {$ (e, \beta) $};
        \node [ob] (eb')  at (0,3) {$ (e, \beta') $};
        \node [ob] (eb'') at (0,6) {$ (e, \beta'') $};
        \node [ob] (nc)   at (8,1.5) {$ (n,b) $};
        \node [ob] (nd)   at (8,4.5) {$ (n,b') $};
        \draw[cd]
          (eb)      to node[below,pos=0.2]{$ (s,\beta) $} (nc)
          (eb')     to node[below,pos=0.2]{$ (s,\beta) $} (nc)
          (eb''.15) to
            node[above,pos=0.2]{$ (s,\beta) $} (nd.165); 
        \draw[color=white,line width=0.3cm]
          (eb) -- (nd);
        \draw[cd]
          (eb)       to node[above,pos=0.2]{$(t,\beta)$} (nd)
          (eb')      to node[above,pos=0.2]{$ (t,\beta') $} (nd)
          (eb''.-15) to
            node[below,pos=0.2]{$ (t,\beta'') $} (nd.195);
        \draw [rounded corners]
          (-1,-1) rectangle (9,7);
      \end{tikzpicture}
    \]
    with the objects corresponding to the circles.
    The presheaf on this category that corresponds to
    $ G \to F $ is given by the $ \int^F \C
    $-shaped diagram in $ \Set $
    \[
      \begin{tikzpicture}
        \node [ob] (eb)   at (0,0) {$ \{\alpha\} $};
        \node [ob] (eb')  at (0,3) {$ \{\alpha'\} $};
        \node [ob] (eb'') at (0,6) {$ \{\alpha''\} $};
        \node [ob] (nc)   at (8,1.5) {$ \{a\} $};
        \node [ob] (nd)   at (8,4.5) {$ \{a',a''\} $};
        \draw[cd]
          (eb)      to node[below]{$ a $} (nc)
          (eb')     to node[above, pos =-.25]{$ a $} (nc)
          (eb''.15) to node[above]{$ a' $} (nd.165); 
        \draw[color=white,line width=0.3cm]
          (eb) edge (nd);
        \draw[cd]
          (eb)       to node[above, pos=0.25]{$ a'  $} (nd)
          (eb')      to node[above]{$ a'' $} (nd)
          (eb''.-15) to node[below]{$ a'' $} (nd.195);
        \draw [rounded corners]
          (-1,-1) rectangle (9,7);
      \end{tikzpicture}
    \]
    where the arrows are labeled to suggest the function
    they represent. The sets in this diagram are given by
    the fibers of $ \theta $. The edge and node functors
    determined by the arrows contain the information about
    where $ Gs $ and $ Gt $ send the elements in the fibers.
  \end{example}
  
  We have now finished the topos theory needed for
  this thesis.  The remaining discussion morally
  belongs to the theory of rewriting and, in
  particular, adhesive category theory. However,
  because all topoi are adhesive and we restrict
  our attention to topoi, we place the discussion
  in here.

  The following two lemmas are used.

  \begin{lemma}[{\cite[Lem.~4.2-3]{lack-sobo_adhesive-cats}}]
    \label{lem:adhesive-properties}
    In a topos, monomorphisms are stable under pushout.
    Also, pushouts along monomorphisms are pullbacks.
  \end{lemma}

  \begin{lemma}[{\cite[Lem.~6.3]{lack-sobo_adhesive-cats}}]
    \label{lem.vk dual}
    In a topos, consider a cube
    \[
      \begin{tikzpicture}
        \node (tb) at (2,2.75) {$\bullet$};
        \node (tl) at (0,2)    {$\bullet$};
        \node (tr) at (3,2)    {$\bullet$};
        \node (tf) at (1,1.25) {$\bullet$};
        \node (bb) at (2,0.75) {$\bullet$};
        \node (bl) at (0,0)    {$\bullet$};
        \node (br) at (3,0)    {$\bullet$};
        \node (bf) at (1,-.75) {$\bullet$};
        \draw [>->] (tb) edge (tr);
        \draw [cd]  (tb) edge (bb);
        \draw [>->] (tb) edge (tl);
        \draw [cd]  (tr) edge (br);
        \draw [>->] (tl) edge (tf);
        \draw [cd]  (tl) edge (bl);
        \draw [>->] (bb) edge (br);
        \draw [>->] (bb) edge (bl);
        \draw [>->] (br) edge (bf);
        \draw [>->] (bl) edge (bf);
        \draw [cd]
        (tf) edge[white,line width=4pt] (bf)
        (tf) edge                       (bf)
        (tr) edge[white,line width=4pt] (tf)
        (tr) edge                       (tf);
      \end{tikzpicture}
    \]
    whose top and bottom faces consist of only
    monomorphisms.  If the top face is a pullback and the
    front faces are pushouts, then the bottom face is a
    pullback if and only if the back faces are pushouts.
 \end{lemma}
 
 Two properties that are desirable for rewriting systems are
 local Church--Rosser and concurrency.  We do not use these
 results in this thesis, so we choose to not discuss them.
 Instead, we point the reader to the existing literature
 \cite{corradini-ehrig_algebraic-graph-grammars,lack-sobo_adhesive-cats}

\end{document}